\documentclass[12pt]{article}
\usepackage{amsmath,amssymb,amsthm,amscd,latexsym,}
 
\usepackage{graphicx}

\makeatletter
\renewcommand{\mod}[1]{\allowbreak \if@display \mkern 8mu \else
\mkern 5mu\fi {\operator@font mod}\,\,#1}
\makeatother

\newcommand{\bc}{\mathbb C}
\newcommand{\bn}{\mathbb N}
 \newcommand{\bq}{\mathbb Q}
\newcommand{\br}{\mathbb R}
\newcommand{\bz}{\mathbb Z}
\newcommand{\bff}{\mathbb F}

\newcommand{\aaa}{\mathbb A}
\newcommand{\ddd}{\mathbb D}
\newcommand{\eee}{\mathbb E}

\DeclareMathOperator{\Aut}{{\rm Aut}\,}

\DeclareMathOperator{\rk}{rk}

\newtheorem{theorem}{Theorem}
\newtheorem{proposition}[theorem]{Proposition}

\newtheorem{remark}[theorem]{Remark}

\newcommand\Hh{\mathcal H}
\newcommand\N{\mathcal N}
\newcommand\La{\mathcal L}

\newcommand\M{\mathcal M}

\newcommand\SSS{\mathfrak S}
\newcommand\AAA{\mathfrak A}
\newcommand\DDD{\mathfrak D}

\begin{document}
\title{K\"ahlerian K3 surfaces and Niemeier lattices}
\date{}
\author{Viacheslav V. Nikulin}
\maketitle

\begin{abstract} Using results (especially see
Remark 1.14.7) of our paper \cite{Nik1}, 1979,
we clarify relation between K\"ahlerian K3 surfaces and Niemeier
lattices. We want to emphasize that all twenty four Niemeier
lattices are important for K3 surfaces, not only the one which is
related to the Mathieu group.
\end{abstract}


\section{Introduction}
\label{introduction}
Studying of finite symplectic automorphism groups of K\"ahlerian K3
surfaces started in our papers \cite{Nik-1/2}
(announcement, 1976) and \cite{Nik0} (1979).
Some general theory of such groups was developed,
and Abelian such groups were classified (14 non-trivial Abelian
groups). Further, in \cite[Remark 1.14.7]{Nik1} (1979) we showed,
in particular,
that all finite symplectic automorphism groups of K3 surfaces
can be obtained from negative definite even unimodular lattices and
their automorphism groups using primitive embeddings of even negative
definite  lattices into such unimodular lattices (see Sect.
\ref{sec:sublattNiem} here for a review).
By our results about existence of primitive embeddings of even
lattices into
even unimodular lattices in \cite{Nik1}, for K3 surfaces it is enough
to use negative definite even unimodular lattices of the rank $24$.
They are Niemeier lattices.

Later, all finite symplectic automorphism groups of K\"ahlerian
K3 surface were classified as abstract groups by Mukai \cite{Muk}
(1988), (see also Xiao \cite{Xiao}, 1996). Kond\=o in \cite{Kon}
(1998)
showed that this classification can be obtained by using primitive
embeddings of lattices into Niemeier lattices
(see also the important appendix to this paper by Mukai).
This is similar to our considerations in \cite[Remark 1.14.7]{Nik1}.
Recently, Hashimoto \cite{Hash} applied similar ideas of using
Niemeier lattices to classify finite symplectic automorphism groups
of K\"alerian K3 surfaces similarly to our results about Abelian
such groups in \cite{Nik0}.

Thus, now it is clear that the methods of using negative
definite even unimodular lattices and Niemeier lattices
are very powerful.

In this paper, we use ideas and results of \cite[Remark 1.14.7]{Nik1}
to show that all twenty four Niemeier lattices are important for
K\"ahlerian K3 surfaces, their geometry and their symplectic
automorphism groups. (Usually, the Niemeier lattice  with the root
system $24A_1$  related to Mathieu group $M_{24}$ is used.)
From our point of view, all twenty four Niemeier lattices
are important for K3 surfaces.

We introduce and use a {\it marking} of a K\"ahlerian
K3 surface $X$ by a Niemeier lattice. Using this marking,
one can study,
in particular, finite symplectic automorphism groups
and non-singular rational curves on $X$.
In \cite{Nik0}, we demonstrated that to study finite symplectic
automorphism groups of K3 surfaces, it is important to work not with
algebraic K3 surfaces but with arbitrary K\"ahlerian K3 surfaces.
General K\"ahlerian K3 surfaces have negative definite Picard lattices
$S_X$ of the rank $\rk S_X\le 19$. By our results in \cite{Nik1}, there
exists a primitive embedding of $S_X$  into one of 24 Niemeier lattices.
One can study some arithmetic and geometry of $S_X$ and of $X$ using
such primitive embedding. It is called a {\it marking of $X$.}
All 24 Niemeier lattices are important for that.
For K\"ahlerian K3 surfaces with semi-negative definite and
hyperbolic Picard lattice $S_X$ we use some
modification of these markings.

In Sect. \ref{sec:primembbunim}, we remind to a reader our
results from \cite{Nik1} about primitive embeddings of lattices
into even unimodular lattices.

In Sect. \ref{sec:sublattNiem}, we remind to a reader classification
of Niemeier lattices and results of \cite[Remark 1.14.7]{Nik1} about
using of them (or any other even negative definite unimodular
lattices)
for studying of automorphism groups which act trivially on the
discriminant group for negative definite even
lattices by applying their primitive embeddings into even
unimodular lattices.

In Sect. \ref{sec:KalerNiemmark}, we remind to a reader definitions
and basic results related to K\"ahlerian K3 surfaces $X$. Using
these results, we introduce markings $S\subset N_i$ of $X$
by Niemeier lattices $N_i$.

In Sect. \ref{sec:applNiemmark}, we consider applications of
markings of K3 surfaces $X$ by Niemeier lattices to study finite
symplectic automorphism groups and non-singular rational
curves of $X$.

In Sect. \ref{sec:examples}, we consider examples of
markings of K\"ahlerian K3 surfaces by concrete Niemeier lattices
$N_i$, $1\le i\le 23$, and their applications.
In particular, we show that
for each of Niemeier lattices $N_i$,
$i=1,\,2,\,3$, $5$---$9$, $11$---$23$,
there exists a K\"ahlerian K3 surface $X$ such that $X$ can be marked
by the Niemeier lattice $N_i$ only. See Theorem  \ref{onlymarkings}.
We believe that similar results are valid for remaining
Niemeier lattices $N_4$, $N_{10}$, $N_{24}$. {\it All Niemeier lattices are important
for marking of K\"ahlerian K3 surfaces.}

\section{Existence of a primitive embedding of an even lattice
into even unimodular lattices, according to \cite{Nik1}}
\label{sec:primembbunim}

In this paper, we use notations, definitions and results of
\cite{Nik1} about lattices
(that is non-degenerate integral (over $\bz$) symmetric bilinear
forms). In
particular, $\oplus$ denotes the orthogonal sum of lattices,
quadratic forms.
For a prime $p$, we denote by $\bz_p$ the ring of $p$-adic integers,
and by $\bz_p^\ast$ its group of invertible elements.

Let $S$ be a lattice. Let $A_S=S^\ast/S$ be its discriminant group,
and $q_S$ its discriminant
quadratic form on $A_S$ where we assume that the lattice $S$ is
even: that is
$x^2$ is even for any $x\in S$. We denote by $l(A_S)$ the minimal
number
of generators of the finite Abelian group $A_S$, and by $|A_S|$
its order.
For a prime $p$, we denote by ${q_{S}}_p=q_{S\otimes \bz_p}$ the
$p$-component of $q_S$ (equivalently,
the discriminant quadratic form of the $p$-adic lattice
$S\otimes \bz_p$). A quadratic form on a group of order $2$
is denoted by  $q_\theta^{(2)}(2)$. A $p$-adic lattice $K({q_{S}}_p)$
or the rank $l({A_S}_p)$ with the discriminant quadratic form
${q_{S}}_p$ is denoted by $K({q_{S}}_p)$. It is unique, up to
isomorphisms,
for $p\not=2$, and for $p=2$, if
${q_{S}}_2\not\cong q_\theta^{(2)}(2)\oplus
q^\prime$. We have the following result where an embedding
$S\subset L$
of lattices is called {\it primitive} if $L/S$ has no torsion.

\begin{theorem} (Theorem 1.12.2 in \cite{Nik1}).

Let $S$ be an even lattice of the signature $(t_{(+)},t_{(-)})$, and
$l_{(+)}$, $l_{(-)}$ are integers.

Then, there exists a primitive embedding of $S$ into one of
even unimodular lattices
of the signature $(l_{(+)},\,l_{(-)})$ if and only if the following
conditions satisfy:

(1) $l_{(+)}-l_{(-)}\equiv 0\mod 8$;

(2) $l_{(+)}-t_{(+)}\ge 0$, $l_{(-)}-t_{(-)}\ge 0$,
$l_{(+)}+l_{(-)}-t_{(+)}-t_{(-)}\ge l(A_S)$;

(3) $(-1)^{l_{(+)}-t_{(+)}}|A_S|\equiv \det {K(q_{S_p})}\mod ({\bz_p}^\ast)^2$
for each odd prime $p$ such that $l_{(+)}+l_{(-)}-t_{(+)}-t_{(-)}=l(A_{S_p})$;

(4) $|A_S|\equiv \pm \det {K(q_{S_2})}\mod ({\bz_2}^\ast)^2$, if
$l_{(+)}+l_{(-)}-t_{(+)}-t_{(-)}= l(A_{S_2})$ and
${q_{S}}_2\not\cong q_\theta^{(2)}(2)\oplus q^\prime$.
\label{th:primembb1}
\end{theorem}

Remark that if the last inequality in (2) is strict then one does not need the
conditions (3) and (4). If
${q_{S}}_2\cong q_\theta^{(2)}(2)\oplus q^\prime$, then one does not need
the condition (4).

\section{Niemeier lattices and their primitive sublattices}
\label{sec:sublattNiem}

Further, negative definite even unimodular lattices $N$
of the rank $24$ are called {\it Niemeier lattices.}
They were classified by
Niemeier \cite{Nie}. See also \cite[Ch. 16, 18]{CS}. All elements with
square $(-2)$ of these lattices define root systems $\Delta(N)$
and generate root sublattices $[\Delta(N)]=N^{(2)}\subset N$
which are orthogonal sums of the root lattices
$A_n$, $n\ge 1$, $D_m$, $m\ge 4$, or $E_k$, $k=6,7,8$ with the
corresponding root systems $\aaa_n$, $\ddd_m$, $\eee_k$. We
denote by the same letters their Dynkin diagrams.
We fix standard bases by simple roots
of these root systems and lattices according to \cite{Bou}
(see Figure \ref{fig:dyndiagr}). If there are several
components, we put an additional second index which numerates
components.

By Niemeier \cite{Nie}, there are $24$ Niemeier lattices $N$,
up to isomorphisms, and they are characterized by there root sublattices.
Here is the list where $\oplus$ denotes the orthogonal sum of lattices:
$$
N^{(2)}=[\Delta(N)]=
$$
$$
(1)\ D_{24},\ (2)\ D_{16}\oplus E_8,\ (3)\ 3E_8,\ (4)\ A_{24},\
(5)\ 2D_{12},\ (6)\ A_{17}\oplus E_7, \ (7)\ D_{10}\oplus 2E_7,
$$
$$
(8)\ A_{15}\oplus D_9,\
(9)\ 3D_8,\
(10)\ 2A_{12},\ (11)\ A_{11}\oplus D_7\oplus E_6,\ (12)\ 4E_6,\
(13)\ 2A_9\oplus D_6,\
$$
$$
(14)\ 4D_6,\
(15)\ 3A_8,\ (16)\ 2A_7\oplus 2D_5,\ (17)\ 4A_6,\ (18)\ 4A_5\oplus D_4,\
(19)\ 6D_4,
$$
$$
(20)\ 6A_4,\
(21)\ 8A_3,\ (22)\ 12A_2,\ (23)\ 24A_1
$$
give $23$ Niemeier lattices $N_i$, where the number $i$ is shown in
brackets above. The last one, the {\it Leech lattice}\ $(24)$
with $N^{(2)}=\{0\}$ has no roots. Further, we also denote by $N(R)$
the Niemeier lattice with the root system $R$.

We recall that a basis $P(N)$ of the root lattice $[\Delta(N)]$
by simple roots is defined up to the reflection group $W(N)$ which is
generated by reflections $s_\delta:x\to x+(x\cdot \delta)\delta$, $x\in N$,
in roots $\delta\in \Delta (N)$. We denote by
$\Gamma(P(N))$ its Dynkin diagram.
Let $A(N)\subset O(N)$ be the group of symmetries
of the fixed basis $P(N)$. Thus, $g\in A(N)$ if
and only if $g(P(N))=P(N)$. Then we have the semi-direct product
$$
O(N)=A(N)\ltimes W(N).
$$
for the automorphism group of $N$. Equivalently, $P(N)$ is
equivalent to a choice of the fundamental chamber for $W(N)$, and
$A(N)$ is the group of symmetries of the fundamental chamber.

We {\it denote by $\N$ the disjoint union of all $24$ Niemeier
lattices $N_i$, $i=1,2,\dots, 24$ with their zeros identified.} Thus,
naturally, $\N_i=N_i\subset \N$ denotes the corresponding Niemeier
sublattice of this set which is called a {\it Niemeier component}
of $\N$. If $K$ is a lattice, then an embedding $K\subset \N$ means an
embedding of $K$ to one of $\N_i\subset \N$ as a sublattice. If $K\not=\{0\}$,
then $\N_i$ is defined uniquely, and it is called the {\it component} of
the embedding $K\subset \N$. The
embedding is called {\it primitive} if $\N_i/K$ has no torsion.

Further, we fix bases $P(\N_i)$ for $\Delta(\N_i)$ for all $24$
components of $\N$. Naturally, the direct product
$$
A(\N)=\prod_{i=1}^{i=24}A(\N_i)
$$
acts on $\N$. By definition, $A(\N_i)$ acts  as $A(N_i)$ on
the component $\N_i=N_i$, and $A(\N_i)$ is identity on all other components
$\N_j$, $j\not=i$.

Since Niemeier lattices
have the signature $(0,24)$, from Theorem \ref{th:primembb1} we obtain

\begin{theorem} (Corollary of Theorem 1.12.2 in \cite{Nik1}).
An even lattice $S$ has a primitive embedding into $\N$ (equivalently,
to one of Niemeier lattices) if and only if the following
conditions satisfy:

(2) $S$ is negative definite and
$\rk S+l(A_S)\le 24$;

(3) $|A_S|\equiv \det {K(q_{S_p})}\mod ({\bz_p}^\ast)^2$
for each odd prime $p$ such that \newline
$\rk S+l(A_{S_p})=24$;

(4) $|A_S|\equiv \pm \det {K(q_{S_2})}\mod ({\bz_2}^\ast)^2$, if
$\rk S+l(A_{S_2})=24$ and \newline
${q_{S}}_2\not\cong q_\theta^{(2)}(2)\oplus q^\prime$.
\label{th:primembb2}
\end{theorem}

Theorem \ref{th:primembb2} gives simple conditions on $S$
when it can be considered as a primitive sublattice of one
of Niemeier lattices (equivalently, of $\N$). We can use it
to get an important information about $S$ and
all primitive sublattices of Niemeier lattices. Later, we shall
apply these results to Picard lattices of K\"ahlerian K3
surfaces.

For a negative definite lattice $S$, we denote by $H(S)$
the kernel of the natural homomorphism
$$
\pi: O(S)\to O(q_S).
$$
Equivalently, an automorphism $\phi\in O(S)$ of the lattice $S$
belongs to $H(S)$ if and only if $\phi$ gives identity on the
discriminant group $S^\ast/S$.

It was observed in \cite[Remark 1.14.7]{Nik1} that
$\pi$ has a non-trivial kernel if and only if
$S$ contains some exceptional sublattices. These exceptional
sublattices can be found from negative definite
even unimodular lattices and their
automorphism groups. Remark that for an unimodular lattice
$L$, the group $H(L)=O(L)$.
To study $H(S)$ for negative definite even lattices
$S$, it was suggested in  \cite[Remark 1.14.7]{Nik1}
to use primitive embeddings of
$S$ into negative definite even unimodular lattices.
For example, for lattices $S$ satisfying
Theorem \ref{th:primembb2}, one can use primitive embeddings
into Niemeier lattices (equivalently, into $\N$). Below,
we summarize these results of \cite[Remark 1.14.7]{Nik1}.

Let $S$ be an even negative definite lattice.
Let $\Delta (S)$ be the set of roots
$\delta \in S$ with square $\delta^2=-2$. It
is easy to see that the reflection $s_\delta$ belongs to
$H(S)$. Thus, $[\delta]=\langle -2 \rangle$ is
the first example of an exceptional sublattice.
The Weyl group $W(S)$ generated by
reflections in all elements of
$\Delta (S)$  is the normal subgroup of $H(S)$. Let us choose
a basis $P(S)$ of the root system $\Delta(S)$ (equivalently, a
fundamental chamber of the Weyl group $W(S)$). Let
$$
A(S)=\{\phi \in H(S)\ |\ \phi(P(S))=P(S)\}
$$
(equivalently, $A(S)\subset H(S)$ is the symmetry subgroup of
the fundamental chamber).
Then $H(S)=A(S)\ltimes W(S)$ is the semi-direct product.
Let
$$
\La (S)=S_{A(S)}=(S^{A(S)})^\perp_S
$$
be the orthogonal complement to the fixed part $S^{A(S)}$ of the action of
$A(S)$ in $S$. Obviously, the sublattice $\La(S)=S_{A(S)}\subset S$
is defined uniquely up to the action of $W(S)$.
It is called {\it the coinvariant
sublattice of $S$ for $A(S)$.}

It was shown in Proposition 1.14.8 of \cite[Remark 1.14.7]{Nik1}
that the primitive sublattice $\La(S)$ satisfies the following
properties: $\La(S)^{A(S)}=\{0\}$ (this is obvious),
$\La (S)$ has no roots with square $(-2)$, $\La(\La(S))=\La(S)$, and
$A(\La(S))=A(S)|\La(S)$. Thus, $H(\La(S))=A(\La(S))$, $A(\La(S))$
defines $A(S)$:  one should continue $A(\La(S))$ identically to the
orthogonal complement of $\La(S)$ in $S$.

The sublattice $\La(S)=S_{A(S)}$ was called in
\cite[Remark 1.14.7]{Nik1} as
{\it the Leech type sublattice of $S$}. Thus, the group
$H(S)$ and its natural subgroups $W(S)$ and $A(S)$ are
completely defined by the basis $P(S)$ and the Leech type sublattice
$\La(S)$. Really, $W(S)$ is generated by reflections in $P(S)$, and
$A(S)=A(\La(S))$ if one continues $A(\La(S))$ identically to the
orthogonal complement $\La(S)^\perp_S$.

Now let us assume that $S\subset N_i$ is a primitive sublattice
of one of Niemeier lattices (or of any other
even unimodular lattice). Then, replacing
$S$ by $w(S)$ where $w\in W(N_i)$, one can assume that $P(S)$ is a subset
of $P(N_i)$. Later, we always assume that the primitive embedding
$S\subset N_i$ satisfies this condition:
\begin{equation}
P(S)=S\cap P(N_i)\ is\  a\  basis\  of\  \Delta(S).
\label{P(S)1}
\end{equation}
Then, continuing $A(S)$ identically to ${S^\perp}_{N_i}$ (it is possible
because $A(S)$ gives identity on $S^\ast/S$), one
obtains a subgroup of $A(N_i)$. It follows that
\begin{equation}
A(S)=\{\phi|S\ |\ \phi\in A(N_i)\ and\ \phi|S^\perp_{N_i}\ is\  identity\}.
\label{A(S)1}
\end{equation}
It follows that $A(S)\subset A(N_i)$ and $P(S)\subset P(N_i)$ satisfy
the obvious important condition:
$$
P(S)\ \ is\  invariant\  with\  respect\  to\  A(S).
$$
Since for Niemeier lattices $N_i$ the sets $P(N_i)$ and their
graphs $\Gamma(P(N_i))$, and
the groups $A(N_i)$ are known, we obtain an effective tool to
construct primitive sublattices $S\subset N_i$ with different graphs
$\Gamma(P(S))$ and groups $A(S)$.

\medskip

{\bf Example.} The Niemeier lattices $N=N_1=N(D_{24}),
\ N_2=N(D_{16}\oplus E_8)$
have the trivial group $A(N)$
(see \cite[Ch. 16]{CS}).
It follows that any primitive sublattice $S\subset N$ of such lattices has
the trivial group $A(S)$ and the trivial Leech type sublattice $\La(S)$.

Let us assume that graph $\Gamma(S^{(2)})$ has a subgraph $\ddd_n$ where
$n\ge 13$. By the classification of Niemeier lattices, then
$S$ can have a primitive embedding only to the lattices $N(D_{24})$
and $N(D_{16}\oplus E_8)$,  and the group $A(S)$ is trivial again
if $S$ satisfies Theorem \ref{th:primembb2}.

\medskip

We shall consider other examples in section \ref{sec:applNiemmark}.

\section{Additional markings of K\"ahlerian K3 surfaces\\
by Niemeier lattices}
\label{sec:KalerNiemmark}

We consider K\"ahlerian K3 surfaces. We recall that they are
K\"ahlerian compact complex surfaces $X$ such that $X$ is simply-connected
and $X$ has a holomorphic 2-dimensional differential form
$\omega_X\in H^{2,0}(X)$ such that $\omega_X$ has no zeros.
One can consider $\omega_X$ as a complex volume form.
See \cite[Chapter 9]{S} and \cite{BR} about
K\"ahlerian K3 surfaces.

It is known that the cohomology lattice $H^2(X,\bz)$ with the
intersection pairing is an even unimodular lattice of the signature
$(3,19)$. It has no torsion. All even unimodular lattices of the
signature $(3,19)$ are isomorphic (the same is valid for all indefinite
even unimodular lattices of the same signature).
$X$ has the Hodge decomposition
$$
H^2(X,\bc)=H^2(X,\bz)\otimes \bc=H^{2,0}(X)+H^{1,1}(X)+H^{0,2}(X)
$$
where $H^{2,0}(X)=\bc \omega_X$ is $1$-dimensional,
$H^{0,2}(X)=\overline{H^{2,0}(X)}$ and
$H^{1,1}(X)=\overline {H^{1,1}(X)}$ is 20-dimensional.
It follows that the {\it Picard lattice}
$$
S_X=H^2(X,\bz)\cap H^{1,1}(X)
$$
of $X$ is a sublattice of the hyperbolic
real subspace $H^{1,1}_\br (X)\subset H^{1,1}(X)$
of the signature $(1,19)$. It follows that
the Picard lattice $S_X$ is a primitive sublattice
of $H^2(X,\bz)$, and it satisfies to one of the following conditions:

\medskip

(a) $S_X$ is negative definite of the rank $0\le \rk S_X\le 19$;

(b) $S_X$ is semi-negative definite with $1$-dimensional kernel,
and $1\le \rk S_X\le 19$;

(c) $S_X$ is hyperbolic (that is it has the signature $(1,\rk S_X-1)$),
and $1\le \rk S_X\le 20$ (this is the case when $X$ is algebraic).

\medskip

Further, we denote by
$\rho(X)=\rk S_X$ the Picard  number of $X$.

Let $L_{K3}$ be an abstract even unimodular lattice of the
signature $(3,19)$. Thus, the  Picard lattice $S_X$ of a
K\"ahlerian K3 surface $X$ has a primitive
embedding $S_X\subset L_{K3}$, and it satisfies to one of conditions
(a), (b) or (c). By the epimorphicity
of the period  map for K3 surfaces \cite{Kul}, \cite{Siu}, \cite{Tod},
Picard lattices $S_X$ of K3 surfaces are characterized by these
properties. An even lattice which is either negative definite, or
semi-negative definite with one-dimensional kernel, or hyperbolic is
isomorphic to the Picard lattice of one of K\"ahlerian K3 surfaces
if and only if it has a primitive embedding into the lattice $L_{K3}$.

By Theorem \ref{th:primembb1} (from \cite{Nik1}), we obtain the
complete description of Picard lattices of K3 surfaces.

\begin{theorem} (Corollary of Theorem 1.12.2 in \cite{Nik1}).

An even negative definite lattice $M$ is isomorphic to the Picard lattice
$S_X$ of one of K\"ahlerian K3 surfaces $X$
(equivalently, $M$ has a primitive embedding into $L_{K3}$)
if and only if

(2) $\rk M\le 19$ and $\rk M+l(A_M)\le 22$;

(3)$-|A_M|\equiv \det {K(q_{M_p})}\mod ({\bz_p}^\ast)^2$
for each odd prime $p$ such that \newline
$\rk M+l(A_{M_p})=22$;

(4) $|A_M|\equiv \pm \det {K(q_{M_2})}\mod ({\bz_2}^\ast)^2$, if
$\rk M+l(A_{M_2})=22$ and
\newline
${q_{M_2}}\not\cong q_\theta^{(2)}(2)\oplus q^\prime$.
\label{th:primembb3}
\end{theorem}

\begin{theorem} (Corollary of Theorem 1.12.2 in \cite{Nik1}).

An even semi-negative definite lattice $M$ with one-dimensional
kernel $Ker\ M$ is isomorphic to the Picard lattice
$S_X$ of one of K\"ahlerian K3 surfaces $X$
(equivalently, $M$ has a primitive embedding into $L_{K3}$),
if and only if for $\widetilde{M}=M/{Ker\ M}$ one has

(2) $\rk \widetilde{M}\le 18$ and
$\rk \widetilde{M}+l(A_{\widetilde{M}})\le 20$;

(3)$|A_{\widetilde{M}}|\equiv \det {K(q_{\widetilde{M}_p})}
\mod ({\bz_p}^\ast)^2$
for each odd prime $p$ such that
\newline
$\rk \widetilde{M}+l(A_{\widetilde{M}_p})=20$;

(4) $|A_{\widetilde{M}}|\equiv \pm \det {K(q_{\widetilde{M}_2})}
\mod ({\bz_2}^\ast)^2$,
if $\rk \widetilde{M}+l(A_{\widetilde{M}_2})=20$ and
\newline
${q_{\widetilde{M}_2}}\not\cong q_\theta^{(2)}(2)\oplus q^\prime$.
\label{th:primembb4}
\end{theorem}

\begin{theorem} (Corollary of Theorem 1.12.2 in \cite{Nik1}).

An even hyperbolic lattice $M$ (that is $M$ has the signature
$(1,\rk M-1)$) is isomorphic to the Picard lattice
$S_X$ of one of algebraic K3 surfaces $X$ over $\bc$
(equivalently, $M$ has a primitive embedding into $L_{K3}$)
if and only if

(2) $\rk M\le 20$ and $\rk M+l(A_M)\le 22$;

(3)$|A_M|\equiv \det {K(q_{M_p})}\mod ({\bz_p}^\ast)^2$
for each odd prime $p$ such that \newline
$\rk M+l(A_{M_p})=22$;

(4) $|A_M|\equiv \pm \det {K(q_{M_2})}\mod ({\bz_2}^\ast)^2$, if
$\rk M+l(A_{M_2})=22$ and
\newline
${q_{M_2}}\not\cong q_\theta^{(2)}(2)\oplus q^\prime$.
\label{th:primembb5}
\end{theorem}

\medskip

We recall the definition of the period domain $\widetilde{\Omega}$
for K\"ahlerian K3 surfaces (see \cite{BR}).
We fix an even unimodular lattice $L_{K3}$ of the signature $(3,19)$.
The $\widetilde{\Omega}$ consists of all triplets
\begin{equation}
(H^{2,0},\, V^+,\,P)
\label{pointofperiods}
\end{equation}
which we describe below.

Here $H^{2,0}\subset L_{K3}\otimes \bc$ is a one-dimensional
complex linear subspace satisfying the condition
\begin{equation}
\omega\cdot \omega=0,\ \ \omega\cdot \overline{\omega}>0
\label{condomega}
\end{equation}
for any $0\not=\omega \in H^{2,0}$. (We extend the symmetric
bilinear form of $L_{K3}$ to the symmetric $\bc$-bilinear form
on $L_{K3}\otimes \bc$.)
Such $H^{2,0}\subset L_{K3}\otimes \bc$ define a $20$-dimensional
complex homogeneous manifold which is denotes by $\Omega$.

For a fixed $H^{2,0}\subset L_{K3}\otimes \bc$
from $\Omega$, we denote
by $H^{1,1}_\br$ the orthogonal complement to $H^{2,0}$
in $L_{K3}\otimes \br$ which is a hyperbolic form of the signature
$(1,19)$. It contains the light cone
\begin{equation}
V=\{x\in H^{1,1}_\br\ |\ x^2>0\}.
\label{lightconeV}
\end{equation}
The $V^+$ denotes one of two halfs (that is connected components)
of this cone $V$. It defines the hyperbolic space ${\Hh}^+=V^+/\br^+$
which is the projectivization of the half-cone $V^+$.

Let $H^{1,1}_\bz=H^{1,1}_{\br}\cap L_{K3}$. The $H^{1,1}_\bz$ is a lattice
which is either negative definite, or semi-negative definite with
one-dimensional kernel, or hyperbolic. Let $\Delta(H^{1,1}_\bz)$ be the
set of all elements with square $(-2)$ of this lattice. The group
$W(H^{1,1}_\bz)$ generated by reflections in all elements from
$\Delta(H^{1,1}_\bz)$ is the discrete reflection group in
$V^+$ and $\Hh^+$. The set $P$ is the set of perpendicular vectors
from $\Delta(H^{1,1}_\bz)$ to one of fundamental chambers $\M$ of
this reflection group which are directed outwards of this chambers.
Thus
\begin{equation}
\M=\{x\in V^+\ |\ x\cdot P \ge 0\},
\label{setP}
\end{equation}
and each codimension one face of $\M$ is
perpendicular to exactly one $\delta\in P$, and vice a versa each
$\delta\in P$ is perpendicular to a codimension one face of $\M$.
Shortly, $P=P(\M)$.

The set of the triplets \eqref{pointofperiods} is
denoted by $\widetilde{\Omega}$. It is a non-Hausdorff
$20$-dimensional complex manifold which gives an
\'etale covering of $\Omega$.

\medskip

We recall that to each K\"ahlerian K3 surface $X$ one can correspond
the canonical triplet
\begin{equation}
(H^{2,0}(X),V^+(X), P(X)).
\label{periodsofX}
\end{equation} See \cite{BR} for
details.

Here $H^{2,0}(X)
\subset H^2(X,\bz)\otimes \bc$ was introduced and considered above.

Here $V^+(X)$
is the half cone containing the K\"ahler class of the light cone
\begin{equation}
V(X)=\{x\in H^{1,1}_\br (X)\ |\ x^2>0\}.
\label{lightconeX}
\end{equation}

Here
\begin{equation}
P(X)\subset S_X
\label{PX}
\end{equation}
is the set of classes of all non-singular rational curves on $X$.
All of them have the square $(-2)$. Any exceptional curve on $X$
(that is an irreducible complex curve $C\subset X$ with $C^2<0$)
is one of them. They define the nef cone
$$
NEF(X)=\{x\in \overline{V^+(X)}\ |\ x\cdot P(X)\ge 0\}
$$
which is the fundamental chamber for the reflection
group $W(S_X)$, and $P(X)$ is the set of perpendicular vectors to
codimension one faces of $NEF(X)$.

We recall that a marking of a K\"ahlerian K3 surface $X$ is
an isomorphism
\begin{equation}
\alpha:H^2(X,\bz)\cong L_{K3}
\label{markX1}
\end{equation}
of lattices where the lattice $L_{K3}$ was introduced above.
A pair $(X,\alpha)$ is called a {\it marked K\"ahlerian K3 surface.}

By taking
$$
\alpha (H^{2,0}(X),V^+(X),P(X))=
$$
\begin{equation}
(\,(\alpha \otimes \bc) (H^{2,0}(X)),\ (\alpha\otimes \br)(V^+(X)),\
\alpha(P(X))\,)\in
\widetilde{\Omega}
\label{periodmapforK3}
\end{equation}
we obtain the {\it period map $\alpha$}
from the moduli space of marked K\"ahlerian K3 surfaces to the
period domain $\widetilde{\Omega}$ of marked K\"ahlerian K3 surface.

By the {\it Global Torelli Theorem for K\"ahlerian K3 surfaces}
(see \cite{PS} (for algebraic case) and \cite{BR}), and
{\it the epimorphicity of the period map for K\"ahlerian K3 surfaces}
(see \cite{Kul} (for algebraic case), and \cite{Siu}, \cite{Tod}),
the period map $\alpha$ is the isomorphism of the complex spaces.

\medskip

We want to introduce an additional marking of K\"ahlerian K3 surfaces
by Niemeier lattices $N_i$ (or by $\N$).

We introduce the corresponding period domain $\widetilde{\Omega}_\N$
which consists of all quadruples
\begin{equation}
(H^{2,0},\, V^+,\,P,\,\tau:S\subset N_i)
\label{pointofperiodsN}
\end{equation}
where the triplet $(H^{2,0},\, V^+,\,P)\in \widetilde{\Omega}$
is as above.

Now, we shall describe $\tau:S\subset N_i$. Here
{\it $S\subset H^{1,1}_\bz$ is the maximal negative definite
sublattice of $H^{1,1}_\bz$.} Thus, $S=H^{1,1}_\bz$ if $H^{1,1}_\bz$
is negative definite; $H^{1,1}_\bz=S\oplus \bz c$ if $H^{1,1}_\bz$
has a one-dimensional kernel $\bz c$;
and $S\subset H^{1,1}_\bz$ is a primitive
negative definite sublattice of $H^{1,1}_\bz$ such that
$(S)^\perp_{H^{1,1}_\bz}=\bz h$ where $h^2>0$, if $H^{1,1}_\bz$
is hyperbolic. For the parabolic case, when $S$
has a one-dimensional kernel $\bz c$,
we additionally require that $P\cap S=P(S)$. For the hyperbolic case,
we additionally require that {\it $h$ is nef, that is $h\in V^+$
and $h\cdot P\ge 0$.} Equivalently, $h$ belongs to the fundamental
chamber for $W(H^{1,1}_\br)$ defined by the $V^+$ and $P$.
{\it The $\tau:S\subset N_i$ is a primitive embedding of $S$ into one of
Niemeier lattices $N_i$ (equivalently, to $\N$) such that
$\tau(P\cap S)\subset P(N_i)$.} Thus, $\tau$ can be considered
as an additional marking by Niemeier lattices. Since $S$ is negative
definite and has a primitive embedding to $L_{K3}$, it satisfies
the Theorem \ref{th:primembb3}. In particular, $\rk S\le 19$ and
$\rk S+ l(A_S)\le 22$. Then $\rk S+l(A_S)<24$. Thus, $S$
satisfies Theorem \ref{th:primembb2} and $S$ has a primitive
embedding to one of $24$ Niemeier lattices $N_i$. The same proof shows
that $S\oplus A_1$ satisfies Theorem \ref{th:primembb2}. It follows that $S$
has a primitive embedding into one of $23$ Niemeier lattices with
non-empty set or roots. {\it This is the important trick due to
Kond\=o \cite{Kon} which permits to avoid the difficult Leech lattice.}

It follows that
$$
\tau:\widetilde{\Omega}_\N\to \widetilde{\Omega}
$$
is a natural ``covering'' of $\widetilde{\Omega}$ which is onto.
Unfortunately, we cannot claim that $\widetilde{\Omega}_\N$ is
a manifold. We can only claim that it is a topological space with
open subsets which are pre-images of open subsets from $\widetilde{\Omega}$.
We can only claim that $\tau$ is one to one
over general points of $\widetilde{\Omega}$ where $H^{1,1}_\bz=\{0\}$ since
we  identify in $\N$ zeros of all Niemeier lattices
$N_i$, $i=1,2,\dots,24$.
This points of $\widetilde{\Omega}$ give a complement to infinite number of
divisors of $\widetilde{\Omega}$. Thus, we can only claim that
$\widetilde{\Omega}_\N$ is similar to a manifold
over these points of $\widetilde{\Omega}$. Moreover,
$\tau$ has finite fibres over points of $\widetilde{\Omega}$ where
$H^{1,1}_\bz$
is negative definite. Thus, $\widetilde{\Omega}_\N$ is very similar to
a manifold over these points. Over other points of $\widetilde{\Omega}$ the
map $\tau$ has countable fibres, in general.

\medskip

Similarly, we introduce additional marking for
K\"ahlerian K3 surfaces $X$. It is
\begin{equation}
\tau:S\subset \N.
\label{NiemarkX}
\end{equation}
 Here {\it $S\subset S_X$ is a maximal negative
definite sublattice.} That is $S=S_X$ if $S_X$ is negative
definite. $S_X=S\oplus \bz c$ if $S_X$ is
semi-negative definite with one-dimensional kernel
generated by the class $c$ of an elliptic curve $C$ on $X$,
and $S\cap P(X)=P(S)$.
If $X$ is algebraic, then
$S=h^\perp_{S_X}$ where $h\in S_X$ is primitive, $h^2>0$, and
{\it $h$ is nef,} that is
$h\cdot D\ge 0$ for every effective divisor $D$ on $X$.
Here {\it $\tau:S\subset \N$ is
a primitive embedding of $S$ into one of $24$ Niemeier
lattices $N_i$ such that $P(X)\cap S=P(S)\subset P(N_i)$.}

\medskip

The standard marking $\alpha :H^2(X,\bz)\cong L_{K3}$ of
a K\"ahlerian K3 surface $X$ with additional
marking $\tau:S\subset \N$ by Niemeier lattices gives its periods
$\alpha(X,\tau)\in \widetilde{\Omega}_\M$. They are
$$
\alpha (H^{2,0}(X),V^+(X),P(X),\tau:S\subset \N)=
$$
\begin{equation}
(\,(\alpha \otimes \bc) (H^{2,0}(X)),\ (\alpha\otimes \br)(V^+(X)),\
\alpha(P(X)),\ \tau\alpha^{-1}:\alpha(S)\subset \N)\in
\widetilde{\Omega}_\N\ .
\label{periodmapforK3N}
\end{equation}

By the Global Torelli Theorem and the Epimorphicity of Torelli map for
K\"ahlerian K3 surfaces (see references above), we obtain isomorphism of
moduli spaces of marked K\"ahlerian K3 surfaces $X$
with additional marking by Niemeier lattices and their
periods space $\widetilde{\Omega}_\N$.

\medskip

\begin{remark}
{\rm
By these definitions and considerations, to construct K\"ahlerian K3
surfaces $X$ with a marking $\tau:S\subset N_i$ by a Niemeier lattice
$N_i$, one has to follow the following procedure.

(a) Check that $S$ satisfies Theorem \ref{th:primembb3}
(equivalently, there exists a primitive embedding $S\subset L_{K3}$).

(b) Choose a primitive embedding $S\subset L_{K3}$.

(c1) To construct $X$ with negative definite $S_X$, choose
$H^{2,0}\subset (S)^\perp_{L_{K3}}\otimes \bc$
which is general enough to have $H^{1,1}_\bz=S$. Choose
$V^+\subset H^{1,1}_\br$, and take $P=P(S)=P(N_i)\cap S$. Then
$(H^{2,0},\,V^+,\,P,\,\tau:S\subset N_i)$ gives periods of a marked
K\"ahlerian K3 surface $(X,\alpha)$ with $S_X=\alpha^{-1}(S)$, and
marking $\tau\alpha:S_X=\alpha^{-1}(S)\subset N_i$ by the Niemeier
lattice $N_i$. The set $P(X)=\alpha^{-1}(P(N_i)\cap S)$.

(c2) To construct $X$ with semi-negative definite $S_X$,
check that $S\subset L_{K3}$ satisfies Theorem \ref{th:primembb4}.
Then there exists a primitive non-zero isotropic
$c\in (S)^\perp_{L_{K3}}$ such that $S\oplus \bz c\subset L_{K3}$
is a primitive
sublattice with the kernel $\bz c$.
Choose
$H^{2,0}\subset (S\oplus \bz c)^\perp_{L_{K3}}\otimes \bc$
which is general enough to have $H^{1,1}_\bz=S\oplus \bz c$. Choose
$V^+\subset H^{1,1}_\br$ such that $c\in \overline{V^+}$.
Take
$P=P(S)=P(N_i)\cap S$, and for each connected component
$P_i$, $i=1,2,\dots, k$,  of the Dynkin diagram of $P$,
take the maximal root $\delta_i$, and denote $p_i=c-\delta_i$. Then
$$
\widetilde{P}=P\cup \{p_1,\,p_2,\dots,\,p_k\}=P(S\oplus \bz c),
$$
and
$(H^{2,0},\,V^+,\,\widetilde{P},\,\tau:S\subset N_i)$
gives periods of a marked
K\"ahlerian K3 surface $(X,\alpha)$ with $S_X=\alpha^{-1}(S\oplus \bz c)$
and with marking $\tau\alpha:\alpha^{-1}(S)\subset N_i$ by the Niemeier
lattice $N_i$. The set $P(X)=\alpha^{-1}(\widetilde{P})$.

(c3) To construct $X$ with hyperbolic $S_X$ (thus, $X$ is algebraic),
take a primitive $h\in (S)^\perp_{L_{K3}}$ such that $h^2>0$ (they
exist obviously). Then the primitive sublattice
$\widetilde{S}=[S\oplus \bz h]_{pr}$ of $L_{K3}$ generated by $S$ and $h$
is hyperbolic. Choose
$H^{2,0}\subset (\widetilde{S})^\perp_{L_{K3}}\otimes \bc$
which is general enough to have $H^{1,1}_\bz=\widetilde{S}$.
Choose $V^+\subset H^{1,1}_\br$ such that $h\in V^+$.
Take
$P=P(S)=P(N_i)\cap S$, and take the fundamental chamber $\M$
for $W(\widetilde{S})$ such that $h\in \M$ and $P\subset P(\M)$
where $P(\M)$ is the set of perpendicular vectors with square $(-2)$
to codimension one faces of $\M$ directed outwards of $\M$.
Then $(H^{2,0},\,V^+,\,{P(\M)},\,\tau:S\subset N_i)$
gives periods of a marked algebraic
 K3 surface $(X,\alpha)$ with $S_X=\alpha^{-1}(\widetilde{S})$,
{\it nef} element $\alpha^{-1}(h)$, and $P(X)=\alpha^{-1}(P(\M))$,
and with the marking $\tau\alpha:\alpha^{-1}(S)\subset N_i$ by the
Niemeier lattice $N_i$. We have $P(X)=\alpha^{-1}(P(\M ))$,
and the linear system $|nh|$, $n>0$, contracts non-singular
rational curves which have
classes from $\alpha^{-1}(S\cap P(N_i))$, and only these curves.
}
\label{rem:Niemmarkconst}
\end{remark}

\section{Applications of marking by Niemeier lattices}
\label{sec:applNiemmark}

Let $X$ be a K\"ahlerian K3 surface $X$ with a
marking $\tau:S\subset N_i$ by a Niemeier lattice $N_i$. By
the conditions on $\tau$ and
our considerations in Sec. \ref{sec:sublattNiem} using
\cite[Remark 1.14.7]{Nik1}, we obtain the following result.

\begin{theorem} Let $X$ be a K\"ahlerian K3 surface with
marking $\tau:S\subset N_i$ by a Niemeier lattice $N_i$.

Then $P(X)\cap S=P(S)$ is the
basis for the root system $\Delta (S)$ and $P(S)=P(X)\cap S$
is a subset of $P(N_i)$. Thus, the Dynkin diagram $\Gamma (P(X)\cap S)$
gives the Dynkin diagram of $\Gamma(P(S))$, and it
is a subdiagram of the Dynkin diagram $\Gamma (P(N_i))$.

In particular, if $S_X$ is negative definite, then $S=S_X$,
and $P(X)=P(N_i)\cap S$
gives the set of classes of all non-singular rational curves on $X$.

If $X$ is algebraic, then $P(X)\cap S=P(S)=P(N_i)\cap S$ gives
the set of classes of all non-singular rational curves on $X$
which are contracted by the linear system $|nh|$ for $n>0$ where
$h$ is the primitive nef element in $S_X$ which generates
the orthogonal complement to $S$ in $S_X$.
\label{th:P(S)X}
\end{theorem}

This shows that for K\"ahlerian K3 surfaces $X$,
marking by Niemeier lattices describes the sets of
non-singular rational curves on $X$
such that their classes are contained in $S\subset S_X$.

\medskip

We recall that an automorphism $\phi$ of a K\"ahlerian K3 surface $X$
is called {\it symplectic} if $\phi$ preserves the holomorphic form
$\omega_X$, that is $\phi^\ast (\omega_X)=\omega_X$. This is equivalent
(see \cite{Nik0}) that
$\phi$ gives the identity on the transcendental part
\begin{equation}
H^2(X,\bz)/S_X.
\label{transcpart}
\end{equation}
It is known (see \cite{PS} and \cite{BR}; formally,
it follows from the Global Torelli Theorem for K3 surfaces) that the
kernel of the action of $\Aut X$ in $H^2(X,\bz)$ is trivial. A subgroup
of $\Aut X$ is called {\it symplectic} if all its elements are symplectic.
We denote by $(\Aut X)_0$ the group of all symplectic automorphisms of $X$.

For a marking $\tau:S\subset N_i$ of $X$ by a Niemeier lattice $N_i$,
we shall
consider the group of automorphisms
\begin{equation}
\Aut(X,S)_0=\{ f \in (\Aut X)_0\ | \ f(S)=S\ and\
f|S^\perp_{H^2(X,\bz)}\ is\
identity \}.
\label{defsympl}
\end{equation}
If $S_X$ is negative definite or hyperbolic, then, in \eqref{defsympl},
the additional condition that $f|S^\perp_{H^2(X,\bz)}$
is identity follows from other
conditions. In all cases, the definition of $\Aut(X,S)_0$ is equivalent to
\begin{equation}
\Aut(X,S)_0=\{ f \in \Aut X\ | \ f(S)=S\ and\
f|S^\perp_{H^2(X,\bz)}\ is\ identity \}.
\label{defsymp2}
\end{equation}

By \cite[Proposition 1.5.1]{Nik1}, $\phi \in O(S)$
can be extended to
$O(H^2(X,\bz))$ identically on $S^\perp_{H^2(X,\bz)}$
if and only if $\phi$
gives identity on $A_S=S^\ast/S$. By the Global Torelli
Theorem for K3 surfaces
(see \cite{PS} and \cite{BR}), then $\phi=f^\ast$ for some $f\in \Aut X$.
By \eqref{transcpart},
$f\in (\Aut X)_0$. Further we identify $\Aut(X,S)_0$ with
its action in $S$. Then,
by our considerations in Sect. \ref{sec:sublattNiem} which
follow from
\cite[Remark 1.14.7]{Nik1}, we obtain the following result.

\begin{theorem} Let $X$ be a K\"ahlerian K3 surface and
$S\subset N_i$ is its marking by a Niemeier lattice
$N_i$, $i=1,2,\dots,24$.

Let $\Aut (X,S)_0$ be the symplectic automorphism group of $X$
which consists of all automorphisms of $X$ which give
identity on $S^\perp_{H^2(X,\bz)}$ (all of them are symplectic).
Then the action of
$\Aut (X,S)_0$ in $S$
identifies it with the subgroup $A(S)\subset A(N_i)$
which is :
\begin{equation}
A(S)=\{\phi\in A(N_i)\ |\
\phi|S^\perp_{N_i}\ is\  identity\}.
\label{A(S)2}
\end{equation}
We have
$$
\Aut (X,S)_0=A(S)|S\,.
$$
The Leech type sublattice $\La(S)\subset S$
(or the coinvariant sublattice) is
\begin{equation}
\La(S)=S_{A(S)}=(S^{A(S)})^\perp_S=((N_i)^{A(S)})^\perp_{N_i}.
\label{Leech(S)2}
\end{equation}

Let $P(X)\cap S$ be the set of all classes of non-singular
rational curves of $X$ which are contained in $S$. Then
$P(X)\cap S$ is invariant with respect to $A(S)$.
\label{th:A(S)X}
\end{theorem}

\begin{remark} {\rm By these theorems, to construct
K\"ahlerian K3 surfaces
with all possible markings $S\subset N_i$
(further, we shall skip $\tau$) and with all possible
$P(X)\cap S$ and $\Aut (X,S)_0$, one can follow the following steps.

(a) For a Niemeier lattice $N_i$ take a subgroup $A\subset A(N_i)$
such that $\La =(N_i)_A=((N_i)^{A})^\perp_{N_i}$
satisfies Theorem \ref{th:primembb3}
(equivalently, there exists a primitive embedding $\La\subset L_{K3}$).
Further, we call such subgroups {\it KahK3 subgroups (that is K\"ahlerian
K3 surfaces subgroups)}. Let $Clos(A)\subset A(N_i)$ (we use notation
from \cite{Hash})  is the maximal
subgroup of $A(N_i)$ with the same coinvariant lattice
$\La =(N_i)_A=((N_i)^{A})^\perp_{N_i}$.
These and only these subgroups $Clos(A)$ correspond to the full
symplectic automorphism groups $\Aut(X,S)_0$ of K\"ahlerian K3 surfaces
which can be marked by $N_i$. The subgroup $A$ can be only a
proper subgroup $A\subset Clos(A)=\Aut(X,S)_0$.

We note that the list of all possible abstract such groups
$A$ for all Niemeier lattices together is known
(see \cite{Nik0} for Abelian groups $A$
and Mukai \cite{Muk}, Xiao \cite{Xiao}
and Kond\=o \cite{Kon} for arbitrary groups $A$).
The corresponding to $A$
Leech type lattices $\La$ (for all Niemeier lattices together)
are also known
(see \cite{Nik0} for Abelian groups
$A$ and Hashimoto \cite{Hash} for arbitrary
groups $A$). Almost for all abstract groups $A$, the Leech type
lattices $\La$ with action of $A$ on $\La$ are uniquely defined up to
isomorphisms.

(b) Choose a subset $P\subset P(N_i)$ which is invariant with
respect to $A$ and such that the primitive sublattice
$S_0=[\La,\ P]_{pr}\subset N_i$ generated by $\La$ and $P$
satisfies Theorem \ref{th:primembb3}
(equivalently, there exists a primitive embedding
$S_0=[\La,\ P]_{pr}\subset L_{K3}$).
Then $P\subset P(S_0)=S_0\cap P(N_i)$ and
$A\subset A(S_0)$.

(c) Extend the primitive sublattice above to a primitive sublattice
$S_0 \subset S\subset N_i$ such that $S\cap P(N_i)=P$,
$$
\{\phi \in A(N_i)\ |\ \phi(S)=S\  and\ \phi|S^\perp_{N_i}\ is\
identity\}=Clos(A)\,,
$$
and $S$ satisfies Theorem \ref{th:primembb3}.
Then $A(S)=Clos(A)$, $\La(S)=\La$, $P(S)=P$,
and there exists a primitive embedding $S\subset L_{K3}$.

(d) Follow Remark \ref{rem:Niemmarkconst}
to construct a K\"ahlerian K3 surface
$X$ with the marking $S\subset N_i$ by the Niemeier
lattice $N_i$. Then
$$
\Aut(X,S)_0=Clos(A),\ \ P(X)\cap S=P(S)=P.
$$
}
\label{rem:constrPAX}
\end{remark}

\section{Markings of K3 surfaces by concrete Niemeier
lattices $N_1$ --- $N_{23}$}
\label{sec:examples}

\begin{figure}
\begin{center}
\includegraphics[width=10cm]{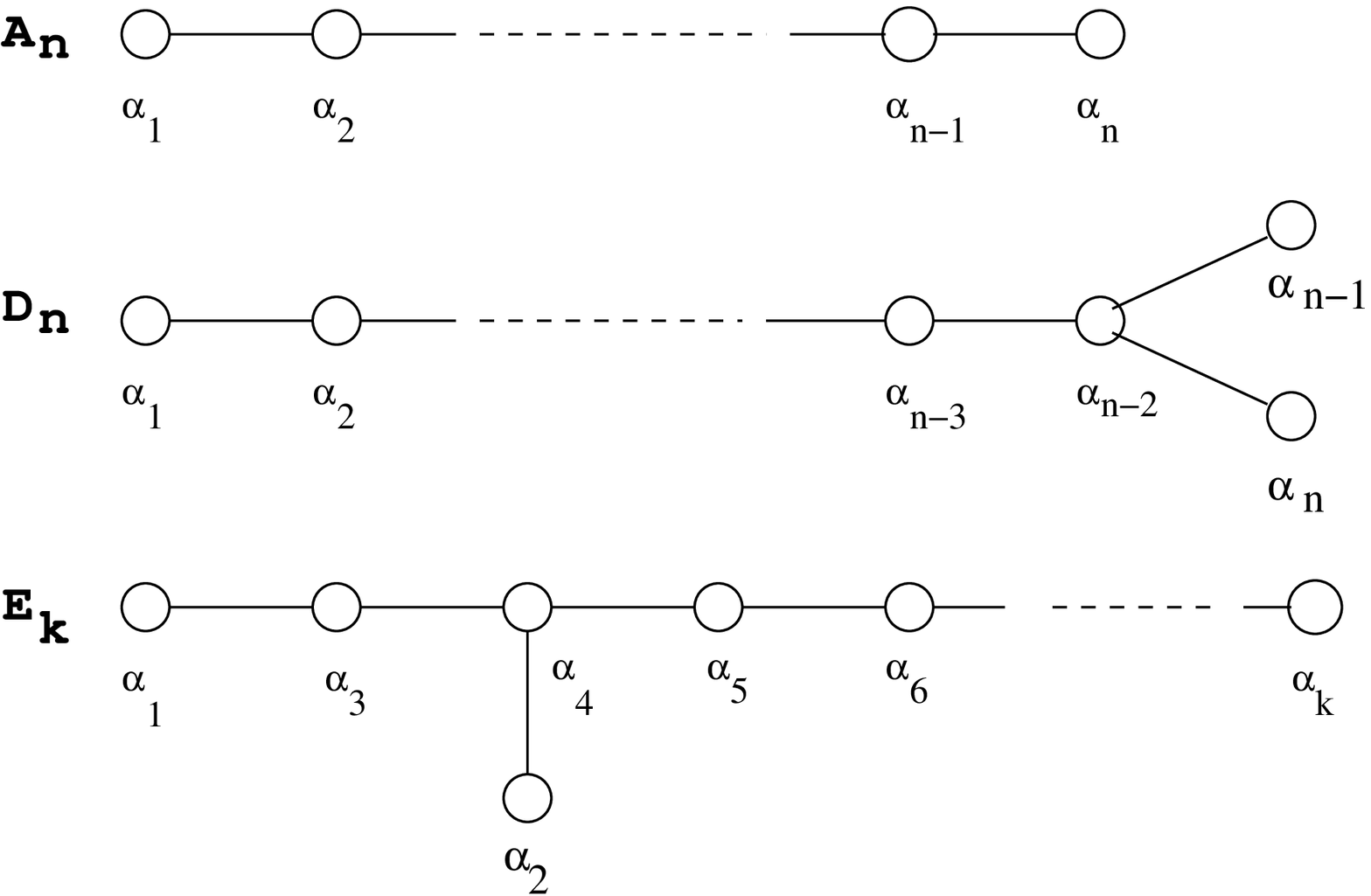}
\end{center}
\caption{Bases of Dynkin diagrams $\aaa_n$, $\ddd_n$, $\eee_k$. }
\label{fig:dyndiagr}
\end{figure}

Below, we use the basis of a root lattice $A_n$, $D_n$ or $E_k$,
$k=6,7,8$,
which is shown on Figure \ref{fig:dyndiagr}.

For $A_n$, $n\ge 1$, we denote
$\varepsilon_1=(\alpha_1+2\alpha_2+\cdots +n\alpha_n)/(n+1)$.
It gives the generator of the discriminant group
$A_n^\ast/A_n\cong \bz/(n+1)\bz$.

For $D_n$, $n\ge 4$ and $n\equiv 0\mod 2$, we denote
$\varepsilon_1=(\alpha_1+\alpha_3+\cdots +\alpha_{n-3}+\alpha_{n-1})/2$,
$\varepsilon_2=(\alpha_{n-1}+\alpha_n)/2$,
$\varepsilon_3=(\alpha_1+\alpha_3+\cdots +\alpha_{n-3}+\alpha_n)/2$. They
give all non-zero elements of the discriminant group
$D_n^\ast/D_n\cong (\bz/2\bz)^2$.

For $D_n$, $n\ge 4$ and $n\equiv 1\mod 2$, we denote
$\varepsilon_1=(\alpha_1+\alpha_3+\cdots +\alpha_{n-2})/2+
\alpha_{n-1}/4-\alpha_n/4$,
$\varepsilon_2=(\alpha_{n-1}+\alpha_n)/2$,
$\varepsilon_3=(\alpha_1+\alpha_3+\cdots +\alpha_{n-2})/2-\alpha_{n-1}/4+
\alpha_n/4$. They give all non-zero elements of
$D_n^\ast/D_n\cong \bz/4\bz$.

For $E_6$, we denote $\varepsilon_1=(\alpha_1-\alpha_3+\alpha_5-\alpha_6)/3$,
$\varepsilon_2=(-\alpha_1+\alpha_3-\alpha_5+\alpha_6)/3$.
They give all non-zero elements of
$E_6^\ast/E_6\cong \bz/3\bz$.

For $E_7$, we denote $\varepsilon_1=(\alpha_2+\alpha_5+\alpha_7)/2$.
It gives the non-zero element of $E_7^\ast/E_7\cong \bz/2\bz$.

If the Dynkin diagram of a root lattice has several connected
components, the second index of a basis numerates
the corresponding connected component.

\medskip

Below we apply the results above to concrete cases of Niemeier
lattices $N_i$. We denote by $X$ a K\"ahlerian K3 surface.

\medskip

{\bf Case 0.} Let us consider marking of $X$
by negative definite even unimodular lattices $K$
of the rank $16$. There are $2$ such lattices: $K_1=\Gamma_{16}$ with
the root sublattice $D_{16}$ and $K_2=2E_8$ with the root sublattice
$2E_8$. For this case, any
primitive sublattice $S\subset K_i$, $i=1,\,2$, satisfies
Theorem \ref{th:primembb3} because $\rk S+l(A_S)\le 16<22$
by Theorem \ref{th:primembb1}. It gives marking of some $X$ if
$P(S)\subset P(K_i)$.
It is why this case is easy.
(Basically, this case was considered in \cite[Remark 1.14.7]{Nik1}.)

In this case, $A(K_1)$ is trivial and $A(K_2)=C_2$ is the group of
order $2$ which permutes two components $E_8$.

Let $X$ be marked by a primitive sublattice $S\subset K_1=\Gamma_{16}$.
We have $\Aut(X,S)_0=A(K_1)$ is trivial. We have $P(X)\cap S=P(S)$ and
$\Gamma(P(S))\subset \Gamma(P(K_1))=\ddd_{16}$. Any such subgraph is possible.

Let $A=A(K_2)$ is the group of order $2$.
Then the coinvariant sublattice (or Leech type sublattice) is
$$
\La_A=(K_2)_A=E_8(2)=[\{\alpha_{i1}-\alpha_{i2}\ |\ 1\le i\le 8\}]
\subset K_2=2E_8.
$$
It is isomorphic to $E_8$ with the form multiplied by $2$.

Let $X$ be marked by $S\subset K_2$. Then $P(X)\cap S=P(S)$.
If $E_8(2)\not\subset S$, then the symplectic group $A(X,S)_0$ is trivial and
$\Gamma (P(X)\cap S)\subset 2\eee_8$.
If $E_8(2)\subset S$, then $A(X,S)_0=A(K_2)$ is the group of order $2$. The
graph $\Gamma(P(X)\cap S)$ is a subgraph of $2\eee_8$ which is invariant
with respect to the permutation of two components of $2\eee_8$. All such
subgraphs are possible.

\medskip

Further, firstly, we consider Niemeier lattices $N_i$ such that $A(N_i)$ has no
non-trivial K\"ahlerian K3 surfaces subgroups
(equivalently, of KahK3 subgroups). Thus, if
$X$ has marking $S\subset N_i$ by such Niemeier lattice $N_i$, then
$\Aut(X,S)_0$ is trivial.

\medskip

{\bf Case 1.} The Niemeier lattice $N_1=N(D_{24})=[D_{24},\varepsilon_1]$
has the trivial group $A(N_1)$ (see \cite[Ch. 16]{CS}).

Let $X$ be marked by a primitive sublattice
$S\subset N_{1}$. Then $S$ must satisfy Theorem \ref{th:primembb3} and
$\Gamma(P(S))\subset \Gamma(P(N_{1}))=\ddd_{24}$. Any such $S$
gives marking of some $X$ and $P(X)\cap S=P(S)$.
The group $\Aut(X,S)_0=A(N_1)$ is trivial.

Let $S=D_{n}$, $17\le n\le 19$. It satisfies
Theorem \ref{th:primembb3}. Thus, there exists $X$ with
marking $S=D_{n}\subset N_1$. By classification
of Niemeier lattices, only $N_1=N(D_{24})$ is possible for
marking of $X$ with such lattice $S\subset S_X$.

\medskip

{\bf Case 2.} The Niemeier lattice $N_2=N(D_{16}\oplus E_8)=
[D_{16},\varepsilon_{11}]\oplus E_8$
has the trivial group $A(N_2)$ (see \cite[Ch. 16]{CS}).

Let $X$ be marked by a primitive sublattice
$S\subset N_{2}$. Then $S$ must satisfy Theorem \ref{th:primembb3} and
$\Gamma(P(S))\subset \Gamma(P(N_{2}))=\ddd_{16}\eee_8$. Any such $S$
gives marking of some $X$ and $P(X)\cap S=P(S)$.
The group $\Aut(X,S)_0=A(N_2)$ is trivial.

Let $S=D_n\oplus E_8$, $8\le n\le 11$. It satisfies
Theorem \ref{th:primembb3}. Thus, there exists $X$ with
marking $S\subset N_2$. By classification
of Niemeier lattices, only $N_2=N(D_{16}\oplus E_8)$
is possible for marking of $X$ with such $S\subset S_X$.

\medskip

{\bf Case 4.} The Niemeier lattice
$N=N_4=N(A_{24})=[A_{24},5\varepsilon_1]$ has the
group $A(N)$ of the order $2$ (see \cite[Ch. 16]{CS}).
It gives the non-trivial involution of the diagram $P(N_4)=\aaa_{24}$. It is
easy to see that $\La=N_{A(N)}=(N^{A(N)})^\perp_N$ does
not satisfy Theorem \ref{th:primembb3}.
Thus, only a trivial subgroup $A\subset A(N)$ is KahK3 subgroup.
(It also follows from results of \cite{Nik0} where
it is shown that $\rk \La=8$ for KahK3 subgroup of order $2$.)

Let $X$ be marked by a primitive sublattice
$S\subset N_{4}$. Then $S$ must satisfy Theorem \ref{th:primembb3} and
$\Gamma(P(S))\subset \Gamma(P(N_{4}))=\aaa_{24}$. Any such $S$
gives marking of some $X$ and $P(X)\cap S=P(S)$.
The group $\Aut(X,S)_0$ is trivial.

Let $S=A_{n}$, $18\le n\le 19$. It satisfies
Theorem \ref{th:primembb3}. Thus, there exists $X$ with
marking $S\subset N_4=N(A_{24})$. By classification
of Niemeier lattices, only $N_1=N(D_{24})$ and $N_4=N(A_{24})$
are possible for marking of $X$ with such $S\subset S_X$.

\medskip

{\bf Case 5.} The Niemeier lattice $N=N_5=N(2D_{12})=
[2D_{12}, \varepsilon_{11}+\varepsilon_{22},
\varepsilon_{21}+\varepsilon_{12}]$
has $A(N)$ of order $2$ which permutes two components $D_{12}$.
It is easy to see that $\La=N_{A(N)}=(N^{A(N)})^\perp_N$ does
not satisfy Theorem \ref{th:primembb3}. Thus, only a trivial
subgroup $A\subset A(N)$ is KahK3 subgroup.

Let $X$ be marked by a primitive sublattice
$S\subset N_{5}$. Then $S$ must satisfy Theorem \ref{th:primembb3} and
$\Gamma(P(S))\subset \Gamma(P(N_{5}))=2\ddd_{12}$. Any such $S$
gives marking of some $X$ and $P(X)\cap S=P(S)$.
The group $\Aut(X,S)_0$ is trivial.

Let $S=D_{10}\oplus D_9$. It satisfies
Theorem \ref{th:primembb3}. Thus, there exists $X$ with
marking $S\subset N_5$. By classification
of Niemeier lattices, only $N_5=N(2D_{12})$ is possible for marking
of $X$ with such $S\subset S_X$.

\medskip

{\bf Case 10.} The Niemeier lattice $N=N_{10}=N(2A_{12})=
[2A_{12}, \varepsilon_{11}+5\varepsilon_{12}]$
has $A(N)$  which is a cyclic group of order $4$.
It is easy to see that $\La=N_{A}=(N^{A})^\perp_N$ does
not satisfy Theorem \ref{th:primembb3} for both its
non-trivial subgroups $A\subset A(N)$.
It is enough to check this for its subgroup of the order $2$ which gives
non-trivial involutions on the graph $\aaa_{12}$ of
each of two components $A_{12}$. Thus, only a trivial
subgroup $A\subset A(N)$ is KahK3 subgroup.

Let $X$ be marked by a primitive sublattice
$S\subset N_{10}$. Then $S$ must satisfy Theorem \ref{th:primembb3} and
$\Gamma(P(S))\subset \Gamma(P(N_{10}))=2\aaa_{12}$. Any such $S$
gives marking of some $X$ and $P(X)\cap S=P(S)$.
The group $\Aut(X,S)_0$ is trivial.

Let $S=A_{10}\oplus A_9$. It satisfies
Theorem \ref{th:primembb3}. Thus, there exists $X$ with
marking $S\subset N_{10}$. By classification
of Niemeier lattices, only $N_1=N(D_{24})$,
$N_4=N(A_{24})$, $N_5=N(2D_{12})$
and $N_{10}=N(2A_{12})$ are possible for marking
of $X$ with such $S\subset S_X$.

\medskip

For next cases, the group $A(N_i)$ has K\"ahlerian K3 surfaces subgroups
(that is KahK3 subgroups) $A$ only of the order $1$ or $2$.
Thus, if $X$ is marked by one of these lattices then the group
$\Aut(X,S)_0$ is either trivial, or it has the order $2$.

\medskip

{\bf Case 3.} For the Niemeier lattice $N=N_3=N(3E_8)$,
the group $A(N)$
is ${\mathfrak S}_3$ (obviously) which acts by permutations on
the three components $E_8$ which we denote by
$(E_8)_j$, $j=1,\,2,\,3$. Simple calculations show that
for $A\subset A(N)={\mathfrak S}_3$ the lattice $\La = N_A=((N^A)^\perp_N)$
does not satisfy Theorem \ref{th:primembb3} if $A$ is ${\mathfrak S}_3$
or the alternating group ${\mathfrak A}_3$.
It is valid only if $A$ is either generated by a transposition,
or $A$ is trivial. They give all KahK3 subgroups $A\subset A(N)$.

For $A=[(kl)]\subset A(N)$ generated by a
transposition $(kl)$, $1\le k<l\le 3$,
the Leech type (or the coinvariant) sublattice is
$$
E_8(2)_{kl}=N_A=[\{\alpha_{ik}-\alpha_{il}\ |\ 1\le i\le 8\}].
$$

Let $X$ be marked by a primitive sublattice
$S\subset N_3=N(3E_8)$. Then $S$ must satisfy
Theorem \ref{th:primembb3} and $\Gamma(P(S))\subset \Gamma(P(N_3))=3\eee_8$.
Any such $S$ gives marking of some $X$ and $P(X)\cap S=P(S)$.
If $S$ does not contain each of sublattices $E_8(2)_{kl}$,
$1\le k<l\le3$,  then $\Aut(X,S)_0$ is trivial.
If $E_8(2)_{kl}\subset S$ for some $1\le k<l\le 3$,
then $\Aut(X,S)_0=[(kl)]$ is the group of order $2$.

Let $S=2E_8$. Then $S$ satisfies Theorem \ref{th:primembb3} and gives
marking $S\subset N_3=N(3E_8)$ of some $X$.
By classification of Niemeier lattices, only
$N_3=N(3E_8)$ is possible for marking of $X$ with such $S\subset S_X$.

\medskip

{\bf Case 6.} For the Niemeier lattice
$N=N_6=N(A_{17}\oplus E_7)=[A_{17}\oplus E_7,
3\varepsilon_{11}+\varepsilon_{12}]$,
the group $A(N_6)$ has order $2$ (see \cite[Ch. 16]{CS}).
Its generator is trivial
on $E_7$ and gives the non-trivial involution of the diagram $\aaa_{17}$.

The coinvariant sublattice for $A(N)$ is equal to
$$
E_8(2)=N_{A(N)}=[\{\alpha_{i1}-\alpha_{(18-i)1}\ |\ 1\le i\le 8\},\
\frac{1}{3}\sum_{i=1}^8{i(\alpha_{i1}-\alpha_{(18-i)1})}]\subset N_{6}.
$$

Let $X$ be marked by a primitive sublattice
$S\subset N_6=N(A_{17}\oplus E_7)$. Then $S$ must satisfy
Theorem \ref{th:primembb3} and
$\Gamma(P(S))\subset \Gamma(P(N_6))=\aaa_{17}\eee_7$. Any such $S$
gives marking of some $X$ and $P(X)\cap S=P(S)$.
If $E_8(2)\not\subset S$, then $\Aut(X,S)_0$ is trivial.
If $E_8(2)\subset S$, then $\Aut(X,S)_0=A(N_6)$ has order two.

In particular, $S=[A_{17},6\varepsilon_{11}] \subset N_6$
satisfies Theorem $\ref{th:primembb3}$ and gives marking of some $X$.
Moreover, $\Aut(X,S)_0$ has order $2$.
By classification of Niemeier lattices and our calculationss above,
$X$ can be marked by $N_6=N(A_{17}\oplus E_7)$ only for such $S\subset S_X$.

\medskip

{\bf Case 7.} For the Niemeier lattice
$N=N_7=N(D_{10}\oplus 2E_7)=[D_{10}\oplus 2E_7,\
\varepsilon_{11}+\varepsilon_{12},\ \varepsilon_{31}+\varepsilon_{13}]$,
the group $A(N_7)$ has order $2$ (see \cite[Ch. 16]{CS}).
Its generator permutes two diagrams $\eee_7$ and
gives a non-trivial involution on the diagram $\ddd_{10}$.

For $A(N)$, the coinvariant
sublattice is equal to
$$
E_8(2)=N_{A(N)}=[\alpha_{91}-\alpha_{10,1},
\{\alpha_{i2}-\alpha_{i3}\ |\ 1\le i\le 7\},\ \ \
$$
$$
(\alpha_{91}-\alpha_{10,1}+\alpha_{22}-\alpha_{23}+
\alpha_{52}-\alpha_{53}+\alpha_{72}-\alpha_{73})/2\,]\subset N_{7}.
$$

Let $X$ be marked by a primitive sublattice
$S\subset N_7=N(D_{10}\oplus 2E_7)$. Then $S$ must satisfy
Theorem \ref{th:primembb3} and
$\Gamma(P(S))\subset \Gamma(P(N_7))=\ddd_{10}2\eee_7$. Any such $S$
gives marking of some $X$ and $P(X)\cap S=P(S)$.
If $E_8(2)\not\subset S$, then $\Aut(X,S)_0$ is trivial.
If $E_8(2)\subset S$, then $\Aut(X,S)_0=A(N_7)$ has order two.

In particular, $S=[\alpha_{91},\alpha_{10,1}, 2E_7]_{pr} \subset N_7$
satisfies Theorem $\ref{th:primembb3}$ and gives marking of some $X$.
Moreover, $\Aut(X,S)_0$ has order two. By classification of Niemeier
lattices and our considerations above, then $X$ can be marked by
$N_7=N(D_{10}\oplus 2E_7)$ only for such $S\subset S_X$.

\medskip

{\bf Case 8.} For the Niemeier lattice
$N=N_8=N(A_{15}\oplus D_9)=[A_{15}\oplus D_9,\
2\varepsilon_{11}+\varepsilon_{12}]$,
the group $A(N_8)$ has order $2$ (see \cite[Ch. 16]{CS}).
Its generator gives a non-trivial involution
on the diagrams $\aaa_{15}$ and $\ddd_9$.

For $A(N)$, the coinvariant
sublattice is equal to
$$
E_8(2)=N_{A}=[\,\{\alpha_{i1}-
\alpha_{(16-i)1}\ |\ 1\le i\le 7\},\ \alpha_{82}-\alpha_{92},
$$
$$
\frac{1}{4}\sum_{i=1}^7{(\alpha_{i1}-\alpha_{(16-i)1})}+
\frac{1}{2}(\alpha_{82}-\alpha_{92})\,]\subset N_{8}.
$$

Let $X$ be marked by a primitive sublattice
$S\subset N_8=N(A_{15}\oplus D_9)$. Then $S$ must satisfy
Theorem \ref{th:primembb3} and
$\Gamma(P(S))\subset \Gamma(P(N_8))=\aaa_{15}\ddd_9$. Any such $S$
gives marking of some $X$ and $P(X)\cap S=P(S)$.
If $E_8(2)\not\subset S$, then $\Aut(X,S)_0$ is trivial.
If $E_8(2)\subset S$, then $\Aut(X,S)_0=A(N_8)$ has order two.

In particular, $S=[A_{15},\alpha_{82},\alpha_{92}]_{pr}
\subset N_8$
satisfies Theorem $\ref{th:primembb3}$ and
gives marking of some $X$. Moreover, $\Aut(X,S)_0$ has order $2$.
By classification of Niemeier lattices and our considerations above,
$X$ can be marked by $N_8=N(A_{15}\oplus D_9)$
only for such $S\subset S_X$.

\medskip

{\bf Case 9.} For the Niemeier lattice
$$
N=N_9=N(3D_8)=
[3D_8,\
\varepsilon_{11}+\varepsilon_{22}+\varepsilon_{23},\
\varepsilon_{21}+\varepsilon_{12}+\varepsilon_{23},\
\varepsilon_{21}+\varepsilon_{22}+\varepsilon_{13}]
$$
the group $A(N)$  is ${\mathfrak S}_3$ which acts by permutations on
the three components $D_8$ and by permutations of
$\alpha_{71}$, $\alpha_{72}$ and $\alpha_{73}$.
Simple calculations show that
for $A\subset A(N)={\mathfrak S}_3$ the lattice
$\La = N_A=((N^A)^\perp_N)$ does not satisfy Theorem
\ref{th:primembb3} if $A$ is ${\mathfrak S}_3$ or
${\mathfrak A}_3$.
It is valid only if $A$ is either generated by a transposition,
or $A$ is trivial. They give all KahK3 subgroups.

Let $A=[(kl)]$, $1\le k<l\le 3$, is a transposition on
$(D_8)_j$, $j=1,2,3$. Then the coinvariant sublattice is
$$
N_A=E_8(2)_{kl}=
$$
$$
[\{\alpha_{ik}-\alpha_{il}\ |\ 1\le i\le 8\},\
(\alpha_{1k}-\alpha_{1l}+\alpha_{3k}-\alpha_{3l}+\alpha_{5k}-\alpha_{5l}+
\alpha_{8k}-\alpha_{8l})/2].
$$

Let $X$ be marked by a primitive sublattice
$S\subset N_9=N(3D_8)$. Then $S$
must satisfy Theorem \ref{th:primembb3} and
$\Gamma(P(S))\subset \Gamma(P(N_9))=3\ddd_8$.
Any such $S$ gives marking of some $X$ and $P(S)=P(X)\cap S$.
If $S$ does not contain each of sublattices $E_8(2)_{kl}$,
$1\le k<l\le3$,  then $\Aut(X,S)_0$ is trivial.
If $E_8(2)_{kl}\subset S$ for some $1\le k<l\le 3$,
then $\Aut(X,S)_0=[(kl)]$ is the group of order $2$.

Let $S=[2D_8]_{pr}\subset N_9$. Then $S$ satisfies
Theorem \ref{th:primembb3} and gives marking of some $X$.
Moreover,  $\Aut(X,S)_0$ has order $2$. By classification
of Niemeier lattices and our considerations above,
$X$ can be marked by the Niemeier lattice $N_9=N(3D_8)$ only
for such $S\subset S_X$.

\medskip

{\bf Case 11.} For the Niemeier lattice
$N=N_{11}=N(A_{11}\oplus D_7\oplus E_6)=[A_{11}\oplus D_{7}\oplus E_6,\
\varepsilon_{11}+\varepsilon_{12}+\varepsilon_{13}]$,
the group $A(N_{11})$ has order $2$ (see \cite[Ch. 16]{CS}).
Its generator gives a non-trivial involution
on the subdiagrams $\aaa_{11}$, $\ddd_7$ and $\eee_6$.
For the group $A=A(N_{11})$, the coinvariant
sublattice is equal to
$$
E_8(2)=N_{A}=[\,\{\alpha_{i1}-
\alpha_{(12-i)1}\ |\ 1\le i\le 5\},\ \alpha_{62}-\alpha_{72},\
\alpha_{13}-\alpha_{63},\ \alpha_{33}-\alpha_{53},
$$
$$
\frac{1}{6}\sum_{i=1}^5{i(\alpha_{i1}-\alpha_{(12-i)1})}+
\frac{1}{2}(\alpha_{62}-\alpha_{72})+
\frac{1}{3}(-\alpha_{13}+\alpha_{63}+\alpha_{33}-\alpha_{53})\,]
\subset N_{11}.
$$

Let $X$ be marked by a primitive sublattice
$S\subset N_{11}=N(A_{11}\oplus D_7\oplus E_6)$. Then $S$
must satisfy Theorem \ref{th:primembb3} and
$\Gamma(P(S))\subset \Gamma(P(N_{11}))=\aaa_{11}\ddd_7\eee_6$.
Any such $S$ gives marking of some $X$ and $P(S)=P(X)\cap S$.
If $S$ does not contain $E_8(2)$, then $\Aut(X,S)_0$ is trivial.
If $E_8(2)\subset S$, then $\Aut(X,S)_0=A(N_{11})$ is the group of order $2$.

In particular, $S=[A_{11}\oplus E_6,\alpha_{62},\alpha_{72}]_{pr}
\subset N_{11}$ satisfies Theorem \ref{th:primembb3} and
gives marking of some $X$. Moreover, $\Aut(X,S)_0$ has order $2$.
By classification of Niemeier lattices and our considerations above,
$X$ can be marked by $N_{11}=N(A_{11}\oplus D_7\oplus E_6)$ only
for such $S\subset S_X$.

\medskip

{\bf Case 15.} For the Niemeier lattice
$$
N=N_{15}=N(3A_8)=
[3A_8,\
4\varepsilon_{11}+\varepsilon_{12}+\varepsilon_{13},\
\varepsilon_{11}+4\varepsilon_{12}+\varepsilon_{13},\
\varepsilon_{11}+\varepsilon_{12}+4\varepsilon_{13}]
$$
the group $A(N)$ has the order 12, and it
is the direct product of the
group of order 2 which gives
non-trivial involutions on all three components $\aaa_8$,
and the group ${\mathfrak S}_3$ which acts by permutations of
the three components $\aaa_8$ and of
$\alpha_{11}$, $\alpha_{12}$ and $\alpha_{13}$
(see \cite[Ch. 16]{CS}).

Simple calculations show that
for $A\subset A(N)$ the coinvariant sublattice
$\La = N_A=((N^A)^\perp_N)$ satisfies Theorem
\ref{th:primembb3} only if $A$ is either trivial or $A=[(kl)]$,
$1\le k<l\le 3$, is generated by the transposition $(kl)$ of
two components $(\aaa_8)_k$ and $(\aaa_8)_l$ and $\alpha_{1k}$,
\and $\alpha_{1l}$, and it is identity on the remaining
component $(\aaa_8)_j$. They give all KahK3 subgroups.
For $A=[(kl)]$, the coinvariant sublattice is
$$
E_8(2)_{kl}=N_A=
[\{\alpha_{ik}-\alpha_{il}\ |\ 1\le i\le 8\},\
\frac{1}{3}\sum_{i=1}^8{i(\alpha_{ik}-\alpha_{il})}]\subset N_{15}\, .
$$

Let $X$ be marked by a primitive sublattice
$S\subset N_{15}=N(3A_8)$. Then $S$
must satisfy Theorem \ref{th:primembb3} and
$\Gamma(P(S))\subset \Gamma(P(N_{15}))=3\aaa_8$.
Any such $S$ gives marking of some $X$ and $P(X)\cap S=P(S)$.
If $S$ does not contain each of sublattices $E_8(2)_{kl}$,
$1\le k<l\le3$,  then $\Aut(X,S)_0$ is trivial.
If $E_8(2)_{kl}\subset S$ for some $1\le k<l\le 3$,
then $\Aut(X,S)_0=[(kl)]$ is
the group of order $2$.

Let $S=[2A_8]_{pr}\subset N_{15}$. Then $S$ satisfies
Theorem \ref{th:primembb3} and gives marking of some $X$.
Moreover, $\Aut(X,S)_0$ has order $2$.
By classification  of Niemeier lattices and our considerations above, $X$
can be marked by the Niemeier lattice $N_{15}=N(3A_8)$ only for
such $S\subset S_X$.

\medskip

{\bf Some general remarks.}
Cases below are more complicated. We use the following simple general
statements and computer Programs (see Appendix, Sect. \ref{sec:programs}).
(Of course, one can also use them for all cases which we considered above
and skipped calculations because they were easy.)

\begin{proposition} Let $N$ be a Niemeier
(or any other even negative definite unimodular lattice),
$P(N)$ a basis of the root system
of $N$ and
$$
A(N)=\{\phi\in O(N)\ |\ \phi(P(N))=P(N)\} .
$$

If $A\subset A(N)$ is a KahK3 subgroup with the
Leech type (or coinvariant) sublattice
$N_A=(N^A)^\perp_N$
(equivalently, there exists a primitive embedding $N_A\subset L_{K3}$),
then its conjugate $A^g=gAg^{-1}$, $g\in A(N)$, is also KahK3 subgroup
with the Leech type (or coinvariant sublattice)
$N_{A^g}=g(N_A)$. They define KahK3 conjugacy classes of $A(N)$.

Thus, to describe KahK3 subgroups $A\subset A(N)$ and their coinvariant
sublattices, it is enough to describe representatives of all
$KahK3$ conjugacy classes  in $A(N)$ and their coinvariant sublattices.
\label{prop:congKahK3}
\end{proposition}

\begin{proof} Indeed, $N_{A^g}$ is isomorphic to $N_A$.
Therefore, if $N_A$ has a primitive embedding $N_A\subset L_{K3}$,
then $N_{A^g}=g(N_A)\subset N$ also has a primitive embedding
$N_{A^g}\subset L_{K3}$.
\end{proof}

To calculate the coinvariant sublattices, we can use

\begin{proposition} Let $N$ be a Niemeier
(or any other even negative definite unimodular lattice), $P(N)$ a
basis of the root
system of $N$ and
$$
A(N)=\{\phi\in O(N)\ |\ \phi(P(N))=P(N)\}.
$$

Let $A_1\subset A(N)$ and $A_2\subset A(N)$ are two subgroups,
and $A=\langle A_1,\,A_2\rangle\subset A(N)$ is generated by
$A_1$ and $A_2$. Then the coinvariant sublattice
$N_{A}$ is the primitive sublattice
$N_{A}=[N_{A_1},\,N_{A_2}]_{pr}\subset N$ of $N$
generated by the coinvariant sublattices
$N_{A_1}$ and $N_{A_2}$ of $A_1$ and $A_2$. Therefore,
$A=\langle A_1,\,A_2\rangle$ is KahK3 subgroup if and
only  $[N_{A_1},\,N_{A_2}]_{pr}$ has a primitive embedding into
$L_{K3}$.

In particular, if $A=\langle g_1,\,\dots\,g_n\rangle$ is generated
by $g_1,\,\dots\,g_n\in A$, then \newline
$N_A=[N_{\langle g_1 \rangle},\,\dots ,
N_{\langle g_n\rangle}]_{pr}\subset N$. Moreover,
$A$ is a KahK3 subgroup if and only if  the sublattice
$[N_{\langle g_1 \rangle},\dots N_{\langle g_n\rangle}]_{pr}\subset N$
has a primitive embedding into $L_{K3}$.
\label{prop:coinvgensubA1A2}
\end{proposition}

\begin{proof} Indeed, since $A_1,\, A_2\subset A$, then
$[N_{A_1},N_{A_2}]_{pr}\subset N_A$.
Let $M=[N_{A_1},N_{A_2}]_{pr}\subset N$.
Since $N_{A_1}\subset M$ and $N_{A_2}\subset M$, it follows that
$A_1\subset A(M)$, $A_2\subset A(M)$, and
$A=\langle A_1,\,A_2\rangle\subset A(M)$. It follows that
$A=\langle A_1,\,A_2\rangle$ gives identity on $M^\perp_N$,
and $M\subset N_{\langle A_1,\,A_2\rangle}$. Therefore,
$M=N_A$.
\end{proof}

To calculate the coinvariant sublattice $N_A$ for a subgroup
$A\subset A(N)$,
we also use the following important and simple statement.

\begin{proposition} Let $N$ be a Niemeier
(or any other even unimodular lattice), $P(N)$ a basis of the root
system of $N$ and
$$
A(N)=\{\phi\in O(N)\ |\ \phi(P(N))=P(N)\}.
$$
Suppose that $P(N)$ generates $N$ over $\bq$ (otherwise, use
an $A$-invariant basis of $N\otimes \bq$ instead of $P(N)$).
Let $P(N)=\{e_1,e_2,\dots,e_n\}$ and $\{e_1^\ast,e_2^\ast,\dots,e_n^\ast\}
\subset N\otimes \bq$ are dual elements, that is
$e_i^\ast\cdot e_j=\delta_{ij}$
where $\delta_{ij}$ is the Kronecker symbol.

Let $A\subset A(N)$ be a subgroup and $N_A\subset N$ its
coinvariant sublattice $N_A=(N^A)^\perp_N$.
Let
$$
\{e_{i_{11}},\dots, e_{i_{1l_1}}\},\ ... \ \{e_{i_{t1}},\dots, e_{i_{tl_t}}\}
$$
be all orbits of $A$ in $P(N)$, and $t$ the number of orbits.

Then
$$
\{e_{i_{11}}^\ast-e_{i_{12}}^\ast,\dots,\
e_{i_{1(l_1-1)}}^\ast-e_{i_{1l_1}}^\ast\},\ ... ,\
\{e_{i_{t1}}^\ast-e_{i_{t2}}^\ast,\dots,\
e_{i_{t(l_t-1)}}^\ast-e_{i_{tl_t}}^\ast\}
$$
give basis of $N_A\otimes \bq$. In particular, $\rk N_A=\rk N-t$.
\label{prop:coinvbasis}
\end{proposition}

A primitive sublattice $K_{pr}\subset N$ of a lattice $N$ is defined
by its rational basis
for the vector space $K\otimes \bq$. Really, $K_{pr}=K\otimes
\bq\cap N\subset N\otimes \bq$.
Thus, Proposition \ref{prop:coinvbasis} gives an efficient method for
calculation of $N_A$ and its invariants
(for Theorem \ref{th:primembb3}) to find out if $N_A$ has a primitive
embedding into $L_{K3}$ and $A$ is KahK3 subgroup.

In Appendix (Sect. \ref{sec:programs}) we give a Program 0
(which uses Programs 1 -- 4)
which calculates a normal (or Smith) basis of the
primitive sublattice $K_{pr}\subset N$ for any sublattice $K$ of
$N$ which is given by generators of $K\otimes \bq$.
We use the root basis $P(N)$ of a Niemeier lattice $N$
(corresponding to basic columns $(0,\dots , 0,1,0, \dots ,0)^t$
of the length $24$),
the integer matrix $r$ of $N$ in this basis (which is equivalent to
Dynkin diagram),
and additional cording data of $N$. The sublattice $K$ is given
by rational columns in this basis (matrix SUBL of size $(24\times\  \cdot\ )$
in Program 0. The Program 0 calculates the normal (elementary divisors or
Smith) basis of
$K_{pr}$ (denoted by SUBLpr for the Program 0) for the embedding
$K_{pr}\subset K_{pr}^\ast$.
It calculates elementary divisors (or Smith) invariants of this
embedding (denoted by
DSUBLpr for the Program 0).  In particular, it calculates $\rk K_{pr}$,
the discriminant group
$A_{K_{pr}}=K_{pr}^\ast/K_{pr}$, gives the number $l(A_{K_{pr}})$
of its minimal generators. Moreover, the normal basis can be used to
calculate the Jordan form of $K_{pr}\otimes \bz_p$ over the rings
of $p$-adic integers,
and the discriminant form $q_{K_{pr}}$.  The last vectors of
this basis give
the unimodular part of $K_{pr}$ over $\bz_p$. Thus, Program 0
(Appendix, Sect. \ref{sec:programs})
gives all necessary invariants  for Theorem \ref{th:primembb3} to
find out
if $K_{pr}$ has a primitive embedding into$L_{K3}$. The program is
very fast for all cases which we consider below.

\medskip

{\bf Case 13.} For the Niemeier lattice
$$
N=N_{13}=N(2A_9\oplus D_6)
=[2A_9\oplus D_6,\
2\varepsilon_{11}+4\varepsilon_{12},\,
5\varepsilon_{11}+\varepsilon_{13},\,
5\varepsilon_{12}+\varepsilon_{33}],
$$
the group $A(N)$ is a cyclic group $C_4$ of
the order $4$ (see \cite[Ch. 16]{CS}).
Its elements are defined by permutations
of the terminals of Dynkin diagrams $\aaa_9$
and $\ddd_6$. Its generator $\varphi$
gives the permutation
$$
\varphi=(\alpha_{11}\alpha_{12}\alpha_{91}\alpha_{92})(\alpha_{53}\alpha_{63}).
$$
By Proposition \ref{prop:coinvbasis},
the coinvariant sublattice $N_{[\varphi]}$ is
$$
N_{[\varphi]}=[\alpha_{11}^\ast-\alpha_{12}^\ast,\,
\alpha_{12}^\ast-\alpha_{91}^\ast,\,\alpha_{91}^\ast-\alpha_{92}^\ast,\,
\alpha_{21}^\ast-\alpha_{22}^\ast,\,\alpha_{22}^\ast-\alpha_{81}^\ast,\,
\alpha_{81}^\ast-\alpha_{82}^\ast,\,
\alpha_{31}^\ast-\alpha_{32}^\ast,\,\alpha_{32}^\ast-\alpha_{71}^\ast,\,
$$
$$
\alpha_{71}^\ast-\alpha_{72}^\ast,\,
\alpha_{41}^\ast-\alpha_{42}^\ast,\,\alpha_{42}^\ast-\alpha_{61}^\ast,\,
\alpha_{61}^\ast-\alpha_{62}^\ast,\,
\alpha_{51}^\ast-\alpha_{52}^\ast,\,
\alpha_{53}^\ast-\alpha_{63}^\ast]_{pr}\subset N.\hfill
$$
Using Program 0 (see Appendix, Sect. \ref{sec:programs}), we obtain
that $\rk N_{[\varphi]}=14$, and $N_{[\varphi]}^\ast/N_{[\varphi]}\cong
(\bz/4\bz)^4\times(\bz/2\bz)^2$,
and $l(N_{[\varphi]}^\ast/N_{[\varphi]})=6$.
By Theorem \ref{th:primembb3}, the lattice $N_{[\varphi]}$ has a primitive
embedding into
$L_{K3}$, and $[\varphi]=A(N)$ is KahK3 group.

Any subgroup of $A(N)$ is also KahK3 subgroup.
The only non-trivial subgroup of $A(N)$ is $[\varphi^2]$
of order $2$. Similar calculations
show that $N_{[\varphi^2]}\cong E_8(2)$.

Let $X$ be marked by a primitive sublattice
$S\subset N_{13}=N(2A_9\oplus D_6)$. Then $S$
must satisfy Theorem \ref{th:primembb3} and
$\Gamma(P(S))\subset \Gamma(P(N_{13}))=2\aaa_9\,\ddd_6$.
Any such $S$ gives marking of some $X$ and $P(X)\cap S=P(S)$.

If $N_{[\varphi]}\subset S$, then
$\Aut(X,S)_0=[\varphi]\cong C_4$ is cyclic group of the order $4$.
Otherwise, if  only $N_{[\varphi^2]}\subset S$,
then $\Aut(X,S)_0=[\varphi^2]\cong C_2$. Otherwise, if
$N_{[\varphi^2]}\not\subset S$, then $\Aut(X,S)_0$ is trivial.

Let $S=[2A_9,N_{[\varphi]}]_{pr}=[\alpha_{11},\,\alpha_{21},\dots ,\,
\alpha_{81},\, N_{[\varphi]}]_{pr}\subset N_{13}$. Then (using
Program 0 in Appendix, Sect. \ref{sec:programs}) we have
$\rk S=19$ and $S^\ast/S\cong \bz/4\bz$. Thus
$S$ satisfies Theorem \ref{th:primembb3} and gives marking of some $X$.
We have  $\Aut(X,S)_0\cong C_4$ and
$2\aaa_9\subset \Gamma(P(X)\cap S)=P(S)$.
By classification  of Niemeier lattices and our
calculations above, $X$
can be marked by the Niemeier lattice $N_{13}=N(2A_9\oplus D_6)$
only for such $S\subset S_X$.

\medskip

{\bf Case 16.} For the Niemeier lattice
$$
N=N_{16}=N(2A_7\oplus 2D_5)
=[2A_7\oplus 2D_5,\
\varepsilon_{11}+\varepsilon_{12}+\varepsilon_{13}+\varepsilon_{24},\
\varepsilon_{11}+7\varepsilon_{12}+
\varepsilon_{23}+\varepsilon_{14}],
$$
the group $A(N)$ is the dihedral group ${\mathfrak D}_8$ of
the order $8$ (see \cite[Ch. 16]{CS}).
We identify it with the group of symmetries of
a quadrat.
Its elements are defined by permutations
of the terminals of Dynkin diagrams $\aaa_7$
and $\ddd_5$. The central symmetry $\varphi_0$
gives the involution
$$
\varphi_0=
(\alpha_{11}\alpha_{71})(\alpha_{12}\alpha_{72})
(\alpha_{43}\alpha_{53})(\alpha_{44}\alpha_{54}).
$$
Two generating symmetries $\varphi_1$ and $\varphi_2$
give involutions
$$
\varphi_1=(\alpha_{11}\alpha_{12})(\alpha_{71}\alpha_{72})(\alpha_{44}\alpha_{54}),\ \
\varphi_2=(\alpha_{12}\alpha_{72})(\alpha_{43}\alpha_{44})(\alpha_{53}\alpha_{54})
$$
where $\varphi_2\varphi_1$ gives a rotation by $90^o$, and has order $4$.

For the cyclic subgroup $H=[\varphi_2\varphi_1]$ of order $4$,
the coinvariant sublattice $N_H$
has $\rk N_H=16$ and $l(N_H^\ast/N_H)=8$ (we use Proposition
\ref{prop:coinvbasis}
and Program 0 of Appendix, Sect. \ref{sec:programs}). Thus, $N_H$ has no
primitive embeddings into $L_{K3}$,
and $H$ is not KahK3 subgroup by Theorem \ref{th:primembb3}. It follows
that $A(N)$ is not KahK3 subgroup either.

For the conjugate in $A(N)$ subgroups $H=[\varphi_0,\, \varphi_1]$ and
$H=[\varphi_0,\,\varphi_2]$ isomorphic
to $C_2\times C_2$, we have $\rk N_H=12$ and
$N_H^\ast/N_H\cong (\bz/4\bz)^2\times (\bz/2\bz)^6$
(we use Proposition \ref{prop:coinvbasis} and Program 0 of Appendix,
Sect. \ref{sec:programs}). Thus, $N_H$ has a primitive
embedding into $L_{K3}$ by Theorem \ref{th:primembb3},  and the
$H\cong C_2\times C_2$ are KahK3 subgroups.
\footnote{These calculations show that for a symplectic group
$G=C_2\times C_2$ on a
K\"ahlerian K3 surface, the group
$S_{(G)}^\ast/S_{(G)}=S_{(2,2)}^\ast /S_{(2,2)}\cong
(\bz/4\bz)^2\times (\bz/2\bz)^6$.
We must correct our calculation of this group in
\cite[Prop. 10.1]{Nik0}.}

All other non-trivial subgroups of $A(N)$ are cyclic subgroups of order $2$
which are $[\varphi_0]$,  and conjugate subgroups
$[\varphi_1]$, $[\varphi_1\varphi_0]$,  $[\varphi_2]$, $[\varphi_2\varphi_0]$
of order $2$. They are KahK3 subgroups since they are subgroups of
$H$ above.
The coninvariant sublattices
$N_{[\varphi_0]}$, $N_{[\varphi_1]}$, $N_{[\varphi_1\varphi_0]}$,
$N_{[\varphi_2]}$, $N_{[\varphi_2\varphi_0]}$ are isomorphic to $E_8(2)$.

Let $X$ be marked by a primitive sublattice
$S\subset N_{16}=N(2A_7\oplus 2D_5)$. Then $S$
must satisfy Theorem \ref{th:primembb3} and
$\Gamma(P(S))\subset \Gamma(P(N_{16}))=2\aaa_7\,2\ddd_5$.
Any such $S$ gives marking of some $X$ and $P(X)\cap S=P(S)$.

If $N_{[\varphi_0,\varphi_1]}\subset S$, then
$\Aut(X,S)_0=[\varphi_0,\,\varphi_1]\cong C_2\times C_2$.
If $N_{[\varphi_0,\varphi_2]}\subset S$, then
$\Aut(X,S)_0=[\varphi_0,\,\varphi_2]\cong C_2\times C_2$.
Otherwise, if  only $N_{[\varphi]}\subset S$ for one of
$\varphi\in \{\varphi_0,\,\varphi_1,\,\varphi_1\varphi_0,\,\varphi_2,\,
\varphi_2\varphi_0\}$,
then $\Aut(X,S)_0=[\varphi]\cong C_2$. Otherwise,
if $N_{[\varphi]}\not\subset S$ for
each such $\varphi$, then $\Aut(X,S)_0$ is trivial.

Let $S=[2A_7,N_{[\varphi_0,\varphi_1]}]_{pr}\subset N_{16}$. Then (using
Program 0 in Appendix, Sect. \ref{sec:programs}) we have
$\rk S=16$ and $S^\ast/S\cong (\bz/2\bz)^4$. Thus
$S$ satisfies Theorem \ref{th:primembb3} and gives marking of some $X$.
We have  $\Aut(X,S)_0\cong C_2\times C_2$ and
$2\aaa_7\subset \Gamma(P(X)\cap S)=P(S)$.
By classification  of Niemeier lattices and our calculations
above, $X$
can be marked by the Niemeier lattice $N_{16}=N(2A_7\oplus 2D_5)$ only for
such $S\subset S_X$.

\medskip

{\it Below, in our calculations, we don't mention using of
Proposition \ref{prop:coinvbasis} and Program 0 of Appendix,
Sec. \ref{sec:programs}. We use them all the time.}

\medskip

{\bf Case 17.} For the Niemeier lattice
$$
N=N_{17}=N(4A_6)=
$$
$$
=[4A_6,\
\varepsilon_{11}+2\varepsilon_{12}+\varepsilon_{13}+6\varepsilon_{14},\
\varepsilon_{11}+6\varepsilon_{12}+2\varepsilon_{13}+\varepsilon_{14},\,
\varepsilon_{11}+\varepsilon_{12}+6\varepsilon_{13}+2\varepsilon_{14}],
$$
the group $A(N)$ has the center $[\varphi_0]$ of order $2$
which preserves components $4\aaa_6$, and \newline
$A(N)/[\varphi_0]={\mathfrak{A}}_4$ is the alternating group
on components of $4\aaa_6$ (see \cite[Ch. 16]{CS}).
Its elements are defined by permutations
of the terminals of the Dynkin diagrams $\aaa_6$.
It is generated by the involution
$$
\varphi_0=(\alpha_{11}\alpha_{61})(\alpha_{12}\alpha_{62})(\alpha_{13}\alpha_{63})(\alpha_{14}\alpha_{64})
$$
and by elements $\varphi_1$ and $\varphi_4$ of order $3$,
$$
\varphi_1=(\alpha_{12}\alpha_{13}\alpha_{14})(\alpha_{62}\alpha_{63}\alpha_{64}),\ \
\varphi_2=(\alpha_{11}\alpha_{12}\alpha_{63})(\alpha_{61}\alpha_{62}\alpha_{13})
$$
where $\varphi_4\varphi_1$ has order $4$ and
$(\varphi_4\varphi_1)^2=\varphi_0$.
All elements of $A(N)$ are conjugate to $\varphi_0$, $\varphi_1$,
$\varphi_0\varphi_1$ (of order $6$) and $\varphi_4\varphi_1$.

The coninvariant sublattice $N_{[\varphi_0]}$ has
$\rk N_{[\varphi_0]}=12$ and
$N_{[\varphi_0]}^\ast /N_{[\varphi_0]}\cong (\bz/2\bz)^{12}$.
By Theorem \ref{th:primembb3}, $N_{[\varphi_0]}$ has no
primitive embeddings into $L_{K3}$, and $[\varphi_0]$ is
not KahK3 subgroup.

For the cyclic group $[\varphi_1]\cong C_3$, we have
$\rk N_{[\varphi_1]}=12$ and
$N_{[\varphi_1]}^\ast/N_{[\varphi_1]}\cong (\bz/3\bz)^6$.
By Theorem \ref{th:primembb3}, $N_{[\varphi_1]}$ has a primitive
embedding into $L_{K3}$, and $[\varphi_1]$
is KahK3 subgroup. Its conjugate give all cyclic subgroups of $A(N)$
of order $3$. They are
$[\varphi_1]$, $[\varphi_4]$ and $[\varphi_2]$, $[\varphi_3]$ where
$\varphi_2=\varphi_1\varphi_4\varphi_1^{-1}$, $\varphi_3=\varphi_1^2\varphi_4\varphi_1^{-2}$.
From the structure of $A(N)$, it follows that {\it conjugate subgroups
$$
H=[\varphi_i]\cong C_3,\ i=1,2,3,4,
$$
are all non-trivial KahK3 subgroups of $A(N_{17})$. We have
$\rk N_H=12$, $N_H^\ast/N_H\cong (\bz/3\bz)^6$.}

Let $X$ be marked by a primitive sublattice
$S\subset N_{17}=N(4A_6)$. Then $S$
must satisfy Theorem \ref{th:primembb3} and
$\Gamma(P(S))\subset \Gamma(P(N_{17}))=4\aaa_6$.
Any such $S$ gives marking of some $X$ and $P(X)\cap S=P(S)$.

If $N_{[\varphi_i]}\subset S$, for one of $i=1,\,2,\,3,\,4$,
then $\Aut(X,S)_0=[\varphi_i]\cong C_3$.
Otherwise, if $N_{[\varphi_i]}\not\subset S$ for all
$i=1,\,2,\,3,\,4$, then $\Aut(X,S)_0$ is trivial.

Let $S=[(A_6)_2=[\alpha_{12},\dots,\,\alpha_{62}],\
N_{[\varphi_1]}]_{pr}\subset N_{17}$.
We have $\rk S=18$, $S^\ast/S\cong \bz/7\bz$.
Thus, $S$ satisfies Theorem \ref{th:primembb3} and gives marking
of some $X$.
We have  $\Aut(X,S)_0\cong C_3$ and $3\aaa_6\subset
\Gamma(P(X)\cap S)=P(S)$.
By classification  of Niemeier lattices and our calculations
above, $X$
can be marked by the Niemeier lattice $N_{17}=N(4A_6)$ only
for such $S\subset S_X$.

\medskip

{\bf Case 14.} For the Niemeier lattice
$$
N=N_{14}=N(4D_6)=
[4D_6,\ even\ permutations\ of\
0_1+\varepsilon_{12}+\varepsilon_{23}+\varepsilon_{34}]=
[4D_6,\ \varepsilon_{12}+\varepsilon_{23}+\varepsilon_{34},\
$$
$$
\varepsilon_{32}+\varepsilon_{13}+\varepsilon_{24},\
\varepsilon_{22}+\varepsilon_{33}+\varepsilon_{14};\
\varepsilon_{11}+\varepsilon_{33}+\varepsilon_{24},\
\varepsilon_{21}+\varepsilon_{13}+\varepsilon_{34},\
\varepsilon_{31}+\varepsilon_{23}+\varepsilon_{14};\
\varepsilon_{21}+\varepsilon_{32}+\varepsilon_{14},\
$$
$$
\varepsilon_{11}+\varepsilon_{22}+\varepsilon_{34},\
\varepsilon_{31}+\varepsilon_{12}+\varepsilon_{24};\
\varepsilon_{31}+\varepsilon_{22}+\varepsilon_{13},\
\varepsilon_{11}+\varepsilon_{32}+\varepsilon_{23},\
\varepsilon_{21}+\varepsilon_{12}+\varepsilon_{33}]
$$
the group $A(N)$ can be identified with ${\mathfrak{S}}_4$ by
its action on components of $4\ddd_6$ (see \cite[Ch. 16]{CS}).
Its elements are defined by permutations
of the terminals of the connected components of
Dynkin diagrams $4\ddd_6$. We numerate these components
by $1,\,2,\,3,\,4$.
Even permutations give the corresponding permutations of
$(\alpha_{51},\,\alpha_{52},\,\alpha_{53},\,\alpha_{54})$ and
$(\alpha_{61},\,\alpha_{62},\,\alpha_{63},\,\alpha_{64})$.
Odd permutations change $\alpha_{5i}$ and $\alpha_{6i}$ places.
For example, the transposition $(12)$ gives the permutation
$$
(12)=(\alpha_{51}\alpha_{62})(\alpha_{61}\alpha_{52})(\alpha_{53}\alpha_{63})(\alpha_{54}\alpha_{64}).
$$

For $(12)(34)\in {\mathfrak S}_4$ of order $2$,  we have
$\rk N_{[(12)(34)]}=12$ and \linebreak
$N_{[(12)(34)]}^\ast/N_{[(12)(34)]}\cong (\bz/2\bz)^{12}$.
By Theorem \ref{th:primembb3}, $N_{[(12)(34)]}$ has no
primitive embeddings into $L_{K3}$, and $[(12)(34)]$ and its
conjugate are
not KahK3 subgroups.

Let $({\mathfrak D}_6)_4={\mathfrak S}_3(1,2,3)$ consists of
all elements of ${\mathfrak S}_4$ which preserve $4$. Its conjugate we
similarly denote as
$({\mathfrak D}_6)_k$, $k=1,2,3,4$. They are isomorphic to
the dihedral group $\mathfrak{D}_6$ of order $6$.
We have
$$
\rk N_{({\mathfrak D}_6)_4}=14,\ \ \
N_{({\mathfrak D}_6)_4}^\ast /N_{({\mathfrak D}_6)_4}
\cong (\bz/6\bz)^2\times (\bz/3\bz)^3.
$$
By Theorem \ref{th:primembb3}, $N_{({\mathfrak D}_6)_4}$ has a
primitive embedding into $L_{K3}$,
and $({\mathfrak D}_6)_4$ is KahK3 subgroup.

Subgroups of $({\mathfrak D}_6)_4$ are also KahK3 subgroups.
Non-trivial ones are conjugate to
$[(12)]\cong C_2$ and $(C_3)_4=[(123)]\cong C_3$.
We have
$N_{[(12)]}\cong E_8(2)$, and
$$
\rk N_{[(123)]}=12,\ \ \
N_{[(123)]}^\ast/N_{[(123)]}\cong (\bz/3\bz)^6.
$$

From the structure of ${\mathfrak S}_4$, we obtain that {\it all
non-trivial
KahK3 subgroups of $A(N_{14})$ are conjugate subgroups
$({\mathfrak D}_6)_k$, $k=1,2,3,4$, isomorphic to ${\mathfrak D}_6$;
conjugate
subgroups $[(ij)]$, $1\le i<j\le 4$, isomorphic to $C_2$; conjugate
subgroups $(C_3)_k\subset ({\mathfrak D}_6)_k$, $k=1,2,3,4$,
isomorphic to $C_3$.}

Let $X$ be marked by a primitive sublattice
$S\subset N_{14}=N(4D_6)$. Then $S$
must satisfy Theorem \ref{th:primembb3} and
$\Gamma(P(S))\subset \Gamma(P(N_{14}))=4\ddd_6$.
Any such $S$ gives marking of some $X$ and $P(X)\cap S=P(S)$.

If $N_{({\mathfrak D}_6)_k}\subset S$ for one of $k=1,\,2,\,3,\,4$,
then $\Aut(X,S)_0=({\mathfrak D}_6)_k\cong {\mathfrak D}_6$.
Otherwise, if $N_{(C_3)_k}\subset S$ for one of $k=1,\,2,\,3,\,4$,
only, then $\Aut(X,S)_0=(C_3)_k\cong C_3$.
Otherwise, if $N_{[(ij)]}\subset S$ for $1\le i<j\le 4$,
only, then $\Aut(X,S)_0=[(ij)]\cong C_2$.
Otherwise, $\Aut(X,S)_0$ is trivial.

Let $S=[(D_6)_1=[\alpha_{11},\dots,\,\alpha_{61}],\
N_{({\mathfrak D}_6)_4}]_{pr}\subset N_{14}$.
We have $\rk S=19$, $S^\ast/S\cong \bz/4\bz$.
Thus, $S$ satisfies Theorem \ref{th:primembb3} and gives
marking of some $X$.
We have  $\Aut(X,S)_0\cong {\mathfrak D}_6$ and $3\ddd_6\subset
\Gamma(P(X)\cap S)=P(S)$.
By classification  of Niemeier lattices and our calculations above,
$X$ can be marked by the Niemeier lattice $N_{14}=N(4D_6)$ only for
such $S\subset S_X$.

\medskip

{\bf Case 12.} For the Niemeier lattice
$$
N=N_{12}=N(4E_6)=
[4E_6,\
\varepsilon_{11}+\varepsilon_{13}+\varepsilon_{24},\
\varepsilon_{11}+\varepsilon_{22}+\varepsilon_{14},\,
\varepsilon_{11}+\varepsilon_{12}+\varepsilon_{23}],
$$
the group $A(N)$ has the center $[\varphi_0]$ of order $2$
which preserves components $4\eee_6$, and \newline
$A(N)/[\varphi_0]={\mathfrak{S}}_4$ is the symmetric group
on components of $4\eee_6$ (see \cite[Ch. 16]{CS}).
Its elements are defined by permutations
of the terminals of the Dynkin diagrams $\eee_6$.
It is generated by the involutions
$$
\varphi_0=(\alpha_{11}\alpha_{61})(\alpha_{12}\alpha_{62})
(\alpha_{13}\alpha_{63})(\alpha_{14}\alpha_{64}),\ \ \
\widetilde{(12)}=(\alpha_{11}\alpha_{12})(\alpha_{61}\alpha_{62})
(\alpha_{14}\alpha_{64}),
$$
$$
\widetilde{(23)}=(\alpha_{11}\alpha_{61})(\alpha_{12}\alpha_{13})(\alpha_{62}\alpha_{63}),\ \ \
\widetilde{(34)}=(\alpha_{11}\alpha_{61})(\alpha_{13}\alpha_{14})(\alpha_{63}\alpha_{64}).
$$
Similarly we define involutions $\widetilde{(ij)}$, $1\le i<j\le 4$,
which act as transpositions of components $(\eee_6)_i$, $(\eee_6)_j$
and elements $\alpha_{1i}$, $\alpha_{1j}$.

The element $\widetilde{(34)}\widetilde{(12)}$ has order $4$, and
$(\widetilde{(34)}\widetilde{(12)})^2=\varphi_0$.
The coinvariant sublattice $N_{[\widetilde{(34)}\widetilde{(12)}]}$ has
$\rk N_{[\widetilde{(34)}\widetilde{(12)}]}=16$, and
$N_{[\widetilde{(34)}\widetilde{(12)}]}^\ast/
N_{[\widetilde{(34)}\widetilde{(12)}]}=
(\bz/4\bz)^4\times (\bz/2\bz)^4$.
By Theorem \ref{th:primembb3}, $N_{[\widetilde{(34)}\widetilde{(12)}]}$
has no
primitive embeddings into $L_{K3}$, and
${[\widetilde{(34)}\widetilde{(12)}]}$ is not KahK3 subgroup.

For $k=1,2,3,4$, let $({\mathfrak D}_{12})_k$ consists
of all elements of $A(N_{12})$
which preserve the component $(\eee_6)_k$.
The conjugate subgroups
$({\mathfrak D}_{12})_k$ are isomorphic to the dihedral
group ${\mathfrak D}_{12}$ of order $12$.
We have
$$
\rk N_{({\mathfrak D}_{12})_k}=16, \ \
N_{({\mathfrak D}_{12})_k}^\ast/N_{({\mathfrak D}_{12})_k}
\cong (\bz/6\bz)^4.
$$
By Theorem \ref{th:primembb3},
$N_{({\mathfrak D}_{12})_k}$ have primitive embeddings into
$L_{K3}$, and $({\mathfrak D}_{12})_k$ are KahK3 subgroups.

From the structure of $A(N)$, it follows that
{\it conjugate subgroups
$$
({\mathfrak D}_{12})_k\cong {\mathfrak D}_{12},\ k=1,2,3,4,
$$
and all their subgroups give all KahK3 subgroups of $A(N_{12})$.}

These non-trivial subgroups
are (where $k$ marks subgroups of $({\mathfrak D}_{12})_k$):

\medskip

\noindent
isomorphic to $C_6$ conjugate to
$(C_6)_4=[\widetilde{(12)}\widetilde{(23)}]$ subgroups
$$
(C_6)_{k},\ k=1,2,3,4,
$$
with $N_{(C_6)_k}=N_{({\mathfrak D}_{12})_k}$
(therefore, $Clos(C_6)_k=({\mathfrak D}_{12})_k$);
isomorphic to ${\mathfrak D}_{6}$ subgroups
$$
({\mathfrak D}_{6})_{ki},\ k=1,2,3,4,\ i=1,2,
$$
with $\rk N_{({\mathfrak D}_{6})_{ki}}=14$, and
$N_{({\mathfrak D}_{6})_{ki}}^\ast/N_{({\mathfrak D}_{6})_{ki}}\cong
(\bz/6\bz)^2\times (\bz/3\bz)^3$;
isomorphic to $C_3$ conjugate subgroups
$$
(C_3)_{k},\ k=1,2,3,4,
$$
with $\rk N_{(C_3)_{k}}=12$ and $N_{(C_3)_{k}}^\ast/N_{(C_3)_{k}}\cong
(\bz/3\bz)^6$;
isomorphic to $C_2\times C_2$ conjugate subgroups
$$
[\widetilde{(ij)},\varphi_0],\ 1\le i<j\le 4,
$$
with $N_{[\widetilde{(ij)},\varphi_0]}=12$, and
$N_{[\widetilde{(ij)},\varphi_0]}^\ast /
N_{[\widetilde{(ij)},\varphi_0]}\cong
(\bz/4\bz)^2\times
(\bz/2\bz)^6$;
isomorphic to $C_2$ subgroups
$$
[\varphi_0],\ \ [\widetilde{(ij)}],\
[\widetilde{(ij)}\varphi_0],\ \ 1\le i<j\le 4
$$
with coinvariant lattices isomorphic to $E_8(2)$.

\medskip

Let $X$ be marked by a primitive sublattice
$S\subset N_{12}=N(4E_6)$. Then $S$
must satisfy Theorem \ref{th:primembb3} and
$\Gamma(P(S))\subset \Gamma(P(N_{17}))=4\eee_6$.
Any such $S$ gives marking of some $X$ and $P(X)\cap S=P(S)$.

If $N_{({\mathfrak D}_{12})_k}\subset S$ for one of
$k=1,\,2,\,3,\,4$,
then $\Aut(X,S)_0={({\mathfrak D}_{12})_k}\cong {\mathfrak D}_{12}$.
Otherwise, if $N_{({\mathfrak D}_{6})_{ki}}\subset S$ only
for one of $k=1,\,2,\,3,\,4$, $i=1,\,2$, then
$\Aut(X,S)_0={({\mathfrak D}_{6})_{ki}}\cong {\mathfrak D}_{6}$.
Otherwise, if $N_{(C_3)_{k}}\subset S$ only for one of
$k=1,\,2,\,3,\,4$, then
$\Aut(X,S)_0={(C_3)_{k}}\cong C_3$.
Otherwise, if $N_{[\widetilde{(ij)},\varphi_0]}\subset S$ only
for one of $1\le i<j\le 4$, then
$\Aut(X,S)_0=[\widetilde{(ij)},\varphi_0]\cong C_2\times C_2$.
Otherwise, if $N_H\subset S$ only
for one of
$H\in \{[\varphi_0]\}\cup \{[\widetilde{(ij)}],\
[\widetilde{(ij)}\varphi_0],\ |\  1\le i<j\le 4\}$, then
$\Aut(X,S)_0=H\cong C_2$. Otherwise,
$\Aut(X,S)_0$ is trivial.

Let $S=[(E_6)_1=[\alpha_{11},\dots,\,
\alpha_{61}],N_{(C_3)_{4}}]_{pr}\subset N_{12}$.
We have $\rk S=18$, $S^\ast/S\cong \bz/3\bz$.
Thus, $S$ satisfies Theorem \ref{th:primembb3} and
gives marking of some $X$.
We have  $\Aut(X,S)_0\cong C_3$ and $3\eee_6\subset
\Gamma(P(X)\cap S)=P(S)$.
By classification  of Niemeier lattices and our calculations
above, $X$ can be marked by the Niemeier lattice $N_{12}=N(4E_6)$
only for such $S\subset S_X$.


\vskip0.5cm

{\bf Case 18.} For the Niemeier lattice
$$
N=N_{18}=N(4A_5\oplus D_4)=
[4A_5\oplus D_4,\
2\varepsilon_{11}+2\varepsilon_{13}+4\varepsilon_{14},\
2\varepsilon_{11}+4\varepsilon_{12}+2\varepsilon_{14},\,
$$
$$
2\varepsilon_{11}+2\varepsilon_{12}+4\varepsilon_{13},\,
3\varepsilon_{11}+3\varepsilon_{12}+\varepsilon_{15},\
3\varepsilon_{11}+3\varepsilon_{13}+\varepsilon_{25},\,
3\varepsilon_{11}+3\varepsilon_{14}+\varepsilon_{35}]
$$
the group $A=A(N)$ has the center $[\varphi_0]$ of order $2$
which preserves components $4\aaa_5$ and $\ddd_4$.
The group $A(N)/[\varphi_0]={\mathfrak{S}}_4$
is the symmetric group on components of $4\aaa_5$
(see \cite[Ch. 16]{CS}).
Elements of $A$ are defined by their actions on
the terminals of the Dynkin diagrams $4\aaa_5\ddd_4$.
It is generated by the involutions
$$
\varphi_0=(\alpha_{11}\alpha_{51})(\alpha_{12}\alpha_{52})(\alpha_{13}\alpha_{53})(\alpha_{14}\alpha_{54}),\ \ \
\widetilde{(12)}=(\alpha_{11}\alpha_{12})(\alpha_{51}\alpha_{52})(\alpha_{14}\alpha_{54})(\alpha_{15}\alpha_{35}),
$$
$$
\widetilde{(23)}=
(\alpha_{11}\alpha_{51})(\alpha_{12}\alpha_{13})(\alpha_{52}\alpha_{53})(\alpha_{15}\alpha_{45}),\ \ \
\widetilde{(34)}=
(\alpha_{11}\alpha_{51})(\alpha_{13}\alpha_{14})(\alpha_{53}\alpha_{54})(\alpha_{15}\alpha_{35})\,.
$$
Similarly we define involutions $\widetilde{(ij)}$, $1\le i<j\le 4$,
which act as transpositions of components $(\aaa_5)_i$, $(\aaa_5)_j$
and elements $\alpha_{1i}$, $\alpha_{1j}$.

The group $A=A(N)$ is isomorphic to the group
$T_{48}=Q_8\rtimes {\mathfrak S}_3$
(standard notation, see \cite{Muk}).
The coinvariant sublattice $N_A$ has $\rk N_A=19$ and
${N_A}^\ast/N_A\cong \bz/24\bz\times \bz/8\bz\times \bz/2\bz$.

By Theorem \ref{th:primembb3}, $N_A$
has a primitive embedding into $L_{K3}$, and $A=A(N)$
is a KahK3 subgroup. Therefore, any its subgroup
$H\subset A$ is also a KahK3 subgroup. It seems, these
facts were first observed in \cite{Hash}.

The group $A=A(N)$ has, up to conjugation,  the following
and only the following non-trivial subgroups $H\subset A$
and invariants of their coinvariant sublattices $N_H$.

(1) A subgroup $H$ isomorphic to $T_{24}$ (the binary tetrachedral group)
consists of
elements of $N$ which give even permutations of four components
of $4\aaa_5$. The coinvariant sublattice $N_H=N_A$. Therefore, $Clos(H)=A$
is isomorphic to $T_{48}$,
and $\rk N_H=19$, $N_H^\ast/N_H\cong \bz/24\bz\times \bz/8\bz\times \bz/2\bz$.

(2) Subgroups $H$ isomorphic to $SD_{16}$ ($2$-Sylow subgroups of $A$)
are conjugate to \newline
$[\widetilde{(34)}\widetilde{(23)}\widetilde{(12)},\,
\widetilde{(13)},\, \widetilde{(24)}]$
with $\rk N_H=18$ and $N_H^\ast/N_H\cong (\bz/8\bz)^2\times \bz/4\bz\times
\bz/2\bz$.

(3) Subgroups $H$ isomorphic to $C_8$ are conjugate to $[\widetilde{(34)}
\widetilde{(23)}\widetilde{(12)}]$,
with \newline
$Clos([\widetilde{(34)}\widetilde{(23)}\widetilde{(12)}])=
[\widetilde{(34)}\widetilde{(23)}\widetilde{(12)},\,\widetilde{(13)},\,
\widetilde{(24)}]$ isomorphic to $SD_{16}$.
Thus,
$
N_{[\widetilde{(34)}\widetilde{(23)}\widetilde{(12)}]}=
N_{[\widetilde{(34)}\widetilde{(23)}\widetilde{(12)},\,
\widetilde{(13)},\,\widetilde{(24)}]}
$
and
$\rk N_H=18$, $N_H^\ast/N_H\cong (\bz/8\bz)^2\times \bz/4\bz\times \bz/2\bz$.

(4) A subgroup $H$ isomorphic to $Q_8$ is
$H=[\widetilde{(34)}\widetilde{(12)},\, \widetilde{(24)}\widetilde{(13)}]$
with $\rk N_H=17$ and $N_H^\ast/N_H\cong (\bz/8\bz)^2\times (\bz/2\bz)^3$.

(5) Subgroups $H$ isomorphic to ${\mathfrak D}_{12}$ are conjugate to
$[\widetilde{(12)},\, \widetilde{(23)}]$ with
$\rk N_H=16$, $N_H^\ast/N_H\cong (\bz/6\bz)^4$.

(6) Subgroups $H$ isomorphic to $C_6$ are conjugate to
$[\widetilde{(23)}\widetilde{(12)}]$ with \linebreak
$Clos\,([\widetilde{(23)}\widetilde{(12)}])=[\widetilde{(12)},\
\widetilde{(23)}]$ isomorphic to
${\mathfrak D}_{12}$. Thus,
$N_{[\widetilde{(23)}\widetilde{(12)}]}=N_{[\widetilde{(12)},\
\widetilde{(23)}]}$ and
$\rk N_H=16$ and $N_H^\ast/N_H\cong (\bz/6\bz)^4$.

(7) Subgroups $H$ isomorphic to ${\mathfrak D}_{8}$ are conjugate to
$[\widetilde{(12)},\, \widetilde{(34)}]$ with
$\rk N_H=15$, $N_H^\ast/N_H\cong (\bz/4\bz)^5$.

(8) Subgroups $H$ isomorphic to ${\mathfrak D}_6$ are conjugate to
$[\varphi_0\widetilde{(23)},\, \widetilde{(12)}]$ or
$[\varphi_0\widetilde{(13)},\ \widetilde{(23)}]$
with $\rk N_H=14$ and $N_H^\ast/N_H\cong (\bz/6\bz)^2\times (\bz/3\bz)^3$.

(9) Subgroups $H$ isomorphic to $C_4$ are conjugate to
$[\widetilde{(34)}\widetilde{(12)}]$,
with $\rk N_H=14$ and $N_H^\ast/N_H\cong (\bz/4\bz)^4\times (\bz/2\bz)^2$.

(10) Subgroups $H$ isomorphic to $C_2\times C_2$ are
$[\widetilde{(ij)},\ \varphi_0]$,
$1\le i<j\le 4$, with $\rk N_H=12$ and
$N_H^\ast/N_H\cong (\bz/4\bz)^2\times (\bz/2\bz)^6$.

(11) Subgroups $H$ isomorphic to $C_3$ are conjugate to
$[\varphi_0\widetilde{(23)}\widetilde{(12)}]$
with $\rk N_H=12$ and $N_H^\ast/N_H\cong (\bz/3\bz)^6$.

(12) Subgroups $H$ isomorphic to $C_2$ are $[\varphi_0]$,
$[\widetilde{(ij)}]$
and $[\widetilde{(ij)}\varphi_0]$, $1\le i<j\le 4$. The lattice
$N_H\cong E_8(2)$.

Let $X$ be marked by a primitive sublattice
$S\subset N=N_{18}=N(4A_5\oplus D_4)$. Then $S$
must satisfy Theorem \ref{th:primembb3} and
$\Gamma(P(S))\subset \Gamma(P(N_{18}))=4\aaa_5\ddd_4$.
Any such $S$ gives marking of some $X$ and $P(X)\cap S=P(S)$.

If $S=N_A$,
then $\Aut(X,S)_0=A\cong T_{48}$.
Otherwise, if only $N_H\subset S$, $H\cong SD_{16}$ or
$H\cong {\mathfrak D}_{12}$,
then $\Aut(X,S)_0=H$.
Otherwise, if only $N_H\subset S$, $H\cong Q_8,\ {\mathfrak D}_8$ or
${\mathfrak D}_6$, then
$\Aut(X,S)_0=H$.
Otherwise, if only $N_H\subset S$, $H\cong C_4$ or $H\cong C_3$,
then
$\Aut(X,S)_0=H$.
Otherwise, if only $N_H\subset S$, $H\cong C_2$, then
$\Aut(X,S)_0=H\cong C_2$. Otherwise,
$\Aut(X,S)_0$ is trivial.

Let $H=[\widetilde{(12)},\, \widetilde{(34)}]\cong {\mathfrak D}_8$.
Let $S=[(A_5)_1=[\alpha_{11},\dots,\,
\alpha_{51}],\ N_H]_{pr}\subset N_{18}$.
We have $\rk S=18$, $S^\ast/S\cong (\bz/4\bz)^4\times (\bz/2\bz)^2$ and
$q_{S_2}=q_\theta^{(2)}(2)\oplus q^\prime$.
Thus, $S$ satisfies Theorem \ref{th:primembb3} and
gives marking of some $X$.
We have  $H\subset \Aut(X,S)_0$, $H\cong {\mathfrak D}_8$,
and $2\aaa_5\subset\Gamma(P(X)\cap S)=P(S)$.
By classification  of Niemeier lattices and our calculations
above, $X$ can be marked by the Niemeier lattice $N_{18}=N(4A_5\oplus D_4)$
only for such $S\subset S_X$.

\vskip0.5cm

{\bf Case 19.} For the Niemeier lattice
$$
N=N_{19}=N(6D_4)=
$$
$$
[6D_4,\ \varepsilon_{13}+\varepsilon_{14}+\varepsilon_{15}+\varepsilon_{16},\,
\varepsilon_{12}+\varepsilon_{14}+\varepsilon_{25}+\varepsilon_{36},\,
\varepsilon_{11}+\varepsilon_{14}+\varepsilon_{35}+\varepsilon_{26},\,
$$
$$
\varepsilon_{23}+\varepsilon_{24}+\varepsilon_{25}+\varepsilon_{26},\,
\varepsilon_{22}+\varepsilon_{24}+\varepsilon_{35}+\varepsilon_{16},\,
\varepsilon_{21}+\varepsilon_{24}+\varepsilon_{15}+\varepsilon_{36}]
$$
the group $A=A(N)$ consists of the cyclic group
$[\varphi ]$ of order $3$ which preserves connected
components of $6\ddd_4$ and $A/[\varphi]={\mathfrak S}_6$
gives the symmetric group on the $6$ components of $6\ddd_4$.
(see \cite[Chs. 16, 18]{CS}).
Its elements are defined by their actions on terminals of
components of $6\ddd_4$. This group is generated by
$$
\varphi =
(\alpha_{11}\alpha_{31}\alpha_{41})
(\alpha_{12}\alpha_{32}\alpha_{42})
(\alpha_{13}\alpha_{33}\alpha_{43})
(\alpha_{14}\alpha_{34}\alpha_{44})
(\alpha_{15}\alpha_{35}\alpha_{45})
(\alpha_{16}\alpha_{36}\alpha_{46}),
$$
and by the involutions
$$
\widetilde{(12)}=
(\alpha_{41}\alpha_{42})(\alpha_{11}\alpha_{32})
(\alpha_{31}\alpha_{12})(\alpha_{13}\alpha_{33})
(\alpha_{14}\alpha_{34})(\alpha_{15}\alpha_{35})
(\alpha_{16}\alpha_{36}),
$$
$$
\widetilde{(23)}=
(\alpha_{11}\alpha_{31})(\alpha_{12}\alpha_{33})
(\alpha_{32}\alpha_{13})(\alpha_{42}\alpha_{43})
(\alpha_{14}\alpha_{34})(\alpha_{15}\alpha_{45})
(\alpha_{36}\alpha_{46}),
$$
$$
\widetilde{(34)}=
(\alpha_{11}\alpha_{31})(\alpha_{12}\alpha_{32})
(\alpha_{43}\alpha_{44})(\alpha_{13}\alpha_{34})
(\alpha_{33}\alpha_{14})(\alpha_{15}\alpha_{35})
(\alpha_{16}\alpha_{36}),
$$
$$
\widetilde{(45)}=
(\alpha_{11}\alpha_{41})(\alpha_{32}\alpha_{42})
(\alpha_{13}\alpha_{33})(\alpha_{44}\alpha_{45})
(\alpha_{14}\alpha_{35})(\alpha_{34}\alpha_{15})
(\alpha_{16}\alpha_{36}),
$$
$$
\widetilde{(56)}=
(\alpha_{11}\alpha_{31})(\alpha_{12}\alpha_{32})
(\alpha_{13}\alpha_{33})(\alpha_{14}\alpha_{34})
(\alpha_{15}\alpha_{36})(\alpha_{35}\alpha_{16})
(\alpha_{45}\alpha_{46}).
$$
Here and further, for the canonical homomorphism
$\pi:A \to {\mathfrak S}_6$, we denote by $\widetilde{x}$
an element in $\pi^{-1}(x)$ for $x\in {\mathfrak S}_6$.
The same for a subgroup $G\subset {\mathfrak S}_6$.
Thus, $\widetilde{G}\subset A$ is a subgroup of $A$ such that
$\pi$ gives an isomorphism of $\widetilde{G}$ and $G$.
By $\widetilde{\widetilde{G}}$
we denote $\pi^{-1}(G)$ for $G\subset {\mathfrak S}_6$.

\medskip

Let us describe KahK3 subgroups $H$ of $A$.
They can be of two types.

\medskip

{\it Type I:} A KahK3 subgroup $H\subset A=A(N_{19})$ contains
the subgroup $[\varphi]$. Then it is
$H=\widetilde{\widetilde{G}}=\pi^{-1}(G)$
for $G\subset {\mathfrak S}_6$. It is sufficient to describe
all possible conjugation classes of $G$ in ${\mathfrak S}_6$.

Since $\rk N_H\le 19$ for KahK3 subgroup $H$, the group $G$ must
have more than two orbits in $\{1,\,2,\,\dots,\,6\}$.
Thus, $G\subset {\mathfrak S}_{1,1,4}$,
$G\subset {\mathfrak S}_{1,2,3}$
or $G\subset {\mathfrak S}_{2,2,2}$. We use the standard notation
from \cite{Muk}.
For example $\SSS_{1,2,3}$ consists of all permutations of
$\SSS_6$ which preserve subsets $\{1\}$, $\{2,\,3\}$ and $\{4,\,5,\,6\}$.
Its alternating subgroup is denoted by $\AAA_{1,2,3}$.

If $H=\widetilde{\widetilde{\SSS}}_{1,1,4}=[\varphi,\, \widetilde{(34)},\,
\widetilde{(45)},\,\widetilde{(56)}]\cong \AAA_{3,4}$, then
$\rk N_H=18$ and $(N_H)^\ast/N_H\cong (\bz/12\bz)^2\times \bz/3\bz$.
By Theorem \ref{th:primembb3},
$N_H$ has a primitive embedding
into $L_{K3}$ and $H$ is KahK3 subgroup.

If $H=\widetilde{\widetilde{\SSS}}_{1,2,3}=[\varphi,\,
\widetilde{(23)},\,
\widetilde{(45)},\,\widetilde{(56)}]\cong \SSS_{3,3}$, then
$\rk N_H=18$ and $(N_H)^\ast/N_H\cong \bz/18\bz\times
\bz/6\bz \times (\bz/3\bz)^2$.
For the 3-component $q_{{N_H}_3}$ of the discriminant
form $q_{N_H}$ of $N_H$
we have
$$
\det K(q_{{N_H}_3})=\det
\left(\begin{array}{cccc}
-6948 & -5022 & 108 & 144 \\
-5022 & -3684 & 78  & 108\\
108   &    78 & -18 & 0   \\
144   &  108  &   0 & -12
\end{array}\right)
$$
over $\bz_3$ and $\det K(q_{{N_H}_3})\equiv 2^5\cdot 3^5\cdot 61\cdot 109
\equiv -3^5\mod (\bz_3^\ast)^2$.
By Theorem \ref{th:primembb3},
$N_H$ has a primitive embedding
into $L_{K3}$ and $H$ is KahK3 subgroup.

If $H=\widetilde{\widetilde{\SSS}}_{2,2,2}=[\varphi,\, \widetilde{(12)},\,
\widetilde{(34)},\,\widetilde{(56)}]$ where
$\SSS_{2,2,2}\cong (C_2)^3$, then
$\rk N_H=18$ and $(N_H)^\ast/N_H\cong (\bz/6\bz)^3\times (\bz/2\bz)^3$.
By Theorem \ref{th:primembb3},
$N_H$ has no a primitive embedding
into $L_{K3}$ and $H$ is not a KahK3 subgroup. The same coinvariant
sublattice $N_H$ have subgroups
$H\subset \widetilde{\widetilde{\SSS}}_{2,2,2}$ which are
conjugate to
$[\varphi, \widetilde{(12)}\widetilde{(34)},\widetilde{(56)}]$,
$[\varphi, \widetilde{(12)}\widetilde{(34)},\,\widetilde{(34)}
\widetilde{(56)}]$
or  $[\varphi, \widetilde{(12)}\widetilde{(34)}\widetilde{(56)}]$.

It follows that KahK3 subgroups of the type I are exactly
subgroups of the type I which are conjugate to subgroups of
$\widetilde{\widetilde{\SSS}}_{1,1,4}=[\varphi,\, \widetilde{(34)},\,
\widetilde{(45)},\,\widetilde{(56)}]\cong \AAA_{3,4}$ and
$\widetilde{\widetilde{\SSS}}_{1,2,3}=[\varphi,\, \widetilde{(23)},\,
\widetilde{(45)},\,\widetilde{(56)}]\cong \SSS_{3,3}$.
It is easy to enumerate such subgroups.

Those of them which are conjugate to type I subgroups of
$\widetilde{\widetilde{\SSS}}_{1,1,4}=[\varphi,\, \widetilde{(34)},\,
\widetilde{(45)},\,\widetilde{(56)}]\cong \AAA_{3,4}$ are conjugate
to one of subgroups (I.1) --- (I.11) below.

(I.1) $H=[\varphi,\, \widetilde{(34)},\,\widetilde{(45)},\,\widetilde{(56)}]=
\widetilde{\widetilde{\SSS}}_{1,1,4}\cong
\AAA_{3,4}$ with
$\rk N_H=18$ and $(N_H)^\ast/N_H\cong (\bz/12\bz)^2\times \bz/3\bz$.

(I.2)
$H=[\varphi,\, \widetilde{(34)}\widetilde{(45)},\,
\widetilde{(34)}\widetilde{(56)}]=
\widetilde{\widetilde{\AAA}}_{1,1,4}\cong
C_3\times \AAA_{4}$ with $Clos(H)=[\varphi,\,
\widetilde{(34)},\,\widetilde{(45)},\,\widetilde{(56)}]$ \newline
$=\widetilde{\widetilde{\SSS}}_{1,1,4}\cong
\AAA_{3,4}$ above and
$\rk N_H=18$, $(N_H)^\ast/N_H\cong (\bz/12\bz)^2\times \bz/3\bz$.

(I.3) $H=[\varphi,\, \widetilde{(34)}\widetilde{(45)}\widetilde{(56)},
\widetilde{(35)}]=
\widetilde{\widetilde{D_8}}\cong D_8\ltimes  C_3$ with
$Clos(H)=[\varphi,\, \widetilde{(34)},\,\widetilde{(45)},\,\widetilde{(56)}]=
\widetilde{\widetilde{\SSS}}_{1,1,4}\cong
\AAA_{3,4}$ above and
$\rk N_H=18$, $(N_H)^\ast/N_H\cong (\bz/12\bz)^2\times \bz/3\bz$.

(I.4) $H=[\varphi,\, \widetilde{(34)}\widetilde{(56)},\,
\widetilde{(34)}\widetilde{(45)}\widetilde{(34)}\widetilde{(56)}
\widetilde{(45)}\widetilde{(56)}]=
\widetilde{\widetilde{K_4}}\cong C_2\times  C_6$ with
$Clos(H)=[\varphi,\, \widetilde{(34)},\,\widetilde{(45)},\,\widetilde{(56)}]=
\widetilde{\widetilde{\SSS}}_{1,1,4}\cong
\AAA_{3,4}$ above and
$\rk N_H=18$, $(N_H)^\ast/N_H\cong (\bz/12\bz)^2\times \bz/3\bz$.
Here
$K_4\cong C_2\times C_2$ is the normal Klein subgroup of
$\SSS_{1,1,4}\cong \SSS_4$.
\footnote{These calculations show that for a symplectic group
$G=C_2\times C_6$ on a
K\"ahlerian K3 surface, the group
$S_{(G)}^\ast/S_{(G)}=S_{(2,6)}^\ast /S_{(2,6)}\cong
(\bz/12\bz)^2\times \bz/3\bz$.
We must correct our calculation of this group in
\cite[Prop. 10.1]{Nik0}.}

(I.5) $H=[\varphi,\, \widetilde{(34)}\widetilde{(45)}\widetilde{(56)}]=
\widetilde{\widetilde{C_4}}\cong Q_{12}$ with
$Clos(H)=[\varphi,\, \widetilde{(34)},\,\widetilde{(45)},\,\widetilde{(56)}]=
\widetilde{\widetilde{\SSS}}_{1,1,4}\cong
\AAA_{3,4}$ above and
$\rk N_H=18$, $(N_H)^\ast/N_H\cong (\bz/12\bz)^2\times \bz/3\bz$.
Here $C_4\subset \SSS_{1,1,4}\cong \SSS_4$ is a cyclic subgroup of
order $4$, and $Q_{12}$ is the binary dihedral group of order $12$.

(I.6) $H=[\varphi,\, \widetilde{(45)},\,\widetilde{(56)}]=
\widetilde{\widetilde{[(45),(56)]}}\cong \AAA_{3,3}$ with
$\rk N_H=16$, $(N_H)^\ast/N_H\cong \bz/9\bz\times (\bz/3\bz)^4$.

(I.7)  $H=[\varphi,\, \widetilde{(45)}\widetilde{(56)}]=
\widetilde{\widetilde{[(45)(56)]}}\cong (C_3)^2$ with
$Clos(H)=[\varphi,\, \widetilde{(45)},\,\widetilde{(56)}]=
\widetilde{\widetilde{[(45),(56)]}}$ $\cong \AAA_{3,3}$ above, and
$\rk N_H=16$, $(N_H)^\ast/N_H\cong \bz/9\bz\times (\bz/3\bz)^4$.
\footnote{These calculations show that for a symplectic group
$G=(C_3)^2$ on a
K\"ahlerian K3 surface, the group
$S_{(G)}^\ast/S_{(G)}=S_{(3,3)}^\ast /S_{(3,3)}\cong
\bz/9\bz\times (\bz/3\bz)^4$.
We must correct our calculation of this group in
\cite[Prop. 10.1]{Nik0}.}

(I.8) $H=[\varphi,\, \widetilde{(34)},\,\widetilde{(56)}]=
\widetilde{\widetilde{[(34),(56)]}}\cong \DDD_{12}$ with
$\rk N_H=16$, $(N_H)^\ast/N_H\cong (\bz/6\bz)^4$.

(I.9) $H=[\varphi,\, \widetilde{(34)}\widetilde{(56)}]=
\widetilde{\widetilde{[(34)(56)]}}\cong C_6$ with
$Clos(H)=[\varphi,\, \widetilde{(34)},\,\widetilde{(56)}]=
\widetilde{\widetilde{[(34),(56)]}}\cong \DDD_{12}$ above and
$\rk N_H=16$, $(N_H)^\ast/N_H\cong (\bz/6\bz)^4$.

(I.10) $H=[\varphi,\,\widetilde{(56)}]=
\widetilde{\widetilde{[(56)]}}\cong \DDD_{6}$ with
$\rk N_H=14$, $(N_H)^\ast/N_H\cong (\bz/6\bz)^2\times (\bz/3\bz)^3$.

(I.11) $H=[\varphi]=
\widetilde{\widetilde{\{e\}}}\cong C_{3}$ with
$\rk N_H=12$, $(N_H)^\ast/N_H\cong (\bz/3\bz)^6$.

Those of them which are conjugate to type I subgroups of
$\widetilde{\widetilde{\SSS}}_{1,2,3}=[\varphi,\, \widetilde{(23)},\,
\widetilde{(45)},\,\widetilde{(56)}]\cong \SSS_{3,3}$ and
different from subgroups (I.1) --- (I.11) above are
conjugate to one of subgroups (I$^\prime$.1) --- (I$^\prime$.2) below.

(I$^\prime$.1) $H=[\varphi,\, \widetilde{(23)},\,
\widetilde{(45)},\,\widetilde{(56)}]=
\widetilde{\widetilde{\SSS}}_{1,2,3}\cong \SSS_{3,3}$, and
$\rk N_H=18$, $(N_H)^\ast/N_H\cong \bz/18\bz\times
\bz/6\bz \times (\bz/3\bz)^2$.

(I$^\prime$.2) $H=[\varphi,\, \widetilde{(23)}\widetilde{(45)}
\widetilde{(56)}]=
\widetilde{\widetilde{[(23)(45)(56)]}}\cong C_3\times \DDD_6$ has \newline
$Clos(H)=[\varphi,\, \widetilde{(23)},\,\widetilde{(45)},\,\widetilde{(56)}]=
\widetilde{\widetilde{\SSS}}_{1,2,3}\cong \SSS_{3,3}$, and
$\rk N_H=18$, $(N_H)^\ast/N_H\cong \bz/18\bz\times \bz/6\bz
\times (\bz/3\bz)^2$.

\medskip

{\it Type II:} A KahK3 subgroup $H\subset A=A(N_{19})$ does
not contain the
subgroup $[\varphi]$. Then $\pi$ gives an isomorphism
$\pi: H\to \pi(H)=G\subset {\mathfrak S}_6$.
It is sufficient to describe all possible conjugation classes of
$G$ in ${\mathfrak S}_6$
and their possible lifts to $H=\widetilde{G}\subset A$.

Since $\rk N_H\le 19$, it follows that $G$ must have at least two orbits
in $\{1,\dots, 6\}$. Thus, $H=\widetilde{G}$ where either
$G\subset \SSS_{1,5}$, or
$G\subset \SSS_{2,4}$, or $G\subset \SSS_{3,3}$ (up to conjugation).


Let us consider cases when $G\subset \SSS_{1,5}$. It is easy to see that
for $G=\SSS_{1,5}$, we obtain (up to conjugation in $A$)

(II.1) $H=[\widetilde{(23)},\,\widetilde{(34)},\,
\widetilde{(45)}\varphi,\,\widetilde{(56)}]=
\widetilde{\SSS }_{1,5}\cong \SSS_5$ with $\rk N_H=19$ and
$(N_H)^\ast/N_H\cong \bz/60\bz\times \bz/5\bz$.
By Theorem \ref{th:primembb3}, $N_{H}$ has a primitive embedding
into $L_{K3}$ and $H$ is a KahK3 subgroup.

In particular, for any
$G\subset \SSS_{1,5}$, we obtain a KahK3 subgroup
$$
\widetilde{G}=\pi^{-1}(G)\cap [\widetilde{(23)},\,
\widetilde{(34)},\,\widetilde{(45)}\varphi,\,\widetilde{(56)}]\cong G.
$$
For each subgroup $G\subset \SSS_{1,5}$, one can
additionally check that it is the only lift of $G$ to
a subgroup $\widetilde{G}\subset A$ up to conjugation in $A$ if $G$
is different from $\SSS_4$ and $\AAA_4$. The groups
$\SSS_4$ and $\AAA_4$ have two lifts each.
Below we list subgroups $H=\widetilde{G}\subset
[\widetilde{(23)},\widetilde{(34)},\,
\widetilde{(45)}\varphi,\,\widetilde{(56)}]$
for all 17 non-trivial conjugacy classes of subgroups $G\subset \SSS_5$
and additional conjugacy classes for $G\cong \SSS_4,\,\AAA_4$.

(II.2) $H=[\widetilde{(23)}\widetilde{(34)}\widetilde{(45)}\varphi
\widetilde{(56)},\,\widetilde{(23)}\widetilde{(34)}]\cong
[(23456),(234)]\cong \AAA_5$ with
$\rk N_H=18$ and $(N_H)^\ast/N_H\cong \bz/30\bz\times \bz/10\bz$.

(II.3) $H=[\widetilde{(23)}\widetilde{(34)}\widetilde{(45)}
\varphi\widetilde{(56)},\,
\widetilde{(34)}\widetilde{(56)}\varphi\widetilde{(45)}]=
\widetilde{[(23456),(3465)]}\cong Hol(C_5)=C_5\rtimes Aut(C_5)$ with
$\rk N_H=18$ and $(N_H)^\ast/N_H\cong (\bz/10\bz)^2\times \bz/5\bz$.

(II.4) $H=[\widetilde{(23)}\widetilde{(34)}\widetilde{(45)}\varphi
\widetilde{(56)},\,\widetilde{(56)}\widetilde{(34)}
\widetilde{(45)}\varphi \widetilde{(34)}\widetilde{(56)}
\widetilde{(45)}\varphi ]=\widetilde{[(23456), (36)(45)]}\cong \DDD_{10}$
with $\rk N_H=16$ and $(N_H)^\ast/N_H\cong (\bz/5\bz)^4$.

(II.5) $H=[\widetilde{(23)}\widetilde{(34)}\widetilde{(45)}\varphi
\widetilde{(56)}]=\widetilde{[(23)(34)(45)(56)]}\cong C_5$ with
\newline $Clos(H)=[\widetilde{(23)}\widetilde{(34)}\widetilde{(45)}
\varphi\widetilde{(56)},\,\widetilde{(56)}\widetilde{(34)}
\widetilde{(45)}\varphi \widetilde{(34)}
\widetilde{(56)}\widetilde{(45)}\varphi ]=
\widetilde{[(23456), (36)(45)]}\cong \DDD_{10}$ above and
$\rk N_H=16$, $(N_H)^\ast/N_H\cong (\bz/5\bz)^4$.

(II.6) $H=[\widetilde{(34)},\,\widetilde{(45)}\varphi,\,\widetilde{(56)}]=
\widetilde{[(34),(45),(56)]}\cong \SSS_4$ with
$\rk N_H=17$ and $(N_H)^\ast/N_H\cong (\bz/12\bz)^2\times \bz/4\bz$.

(II.6$^\prime$) $H=[\widetilde{(34)},\,\widetilde{(45)},\,\widetilde{(56)}]=
\widetilde{[(34),(45),(56)]}\cong \SSS_4$ with
$\rk N_H=17$ and $(N_H)^\ast/N_H\cong (\bz/12\bz)^2\times \bz/4\bz$.

(II.7) $H=[\widetilde{(34)}\widetilde{(56)},\,
\widetilde{(45)}\varphi\widetilde{(56)}]=
\widetilde{[(34)(56),(456)]}\cong \AAA_4$
with $\rk N_H=16$ and \newline
$(N_H)^\ast/N_H\cong (\bz/12\bz)^2\times (\bz/2\bz)^2$.

(II.7$^\prime$) $H=[\widetilde{(34)}\widetilde{(56)},\,
\widetilde{(45)}\widetilde{(56)}]=
\widetilde{[(34)(56),(456)]}\cong \AAA_4$
with $\rk N_H=16$ and \newline
$(N_H)^\ast/N_H\cong (\bz/12\bz)^2\times (\bz/2\bz)^2$.

(II.8) $H=[\widetilde{(23)},\,\widetilde{(45)}\varphi,\,\widetilde{(56)}]=
\widetilde{[(23),\,(45),\,(56)]}\cong \DDD_{12}$
with $\rk N_H=16$ and \newline
$(N_H)^\ast/N_H\cong (\bz/6\bz)^4$.

(II.9) $H=[\widetilde{(23)}\widetilde{(45)}\varphi\widetilde{(56)}]=
\widetilde{[(23)(45)(56)]}\cong C_6$ with $Clos(H)=[\widetilde{(23)},\,
\widetilde{(45)}\varphi,\,\widetilde{(56)}]
\cong \DDD_{12}$ above and $\rk N_H=16$, $(N_H)^\ast/N_H\cong (\bz/6\bz)^4$.

(II.10) $H=[\widetilde{(23)}\widetilde{(45)}\varphi,
\,\widetilde{(23)}\widetilde{(56)}]=
[(23)(45),\,(23)(56)]\cong \DDD_6$ with
$\rk N_H=14$ and $(N_H)^\ast/N_H\cong (\bz/6\bz)^2\times (\bz/3\bz)^3$.

(II.11) $H=[\widetilde{(45)}\varphi,\,\widetilde{(56)}]=
\widetilde{[(45),\,(56)]}\cong \DDD_6$ with
$\rk N_H=14$ and $(N_H)^\ast/N_H\cong (\bz/6\bz)^2\times (\bz/3\bz)^3$.

(II.12) $H=[\widetilde{(34)}\widetilde{(45)}\varphi\widetilde{(56)},\,
\widetilde{(34)}\widetilde{(45)}\varphi \widetilde{(34)}]=
\widetilde{[(34)(45)(56),\,(35)]}\cong \DDD_8$ with
$\rk N_H=15$ and $(N_H)^\ast/N_H\cong (\bz/4\bz)^5$.

(II.13) $H=[\widetilde{(34)}\widetilde{(45)}\varphi\widetilde{(56)}]=
\widetilde{[(34)(45)(56)]}\cong C_4$
with $\rk N_H=14$ and $(N_H)^\ast/N_H\cong (\bz/4\bz)^4\times (\bz/2\bz)^2$.

(II.14) $H=[\widetilde{(45)}\varphi\widetilde{(56)}]=
\widetilde{[(45)(56)]}\cong C_3$
with $\rk N_H=12$ and $(N_H)^\ast/N_H\cong (\bz/3\bz)^6$.

(II.15) $H=[\widetilde{(34)}\widetilde{(56)},\widetilde{(45)}\varphi
\widetilde{(34)}\widetilde{(56)}\widetilde{(45)}\varphi ]
=\widetilde{[(34)(56),(35)(46)]}\cong (C_2)^2$ with
$\rk N_H=12$ and $(N_H)^\ast/N_H\cong (\bz/4\bz)^2\times (\bz/2\bz)^6$.

(II.16) $H=[\widetilde{(34)},\,\widetilde{(56)}]=
\widetilde{[(34),(56)]}\cong (C_2)^2$
with $\rk N_H=12$ and $(N_H)^\ast/N_H\cong (\bz/4\bz)^2\times (\bz/2\bz)^6$.

(II.17) $H=[\widetilde{(34)}\widetilde{(56)}]=\widetilde{[(34)(56)]}\cong C_2$ with
$\rk N_H=8$ and $(N_H)^\ast/N_H\cong (\bz/2\bz)^8$.

(II.18) $H=[\widetilde{(56)}]=\widetilde{[(56)]}\cong C_2$ with
$\rk N_H=8$ and $(N_H)^\ast/N_H\cong (\bz/2\bz)^8$.

\vskip0.5cm


It is easy to see that if  $G\subset \SSS_{3,3}$ and
$G$ does not have a one-element orbit (equivalently, $G$ is not conjugate to
a subgroup of $\SSS_{1,5}$), then $G$ contains
an element of the cycle type $[3,3]$. Then $H=\widetilde{G}$
contains one of subgroups
$H_1=[\widetilde{(12)}\widetilde{(23)}\widetilde{(45)}
\widetilde{(56)}]$,
$[\widetilde{(12)}\widetilde{(23)}\widetilde{(45)}
\widetilde{(56)\varphi}]$ or
$[\widetilde{(12)}\widetilde{(23)}\widetilde{(45)}
\widetilde{(56)}\varphi^2]$
(up to conjugation) with $\rk N_{H_1}=16$ and
$N_{H_1}^\ast/N_{H_1}\cong (\bz/3\bz)^8$.
By Theorem \ref{th:primembb3}, neither of
these $N_{H_1}$ has a primitive embedding
into $L_{K3}$ and neither of these $H_1$ is a KahK3 subgroup.
Then $H=\widetilde{G}$ is not KahK3 subgroup too.

It is easy to see that if  $G\subset \SSS_{2,4}$ and
$G$ does not have a one-element orbit, then either $G$ contains
an element of the cycle type $[2,2,2]$, or $[2,4]$, or
$G$ is conjugate to the subgroup $[(12)(34), (12)(56)]$
isomorphic to $(C_2)^2$. For the first case, a subgroup
$H=\widetilde{G}$ contains a subgroup
$H_1=[\widetilde{(12)}\widetilde{(34)}\widetilde{(56)}]$
(up to conjugation) with $\rk N_{H_1}=12$ and
$N_{H_1}^\ast/N_{H_1}\cong (\bz/2\bz)^{12}$.
For the second case, a subgroup $H=\widetilde{G}$
contains a subgroups
$H_1=[\widetilde{(12)}\widetilde{(34)}\widetilde{(45)}\widetilde{(56)}]$
(up to conjugation) with $\rk N_{H_1}=16$ and
$N_{H_1}^\ast/N_{H_1}\cong (\bz/4\bz)^4\times (\bz/2\bz)^{4}$.
By Theorem \ref{th:primembb3}, neither of
these $N_{H_1}$ has a primitive embedding
into $L_{K3}$, and neither of these $H_1$ is a KahK3 subgroup.
Then $H=\widetilde{G}$ is not a KahK3 subgroup too.
The last case, gives a subgroup

(II$^\prime$,1) $H=[\widetilde{(12)}\widetilde{(34)},\,
\widetilde{(12)}\widetilde{(56)}]\cong (C_2)^2$
(up to conjugation) with $\rk N_H=12$ and \newline
$(N_H)^\ast/N_H\cong (\bz/4\bz)^2\times (\bz/2\bz)^6$.
By Theorem \ref{th:primembb3}, $N_{H}$ has a primitive embedding
into $L_{K3}$ and $H$ is a KahK3 subgroup. This case is interesting
because the invariant lattice $N^H$ does not have roots.

\medskip

Thus, finally we obtained the list of all KahK3 conjugacy
classes of subgroups of $A=A(N_{19})$
which are subgroups (I.1) --- (I.11), (I$^\prime$.1),
(I$^\prime$.2), (II.1) --- (II.18),
(II.6$^\prime$), (II.7$^\prime$), (II$^\prime$.1) above.

\medskip

Let $X$ be marked by a primitive sublattice
$S\subset N=N_{19}=N(6D_4)$. Then $S$
must satisfy Theorem \ref{th:primembb3} and
$\Gamma(P(S))\subset \Gamma(P(N_{19}))=6\ddd_4$.
Any such $S$ gives marking of some $X$ and $P(X)\cap S=P(S)$.

If $N_H\subset S$ where $H$ has type (I.1) ($H\cong \AAA_{3,4}$),
(I$^\prime$.1) ($H\cong \SSS_{3,3}$), (II.1) ($H\cong \SSS_5$) or
(II$^\prime$.1) ($H\cong (C_2)^2$)
then $\Aut(X,S)_0=H$.
Otherwise, if only $N_H\subset S$ where $H$ has type
(I.6) ($H\cong \AAA_{3,3}$), (I.8) ($H\cong \DDD_{12}$),
(II.2) ($H\cong \AAA_5$), (II.3) ($H\cong Hol(C_5)$),
(II.6) ($H\cong \SSS_4$), (II.8) ($H\cong \DDD_{12}$)
or (II.7$^\prime$) ($H\cong \AAA_4$) then $\Aut(X,S)_0=H$.
Otherwise, if only $N_H\subset S$ where $H$ has type
(I.10) ($H\cong \DDD_6$), (II.4) ($H\cong \DDD_{10}$),
(II.7) ($H\cong \AAA_4$), (II.10) ($H\cong \DDD_6$),
(II.11) ($H\cong \DDD_6$) or (II.12) ($H\cong \DDD_8$)
then $\Aut(X,S)_0\cong H$. Otherwise, if only
$N_H\subset S$ where $H$ has type (II.13) ($H\cong C_4$),
(II.14) ($H\cong C_3$), (II.15) ($H\cong (C_2)^2$) or
(II.16) ($H\cong (C_2)^2$) then
$\Aut (X,S)_0=H$. Otherwise, if only
$N_H\subset S$ where $H$ has type (II.17)
($H\cong C_2$) or (II.18) ($H\cong C_2$)
then $\Aut(X,S)_0\cong H$. Otherwise,
$\Aut(X,S)_0$ is trivial.

Let $H=[\widetilde{(34)},\,\widetilde{(45)}\varphi,\,\widetilde{(56)}]=
\widetilde{[(34),(45),(56)]}\cong \SSS_4$ has the type (II.6).
Let $S=[(D_4)_6=[\alpha_{16},\dots,\,
\alpha_{46}],\ N_H]_{pr}\subset N_{19}$.
We have $\rk S=19$ and $S^\ast/S\cong \bz/12\bz$.
Thus, $S$ satisfies Theorem \ref{th:primembb3} and
gives marking of some $X$.
We have  $H\subset \Aut(X,S)_0$ where $H\cong \SSS_4$, and
$4\ddd_4\subset\Gamma(P(X)\cap S)=P(S)$.
By classification  of Niemeier lattices and our calculations
above, $X$ can be marked by the Niemeier lattice $N_{19}=N(6D_4)$
only for such $S\subset S_X$.


{\bf Case 20.} For the Niemeier lattice
$$
N=N_{20}=N(6A_4)=[6A_4,[1(01441)]]=
$$
$$
[6A_4,\ \varepsilon_{11}+\varepsilon_{13}-\varepsilon_{14}-
\varepsilon_{15}+\varepsilon_{16},\,
\varepsilon_{11}+\varepsilon_{12}+\varepsilon_{14}-
\varepsilon_{15}-\varepsilon_{16},\,
\varepsilon_{11}-\varepsilon_{12}+\varepsilon_{13}+
\varepsilon_{15}-\varepsilon_{16},\,
$$
$$
\varepsilon_{11}-\varepsilon_{12}-\varepsilon_{13}+
\varepsilon_{14}+\varepsilon_{16},\,
\varepsilon_{11}+\varepsilon_{12}-\varepsilon_{13}-
\varepsilon_{14}+\varepsilon_{15}]
$$
the group $A=A(N_{20})$ consists of the cyclic group
$[\varphi ]$ of order $2$ which preserves connected
components of $6\aaa_4$ and $A/[\varphi]=\SSS_5$
which acts faithfully on $6$ connected components of $6\aaa_4$.
(see \cite[Chs. 16, 18]{CS}). The $\SSS_5$ acts on
five triplets of classes of elements of $N_{20}/{6A_4}$ below:
$$
\tilde{1}=\{2\varepsilon_{11}-2\varepsilon_{12}+
\varepsilon_{14}+\varepsilon_{15},\
\varepsilon_{13}+2\varepsilon_{14}-2\varepsilon_{15}-
\varepsilon_{16},\
2\varepsilon_{11}+2\varepsilon_{12}-\varepsilon_{13}-\varepsilon_{16}\},
$$
$$
\tilde{2}=\{2\varepsilon_{11}-\varepsilon_{12}+2\varepsilon_{13}-
\varepsilon_{14},\
\varepsilon_{12}-\varepsilon_{14}-2\varepsilon_{15}+2\varepsilon_{16},\
2\varepsilon_{11}-2\varepsilon_{13}+\varepsilon_{15}+\varepsilon_{16}\},
$$
$$
\tilde{3}=\{2\varepsilon_{11}-\varepsilon_{13}+2\varepsilon_{14}-
\varepsilon_{15},\
\varepsilon_{12}-2\varepsilon_{13}+2\varepsilon_{15}-\varepsilon_{16},\
-2\varepsilon_{11}-\varepsilon_{12}+2\varepsilon_{14}-\varepsilon_{16}\},
$$
$$
\tilde{4}=\{-2\varepsilon_{11}-\varepsilon_{12}-\varepsilon_{13}+
2\varepsilon_{15},\
\varepsilon_{12}-\varepsilon_{13}+2\varepsilon_{14}-2\varepsilon_{16},\
2\varepsilon_{11}-\varepsilon_{14}+2\varepsilon_{15}-\varepsilon_{16}\},
$$
$$
\tilde{5}=\{2\varepsilon_{11}+\varepsilon_{13}+
\varepsilon_{14}-2\varepsilon_{16},\
\varepsilon_{12}+2\varepsilon_{13}-2\varepsilon_{14}-\varepsilon_{15},\
-\varepsilon_{11}-2\varepsilon_{12}-2\varepsilon_{15}-\varepsilon_{16}\}.
$$
One can give exact definitions easily.

The group $A=A(N_{20})$ is defined by its actions on terminals of
components of $6\aaa_4$. It is generated by the involution
$$
\varphi=(\alpha_{11}\alpha_{41})
(\alpha_{12}\alpha_{42})(\alpha_{13}\alpha_{43})
(\alpha_{14}\alpha_{44})(\alpha_{15}\alpha_{45})(\alpha_{16}\alpha_{46}),
$$
above and the corresponding transpositions of $\tilde{1},
\dots,\ \tilde{5}$ above:
$$
\widetilde{(12)}=(\alpha_{11}\alpha_{15}\alpha_{41}\alpha_{45})
(\alpha_{12}\alpha_{46}\alpha_{42}\alpha_{16})
(\alpha_{13}\alpha_{44}\alpha_{43}\alpha_{14}),
$$
$$
\widetilde{(23)}=(\alpha_{11}\alpha_{16}\alpha_{41}\alpha_{46})
(\alpha_{12}\alpha_{13}\alpha_{42}\alpha_{43})
(\alpha_{14}\alpha_{45}\alpha_{44}\alpha_{15}),
$$
$$
\widetilde{(34)}=(\alpha_{11}\alpha_{12}\alpha_{41}\alpha_{42})
(\alpha_{13}\alpha_{14}\alpha_{43}\alpha_{44})
(\alpha_{15}\alpha_{46}\alpha_{45}\alpha_{16}),
$$
$$
\widetilde{(45)}=(\alpha_{11}\alpha_{13}\alpha_{41}\alpha_{43})
(\alpha_{12}\alpha_{16}\alpha_{42}\alpha_{46})
(\alpha_{14}\alpha_{15}\alpha_{44}\alpha_{45})
$$
have the order $4$ in $A$.

For $H=[\varphi]$ we have $\rk N_H=12$ and
$(N_H)^\ast/N_H\cong (\bz/2\bz)^{12}$.
By Theorem \ref{th:primembb3}, $N_{H}$ does not have a primitive embedding
into $L_{K3}$ and $H=[\varphi]$ is not a KahK3 subgroup.

It follows that the canonical projection $\pi:A\to \SSS_5$ gives an
isomorphism $\pi|H:H\to \pi (H)=G\subset \SSS_5$ for any
KahK3 subgroup $H\subset A$.
Then we denote $H=\widetilde{G}$. Let us consider possible $G$ and $H$
for KahK3 subgroups $H\subset A$ up to conjugation (in $\SSS_5$ and $A$).

We have $\widetilde{(12)}^2=(\widetilde{(12)}\widetilde{(23)})^3 =
(\widetilde{(12)}\widetilde{(23)}\widetilde{(34)}\widetilde{(45)}\varphi)^5=
(\widetilde{(12)}\widetilde{(34)}\widetilde{(45)})^6=\varphi$.
It follows that $[\widetilde{(12)}]$,  $[\widetilde{(12)}\varphi]$,
$[\widetilde{(12)}\widetilde{(23)}]=\widetilde{[(123)]}$,
$[\widetilde{(12)}\widetilde{(23)}\widetilde{(34)}\widetilde{(45)}\varphi]=
\widetilde{[(12345)]}$, $[\widetilde{(12)}\widetilde{(34)}\widetilde{(45)}]$,
$[\widetilde{(12)}\widetilde{(34)}\widetilde{(45)}\varphi]=
\widetilde{[(12)(345)]}$
are not KahK3 subgroups.

We have $H=[\widetilde{(12)}\widetilde{(34)}\varphi]\cong C_2$ has
$\rk N_H=12$ and $(N_H)^\ast/N_H\cong (\bz/2\bz)^{12}$. By Theorem
\ref{th:primembb3}, it is not a KahK3 subgroup.

From these calculations and the known list of conjugation classes of
subgroups in $\SSS_5$, we obtain the list of KahK3 subgroups of $A=A(6A_4)$
up to conjugation. They are

(1) $H=[\widetilde{(12)}\widetilde{(23)}\widetilde{(34)}
\widetilde{(45)},\
\widetilde{(45)}\widetilde{(23)}\widetilde{(34)}]=
\widetilde{[(12345), (2354)]}\cong Hol(C_5)$ and
$H=[\widetilde{(12)}\widetilde{(23)}\widetilde{(34)}
\widetilde{(45)},\
\widetilde{(45)}\widetilde{(23)}\widetilde{(34)}\varphi]=
\widetilde{[(12345), (2354)]}\cong Hol(C_5)$
with $\rk N_H=18$, \newline
$(N_H)^\ast/N_H\cong (\bz/10\bz)^2 \times \bz/5\bz$.

\medskip

(2) $H=[\widetilde{(12)}\widetilde{(23)}\widetilde{(34)}
\widetilde{(45)},\
\widetilde{(45)}\widetilde{(23)}\widetilde{(34)}
\widetilde{(23)}\widetilde{(45)}\widetilde{(34)}]=
\widetilde{[(12345), (25)(34)]}\cong \DDD_{10}$
with $\rk N_H=16$ and $(N_H)^\ast/N_H\cong (\bz/5\bz)^4$.

\medskip

(3) $H=[\widetilde{(12)}\widetilde{(23)}\widetilde{(34)}
\widetilde{(45)}]=\widetilde{[(12345)]}\cong C_5$
with $Clos(H)=$ \newline
$[\widetilde{(12)}\widetilde{(23)}\widetilde{(34)}\widetilde{(45)},\
\widetilde{(45)}\widetilde{(23)}\widetilde{(34)}
\widetilde{(23)}\widetilde{(45)}\widetilde{(34)}]\cong \DDD_{10}$
above and $\rk N_H=16$, $(N_H)^\ast/N_H\cong (\bz/5\bz)^4$.

\medskip

(4) $H=[\widetilde{(12)}\widetilde{(23)}\widetilde{(34)}]=
\widetilde{[(1234)]}\cong C_4$
and $H=[\widetilde{(12)}\widetilde{(23)}\widetilde{(34)}\varphi]=
\widetilde{[(1234)]}\cong C_4$
with $\rk N_H=14$ and $(N_H)^\ast/N_H\cong (\bz/4\bz)^4\times (\bz/2\bz)^2$.

\medskip

(5) $H=[\widetilde{(12)}\widetilde{(34)}]=\widetilde{[(12)(34)]}\cong C_2$
with $\rk N_H=8$ and $(N_H)^\ast/N_H\cong (\bz/2\bz)^8$.

Let $X$ be marked by a primitive sublattice
$S\subset N=N_{20}=N(6A_4)$. Then $S$
must satisfy Theorem \ref{th:primembb3} and
$\Gamma(P(S))\subset \Gamma(P(N_{20}))=6\aaa_4$.
Any such $S$ gives marking of some $X$ and $P(X)\cap S=P(S)$.

If $N_H\subset S$ where $H$ has type (1) ($H\cong Hol(C_5)$),
(2) ($H\cong \DDD_{10}$) or (4) ($H\cong C_4$)
then $\Aut(X,S)_0=H$.
Otherwise, if only $N_H\subset S$ where $H$ has type
(5) ($H\cong C_2$) then $\Aut(X,S)_0=H\cong C_2$.
Otherwise, $\Aut(X,S)_0$ is trivial.

Let $H=[\widetilde{(12)}\widetilde{(23)}\widetilde{(34)}]\cong C_4$
has the type (4). Let $S=[(A_4)_1=[\alpha_{11},\dots,\,
\alpha_{41}],\ A_2=[\alpha_{22},\alpha_{32}],\ N_H]_{pr}\subset N_{20}$.
We have $\rk S=19$ and $S^\ast/S\cong \bz/10\bz$.
Thus, $S$ satisfies Theorem \ref{th:primembb3} and
gives marking of some $X$.
We have  $H\subset \Aut(X,S)_0$ where $H\cong C_4$, and
$4\aaa_4\aaa_2\subset\Gamma(P(X)\cap S)=P(S)$ where $H\cong C_4$ acts transitively
on $4$ components of $4\aaa_4$ and permutes elements $\alpha_{22}$,
$\alpha_{32}$ of the component $\aaa_2$.
By classification  of Niemeier lattices and our calculations
above, $X$ can be marked by the Niemeier lattices $N_{20}=N(6A_4)$
only for such $S\subset S_X$.

\medskip


{\bf Case 21.} For the Niemeier lattice
$$
N=N_{21}=N(8A_3)=[8A_3,[3(2001011)]]=
$$
$$
[8A_3,\ -\varepsilon_{11}+2\varepsilon_{12}+
\varepsilon_{15}+\varepsilon_{17}+\varepsilon_{18},\
-\varepsilon_{11}+\varepsilon_{12}+2\varepsilon_{13}+
\varepsilon_{16}+\varepsilon_{18},\,
$$
$$
-\varepsilon_{11}+\varepsilon_{12}+\varepsilon_{13}+
2\varepsilon_{14}+\varepsilon_{17},\
-\varepsilon_{11}+\varepsilon_{13}+\varepsilon_{14}+
2\varepsilon_{15}+\varepsilon_{18},\
-\varepsilon_{11}+\varepsilon_{12}+\varepsilon_{14}+
\varepsilon_{15}+2\varepsilon_{16},
$$
$$
-\varepsilon_{11}+\varepsilon_{13}+\varepsilon_{15}+
\varepsilon_{16}+2\varepsilon_{17},\
-\varepsilon_{11}+\varepsilon_{14}+\varepsilon_{16}+
\varepsilon_{17}+2\varepsilon_{18}]
$$
the group $A=A(N_{21})$ consists of the cyclic group
$[\varphi ]$ of order $2$ which preserves connected
components of $8\aaa_3$ and $A/[\varphi]=\text{Aff}(3,\bff_2)$
is the affine group of 3-dimensional affine space over
$\bff_2$ defined by $8$ components $(\aaa_3)_i$, $1\le i\le 8$,
of $8\aaa_3$. Here a 4-element subset $\{i,j,k,l\}\subset \{1,2,\dots,8\}$
is a plane of the affine space if and only if
$2(\varepsilon_{1i}+\varepsilon_{1j}+\varepsilon_{1k}+
\varepsilon_{1l})\in N/8A_3$.
For example, from definition above, $\{1,5,7,8\}$,
$\{1,2,6,8\}$, $\{1,2,3,7\}$,
$\{1,3,4,8\}$, $\{1,2,4,5\}$, $\{1,3,5,6\}$ and $\{1,4,6,7\}$ are
all planes containing $1$.
Thus, one can take affine coordinates where $1=(000)$,
$2=(100)$, $3=(010)$, $4=(001)$, $5=(101)$, $6=(111)$,
$7=(110)$, $8=(111)$ (where elements of the triplets are taken
$\mod 2$). Elements of $A$ are defined on terminals of $8\aaa_3$.
The involution $\varphi$ is
$$
\varphi=(\alpha_{11}\alpha_{31})(\alpha_{12}\alpha_{32})
(\alpha_{13}\alpha_{33})
(\alpha_{14}\alpha_{34})(\alpha_{15}\alpha_{35})
(\alpha_{16}\alpha_{36})(\alpha_{17}\alpha_{37})(\alpha_{18}\alpha_{38}).
$$
The translation part of $A$ is generated by the translations
$$
\widetilde{T}_{12}=
(\alpha_{11}\alpha_{12})(\alpha_{31}\alpha_{32})
(\alpha_{13}\alpha_{17})(\alpha_{33}\alpha_{37})
(\alpha_{14}\alpha_{15})(\alpha_{34}\alpha_{35})
(\alpha_{16}\alpha_{18})(\alpha_{36}\alpha_{38}),
$$
$$
\widetilde{T}_{13}=
(\alpha_{11}\alpha_{13})(\alpha_{31}\alpha_{33})
(\alpha_{12}\alpha_{17})(\alpha_{32}\alpha_{37})
(\alpha_{14}\alpha_{18})(\alpha_{34}\alpha_{38})
(\alpha_{15}\alpha_{16})(\alpha_{35}\alpha_{36}),
$$
$$
\widetilde{T}_{14}=
(\alpha_{11}\alpha_{14})(\alpha_{31}\alpha_{34})
(\alpha_{12}\alpha_{15})(\alpha_{32}\alpha_{35})
(\alpha_{13}\alpha_{18})(\alpha_{33}\alpha_{38})
(\alpha_{16}\alpha_{17})(\alpha_{36}\alpha_{37}).
$$
Here $T_{ij}(i)=j$ and $\pi(\widetilde{T}_{ij})=T_{ij}$
for the canonical homomorphism $\pi:A\to \text{Aff}(3,\bff_2)$.

The subgroup $H=[\varphi]\cong C_2$ has $\rk N_H=8$ and
$(N_H)^\ast/N_H\cong (\bz/2\bz)^8$. Thus, $N_H$ satisfies
Theorem \ref{th:primembb3} and $H=[\varphi]\cong C_2$ is a KahK3
subgroup.

The subgroup $H=[\widetilde{T}_{12}]\cong C_2$ has
$\rk N_H=12$ and $(N_H)^\ast/N_H\cong (\bz/2\bz)^{12}$.
Thus, $N_H$ does not satisfy Theorem \ref{th:primembb3}
and $H=[\widetilde{T}_{12}]\cong C_2$ is not a KahK3
subgroup. The same is valid for $H=[\varphi \widetilde{T}_{12}]\cong C_2$.

It follows that a KahK3 subgroup $H\subset A$ has a trivial
translation part in $\text{Aff}(3,\bff_2)$:

(a) {\it for the canonical
homomorphism $\overline{\pi}:A\to \text{Gl}(3,\bff_2)$, the kernel
of  $\overline{\pi}|H:H\to \text{Gl}(3,\bff_2)$ is
contained in $[\varphi]\cong C_2$.}

Let us describe possible KahK3 subgroups $H\subset A$ and their images
$G=\pi(H)$ for the canonical homomoprhism $\pi:A\to \text{Aff}(3,\bff_2)$.

Let $H\subset A$ be a KahK3 subgroup that is $N_H$
satisfies Theorem \ref{th:primembb3}. Then $\rk N_H\le 19$. It follows that
$G=\pi(H)$ must have at least two orbits in $\{1,2,\dots,8\}$. We want to show
that $G$ has a 1-element orbit.

Let $G$ has a 2-elements orbit $\{1,2\}$ and it does not have a 1-element
orbit. Then $G$ preserves the vector $12=(100)$, and its linear part
$\overline{G}\subset \SSS_4$ where
$$
\SSS_4=\{f\in \text{Aff}(3,\bff_2)\ |\ f(1)=1,\ f(2)=2\}
$$
consists of matrices with the zero $1$ and the eigen vector
$\{1,2\}=(100)$.
For the zero $1$ and the basis $\{1,2\}=(100)$, $\{1,3\}=(010)$,
$\{1,4\}=(001)$,
this is the symmetric group $\SSS_4$ by its action by conjugations
on four cyclic subgroups of order $3$ of $\SSS_4$ which are
$$
C1_3=\left[\left(\begin{array}{ccc}
1&0&0\\
0&0&1\\
0&1&1
\end{array}
\right)\right],\
C2_3=\left[\left(\begin{array}{ccc}
1&1&0\\
0&0&1\\
0&1&1
\end{array}
\right)\right],\
$$
$$
C3_3=\left[\left(\begin{array}{ccc}
1&0&1\\
0&0&1\\
0&1&1
\end{array}
\right)\right],\
C4_3=\left[\left(\begin{array}{ccc}
1&1&1\\
0&0&1\\
0&1&1
\end{array}
\right)\right].
$$
Thus, to construct $G$ and finally $H$, we should
take a subgroup $\overline{G}\subset \SSS_4$
and change some of its generators $g$ by $T_{12}g$.
The resulting group $G$ must be
isomorphic to its linear part $\overline{G}$,
and it must not have one-element orbits
in $\{1,\dots, 8\}$.

Let $\overline{G}=[F1_3]=C1_3\cong C_3$ where
$F1_3=\left(\begin{array}{ccc}
1&0&0\\
0&0&1\\
0&1&1
\end{array}
\right)
$. The transformation $T_{12}F1_3$ gives the permutation
$(12)(358746)$ which has the order $6$ instead of $3$.
Thus we don't get a KahK3 subgroup because of the property (a) above.

Let $\overline{G}=[F_4]\cong C_4$ where
$F_4=\left(\begin{array}{ccc}
1&0&1\\
0&0&1\\
0&1&0
\end{array}
\right)
$. The transformation $T_{12}F_4$ gives the permutation
$(12)(3574)$ which has the same order $4$. But, it has 1-element
orbits $\{6\}$ and $\{8\}$. Thus, we don't get a KahK3 subgroup
we are looking for.

Let $\overline{G}=[F1_2]\cong C_2$ where
$F1_2=\left(\begin{array}{ccc}
1&0&0\\
0&1&1\\
0&0&1
\end{array}
\right)
$ is an odd element of $\SSS_4$.
The transformation $T_{12}F1_2$ gives the
permutation $(12)(37)(46)(58)$. It gives
the translations $T_{12}$ in the plane $\{1,2,3,7\}$ and
a different translation $T_{46}$
in the parallel plane $\{4,5,6,8\}$. It is lifted to an
element
$$
\widetilde{T_{12}F1_2}=
(\alpha_{11}\alpha_{12}\alpha_{31}\alpha_{32})
(\alpha_{13}\alpha_{17}\alpha_{33}\alpha_{37})
(\alpha_{14}\alpha_{36}\alpha_{34}\alpha_{16})
(\alpha_{15}\alpha_{18}\alpha_{35}\alpha_{38})
$$
of $A$ such that $\widetilde{T_{12}F1_2}^2=\varphi$.
The subgroup $H=[\widetilde{T_{12}F1_2}]\subset A$
has $\rk N_H=16$ and $(N_H)^\ast/N_H\cong
(\bz/4\bz)^4\times (\bz/2\bz)^4$. The $N_H$
does not satisfy Theorem  \ref{th:primembb3},
and $H$ is not a KahK3 subgroup.

Let $\overline{G}=[F2_2]\cong C_2$ where
$F2_2=\left(\begin{array}{ccc}
1&1&0\\
0&1&0\\
0&0&1
\end{array}
\right)
$ is an even element of $\SSS_4$.
The transformation $T_{12}F2_2$ gives the
permutation $(12)(45)$. It gives
the translations $T_{12}$ in the plane $\{1,2,4,5\}$ and
the identity in the parallel plane $\{3,6,7,8\}$.
Thus, it has one-element orbits and does not give
KahK3 subgroups $H$ we are looking for.

Let $\overline{G}=[F1_2,\,F1^\prime_2]\subset \SSS_4$
is a non-normal Klein subgroup where $F1_2$ and
$F1^\prime_2$ are commuting odd transpositions.
For the group $G$, one of $F1_2$, $F1^\prime_2$
must be replaced by $T_{12}F1_2$, $T_{12}F1^\prime_2$ which gives
a non KahK3 subgroup as we have seen above.

Let $\overline{G}=[F2_2,\,F2^\prime_2]\subset \SSS_4$
is a normal Klein subgroup where $F2_2$ and
$F2^\prime_2$ are different even elements of order $2$.
For the group $G$, some of $F2_2$, $F2^\prime_2$
must be replaced by $T_{12}F2_2$, $T_{12}F2^\prime_2$
but the group $G$ will have one-element orbits
which follows from our considerations above. Thus,
we don't obtain a KahK3 subgroup we are looking for.

Let $\overline{G}=[F_4,\,F1^\prime_2]\cong \ddd_8$ where
$F_4$ is the matrix above and
$F1^\prime_2=
\left(\begin{array}{ccc}
1&0&0\\
0&0&1\\
0&1&0
\end{array}
\right)
$ is an odd element of $\SSS_4$ of order $2$.
The group
$G$ must be $G=[T_{12}F_4,\,F1^\prime_2]$ as we have seen above.
Here $T_{12}F_4$ gives the permutation $(12)(3574)$ and
$F1^\prime_2$ gives the permutation $(34)(57)$. The group $G$
has one-element orbits $\{6\}$ and $\{8\}$. Thus, we
don't obtain a KahK3 subgroup we are looking for.

Let $\overline{G}=[F1_3,\,F1^\prime_2]\cong \SSS_3$ where
$F1_3$ is the matrix above and
$F1^\prime_2=
\left(\begin{array}{ccc}
1&0&0\\
0&0&1\\
0&1&0
\end{array}
\right)
$ is an odd element of $\SSS_4$ of order $2$ such that
$F1^\prime_2 F1_3 F1^\prime_2=(F1_3)^2$.
As we have seen above, we cannot replace
$F1_3$ and $F1^\prime_2$ by
$T_{12}F1_3$ and $T_{12}F1^\prime_2$ respectively
because we then obtain a not KahK3 subgroup. Thus, from this
case, we don't obtain a KahK3 subgroup we are looking for.

Let $\overline{G}=[F1_3,\,F1^\prime_3]\cong \AAA_4$ where
$F1_3$ is the matrix above and $F1^\prime_3$ is another
matrix of order $3$.
As we have seen above, then we cannot replace
$F1_3$ and $F1^\prime_3$ by
$T_{12}F1_3$ and $T_{12}F1^\prime_3$ respectively
because we then obtain a not KahK3 subgroup. Thus,
we don't obtain a KahK3 subgroup we are looking for.

Let $\overline{G}=[F1_2,\,F1^\prime_2,\,F1^{\prime\prime}_2]=\SSS_4$
where $F1_2$, $F1^\prime_2$, $F1^{\prime\prime}_2$
are odd elements of order $2$.
As we have seen, we cannot replace
$F1_2$, $F1^\prime_2$, $F1^{\prime\prime}_2$
by
$T_{12}F1_2$, $T_{12}F1^\prime_2$, $T_{12}F1^{\prime\prime}_2$
respectively because we then obtain a not KahK3 subgroup. Thus,
we don't obtain a KahK3 subgroup we are looking for.

Thus, finally we proved that $G$ must have a 1-elements orbit if
it has a 2-elements orbit.

Now, assume that $G$ does not have a 1-elements orbit and a 2-elements
orbit, but it has a 3-elements orbit. This three elements generate
a plane with 4 elements which is invariant with respect to $G$.
The remaining element of the plane then gives a 1-elements orbit.
We obtain a contradiction.

Now, assume that $G$ has two 4-elements orbits.

At first, let us assume that these 4-elements orbits are in
general position, they don't define planes. We can assume that
one of orbits is $\{1,2,3,4\}$. Then $G$ must contain
an even permutation of the order $2$ of $\{1,2,3,4\}$.
We can assume that it is $(13)(24)$. On all elements it is then
$(13)(24)(78)(65)$, and it is lifted to the automorphism
$$
g=
(\alpha_{11}\alpha_{13}\alpha_{31}\alpha_{33})
(\alpha_{12}\alpha_{14}\alpha_{32}\alpha_{34})
(\alpha_{17}\alpha_{38}\alpha_{37}\alpha_{18})
(\alpha_{16}\alpha_{15}\alpha_{36}\alpha_{35})
$$
from $A$ such that $g^2=\varphi$. Then $H=[g]$
must be a KahK3 subgroup. We have
$\rk N_H=16$ and $(N_H)^\ast/N_H\cong (\bz/4\bz)^4\times
(\bz/2\bz)^4$. By Theorem \ref{th:primembb3}, $H$
 cannot be a KahK3 subgroup, and we get a contradiction.

Secondly, let us assume that these 4-elements orbits
define parallel planes. We can assume that these planes are
$\{1,3,5,6\}$ and $\{2,4,7,8\}$. Then $G$ must contain an
even permutation $g$ of order $2$ of $\{1,3,5,6\}$. We can
assume that it is $(13)(65)$. If $g$ is not identity
in $\{2,4,7,8\}$ and, for example $g(2)=4$, then
$g=(13)(24)(78)(65)$. Like above, we obtain a contradiction.
Thus, $g$ must be identity in $\{2,4,7,8\}$ and
$g=(13)(56)$.
Considering another plane $\{2,4,7,8\}$,
similarly, we can construct an even permutation $g^\prime\in G$
of order $2$ of $\{2,4,7,8\}$ which is identity
on $\{1,3,5,6\}$. For example, we can assume that
$g^\prime=(24)(78)$. Then $gg^\prime=(13)(24)(78)(65)$,
and we obtain a contradiction like above.

\medskip

Further we denote
$$
\text{Gl}(3,\bff_2)=\{f\in \text{Aff}(3,\bff_2),\ f(1)=1\}.
$$
of order $2^3\cdot 3\cdot 7=168$.
Thus, we identify zero with $1=(000)$. We use
the basis $2=(100)$, $3=(010)$ and $4=(001)$.
Let us consider possible subgroups
$G\subset \text{Gl}(3,\bff_2)$ and their lifts to
KahK3 subgroups $H\subset A$ up to conjugation.
As we have seen, they give all KahK3 conjugation
classes of subroups in $A$.

For $F\in \text{Gl}(3,\bff_2)$, we denote by $\widetilde{F}$
its lift to $\widetilde{F}\subset A$ such that
$\widetilde{F}(\alpha_{11})=\alpha_{11}$.
For a subgroup $G\subset \text{Gl}(3,\bff_2)$,
$$
\widetilde{G}=\{\widetilde{F}\ |\ F\in G \}.
$$
Obviously, $\pi|\widetilde{G}:\widetilde{G}\cong G$
is an isomorphism.

We consider the following elements from $\text{Gl}(3,\bff_2)$ and $A$:
$$
F1_7=
\left(\begin{array}{ccc}
0&0&1\\
1&0&0\\
0&1&1
\end{array}
\right)=(2345678),\
\widetilde{F1}_7=(\alpha_{12}\alpha_{13}\alpha_{14}\alpha_{15}
\alpha_{16}\alpha_{17}\alpha_{18})
(\alpha_{32}\alpha_{33}\alpha_{34}\alpha_{35}\alpha_{36}
\alpha_{37}\alpha_{38})
$$
or order $7$;
$$
F1_4=
\left(\begin{array}{ccc}
1&0&1\\
0&0&1\\
0&1&0
\end{array}
\right)=(3475)(68),\ \
\widetilde{F1}_4=(\alpha_{13}\alpha_{14}\alpha_{37}\alpha_{15})
(\alpha_{33}\alpha_{34}\alpha_{17}\alpha_{35})
(\alpha_{16}\alpha_{18}\alpha_{36}\alpha_{38});
$$
$$
F2_4=
\left(\begin{array}{ccc}
0&0&1\\
1&0&1\\
0&1&1
\end{array}
\right)=(2346)(78),
$$
$$
\widetilde{F2}_4=(\alpha_{12}\alpha_{13}\alpha_{34}\alpha_{16})
(\alpha_{32}\alpha_{33}\alpha_{14}\alpha_{36})
(\alpha_{15}\alpha_{35})(\alpha_{17}\alpha_{38}
\alpha_{37}\alpha_{18})
$$
of order $4$;
$$
F1_3=
\left(\begin{array}{ccc}
1&0&1\\
0&0&1\\
0&1&1
\end{array}
\right)=(346)(587),\ \
\widetilde{F1}_3=(\alpha_{13}\alpha_{14}\alpha_{16})
(\alpha_{33}\alpha_{34}\alpha_{36})(\alpha_{15}\alpha_{18}\alpha_{17})
(\alpha_{35}\alpha_{38}\alpha_{37})
$$
such that $F1_3F1_7(F1_3)^{-1}=(F1_7)^2$,
$$
F2_3=
\left(\begin{array}{ccc}
0&0&1\\
1&0&0\\
0&1&0
\end{array}
\right)=(234)(578),\ \
\widetilde{F2}_3=(\alpha_{12}\alpha_{33}\alpha_{34})
(\alpha_{32}\alpha_{13}\alpha_{14})(\alpha_{15}\alpha_{17}\alpha_{38})
(\alpha_{35}\alpha_{37}\alpha_{18})
$$
of order $3$;
$$
F1_2=\left(\begin{array}{ccc}
1&0&0\\
0&1&1\\
0&0&1
\end{array}
\right)=(48)(56),\ \ \widetilde{F1}_2=
(\alpha_{12}\alpha_{32})(\alpha_{14}\alpha_{38})
(\alpha_{34}\alpha_{18})(\alpha_{15}\alpha_{16})
(\alpha_{35}\alpha_{36})(\alpha_{17}\alpha_{37})
$$
of order $2$.

By considering possible subgroups $G\subset \text{Gl}(3,\bff_2)$
and their lifts to subgroups of $A$, we obtain the following
KahK3 subgroups of $A$ up to conjugation.

Cases (II.1) --- (II.3) below give all $H$ with 7-elements
orbit $\{2,3,4,5,6,7,8\}$,
equivalently order of $H$ is divisible by $7$.

(II.1) $H=[\widetilde{F1}_7,\widetilde{F1}_4]=
\widetilde{\text{Gl}(3,\bff_2)}\cong L_2(7)=PSL(2,\bff_7)$
with $\rk N_H=19$  and \newline
$(N_H)^\ast/N_H\cong \bz/28\bz\times \bz/7\bz$.

(II.2) $H=[\widetilde{F1}_7,\widetilde{F1}_3]=[\widetilde{F1}_7]
\rtimes [\widetilde{F1}_3]\cong C_7\rtimes C_3$
with $\rk N_H=18$ and $(N_H)^\ast/N_H\cong (\bz/7\bz)^3$.

(II.3) $H=[\widetilde{F1}_7]\cong C_7$ with $Clos(H)=
[\widetilde{F1}_7,\widetilde{F1}_3]\cong C_7\rtimes C_3$ above
and $\rk N_H=18$, $(N_H)^\ast/N_H\cong (\bz/7\bz)^3$.

Cases below give all $H$ with 1-elements orbits $\{1\}$ and $\{2\}$.

(I.1) $H=[\varphi,\,\widetilde{F1}_4,\,\widetilde{F1}_3]=[\varphi]
\times [\widetilde{F1}_4,\,\widetilde{F_1}_3]
\cong C_2\times \SSS_4$ with $\rk N_H=18$,
$(N_H)^\ast/N_H\cong (\bz/12\bz)^2\times (\bz/2\bz)^2$ and
$$
\det(K((N_H)_2)\equiv \det \left(
\begin{array}{cccc}
-75192 & 85308 &63780  &-36\\
85308  &-97356 &-72420 &48 \\
63780  &-72420 &-54112 &34 \\
-36    & 48    &34     &-4
\end{array}
\right)=
$$
$$
2^6\cdot 3^4\cdot 23^2\cdot239\equiv \pm 12^2\cdot 2^2\mod (\bz_2^\ast)^2.
$$
Thus, $N_H$ satisfies Theorem \ref{th:primembb3} and $H$ is a KahK3 subgroup.
All cases (1.2) -- (II.11$^\prime$) below give its subgroups.

(I.2)
$H=[\varphi,\,\widetilde{F1}_3,\,\widetilde{F1}_4\widetilde{F1}_3
(\widetilde{F1}_4)^{-1}]=
[\varphi]\times [\widetilde{F1}_3,\,\widetilde{F1}_4\widetilde{F1}_3
(\widetilde{F1}_4)^{-1}]
\cong C_2\times \AAA_4$ with $Clos(H)=[\varphi,\,\widetilde{F1}_4,\,
\widetilde{F1}_3]\cong C_2\times \SSS_4$
above and $\rk N_H=18$, $(N_H)^\ast/N_H\cong
(\bz/12\bz)^2\times (\bz/2\bz)^2$.

(I.3)
$H=[\varphi,\,\widetilde{F1}_2,\,\widetilde{F1}_4\widetilde{F1}_3
(\widetilde{F1}_4)^{-1}]=
[\varphi]\times [\widetilde{F1}_2,\,\widetilde{F1}_4\widetilde{F1}_3
(\widetilde{F1}_4)^{-1}]
\cong C_2\times \DDD_6\cong \DDD_{12}$ with
$\rk N_H=16$ and $(N_H)^\ast/N_H\cong (\bz/6\bz)^4$.

(I.4)
$H=[\varphi\widetilde{F1}_4\widetilde{F1}_3(\widetilde{F1}_4)^{-1}]
\cong C_6$ with $Clos(H)=[\varphi,\widetilde{F1}_2,\,
\widetilde{F1}_4\widetilde{F1}_3(\widetilde{F1}_4)^{-1}]
\cong C_2\times \DDD_6\cong \DDD_{12}$ above and
$\rk N_H=16$, $(N_H)^\ast/N_H\cong (\bz/6\bz)^4$.

(I.5)
$H=[\varphi,\,\widetilde{F1}_3\widetilde{F1}_4
(\widetilde{F1}_3)^{-1},\,\widetilde{F1}_2]=
[\varphi]\times [\widetilde{F1}_3\widetilde{F1}_4
(\widetilde{F1}_3)^{-1}, \widetilde{F1}_2]
\cong C_2\times \DDD_8$ with
$\rk N_H=16$ and $(N_H)^\ast/N_H\cong (\bz/4\bz)^4\times (\bz/2\bz)^2$.

(I.6)
$H=[\varphi,\,\widetilde{F1}_3(\widetilde{F1}_4)^2
(\widetilde{F1}_3)^{-1},\,\widetilde{F1}_2]=
[\varphi]\times [\widetilde{F1}_3(\widetilde{F1}_4)^2
(\widetilde{F1}_3)^{-1},\, \widetilde{F1}_2]
\cong (C_2)^3$ with
$\rk N_H=14$ and $(N_H)^\ast/N_H\cong (\bz/4\bz)^2\times (\bz/2\bz)^6$.
\footnote{These calculations show that for a symplectic group
$G=(C_2)^3$ on a
K\"ahlerian K3 surface, the group
$S_{(G)}^\ast/S_{(G)}=S_{(2,2,2)}^\ast /S_{(2,2,2)}\cong
(\bz/4\bz)^2\times (\bz/2\bz)^6$.
We must correct our calculation of this group in
\cite[Prop. 10.1]{Nik0}.}

(I.7)
$H=[\varphi,\,\widetilde{F1}_3(\widetilde{F1}_4)^2(\widetilde{F1}_3)^{-1},\,
\widetilde{F1}_3\widetilde{F1}_4(\widetilde{F1}_3)^{-1}\widetilde{F1}_2]=$

$[\varphi]\times [\widetilde{F1}_3(\widetilde{F1}_4)^2(\widetilde{F1}_3)^{-1},\,
\widetilde{F1}_3\widetilde{F1}_4(\widetilde{F1}_3)^{-1}\widetilde{F1}_2]
\cong (C_2)^3$ with
$\rk N_H=14$ and $(N_H)^\ast/N_H\cong (\bz/4\bz)^2\times (\bz/2\bz)^6$.

(I.8)
$H=[\varphi,\,\widetilde{F1}_3\widetilde{F1}_4(\widetilde{F1}_3)^{-1}]=
[\varphi]\times [\widetilde{F1}_3\widetilde{F1}_4(\widetilde{F1}_3)^{-1}]
\cong C_2\times C_4$ with $Clos(H)=
[\varphi,\,\widetilde{F1}_3\widetilde{F1}_4(\widetilde{F1}_3)^{-1},\,
\widetilde{F1}_2]
\cong C_2\times \DDD_8$ above and
$\rk N_H=16$, $(N_H)^\ast/N_H\cong (\bz/4\bz)^4\times (\bz/2\bz)^2$.

(I.9)
$H=[\varphi,\,\widetilde{F1}_2]=
[\varphi]\times [\widetilde{F1}_2]
\cong (C_2)^2$ with
$\rk N_H=12$ and $(N_H)^\ast/N_H\cong (\bz/4\bz)^2\times (\bz/2\bz)^6$.

(I.10)
$H=[\varphi] \cong C_2$ with
$\rk N_H=8$ and $(N_H)^\ast/N_H\cong (\bz/2\bz)^8$.

\medskip

(II.3) $H=[\widetilde{F1}_4,\,\widetilde{F1}_3]\cong \SSS_4$ with
$\rk N_H=17$ and $(N_H)^\ast/N_H\cong (\bz/12\bz)^2\times (\bz/4\bz)$.

(II.3$^\prime$) $H=[\varphi \widetilde{F1}_4,\,\widetilde{F1}_3]\cong \SSS_4$ with
$\rk N_H=17$ and $(N_H)^\ast/N_H\cong (\bz/12\bz)^2\times (\bz/4\bz)$.

(II.4)
$H=[\widetilde{F1}_3,\widetilde{F1}_4\widetilde{F1}_3
(\widetilde{F1}_4)^{-1}]\cong \AAA_4$
with $\rk N_H=16$ and $(N_H)^\ast/N_H\cong (\bz/12\bz)^2\times (\bz/2\bz)^2$.

(II.5)
$H=[\widetilde{F1}_2,\,\widetilde{F1}_4\widetilde{F1}_3(\widetilde{F1}_4)^{-1}]
\cong \DDD_6$ with
$\rk N_H=14$ and $(N_H)^\ast/N_H\cong (\bz/6\bz)^2\times (\bz/3\bz)^3$.

(II.5$^\prime$)
$H=[\varphi\widetilde{F1}_2,\,\widetilde{F1}_4
\widetilde{F1}_3(\widetilde{F1}_4)^{-1}]
\cong \DDD_6$ with
$\rk N_H=14$ and $(N_H)^\ast/N_H\cong (\bz/6\bz)^2\times (\bz/3\bz)^3$.

(II.6)
$H=[\widetilde{F1}_4\widetilde{F1}_3(\widetilde{F1}_4)^{-1}]
\cong C_3$ with
$\rk N_H=12$ and $(N_H)^\ast/N_H\cong (\bz/3\bz)^6$.

(II.7)
$H=[\widetilde{F1}_3\widetilde{F1}_4(\widetilde{F1}_3)^{-1},\,\widetilde{F1}_2]
\cong \DDD_8$ with
$\rk N_H=15$ and $(N_H)^\ast/N_H\cong (\bz/4\bz)^5$.

(II.7$^\prime$)
$H=[\varphi \widetilde{F1}_3\widetilde{F1}_4
(\widetilde{F1}_3)^{-1},\,\widetilde{F1}_2]
\cong \DDD_8$ with
$\rk N_H=15$ and $(N_H)^\ast/N_H\cong (\bz/4\bz)^5$.

(II.7$^{\prime\prime}$)
$H=[\widetilde{F1}_3\widetilde{F1}_4(\widetilde{F1}_3)^{-1},\,
\varphi\widetilde{F1}_2]
\cong \DDD_8$ with
$\rk N_H=15$ and $(N_H)^\ast/N_H\cong (\bz/4\bz)^5$.

(II.7$^{\prime\prime\prime}$)
$H=[\varphi\widetilde{F1}_3\widetilde{F1}_4(\widetilde{F1}_3)^{-1}, \,
\varphi\widetilde{F1}_2]
\cong \DDD_8$ with
$\rk N_H=15$ and $(N_H)^\ast/N_H\cong (\bz/4\bz)^5$.

(II.8)
$H=[\widetilde{F1}_3(\widetilde{F1}_4)^2(\widetilde{F1}_3)^{-1},\,
\widetilde{F1}_2]
\cong (C_2)^2$ with
$\rk N_H=12$ and \newline
$(N_H)^\ast/N_H\cong (\bz/4\bz)^2\times (\bz/2\bz)^6$
($H$ is a non-normal Klein subgroup of $\SSS_4$).

(II.8$^\prime$)
$H=[\varphi\widetilde{F1}_3(\widetilde{F1}_4)^2(\widetilde{F1}_3)^{-1},\,
\widetilde{F1}_2]
\cong (C_2)^2$ with
$\rk N_H=12$ and $(N_H)^\ast/N_H\cong (\bz/4\bz)^2\times (\bz/2\bz)^6$.

(II.8$^{\prime\prime}$)
$H=[\varphi \widetilde{F1}_3(\widetilde{F1}_4)^2(\widetilde{F1}_3)^{-1},\,
\varphi\widetilde{F1}_2]
\cong (C_2)^2$ with
$\rk N_H=12$ and $(N_H)^\ast/N_H\cong (\bz/4\bz)^2\times (\bz/2\bz)^6$.

(II.9)
$H=[\widetilde{F1}_3(\widetilde{F1}_4)^2(\widetilde{F1}_3)^{-1},\,
\widetilde{F1}_3\widetilde{F1}_4(\widetilde{F1}_3)^{-1}\widetilde{F1}_2]
\cong (C_2)^2$ with
$\rk N_H=12$ and $(N_H)^\ast/N_H\cong (\bz/4\bz)^2\times (\bz/2\bz)^6$
($H$ is the normal Klein subgroup of $\SSS_4$).

(II.9$^\prime$)
$H=[\varphi \widetilde{F1}_3(\widetilde{F1}_4)^2(\widetilde{F1}_3)^{-1},\,
\widetilde{F1}_3\widetilde{F1}_4(\widetilde{F1}_3)^{-1}\widetilde{F1}_2]
\cong (C_2)^2$ with
$\rk N_H=12$ and $(N_H)^\ast/N_H\cong (\bz/4\bz)^2\times (\bz/2\bz)^6$.

(II.10)
$H=[\widetilde{F1}_4]\cong C_4$ with
$\rk N_H=14$ and $(N_H)^\ast/N_H\cong (\bz/4\bz)^4\times (\bz/2\bz)^2$.

(II.10$^\prime$)
$H=[\varphi \widetilde{F1}_4]\cong C_4$ with
$\rk N_H=14$ and $(N_H)^\ast/N_H\cong (\bz/4\bz)^4\times (\bz/2\bz)^2$.

(II.11)
$H=[\widetilde{F1}_2]\cong C_2$ with
$\rk N_H=8$ and $(N_H)^\ast/N_H\cong(\bz/2\bz)^8$.

(II.11$^\prime$)
$H=[\varphi \widetilde{F1}_2]\cong C_2$ with
$\rk N_H=8$ and $(N_H)^\ast/N_H\cong(\bz/2\bz)^8$.

\medskip

Cases below give all $H$ which have 1-element orbit $\{1\}$, no
other 1-element orbits, and  4-elements
orbit $\{2,3,4,6\}$ which gives a plane.

(III.1) $H=[\varphi,\, \widetilde{F2}_4,\,\widetilde{F2}_3]=
[\varphi]\times [\widetilde{F2}_4,\,\widetilde{F2}_3]\cong C_2\times \SSS_4$
with $\rk N_H=18$, $(N_H)^\ast/N_H\cong (\bz/12\bz)^2\times
(\bz/2\bz)^2$ and
$$
\det(K((N_H)_2)\equiv \det \left(
\begin{array}{cccc}
-14460 & -12168 &-6276  &-48\\
-12168  &-11724 &-6072 &-36 \\
6276  &-6072 &-3148 & -18 \\
-48    & -36    & -18     & -4
\end{array}
\right)=
$$
$$
2^6\cdot 3^2\cdot 41\cdot 9767 \equiv \pm 12^2\cdot 2^2\mod (\bz_2^\ast)^2.
$$
Thus, $N_H$ satisfies Theorem \ref{th:primembb3} and $H$ is a KahK3 subgroup.
All cases (III.2) -- (IV.2$^\prime$) below give its subgroups.

(III.2) $H=[\varphi,\, (\widetilde{F2}_4)^2,\,\widetilde{F2}_3]=
[\varphi]\times [(\widetilde{F2}_4)^2,\,\widetilde{F2}_3]\cong
C_2\times \AAA_4$
with \newline
$Clos(H)=[\varphi,\, \widetilde{F2}_4,\,\widetilde{F2}_3]\cong
C_2\times\SSS_4$
above and $\rk N_H=18$, $(N_H)^\ast/N_H\cong (\bz/12\bz)^2\times (\bz/2\bz)^2$.

(IV.1) $H=[\widetilde{F2}_4,\,\widetilde{F2}_3]\cong \SSS_4$
with $\rk N_H=17$, $(N_H)^\ast/N_H\cong (\bz/12\bz)^2\times \bz/4\bz$.

(IV.1$^\prime$) $H=[\varphi\widetilde{F2}_4,\,\widetilde{F2}_3]\cong \SSS_4$
with $\rk N_H=17$, $(N_H)^\ast/N_H\cong (\bz/12\bz)^2\times \bz/4\bz$.

(IV.2) $H=[(\widetilde{F2}_4)^2,\,\widetilde{F2}_3]\cong \AAA_4$
with $\rk N_H=16$, $(N_H)^\ast/N_H\cong (\bz/12\bz)^2\times (\bz/2\bz)^2$.

(IV.2$^\prime$) $H=[\varphi (\widetilde{F2}_4)^2,\,\widetilde{F2}_3]\cong \AAA_4$
with $\rk N_H=16$, $(N_H)^\ast/N_H\cong (\bz/12\bz)^2\times (\bz/2\bz)^2$.

\medskip

Let $X$ be marked by a primitive sublattice
$S\subset N=N_{21}=N(8A_3)$. Then $S$
must satisfy Theorem \ref{th:primembb3} and
$\Gamma(P(S))\subset \Gamma(P(N_{21}))=8\aaa_3$.
Any such $S$ gives marking of some $X$ and $P(X)\cap S=P(S)$.

If $N_H\subset S$ where $H$ has type (I.1), (III.1) ($H\cong C_2\times \SSS_4$), or
(II.1) ($H\cong L_2(7)\cong GL(3,\bff_2)$)
then $\Aut(X,S)_0=H$.
Otherwise, if only $N_H\subset S$ where $H$ has type
(I.3) ($H\cong C_2\times \DDD_6\cong \DDD_{12}$),
(I.5) ($H\cong C_2\times \DDD_8$), or (II.2) ($H\cong C_7\rtimes C_3$),
then $\Aut(X,S)_0=H$.
Otherwise, if only $N_H\subset S$ where
$H$ has type
(I.6), (I.7) ($H\cong (C_2)^3$),
(II.3), (II.3$^\prime$),
(IV.1) or (IV.1$^\prime$) ($H\cong \SSS_4$)
then $\Aut(X,S)_0=H$.
Otherwise, if only
$N_H\subset S$ where
$H$ has type
(I.9) ($H\cong (C_2)^2$), (II.4), (IV.2), (IV.2$^\prime$) ($H\cong \AAA_4$),
(II.7), (II.7$^\prime$), (II.7$^{\prime\prime}$), or
(II.7$^{\prime\prime\prime}$) ($H\cong \DDD_8$)
then $\Aut(X,S)_0=H$.
Otherwise, if only
$N_H\subset S$ where
$H$ has type
(I.10) ($H\cong C_2$), (II.5), or (II.5$^\prime$) ($H\cong \DDD_6$),
then $\Aut(X,S)_0=H$.
Otherwise, if only
$N_H\subset S$ where
$H$ has type
(II.6) ($H\cong C_3$), (II.8), (II.8$^\prime$), (II.8$^{\prime\prime}$), (II.9),
(II.9$^\prime$) ($H\cong (C_2)^2$), (II.10), or (II.10$^\prime$) ($H\cong C_4$)
then $\Aut(X,S)_0=H$.
Otherwise, if only
$N_H\subset S$ where $H$ has type
(II.11), or (II,11$^\prime$) ($H\cong C_2$),
then $\Aut(X,S)_0=H$.
Otherwise, $\Aut(X,S)_0$ is trivial.

Let $H=[\widetilde{F2}_4,\,\widetilde{F2}_3]\cong \SSS_4$
has type (IV.1). Let
$S=[(A_3)_2=[\alpha_{12},\alpha_{22},
\alpha_{32}],\, N_H]_{pr}\subset N_{21}$.
We have $\rk S=19$ and $S^\ast/S\cong \bz/12\bz\times \bz/3\bz$.
Thus, $S$ satisfies Theorem \ref{th:primembb3} and
gives marking of some $X$.
We have  $H\subset \Aut(X,S)_0$ where $H\cong \SSS_4$, and
$(\aaa_3)_2(\aaa_3)_3(\aaa_3)_4(\aaa_3)_6=\Gamma(P(X)\cap S)=P(S)$.
By classification  of Niemeier lattices and our calculations
above, $X$ can be marked by the Niemeier lattice $N_{21}=N(8A_3)$
only for such $S\subset S_X$.

\medskip

Below we consider cases 22 and 23 when $\Aut(X,S)_0$ can be very large. 
We use Mukai's classification of abstract finite symplectic automorphism 
groups of K3 surfaces and its amplification by Xiao \cite{Xiao}. 
We follow Hashimoto \cite{Hash} in numbering them by $n=1, 2, \dots 81$
(see cases 22 and 23 below). For fixed $n$ and the corresponding abstract 
group $H$, we follow to Hashimoto \cite{Hash} in using GAP Program \cite{GAP} 
for finding of conjugacy classes of $H$ in
$A22=A(N_{22})$ and $A23=A(N_{23})$. In particular,
the GAP invariant $i$ of these groups (found by Hashimoto) is very useful.
Then the abstract finite group $H$ is given by the GAP command
H:=SmallGroup($|H|$,$i$).  Its conjugacy classes in $A22$ and $A23$ are
given by the GAP command IsomorphicSubgroups(A22,H) and 
IsomorphicSubgroups(A23,H) respectively
when we identify $A22$ and $A23$ with the corresponding subgroups 
of the permutation group 
${\frak S}_{24}$. For each of the conjugacy classes of $H$, 
we calculate (using algorithms and 
programs above) invariants of Theorem
\ref{th:primembb3} to find out  if $N_H$ has a primitive embedding 
into $L_{K3}$ and it 
is a KahK3 conjugacy class. We give $\rk N_H$, $(N_H)^\ast/N_H$, and
some additional invariants, if necessary, which show that
the conjugacy class is a KahK3 conjugacy class according to
Theorem \ref{th:primembb3}. All other conjugacy classes of $H$ do not satisfy
Theorem \ref{th:primembb3}, and they don't give KahK3 conjugacy classes.
Surprisingly, these invariants are the same for all
KahK3 conjugacy classses for the fixed $n$ as our calculations show
(we must say that our calculations agree to
calculations by Hashimoto in \cite{Hash}). We give them for each $n$.
We numerate these conjugacy classes as $H_{n,m}$, where
$m\in \bn$ (and by $H_{41,1,1}$ and $H_{41,1,2}$ for
$n=41$ and $A(N_{23})$). We must say that some of our calculations for $A(N_{23})$
repeat calculations by Mukai in Appendix to \cite{Kon}.

\medskip

{\bf Case 22.} For the Niemeier lattice $N_{22}$, we have
$$
N=N_{22}=N(12A_2)=[12A_2,\ [2(11211122212)]]=[12A_2,\
$$
$$
-\varepsilon_{1}+\varepsilon_{2}+\varepsilon_{3}-\varepsilon_{4}+
\varepsilon_{5}+\varepsilon_{6}+\varepsilon_{7}-\varepsilon_{8}
-\varepsilon_{9}-\varepsilon_{10}+\varepsilon_{11}-\varepsilon_{12},
$$
$$
-\varepsilon_{1}-\varepsilon_{2}+\varepsilon_{3}+\varepsilon_{4}-
\varepsilon_{5}+\varepsilon_{6}+\varepsilon_{7}+\varepsilon_{8}
-\varepsilon_{9}-\varepsilon_{10}-\varepsilon_{11}+\varepsilon_{12},
$$
$$
-\varepsilon_{1}+\varepsilon_{2}-\varepsilon_{3}+\varepsilon_{4}+
\varepsilon_{5}-\varepsilon_{6}+\varepsilon_{7}+\varepsilon_{8}
+\varepsilon_{9}-\varepsilon_{10}-\varepsilon_{11}-\varepsilon_{12},
$$
$$
-\varepsilon_{1}-\varepsilon_{2}+\varepsilon_{3}-\varepsilon_{4}+
\varepsilon_{5}+\varepsilon_{6}-\varepsilon_{7}+\varepsilon_{8}
+\varepsilon_{9}+\varepsilon_{10}-\varepsilon_{11}-\varepsilon_{12},
$$
$$
-\varepsilon_{1}-\varepsilon_{2}-\varepsilon_{3}+\varepsilon_{4}-
\varepsilon_{5}+\varepsilon_{6}+\varepsilon_{7}-\varepsilon_{8}
+\varepsilon_{9}+\varepsilon_{10}+\varepsilon_{11}-\varepsilon_{12},
$$
$$
-\varepsilon_{1}-\varepsilon_{2}-\varepsilon_{3}-\varepsilon_{4}+
\varepsilon_{5}-\varepsilon_{6}+\varepsilon_{7}+\varepsilon_{8}
-\varepsilon_{9}+\varepsilon_{10}+\varepsilon_{11}+\varepsilon_{12},
$$
$$
-\varepsilon_{1}+\varepsilon_{2}-\varepsilon_{3}-\varepsilon_{4}-
\varepsilon_{5}+\varepsilon_{6}-\varepsilon_{7}+\varepsilon_{8}
+\varepsilon_{9}-\varepsilon_{10}+\varepsilon_{11}+\varepsilon_{12},
$$
$$
-\varepsilon_{1}+\varepsilon_{2}+\varepsilon_{3}-\varepsilon_{4}-\varepsilon_{5}-
\varepsilon_{6}+\varepsilon_{7}-\varepsilon_{8}
+\varepsilon_{9}+\varepsilon_{10}-\varepsilon_{11}+\varepsilon_{12},
$$
$$
-\varepsilon_{1}+\varepsilon_{2}+\varepsilon_{3}+\varepsilon_{4}-
\varepsilon_{5}-\varepsilon_{6}-\varepsilon_{7}+\varepsilon_{8}
-\varepsilon_{9}+\varepsilon_{10}+\varepsilon_{11}-\varepsilon_{12},
$$
$$
-\varepsilon_{1}-\varepsilon_{2}+\varepsilon_{3}+\varepsilon_{4}+
\varepsilon_{5}-\varepsilon_{6}-\varepsilon_{7}-\varepsilon_{8}
+\varepsilon_{9}-\varepsilon_{10}+\varepsilon_{11}+\varepsilon_{12},
$$
$$
-\varepsilon_{1}+\varepsilon_{2}-\varepsilon_{3}+\varepsilon_{4}+
\varepsilon_{5}+\varepsilon_{6}-\varepsilon_{7}-\varepsilon_{8}
-\varepsilon_{9}+\varepsilon_{10}-\varepsilon_{11}+\varepsilon_{12}]\,
$$
(see \cite[Ch. 16]{CS}) where $\varepsilon_{k}=\varepsilon_{1,k}$, 
$k=1,2,\dots, 24$. Equivalently,
$$
N=N_{22}=N(12A_2)=[12A_2,\
\varepsilon_{1}+\varepsilon_{7}-\varepsilon_{9}+\varepsilon_{10}-
\varepsilon_{11}-\varepsilon_{12},\
$$
$$
\varepsilon_{2}+\varepsilon_{7}-\varepsilon_{8}-\varepsilon_{9}-
\varepsilon_{10}+\varepsilon_{11},\ 
\varepsilon_{3}+\varepsilon_{8}-\varepsilon_{9}-\varepsilon_{10}-
\varepsilon_{11}+\varepsilon_{12},\
\varepsilon_{4}-\varepsilon_{7}+\varepsilon_{8}-\varepsilon_{9}+
\varepsilon_{11}-\varepsilon_{12},\
$$
$$
\varepsilon_{5}+\varepsilon_{7}+\varepsilon_{8}+\varepsilon_{10}+
\varepsilon_{11}+\varepsilon_{12},\
\varepsilon_{6}-\varepsilon_{7}-\varepsilon_{8}-\varepsilon_{9}+
\varepsilon_{10}+\varepsilon_{12}]\
$$
for the reduced basis of the cord group.
The group $A=A(N_{22})$ consists of the cyclic group
$[\varphi_0]$ of order $2$ where
$$
\varphi_0=(\alpha_{1,1}\alpha_{2,1})
(\alpha_{1,2}\alpha_{2,2})\cdots (\alpha_{1,12}\alpha_{2,12})
$$
gives non-trivial
involutions on all $12$ components $12\aaa_{2}$
and $A/[\varphi_0]=M_{12}$ is the Mathieu
group $M_{12}$ on $12$ components $12\aaa_{2}$.
We have:
$$
M_{12}=[\varphi_1=(1)(2,3,4,5,6,7,8,9,10,11,12),\
$$
$$
\varphi_2=(1)(2)(3)(10)(4,8,12,9)(5,11,6,7),\ \
\varphi_3=(1,2)(3,12)(4,7)(5,9)(6,10)(8,11)].
$$
where $1,\dots, 12$ numerate components of $12\aaa_2$. Possible lifts
of the generators $\varphi_1,\varphi_2,\varphi_3$ of $M_{12}$
to elements of $A(N_{22})$ are respectively
$$
\widetilde{\varphi}_1=
(\alpha_{1,1})(\alpha_{2,1})(\alpha_{1,2}\alpha_{1,3}\dots
\alpha_{1,11}\alpha_{1,12})
(\alpha_{2,2}\alpha_{2,3}\dots\alpha_{1,11}\alpha_{1,12}),
$$
$$
\widetilde{\varphi}_2=(\alpha_{1,1})(\alpha_{2,1})(\alpha_{1,2}\alpha_{2,2})
(\alpha_{1,3}\alpha_{2,3})(\alpha_{1,10})(\alpha_{2,10})
(\alpha_{1,4}\alpha_{1,8}\alpha_{1,12}\alpha_{1,9})
$$
$$
(\alpha_{2,4}\alpha_{2,8}\alpha_{2,12}\alpha_{2,9})
(\alpha_{1,5}\alpha_{2,11}\alpha_{1,6}\alpha_{2,7})
(\alpha_{2,5}\alpha_{1,11}\alpha_{2,6}\alpha_{1,7}),
$$
$$
\widetilde{\varphi}_3=(\alpha_{1,1}\alpha_{1,2}\alpha_{2,1}\alpha_{2,2})
(\alpha_{1,3}\alpha_{2,12}\alpha_{2,3}\alpha_{1,12})
(\alpha_{1,4}\alpha_{1,7}\alpha_{2,4}\alpha_{2,7})
$$
$$
(\alpha_{1,5}\alpha_{2,9}\alpha_{2,5}\alpha_{1,9})
(\alpha_{1,6}\alpha_{2,10}\alpha_{2,6}\alpha_{1,10})
(\alpha_{1,8}\alpha_{1,11}\alpha_{2,8}\alpha_{2,11}).
$$
Thus, $A(M_{22})=[\varphi_0,\widetilde{\varphi}_{1},\widetilde{\varphi}_2,
\widetilde{\varphi}_3]$.

Using GAP Progam \cite{GAP}, we obtain the following classification.

\newpage

\centerline {\bf Classification of KahK3 conjugacy classes for $A(N_{22})$.}

\vskip1cm

{\bf n=79,} $H\cong {\frak A}_{6}$ ($|H|=360$, $i=118$):
$\rk N_H=19$ and $(N_H)^\ast/N_H\cong \bz/60\bz\times \bz/3\bz$.
$$
H_{79,1}=
$$
$$
[(\alpha_{1,1}\alpha_{1,3}\alpha_{1,7}\alpha_{1,8}\alpha_{1,11})
(\alpha_{2,1}\alpha_{2,3}\alpha_{2,7}\alpha_{2,8}\alpha_{2,11})
(\alpha_{1,4}\alpha_{1,12}\alpha_{1,5}\alpha_{1,6}\alpha_{1,10})
(\alpha_{2,4}\alpha_{2,12}\alpha_{2,5}\alpha_{2,6}\alpha_{2,10}),\
$$
$$
(\alpha_{1,1}\alpha_{1,3}\alpha_{2,11}\alpha_{2,10}\alpha_{2,7})
(\alpha_{2,1}\alpha_{2,3}\alpha_{1,11}\alpha_{1,10}\alpha_{1,7})
(\alpha_{1,4}\alpha_{2,5}\alpha_{2,12}\alpha_{1,6}\alpha_{2,8})
(\alpha_{2,4}\alpha_{1,5}\alpha_{1,12}\alpha_{2,6}\alpha_{1,8})].
$$
with orbits (here and in what follows we show orbits with more than one elements only)
$
\{\alpha_{1,1},\alpha_{1,3},\alpha_{1,7},\alpha_{2,11},\alpha_{1,8},
\alpha_{2,1},\alpha_{2,10},\alpha_{1,11},\alpha_{2,4},
\alpha_{2,3},\alpha_{2,7},\alpha_{1,10},\alpha_{2,12},
\alpha_{1,5},\alpha_{2,8},\alpha_{1,4},
\newline
\alpha_{2,5},\alpha_{1,6},\alpha_{1,12},\alpha_{2,6}\}
$.

\medskip

{\bf n=70,} $H\cong {\frak S}_5$ ($|H|=120$, $i=34$):
$\rk N_H=19$ and $(N_H)^\ast/N_H\cong \bz/60\bz \times \bz/5\bz$.
$$
H_{70,1}=
$$
$$
[(\alpha_{1,2}\alpha_{2,3}\alpha_{1,4}\alpha_{1,5}\alpha_{2,11})
(\alpha_{2,2}\alpha_{1,3}\alpha_{2,4}\alpha_{2,5}\alpha_{1,11})
(\alpha_{1,6}\alpha_{2,8}\alpha_{2,7}\alpha_{1,10}\alpha_{1,12})
(\alpha_{2,6}\alpha_{1,8}\alpha_{1,7}\alpha_{2,10}\alpha_{2,12}),
$$
$$
(\alpha_{1,3}\alpha_{2,4})(\alpha_{2,3}\alpha_{1,4})
(\alpha_{1,6}\alpha_{2,12})(\alpha_{2,6}\alpha_{1,12})
(\alpha_{1,7}\alpha_{2,8})(\alpha_{2,7}\alpha_{1,8})
(\alpha_{1,9}\alpha_{1,10})(\alpha_{2,9}\alpha_{2,10})]
$$
with orbits
$
\{\alpha_{1,2},\alpha_{2,3},\alpha_{1,4},\alpha_{1,5},\alpha_{2,11}\},
\{\alpha_{2,2},\alpha_{1,3},\alpha_{2,4},\alpha_{2,5},\alpha_{1,11}\},
\newline
\{\alpha_{1,6},\alpha_{2,8},\alpha_{2,12},\alpha_{2,7},\alpha_{1,7},\alpha_{2,6},
\alpha_{1,10},\alpha_{1,8},\alpha_{2,10},\alpha_{1,12},\alpha_{1,9},\alpha_{2,9}\}
$.

\medskip

{\bf n=63,} $H\cong M_9$ ($|H|=72$, $i=41$):
$\rk N_H=19$ and
$(N_H)^\ast/N_H\cong \bz/18\bz \times \bz/6\bz\times \bz/2\bz$.
$$
H_{63,1}=
$$
$$
[(\alpha_{1,3}\alpha_{2,3})
(\alpha_{1,4}\alpha_{2,9}\alpha_{1,5}\alpha_{2,12})
(\alpha_{2,4}\alpha_{1,9}\alpha_{2,5}\alpha_{1,12})
(\alpha_{1,6}\alpha_{1,7}\alpha_{1,8}\alpha_{1,10})
(\alpha_{2,6}\alpha_{2,7}\alpha_{2,8}\alpha_{2,10})
(\alpha_{1,11}\alpha_{2,11}),
$$
$$
(\alpha_{1,1}\alpha_{2,1})
(\alpha_{1,3}\alpha_{2,10}\alpha_{2,4}\alpha_{2,9})
(\alpha_{2,3}\alpha_{1,10}\alpha_{1,4}\alpha_{1,9})
(\alpha_{1,6}\alpha_{1,8}\alpha_{2,12}\alpha_{2,7})
(\alpha_{2,6}\alpha_{2,8}\alpha_{1,12}\alpha_{1,7})
(\alpha_{1,11}\alpha_{2,11})]
$$
with orbits
$
\{\alpha_{1,1},\alpha_{2,1}\},
\{\alpha_{1,3},\alpha_{2,3},\alpha_{2,10},\alpha_{1,10},\alpha_{2,6},
\alpha_{2,4},\alpha_{1,6},\alpha_{1,4},\alpha_{2,7},\alpha_{2,8},\alpha_{1,9},\alpha_{2,9},
\newline
\alpha_{1,7},\alpha_{1,8},\alpha_{1,12},\alpha_{2,5},\alpha_{1,5},\alpha_{2,12}\},
\{\alpha_{1,11},\alpha_{2,11}\}
$;
$$
H_{63,2}=
$$
$$
[(\alpha_{1,2}\alpha_{2,2})
(\alpha_{1,3}\alpha_{2,4}\alpha_{1,8}\alpha_{2,10})
(\alpha_{2,3}\alpha_{1,4}\alpha_{2,8}\alpha_{1,10})
(\alpha_{1,5}\alpha_{1,9}\alpha_{1,12}\alpha_{1,7})
(\alpha_{2,5}\alpha_{2,9}\alpha_{2,12}\alpha_{2,7})
(\alpha_{1,11}\alpha_{2,11}),
$$
$$
(\alpha_{1,1}\alpha_{2,1})
(\alpha_{1,3}\alpha_{2,10}\alpha_{2,4}\alpha_{2,9})
(\alpha_{2,3}\alpha_{1,10}\alpha_{1,4}\alpha_{1,9})
(\alpha_{1,6}\alpha_{1,8}\alpha_{2,12}\alpha_{2,7})
(\alpha_{2,6}\alpha_{2,8}\alpha_{1,12}\alpha_{1,7})
(\alpha_{1,11}\alpha_{2,11})]
$$
with orbits
$
\{\alpha_{1,1},\alpha_{2,1}\},
\{\alpha_{1,2},\alpha_{2,2}\},
\{\alpha_{1,3},\alpha_{2,4},\alpha_{2,10},\alpha_{1,8},\alpha_{2,9},
\alpha_{2,12},\alpha_{2,7},\alpha_{2,5},\alpha_{1,6}\},
\newline
\{\alpha_{2,3},\alpha_{1,4},\alpha_{1,10},\alpha_{2,8},\alpha_{1,9},
\alpha_{1,12},\alpha_{1,7},\alpha_{1,5},\alpha_{2,6}\},
\{\alpha_{1,11},\alpha_{2,11}\}
$.

\medskip

{\bf n=62,} $H\cong N_{72}$ ($|H|=72$, $i=40$):
$\rk N_H=19$,
$(N_H)^\ast/N_H\cong \bz/36\bz \times (\bz/3\bz)^2$ and
$\det(K((q_{N_H})_3))\equiv - 2^2\cdot 3^4\mod (\bz_3^\ast)^2$.
$$
H_{62,1}=
[(\alpha_{1,3}\alpha_{2,12})(\alpha_{2,3}\alpha_{1,12})
(\alpha_{1,4}\alpha_{2,5})(\alpha_{2,4}\alpha_{1,5})
(\alpha_{1,8}\alpha_{1,10})(\alpha_{2,8}\alpha_{2,10})
(\alpha_{1,9}\alpha_{2,11})(\alpha_{2,9}\alpha_{1,11}),
$$
$$
(\alpha_{1,1}\alpha_{1,2}\alpha_{1,3}\alpha_{1,5}\alpha_{2,4}\alpha_{2,11})
(\alpha_{2,1}\alpha_{2,2}\alpha_{2,3}\alpha_{2,5}\alpha_{1,4}\alpha_{1,11})
(\alpha_{1,7}\alpha_{2,12}\alpha_{1,9})
$$
$$
(\alpha_{2,7}\alpha_{1,12}\alpha_{2,9})
(\alpha_{1,8}\alpha_{2,10})(\alpha_{2,8}\alpha_{1,10})]
$$
with orbits
$
\{\alpha_{1,1},\alpha_{1,2},\alpha_{1,3},\alpha_{2,12},\alpha_{1,5},
\alpha_{1,9},\alpha_{2,4},\alpha_{2,11},\alpha_{1,7}\},
\{\alpha_{2,1},\alpha_{2,2},\alpha_{2,3},\alpha_{1,12},\alpha_{2,5},
\newline
\alpha_{2,9},\alpha_{1,4},\alpha_{1,11},\alpha_{2,7}\},
\{\alpha_{1,8},\alpha_{1,10},\alpha_{2,10},\alpha_{2,8}\}
$.

\medskip

{\bf n=55,} $H\cong {\frak A}_5$ ($|H|=60$, $i=5$):
$\rk N_H=18$ and $(N_H)^\ast/N_H \cong\bz/30\bz\times \bz/10\bz$.
$$
H_{55,1}=
$$
$$
[(\alpha_{1,2}\alpha_{2,3}\alpha_{2,11}\alpha_{1,4}\alpha_{1,5})
(\alpha_{2,2}\alpha_{1,3}\alpha_{1,11}\alpha_{2,4}\alpha_{2,5})
(\alpha_{1,6}\alpha_{2,7}\alpha_{2,9}\alpha_{1,12}\alpha_{2,8})
(\alpha_{2,6}\alpha_{1,7}\alpha_{1,9}\alpha_{2,12}\alpha_{1,8}),
$$
$$
(\alpha_{1,1}\alpha_{1,3}\alpha_{2,4})
(\alpha_{2,1}\alpha_{2,3}\alpha_{1,4})
(\alpha_{1,2}\alpha_{1,5}\alpha_{2,11})
(\alpha_{2,2}\alpha_{2,5}\alpha_{1,11})
(\alpha_{1,7}\alpha_{1,9}\alpha_{2,12})
(\alpha_{2,7}\alpha_{2,9}\alpha_{1,12})]
$$
with orbits
$
\{\alpha_{1,1},\alpha_{1,3},\alpha_{1,11},\alpha_{2,4},\alpha_{2,2},\alpha_{2,5}\},
\{\alpha_{2,1},\alpha_{2,3},\alpha_{2,11},\alpha_{1,4},\alpha_{1,2},\alpha_{1,5}\},
\newline
\{\alpha_{1,6},\alpha_{2,7},\alpha_{2,9},\alpha_{1,12},\alpha_{2,8}\},
\{\alpha_{2,6},\alpha_{1,7},\alpha_{1,9},\alpha_{2,12},\alpha_{1,8}\}
$;
$$
H_{55,2}=
$$
$$
[(\alpha_{1,1}\alpha_{1,2}\alpha_{2,4}\alpha_{2,7}\alpha_{1,5})
(\alpha_{2,1}\alpha_{2,2}\alpha_{1,4}\alpha_{1,7}\alpha_{2,5})
(\alpha_{1,3}\alpha_{2,9}\alpha_{1,12}\alpha_{2,11}\alpha_{2,10})
(\alpha_{2,3}\alpha_{1,9}\alpha_{2,12}\alpha_{1,11}\alpha_{1,10}),
$$
$$
(\alpha_{1,1}\alpha_{1,3}\alpha_{2,4})
(\alpha_{2,1}\alpha_{2,3}\alpha_{1,4})
(\alpha_{1,2}\alpha_{1,5}\alpha_{2,11})
(\alpha_{2,2}\alpha_{2,5}\alpha_{1,11})
(\alpha_{1,7}\alpha_{1,9}\alpha_{2,12})
(\alpha_{2,7}\alpha_{2,9}\alpha_{1,12})]
$$
with orbits
$
\{\alpha_{1,1},\alpha_{1,2},\alpha_{1,3},\alpha_{2,4},\alpha_{1,5},
\alpha_{2,9},\alpha_{2,7},\alpha_{2,11},\alpha_{1,12},\alpha_{2,10}\},
\newline
\{\alpha_{2,1},\alpha_{2,2},\alpha_{2,3},\alpha_{1,4},\alpha_{2,5},
\alpha_{1,9},\alpha_{1,7},\alpha_{1,11},\alpha_{2,12},\alpha_{1,10}\}
$.

\medskip

{\bf n=54,} $H\cong T_{48}$ ($|H|=48$, $i=29$):
$\rk N_H=19$ and $(N_H)^\ast/N_H\cong
\bz/24\bz\times \bz/8\bz\times \bz/2\bz$.
$$
H_{54,1}=[
(\alpha_{1,2}\alpha_{1,11})(\alpha_{2,2}\alpha_{2,11})
(\alpha_{1,3}\alpha_{1,6})(\alpha_{2,3}\alpha_{2,6})
(\alpha_{1,4}\alpha_{1,12})(\alpha_{2,4}\alpha_{2,12})
(\alpha_{1,9}\alpha_{1,10})(\alpha_{2,9}\alpha_{2,10}),
$$
$$
(\alpha_{1,1}\alpha_{2,11}\alpha_{2,1}\alpha_{1,11})
(\alpha_{1,2}\alpha_{2,2})
(\alpha_{1,3}\alpha_{1,6}\alpha_{2,10}\alpha_{1,8}
\alpha_{2,4}\alpha_{2,12}\alpha_{2,9}\alpha_{2,7})
$$
$$
(\alpha_{2,3}\alpha_{2,6}\alpha_{1,10}\alpha_{2,8}
\alpha_{1,4}\alpha_{1,12}\alpha_{1,9}\alpha_{1,7})]
$$
with orbits
$
\{\alpha_{1,1},\alpha_{2,11},\alpha_{2,2},\alpha_{2,1},\alpha_{1,2},\alpha_{1,11}\},
\{\alpha_{1,3},\alpha_{1,6},\alpha_{2,10},\alpha_{2,9},
\alpha_{1,8},\alpha_{2,7},\alpha_{2,4},\alpha_{2,12}\}
\newline
\{\alpha_{2,3},\alpha_{2,6},\alpha_{1,10},\alpha_{1,9},
\alpha_{2,8},\alpha_{1,7},\alpha_{1,4},\alpha_{1,12}\}
$.

\medskip

{\bf n=48,} $H\cong {\frak S}_{3,3}$ ($|H|=36$, $i=10$):
$\rk N_H=18$, $(N_H)^\ast/N_H \cong
\bz/18\bz\times
\bz/6\bz\times (\bz/3\bz)^2$ and
$\det(K((q_{N_H})_3))\equiv -2^2\cdot 3^5\mod (\bz_3^\ast)^2$.
$$
H_{48,1}=
(\alpha_{1,2}\alpha_{1,7})(\alpha_{2,2}\alpha_{2,7})
(\alpha_{1,3}\alpha_{2,4})(\alpha_{2,3}\alpha_{1,4})
(\alpha_{1,5}\alpha_{2,12})(\alpha_{2,5}\alpha_{1,12})
(\alpha_{1,9}\alpha_{2,11})(\alpha_{2,9}\alpha_{1,11}),
$$
$$
(\alpha_{1,1}\alpha_{1,2}\alpha_{1,3}\alpha_{1,5}\alpha_{2,4}\alpha_{2,11})
(\alpha_{2,1}\alpha_{2,2}\alpha_{2,3}\alpha_{2,5}\alpha_{1,4}\alpha_{1,11})
(\alpha_{1,7}\alpha_{2,12}\alpha_{1,9})
(\alpha_{2,7}\alpha_{1,12}\alpha_{2,9})
$$
$$
(\alpha_{1,8}\alpha_{2,10})
(\alpha_{2,8}\alpha_{1,10})]
$$
with orbits
$
\{\alpha_{1,1},\alpha_{1,2},\alpha_{1,7},\alpha_{1,3},\alpha_{2,12},
\alpha_{2,4},\alpha_{1,5},\alpha_{1,9},\alpha_{2,11}\},
\newline
\{\alpha_{2,1},\alpha_{2,2},\alpha_{2,7},\alpha_{2,3},\alpha_{1,12},
\alpha_{1,4},\alpha_{2,5},\alpha_{2,9},\alpha_{1,11}\},
\{\alpha_{1,8},\alpha_{2,10}\},\{\alpha_{2,8},\alpha_{1,10}\}
$;
$$
H_{48,2}=[
(\alpha_{1,3}\alpha_{2,4})(\alpha_{2,3}\alpha_{1,4})
(\alpha_{1,5}\alpha_{2,11})(\alpha_{2,5}\alpha_{1,11})
(\alpha_{1,6}\alpha_{1,10})(\alpha_{2,6}\alpha_{2,10})
(\alpha_{1,9}\alpha_{2,12})(\alpha_{2,9}\alpha_{1,12}),
$$
$$
(\alpha_{1,1}\alpha_{1,2}\alpha_{1,3}\alpha_{1,5}\alpha_{2,4}\alpha_{2,11})
(\alpha_{2,1}\alpha_{2,2}\alpha_{2,3}\alpha_{2,5}\alpha_{1,4}\alpha_{1,11})
(\alpha_{1,7}\alpha_{2,12}\alpha_{1,9})
(\alpha_{2,7}\alpha_{1,12}\alpha_{2,9})
$$
$$
(\alpha_{1,8}\alpha_{2,10})
(\alpha_{2,8}\alpha_{1,10})]
$$
with orbits
$
\{\alpha_{1,1},\alpha_{1,2},\alpha_{1,3},\alpha_{2,4},\alpha_{1,5},\alpha_{2,11}\},
\{\alpha_{2,1},\alpha_{2,2},\alpha_{2,3},\alpha_{1,4},\alpha_{2,5},\alpha_{1,11}\},
\{\alpha_{1,6},\alpha_{1,10},\alpha_{2,8}\},
\newline
\{\alpha_{2,6},\alpha_{2,10},\alpha_{1,8}\},
\{\alpha_{1,7},\alpha_{2,12},\alpha_{1,9}\},
\{\alpha_{2,7},\alpha_{1,12},\alpha_{2,9}\}
$.

\medskip

{\bf n=46,} $H\cong 3^2 C_4$ ($|H|=36$, $i=9$):
$\rk N_H=18$ and $(N_H)^\ast/N_H\cong \bz/18\bz\times \bz/6\bz\times \bz/3\bz$.
$$
H_{46,1}=
$$
$$
[(\alpha_{1,2}\alpha_{1,3}\alpha_{1,7}\alpha_{2,4})
(\alpha_{2,2}\alpha_{2,3}\alpha_{2,7}\alpha_{1,4})
(\alpha_{1,5}\alpha_{1,9}\alpha_{2,12}\alpha_{2,11})
(\alpha_{2,5}\alpha_{2,9}\alpha_{1,12}\alpha_{1,11})
(\alpha_{1,6}\alpha_{2,6})(\alpha_{1,10}\alpha_{2,10}),
$$
$$
(\alpha_{1,1}\alpha_{1,3}\alpha_{2,4})
(\alpha_{2,1}\alpha_{2,3}\alpha_{1,4})
(\alpha_{1,2}\alpha_{1,5}\alpha_{2,11})
(\alpha_{2,2}\alpha_{2,5}\alpha_{1,11})
(\alpha_{1,7}\alpha_{1,9}\alpha_{2,12})
(\alpha_{2,7}\alpha_{2,9}\alpha_{1,12})]
$$
with orbits
$
\{\alpha_{1,1},\alpha_{1,3},\alpha_{1,7},\alpha_{2,4},
\alpha_{1,9},\alpha_{1,2},\alpha_{2,12},\alpha_{1,5},\alpha_{2,11}\},
\{\alpha_{2,1},\alpha_{2,3},\alpha_{2,7},\alpha_{1,4},\alpha_{2,9},
\newline
\alpha_{2,2},\alpha_{1,12},\alpha_{2,5},\alpha_{1,11}\},
\{\alpha_{1,6},\alpha_{2,6}\},\{\alpha_{1,10},\alpha_{2,10}\}
$;
$$
H_{46,2}=
$$
$$
[(\alpha_{1,1}\alpha_{2,1})
(\alpha_{1,2}\alpha_{2,3}\alpha_{1,7}\alpha_{1,4})
(\alpha_{2,2}\alpha_{1,3}\alpha_{2,7}\alpha_{2,4})
(\alpha_{1,5}\alpha_{2,9}\alpha_{2,12}\alpha_{1,11})
(\alpha_{2,5}\alpha_{1,9}\alpha_{1,12}\alpha_{1,11})
(\alpha_{1,8}\alpha_{2,8}),
$$
$$
(\alpha_{1,1}\alpha_{1,3}\alpha_{2,4})
(\alpha_{2,1}\alpha_{2,3}\alpha_{1,4})
(\alpha_{1,2}\alpha_{1,5}\alpha_{2,11})
(\alpha_{2,2}\alpha_{2,5}\alpha_{1,11})
(\alpha_{1,7}\alpha_{1,9}\alpha_{2,12})
(\alpha_{2,7}\alpha_{2,9}\alpha_{1,12})]
$$
with orbits
$
\{\alpha_{1,1},\alpha_{2,1},\alpha_{1,3},\alpha_{2,3},
\alpha_{2,7},\alpha_{2,4},\alpha_{1,7},\alpha_{1,4},\alpha_{2,9},
\alpha_{2,2},\alpha_{1,9},\alpha_{1,2},\alpha_{2,12},\alpha_{1,12},
\alpha_{2,5},\alpha_{1,5},
\newline
\alpha_{1,11},\alpha_{2,11}\},
\{\alpha_{1,8},\alpha_{2,8}\}
$.

\medskip

{\bf n=38,} $H\cong T_{24}$ ($|H|=24$, $i=3$):
$$
H_{38,1}=[
(\alpha_{1,1}\alpha_{1,2}\alpha_{2,11})
(\alpha_{2,1}\alpha_{2,2}\alpha_{1,11})
(\alpha_{1,3}\alpha_{2,7}\alpha_{2,10})
(\alpha_{2,3}\alpha_{1,7}\alpha_{1,10})
(\alpha_{1,4}\alpha_{2,8}\alpha_{1,9})
(\alpha_{2,4}\alpha_{1,8}\alpha_{2,9}),
$$
$$
(\alpha_{1,2}\alpha_{2,2})
(\alpha_{1,3}\alpha_{1,6}\alpha_{2,4}\alpha_{2,12})
(\alpha_{2,3}\alpha_{2,6}\alpha_{1,4}\alpha_{1,12})
(\alpha_{1,7}\alpha_{1,9}\alpha_{2,8}\alpha_{1,10})
(\alpha_{2,7}\alpha_{2,9}\alpha_{1,8}\alpha_{2,10})
(\alpha_{1,11}\alpha_{2,11})]
$$
with $Clos(H_{38,1})=H_{54,1}$.

\medskip

{\bf n=34,} $H\cong {\frak S}_4$ ($|H|=24$, $i=12$):
$\rk N_H=17$ and $(N_H)^\ast/N_H\cong (\bz/12\bz)^2\times \bz/4\bz$.
$$
H_{34,1}=[
(\alpha_{1,2}\alpha_{2,5})(\alpha_{2,2}\alpha_{1,5})
(\alpha_{1,4}\alpha_{2,8})(\alpha_{2,4}\alpha_{1,8})
(\alpha_{1,7}\alpha_{2,9})(\alpha_{2,7}\alpha_{1,9})
(\alpha_{1,11}\alpha_{1,12})(\alpha_{2,11}\alpha_{2,12}),
$$
$$
(\alpha_{1,1}\alpha_{1,3}\alpha_{2,4})
(\alpha_{2,1}\alpha_{2,3}\alpha_{1,4})
(\alpha_{1,2}\alpha_{1,5}\alpha_{2,11})
(\alpha_{2,2}\alpha_{2,5}\alpha_{1,11})
(\alpha_{1,7}\alpha_{1,9}\alpha_{2,12})
(\alpha_{2,7}\alpha_{2,9}\alpha_{1,12})]
$$
with orbits
$
\{\alpha_{1,1},\alpha_{1,3},\alpha_{2,4},\alpha_{1,8}\},
\{\alpha_{2,1},\alpha_{2,3},\alpha_{1,4},\alpha_{2,8}\},
\newline
\{\alpha_{1,2},\alpha_{2,5},\alpha_{1,5},\alpha_{1,11},\alpha_{2,2},\alpha_{2,11},
\alpha_{1,12},\alpha_{2,12},\alpha_{2,7},\alpha_{1,7},\alpha_{1,9},\alpha_{2,9}\}
$.

\medskip

{\bf n=32,} $H\cong Hol(C_5)$ ($|H|=20$, $i=3$):
$\rk N_H=18$ and $(N_H)^\ast/N_H\cong (\bz/10\bz)^2\times \bz/5\bz$.
$$
H_{32,1}=[
(\alpha_{1,4}\alpha_{1,9})(\alpha_{2,4}\alpha_{2,9})
(\alpha_{1,5}\alpha_{2,10})(\alpha_{2,5}\alpha_{1,10})
(\alpha_{1,7}\alpha_{2,12})(\alpha_{2,7}\alpha_{1,12})
(\alpha_{1,8}\alpha_{2,11})(\alpha_{2,8}\alpha_{1,11})
$$
$$
(\alpha_{1,1}\alpha_{2,1})
(\alpha_{1,3}\alpha_{2,10}\alpha_{2,4}\alpha_{2,9})
(\alpha_{2,3}\alpha_{1,10}\alpha_{1,4}\alpha_{1,9})
(\alpha_{1,6}\alpha_{1,8}\alpha_{2,12}\alpha_{2,7})
(\alpha_{2,6}\alpha_{2,8}\alpha_{1,12}\alpha_{1,7})
(\alpha_{1,11}\alpha_{2,11})]
$$
with orbits
$
\{\alpha_{1,1},\alpha_{2,1}\},
\{\alpha_{1,3},\alpha_{2,10},\alpha_{1,5},\alpha_{2,4},\alpha_{2,9}\},
\{\alpha_{2,3},\alpha_{1,10},\alpha_{2,5},\alpha_{1,4},\alpha_{1,9}\},
\newline
\{\alpha_{1,6},\alpha_{1,8},\alpha_{2,11},\alpha_{2,12},\alpha_{1,11},
\alpha_{1,7},\alpha_{2,7},\alpha_{2,8},\alpha_{2,6},\alpha_{1,12}\}
$.

\medskip

{\bf n=31,} $H\cong C_3\times D_6$, ($|H|=18$, $i=3$):
$$
H_{31,1}=[
(\alpha_{1,1}\alpha_{1,3}\alpha_{2,4})
(\alpha_{2,1}\alpha_{2,3}\alpha_{1,4})
(\alpha_{1,2}\alpha_{1,5}\alpha_{2,11})
(\alpha_{2,2}\alpha_{2,5}\alpha_{1,11})
(\alpha_{1,7}\alpha_{1,9}\alpha_{2,12})
(\alpha_{2,7}\alpha_{2,9}\alpha_{1,12}),
$$
$$
(\alpha_{1,1}\alpha_{1,2}\alpha_{2,12}\alpha_{2,4}\alpha_{1,5}\alpha_{1,9})
(\alpha_{2,1}\alpha_{2,2}\alpha_{1,12}\alpha_{1,4}\alpha_{2,5}\alpha_{2,9})
(\alpha_{1,3}\alpha_{2,11}\alpha_{1,7})
(\alpha_{2,3}\alpha_{1,11}\alpha_{2,7})
$$
$$
(\alpha_{1,8}\alpha_{2,10})
(\alpha_{2,8}\alpha_{1,10})]
$$
with $Clos(H_{31,1})=H_{48,1}$ above;
$$
H_{31,2}=[
(\alpha_{1,1}\alpha_{1,3}\alpha_{2,4})
(\alpha_{2,1}\alpha_{2,3}\alpha_{1,4})
(\alpha_{1,2}\alpha_{1,5}\alpha_{2,11})
(\alpha_{2,2}\alpha_{2,5}\alpha_{1,11})
(\alpha_{1,7}\alpha_{1,9}\alpha_{2,12})
(\alpha_{2,7}\alpha_{2,9}\alpha_{1,12}),
$$
$$
(\alpha_{1,1}\alpha_{1,2}\alpha_{1,3}\alpha_{2,11}\alpha_{2,4}\alpha_{1,5})
(\alpha_{2,1}\alpha_{2,2}\alpha_{2,3}\alpha_{1,11}\alpha_{1,4}\alpha_{2,5})
(\alpha_{1,6}\alpha_{2,8}\alpha_{1,10})
(\alpha_{2,6}\alpha_{1,8}\alpha_{2,10})
$$
$$
(\alpha_{1,7}\alpha_{2,12})
(\alpha_{2,7}\alpha_{1,12})]
$$
with $Clos(H_{31,2})=H_{48,2}$ above.

\medskip

{\bf n=30,} $H\cong {\frak A}_{3,3}$ ($|H|=18$, $i=4$):
$\rk N_H=16$ and $(N_H)^\ast/N_H\cong \bz/9\bz\times (\bz/3\bz)^4$.
$$
H_{30,1}=[
(\alpha_{1,2}\alpha_{1,7})(\alpha_{2,2}\alpha_{2,7})
(\alpha_{1,3}\alpha_{2,4})(\alpha_{2,3}\alpha_{1,4})
(\alpha_{1,5}\alpha_{2,12})(\alpha_{2,5}\alpha_{1,12})
(\alpha_{1,9}\alpha_{2,11})(\alpha_{2,9}\alpha_{1,11}),
$$
$$
(\alpha_{1,1}\alpha_{1,2}\alpha_{1,7})(\alpha_{2,1}\alpha_{2,2}\alpha_{2,7})
(\alpha_{1,3}\alpha_{1,5}\alpha_{1,9})(\alpha_{2,3}\alpha_{2,5}\alpha_{2,9})
(\alpha_{1,4}\alpha_{1,11}\alpha_{1,12})(\alpha_{2,4}\alpha_{2,11}\alpha_{2,12}),
$$
$$
(\alpha_{1,1}\alpha_{1,3}\alpha_{2,4})(\alpha_{2,1}\alpha_{2,3}\alpha_{1,4})
(\alpha_{1,2}\alpha_{1,5}\alpha_{2,11})(\alpha_{2,2}\alpha_{2,5}\alpha_{1,11})
(\alpha_{1,7}\alpha_{1,9}\alpha_{2,12})(\alpha_{2,7}\alpha_{2,9}\alpha_{1,12})]
$$
with orbits
$
\{\alpha_{1,1},\alpha_{1,2},\alpha_{1,3},\alpha_{1,7},\alpha_{1,5},
\alpha_{2,4},\alpha_{1,9},\alpha_{2,12},\alpha_{2,11}\},
\{\alpha_{2,1},\alpha_{2,2},\alpha_{2,3},\alpha_{2,7},\alpha_{2,5},
\alpha_{1,4},\alpha_{2,9},
\newline
\alpha_{1,12},\alpha_{1,11}\}
$;
$$
H_{30,2}=[
(\alpha_{1,3}\alpha_{2,4})(\alpha_{2,3}\alpha_{1,4})
(\alpha_{1,5}\alpha_{2,11})(\alpha_{2,5}\alpha_{1,11})
(\alpha_{1,6}\alpha_{1,10})(\alpha_{2,6}\alpha_{2,10})
(\alpha_{1,9}\alpha_{2,12})(\alpha_{2,9}\alpha_{1,12}),
$$
$$
(\alpha_{1,2}\alpha_{1,5}\alpha_{2,11})(\alpha_{2,2}\alpha_{2,5}\alpha_{1,11})
(\alpha_{1,6}\alpha_{1,10}\alpha_{2,8})(\alpha_{2,6}\alpha_{2,10}\alpha_{1,8})
(\alpha_{1,7}\alpha_{2,12}\alpha_{1,9})(\alpha_{2,7}\alpha_{1,12}\alpha_{2,9}),
$$
$$
(\alpha_{1,1}\alpha_{1,3}\alpha_{2,4})(\alpha_{2,1}\alpha_{2,3}\alpha_{1,4})
(\alpha_{1,2}\alpha_{1,5}\alpha_{2,11})(\alpha_{2,2}\alpha_{2,5}\alpha_{1,11})
(\alpha_{1,7}\alpha_{1,9}\alpha_{2,12})(\alpha_{2,7}\alpha_{2,9}\alpha_{1,12})]
$$
with orbits
$
\{\alpha_{1,1},\alpha_{1,3},\alpha_{2,4}\},\{\alpha_{2,1},\alpha_{2,3},\alpha_{1,4}\},
\{\alpha_{1,2},\alpha_{1,5},\alpha_{2,11}\},\{\alpha_{2,2},\alpha_{2,5},\alpha_{1,11}\},
\newline 
\{\alpha_{1,6},\alpha_{1,10},\alpha_{2,8}\},
\{\alpha_{2,6},\alpha_{2,10},\alpha_{1,8}\},
\{\alpha_{1,7},\alpha_{2,12},\alpha_{1,9}\},
\{\alpha_{14 2,7},\alpha_{1,12},\alpha_{2,9}\}
$.

\medskip

{\bf n=26,} $H\cong SD_{16}$ ($|H|=16$, $i=8$):
$\rk N_H=18$ and
$(N_H)^\ast/N_H\cong (\bz/8\bz)^2\times \bz/4\bz\times \bz/2\bz$.
$$
H_{26,1}=[
(\alpha_{1,2}\alpha_{2,2})(\alpha_{1,4}\alpha_{2,4})
(\alpha_{1,5}\alpha_{2,6}\alpha_{2,11}\alpha_{2,9})
(\alpha_{2,5}\alpha_{1,6}\alpha_{1,11}\alpha_{1,9})
(\alpha_{1,7}\alpha_{2,10}\alpha_{2,8}\alpha_{1,12})
(\alpha_{2,7}\alpha_{1,10}
$$
$$
\alpha_{1,8}\alpha_{2,12}),
(\alpha_{1,3}\alpha_{2,4})(\alpha_{2,3}\alpha_{1,4})
(\alpha_{1,6}\alpha_{2,12})(\alpha_{2,6}\alpha_{1,12})
(\alpha_{1,7}\alpha_{2,8})(\alpha_{2,7}\alpha_{1,8})
(\alpha_{1,9}\alpha_{1,10})(\alpha_{2,9}\alpha_{2,10})]
$$
with orbits
$
\{\alpha_{1,2},\alpha_{2,2}\},
\{\alpha_{1,3},\alpha_{2,4},\alpha_{1,4},\alpha_{2,3}\},
\{\alpha_{1,5},\alpha_{2,6},\alpha_{2,11},\alpha_{1,12},
\alpha_{2,9},\alpha_{1,7},\alpha_{2,10},\alpha_{2,8}\},
\newline
\{\alpha_{2,5},\alpha_{1,6},\alpha_{1,11},\alpha_{2,12},
\alpha_{1,9},\alpha_{2,7},\alpha_{1,10},\alpha_{1,8}\}$.

\medskip

{\bf n=18,} $H\cong D_{12}$ ($|H|=12$, $i=4$):
$\rk N_H=16$ and $(N_H)^\ast/N_H\cong (\bz/6\bz)^4$.
$$
H_{18,1}=[
(\alpha_{1,2}\alpha_{1,3})(\alpha_{2,2}\alpha_{2,3})
(\alpha_{1,5}\alpha_{1,12})(\alpha_{2,5}\alpha_{2,12})
(\alpha_{1,6}\alpha_{1,9})(\alpha_{2,6}\alpha_{2,9})
(\alpha_{1,10}\alpha_{1,11})(\alpha_{2,10}\alpha_{2,11}),
$$
$$
(\alpha_{1,3}\alpha_{2,4})(\alpha_{2,3}\alpha_{1,4})
(\alpha_{1,6}\alpha_{2,12})(\alpha_{2,6}\alpha_{1,12})
(\alpha_{1,7}\alpha_{2,8})(\alpha_{2,7}\alpha_{1,8})
(\alpha_{1,9}\alpha_{1,10})(\alpha_{2,9}\alpha_{2,10})]
$$
with orbits
$
\{\alpha_{1,2},\alpha_{1,3},\alpha_{2,4}\},
\{\alpha_{2,2},\alpha_{2,3},\alpha_{1,4}\},
\{\alpha_{1,5},\alpha_{1,12},\alpha_{2,6},\alpha_{2,9},\alpha_{2,10},\alpha_{2,11}\},
\newline
\{\alpha_{2,5},\alpha_{2,12},\alpha_{1,6},\alpha_{1,9},\alpha_{1,10},\alpha_{1,11}\},
\{\alpha_{1,7},\alpha_{2,8}\},
\{\alpha_{2,7},\alpha_{1,8}\}
$.

\medskip

{\bf n=17,} $H={\frak A}_4$ ($|H|=12$, $i=3$):
$\rk N_H=16$ and $(N_H)^\ast/N_H\cong (\bz/12\bz)^2\times (\bz/2\bz)^2$.
$$
H_{17,1}=[
(\alpha_{1,4}\alpha_{2,7}\alpha_{1,8})(\alpha_{2,4}\alpha_{1,7}\alpha_{2,8})
(\alpha_{1,5}\alpha_{1,6}\alpha_{2,10})(\alpha_{2,5}\alpha_{2,6}\alpha_{1,10})
(\alpha_{1,9}\alpha_{2,11}\alpha_{1,12})(\alpha_{2,9}\alpha_{1,11}\alpha_{2,12}),
$$
$$
(\alpha_{1,3}\alpha_{2,4})(\alpha_{2,3}\alpha_{1,4})
(\alpha_{1,6}\alpha_{2,12})(\alpha_{2,6}\alpha_{1,12})
(\alpha_{1,7}\alpha_{2,8})(\alpha_{2,7}\alpha_{1,8})
(\alpha_{1,9}\alpha_{1,10})(\alpha_{2,9}\alpha_{2,10})]
$$
with orbits
$
\{\alpha_{1,3},\alpha_{2,4},\alpha_{1,7},\alpha_{2,8}\},
\{\alpha_{2,3},\alpha_{1,4},\alpha_{2,7},\alpha_{1,8}\},
\{\alpha_{1,5},\alpha_{1,6},\alpha_{2,10},\alpha_{2,12},\alpha_{2,9},\alpha_{1,11}\},
\newline
\{\alpha_{2,5},\alpha_{2,6},\alpha_{1,10},\alpha_{1,12},\alpha_{1,9},\alpha_{2,11}\}
$.

\medskip

{\bf n=16,} $H\cong D_{10}$ ($|H|=10$, $i=1$):
$\rk N_H=16$ and $(N_H)^\ast/N_H\cong (\bz/5\bz)^4$.
$$
H_{16,1}=
[(\alpha_{1,4}\alpha_{1,6})(\alpha_{2,4}\alpha_{2,6})
(\alpha_{1,5}\alpha_{2,7})(\alpha_{2,5}\alpha_{1,7})
(\alpha_{1,8}\alpha_{2,10})(\alpha_{2,8}\alpha_{1,10})
(\alpha_{1,11}\alpha_{2,12})(\alpha_{2,11}\alpha_{1,12}),
$$
$$
(\alpha_{1,3}\alpha_{2,4})(\alpha_{2,3}\alpha_{1,4})
(\alpha_{1,6}\alpha_{2,12})(\alpha_{2,6}\alpha_{1,12})
(\alpha_{1,7}\alpha_{2,8})(\alpha_{2,7}\alpha_{1,8})
(\alpha_{1,9}\alpha_{1,10})(\alpha_{2,9}\alpha_{2,10})]
$$
with orbits
$
\{\alpha_{1,3},\alpha_{2,4},\alpha_{2,6},\alpha_{1,12},\alpha_{2,11} \},
\{ \alpha_{2,3},\alpha_{1,4},\alpha_{1,6},\alpha_{2,12},\alpha_{1,11} \},
\newline 
\{\alpha_{1,5},\alpha_{2,7},\alpha_{1,8},\alpha_{2,10},\alpha_{2,9}\},
\{\alpha_{2,5},\alpha_{1,7},\alpha_{2,8},\alpha_{1,10},\alpha_{1,9}\}
$.

\medskip

{\bf n=15,} $H\cong C_3^2$ ($|H|=9$, $i=2$):
$$
H_{15,1}=[
(\alpha_{1,1}\alpha_{1,2}\alpha_{1,7})(\alpha_{2,1}\alpha_{2,2}\alpha_{2,7})
(\alpha_{1,3}\alpha_{1,5}\alpha_{1,9})(\alpha_{2,3}\alpha_{2,5}\alpha_{2,9})
(\alpha_{1,4}\alpha_{1,11}\alpha_{1,12})(\alpha_{2,4}\alpha_{2,11}\alpha_{2,12}),
$$
$$
(\alpha_{1,1}\alpha_{1,3}\alpha_{2,4})(\alpha_{2,1}\alpha_{2,3}\alpha_{1,4})
(\alpha_{1,2}\alpha_{1,5}\alpha_{2,11})(\alpha_{2,2}\alpha_{2,5}\alpha_{1,11})
(\alpha_{1,7}\alpha_{1,9}\alpha_{2,12})(\alpha_{2,7}\alpha_{2,9}\alpha_{1,12})]
$$
with $Clos(H_{15,1})=H_{30,1}$ above;
$$
H_{15,2}=[
(\alpha_{1,2}\alpha_{1,5}\alpha_{2,11})(\alpha_{2,2}\alpha_{2,5}\alpha_{1,11})
(\alpha_{1,6}\alpha_{1,10}\alpha_{2,8})(\alpha_{2,6}\alpha_{2,10}\alpha_{1,8})
(\alpha_{1,7}\alpha_{2,12}\alpha_{1,9})(\alpha_{2,7}\alpha_{1,12}\alpha_{2,9}),
$$
$$
(\alpha_{1,1}\alpha_{1,3}\alpha_{2,4})(\alpha_{2,1}\alpha_{2,3}\alpha_{1,4})
(\alpha_{1,2}\alpha_{1,5}\alpha_{2,11})(\alpha_{2,2}\alpha_{2,5}\alpha_{1,11})
(\alpha_{1,7}\alpha_{1,9}\alpha_{2,12})(\alpha_{2,7}\alpha_{2,9}\alpha_{1,12})]
$$
with $Clos(H_{15,2})=H_{30,2}$ above.

\medskip

{\bf n=14,} $H\cong C_8$ ($|H|=8$, $i=1$):
$$
H_{14,1}=
$$
$$
[(\alpha_{1,2}\alpha_{2,2})
(\alpha_{1,3}\alpha_{1,4}\alpha_{2,3}\alpha_{2,4})
(\alpha_{1,5}\alpha_{2,6}\alpha_{1,7}\alpha_{1,12}
\alpha_{2,11}\alpha_{2,9}\alpha_{2,8}\alpha_{2,10})
$$
$$
(\alpha_{2,5}\alpha_{1,6}\alpha_{2,7}\alpha_{2,12}
\alpha_{1,11}\alpha_{1,9}\alpha_{1,8}\alpha_{1,10})]
$$
with $Clos(H_{14,1})=H_{26,1}$ above.

\medskip

{\bf n=12,} $H\cong Q_8$ ($|H|=8$, $i=4$):
$\rk N_H=17$, $(N_H)^\ast/N_H \cong
(\bz/8\bz)^2\times (\bz/2\bz)^3$, and
$K((q_{N_H})_2)\cong q_\theta^{(2)}(2)\oplus q^\prime$.
$$
H_{12,1}=
$$
$$
[(\alpha_{1,1}\alpha_{1,4}\alpha_{1,10}\alpha_{1,7})
(\alpha_{2,1}\alpha_{2,4}\alpha_{2,10}\alpha_{2,7})
(\alpha_{1,2}\alpha_{2,6}\alpha_{1,9}\alpha_{2,12})
(\alpha_{2,2}\alpha_{1,6}\alpha_{2,9}\alpha_{1,12})
(\alpha_{1,8}\alpha_{2,8})(\alpha_{1,11}\alpha_{2,11}),
$$
$$
(\alpha_{1,1}\alpha_{2,2}\alpha_{1,10}\alpha_{2,9})
(\alpha_{2,1}\alpha_{1,2}\alpha_{2,10}\alpha_{1,9})
(\alpha_{1,4}\alpha_{1,12}\alpha_{1,7}\alpha_{1,6})
(\alpha_{2,4}\alpha_{2,12}\alpha_{2,7}\alpha_{2,6})
(\alpha_{1,5}\alpha_{2,5})(\alpha_{1,11}\alpha_{2,11})]
$$
with orbits
$
\{\alpha_{1,1},\alpha_{1,4},\alpha_{2,2},\alpha_{1,10},
\alpha_{1,12},\alpha_{1,6},\alpha_{1,7},\alpha_{2,9}\},
\newline 
\{\alpha_{2,1},\alpha_{2,4},\alpha_{1,2},\alpha_{2,10},
\alpha_{2,12},\alpha_{2,6},\alpha_{2,7},\alpha_{1,9}\},
\{\alpha_{1,5},\alpha_{2,5}\},\{\alpha_{1,8},\alpha_{2,8}\},
\{\alpha_{1,11},\alpha_{2,11}\}$.

\medskip

{\bf n=10,} $H\cong D_8$ ($|H|=8$, $i=3$):
$\rk N_H=15$ and $(N_H)^\ast/N_H \cong (\bz/4\bz)^5$.
$$
H_{10,1}=[
(\alpha_{1,4}\alpha_{2,7})(\alpha_{2,4}\alpha_{1,7})
(\alpha_{1,5}\alpha_{1,12})(\alpha_{2,5}\alpha_{2,12})
(\alpha_{1,6}\alpha_{2,11})(\alpha_{2,6}\alpha_{1,11})
(\alpha_{1,9}\alpha_{2,10})(\alpha_{2,9}\alpha_{1,10}),
$$
$$
(\alpha_{1,3}\alpha_{2,4})(\alpha_{2,3}\alpha_{1,4})
(\alpha_{1,6}\alpha_{2,12})(\alpha_{2,6}\alpha_{1,12})
(\alpha_{1,7}\alpha_{2,8})(\alpha_{2,7}\alpha_{1,8})
(\alpha_{1,9}\alpha_{1,10})(\alpha_{2,9}\alpha_{2,10})]
$$
with orbits
$
\{\alpha_{1,3},\alpha_{2,4},\alpha_{1,7},\alpha_{2,8}\},
\{\alpha_{2,3},\alpha_{1,4},\alpha_{2,7},\alpha_{1,8}\},
\{\alpha_{1,5},\alpha_{1,12},\alpha_{2,6},\alpha_{1,11}\},
\{\alpha_{2,5},\alpha_{2,12},
\newline
\alpha_{1,6},\alpha_{2,11}\},
\{\alpha_{1,9},\alpha_{2,10},\alpha_{1,10},\alpha_{2,9}\}
$.

\medskip

{\bf n=7,} $H\cong C_6$, ($|H|=6$, $i=2$):
$$
H_{7,1}=
[(\alpha_{1,2}\alpha_{1,3}\alpha_{2,4})
(\alpha_{2,2}\alpha_{2,3}\alpha_{1,4})
(\alpha_{1,5}\alpha_{1,12}\alpha_{2,9}\alpha_{2,11}\alpha_{2,10}\alpha_{2,6})
(\alpha_{2,5}\alpha_{2,12}\alpha_{1,9}\alpha_{1,11}\alpha_{1,10}\alpha_{1,6})
$$
$$
(\alpha_{1,7}\alpha_{2,8})(\alpha_{2,7}\alpha_{1,8})]
$$
with $Clos(H_{7,1})=H_{18,1}$ above.

\medskip

{\bf n=6,} $H\cong D_6$ ($|H|=6$, $i=1$):
$\rk N_H=14$, $(N_H)^\ast/N_H \cong (\bz/6\bz)^2\times (\bz/3\bz)^3$.
$$
H_{6,1}=[
(\alpha_{1,3}\alpha_{2,4})(\alpha_{2,3}\alpha_{1,4})
(\alpha_{1,5}\alpha_{2,11})(\alpha_{2,5}\alpha_{1,11})
(\alpha_{1,6}\alpha_{1,10})(\alpha_{2,6}\alpha_{2,10})
(\alpha_{1,9}\alpha_{2,12})(\alpha_{2,9}\alpha_{1,12}),
$$
$$
(\alpha_{1,1}\alpha_{1,3}\alpha_{2,4})(\alpha_{2,1}\alpha_{2,3}\alpha_{1,4})
(\alpha_{1,2}\alpha_{1,5}\alpha_{2,11})(\alpha_{2,2}\alpha_{2,5}\alpha_{1,11})
(\alpha_{1,7}\alpha_{1,9}\alpha_{2,12})(\alpha_{2,7}\alpha_{2,9}\alpha_{1,12})]
$$
with orbits
$
\{\alpha_{1,1},\alpha_{1,3},\alpha_{2,4}\},\{\alpha_{2,1},\alpha_{2,3},\alpha_{1,4}\},
\{\alpha_{1,2},\alpha_{1,5},\alpha_{2,11}\},\{\alpha_{2,2},\alpha_{2,5},\alpha_{1,11}\},
\{\alpha_{1,6},\alpha_{1,10}\},
\newline
\{\alpha_{2,6},\alpha_{2,10}\},
\{\alpha_{1,7},\alpha_{1,9},\alpha_{2,12}\},\{\alpha_{2,7},\alpha_{2,9},\alpha_{1,12}\}
$;
$$
H_{6,2}=[
(\alpha_{1,2}\alpha_{1,7})(\alpha_{2,2}\alpha_{2,7})
(\alpha_{1,3}\alpha_{2,4})(\alpha_{2,3}\alpha_{1,4})
(\alpha_{1,5}\alpha_{2,12})(\alpha_{2,5}\alpha_{1,12})
(\alpha_{1,9}\alpha_{2,11})(\alpha_{2,9}\alpha_{1,11}),
$$
$$
(\alpha_{1,1}\alpha_{1,3}\alpha_{2,4})(\alpha_{2,1}\alpha_{2,3}\alpha_{1,4})
(\alpha_{1,2}\alpha_{1,5}\alpha_{2,11})(\alpha_{2,2}\alpha_{2,5}\alpha_{1,11})
(\alpha_{1,7}\alpha_{1,9}\alpha_{2,12})(\alpha_{2,7}\alpha_{2,9}\alpha_{1,12})]
$$
with orbits
$
\{\alpha_{1,1},\alpha_{1,3},\alpha_{2,4}\},
\{\alpha_{2,1},\alpha_{2,3},\alpha_{1,4}\},
\{\alpha_{1,2},\alpha_{1,7},\alpha_{1,5},\alpha_{1,9},\alpha_{2,12},\alpha_{2,11}\},
\newline
\{\alpha_{2,2},\alpha_{2,7},\alpha_{2,5},\alpha_{2,9},\alpha_{1,12},\alpha_{1,11}\}
$.

\medskip

{\bf n=5,} $H\cong C_5$ ($|H|=5$, $i=1$):
$$
H_{5,1}=
$$
$$
[(\alpha_{1,3}\alpha_{2,4}\alpha_{1,12}\alpha_{2,11}\alpha_{2,6})
(\alpha_{2,3}\alpha_{1,4}\alpha_{2,12}\alpha_{1,11}\alpha_{1,6})
(\alpha_{1,5}\alpha_{1,8}\alpha_{2,9}\alpha_{2,10}\alpha_{2,7})
(\alpha_{2,5}\alpha_{2,8}\alpha_{1,9}\alpha_{1,10}\alpha_{1,7})]
$$
with $Clos(H_{5,1})=H_{16,1}$ above.

\medskip

{\bf n=4,} $H\cong C_4$ ($|H|=4$, $i=1$):
$\rk N_H=14$ and $(N_H)^\ast/N_H\cong (\bz/4\bz)^4\times (\bz/2\bz)^2$.
$$
H_{4,1}=
$$
$$
[(\alpha_{1,1}\alpha_{2,1})(\alpha_{1,3}\alpha_{2,10}\alpha_{2,4}\alpha_{2,9})
(\alpha_{2,3}\alpha_{1,10}\alpha_{1,4}\alpha_{1,9})
(\alpha_{1,6}\alpha_{1,8}\alpha_{2,12}\alpha_{2,7})
(\alpha_{2,6}\alpha_{2,8}\alpha_{1,12}\alpha_{1,7})
(\alpha_{1,11}\alpha_{2,11})]
$$
with orbits
$
\{\alpha_{1,1},\alpha_{2,1}\}, \{ \alpha_{1,3}, \alpha_{2,10},\alpha_{2,4},\alpha_{2,9}\},
\{\alpha_{2,3},\alpha_{1,10},\alpha_{1,4},\alpha_{1,9}\},
\{\alpha_{1,6},\alpha_{1,8},\alpha_{2,12},\alpha_{2,7}\},
\newline
 \{\alpha_{2,6},\alpha_{2,8},\alpha_{1,12},\alpha_{1,7}\},\{\alpha_{1,11},\alpha_{2,11}\}
$.

\medskip

{\bf n=3,} $H\cong C_2^2$ ($|H|=4$, $i=2$):
$\rk N_H=12$ and $(N_H)^\ast/N_H\cong (\bz/4\bz)^2\times (\bz/2\bz)^6$.
$$
H_{3,1}=
[(\alpha_{1,5}\alpha_{2,11})(\alpha_{2,5}\alpha_{1,11})(\alpha_{1,6}\alpha_{1,9})
(\alpha_{2,6}\alpha_{2,9})(\alpha_{1,7}\alpha_{2,8})(\alpha_{2,7}\alpha_{1,8})
(\alpha_{1,10}\alpha_{2,12})(\alpha_{2,10}\alpha_{1,12}),
$$
$$
(\alpha_{1,3}\alpha_{2,4})(\alpha_{2,3}\alpha_{1,4})
(\alpha_{1,6}\alpha_{2,12})(\alpha_{2,6}\alpha_{1,12})(\alpha_{1,7}\alpha_{2,8})
(\alpha_{2,7}\alpha_{1,8})
(\alpha_{1,9}\alpha_{1,10})(\alpha_{2,9}\alpha_{2,10})]
$$
with orbits
$
\{\alpha_{1,3}, \alpha_{2,4}\}, \{ \alpha_{2,3}, \alpha_{1,4} \},
\{ \alpha_{1,5}, \alpha_{2,11}\},
\{\alpha_{2,5}, \alpha_{1,11}\},  
\{ \alpha_{1,6}, \alpha_{1,9}, \alpha_{2,12}, \alpha_{1,10} \},
\newline 
\{\alpha_{2,6}, \alpha_{2,9}, \alpha_{1,12}, \alpha_{2,10} \},
\{ \alpha_{1,7}, \alpha_{2,8} \}, \{ \alpha_{2,7}, \alpha_{1,8}\}$.

\medskip

{\bf n=2,} $H\cong C_3$ ($|H|=3$, $i=1$):
$\rk N_H=12$ and $(N_H)^\ast/N_H\cong (\bz/3\bz)^6$.
$$
H_{2,1}=
[(\alpha_{1,1}\alpha_{1,3}\alpha_{2,4})
(\alpha_{2,1}\alpha_{2,3}\alpha_{1,4})
(\alpha_{1,2}\alpha_{1,5}\alpha_{2,11})
(\alpha_{2,2}\alpha_{2,5}\alpha_{1,11})
(\alpha_{1,7}\alpha_{1,9}\alpha_{2,12})
(\alpha_{2,7}\alpha_{2,9}\alpha_{1,12})]
$$
with orbits
$
\{ \alpha_{1,1}, \alpha_{1,3}, \alpha_{2,4} \},
\{ \alpha_{2,1}, \alpha_{2,3}, \alpha_{1,4} \},
\{ \alpha_{1,2}, \alpha_{1,5}, \alpha_{2,11} \},
\{ \alpha_{2,2}, \alpha_{2,5}, \alpha_{1,11} \},
\newline 
\{\alpha_{1,7}, \alpha_{1,9}, \alpha_{2,12} \},
\{\alpha_{2,7}, \alpha_{2,9}, \alpha_{1,12}\}
$.

\medskip

{\bf n=1,} $H\cong C_2$ ($|H|=2$, $i=1$):
$\rk N_H=8$ and $(N_H)^\ast/N_H\cong (\bz/2\bz)^8$.
$$
H_{1,1}=
[(\alpha_{1,3}\alpha_{2,4})(\alpha_{2,3}\alpha_{1,4})
(\alpha_{1,6}\alpha_{2,12})(\alpha_{2,6}\alpha_{1,12})
(\alpha_{1,7}\alpha_{2,8})(\alpha_{2,7}\alpha_{1,8})
(\alpha_{1,9}\alpha_{1,10})(\alpha_{2,9}\alpha_{2,10})]
$$
with orbits
$
\{\alpha_{1,3},\alpha_{2,4} \},\{ \alpha_{2,3},\alpha_{1,4}\},
\{\alpha_{1,6},\alpha_{2,12}\},\{\alpha_{2,6},\alpha_{1,12}\},
\{\alpha_{1,7},\alpha_{2,8}\},\{\alpha_{2,7},\alpha_{1,8}\},
\newline
\{\alpha_{1,9},\alpha_{1,10}\},\{\alpha_{2,9},\alpha_{2,10}\}
$

\medskip

Let $X$ be marked by a primitive sublattice
$S\subset N=N_{22}=N(12A_2)$. Then $S$
must satisfy Theorem \ref{th:primembb3} and
$\Gamma(P(S))\subset \Gamma(P(N_{22}))=12\aaa_2$.
Any such $S$ gives marking of some $X$ and $P(X)\cap S=P(S)$.

If $N_H\subset S$ where $H$ has the type $n=79$, $70$, $63$, $62$ or $54$ ($\rk N_H=19$),
then $\Aut(X,S)_0=H$.
Otherwise, if only $N_H\subset S$ where
$H$ has the type $n=55$, $48$, $46$, $32$ or $26$ ($\rk N_H=18$), then
$\Aut(X,S)_0=H$.
Otherwise, if only $N_H\subset S$ where $H$ has the type $n=34$ or $12$
($\rk N_H=17$), then $\Aut(X,S)_0=H$.
Otherwise, if only $N_H\subset S$ where $H$ has the type $n=30$, $18$, $17$ or $16$
($\rk N_H=16$), then $\Aut(X,S)_0=H$.
Otherwise, if only $N_H\subset S$ where $H$ has the type $n=10$
($\rk N_H=15$), then $\Aut(X,S)_0=H$.
Otherwise, if only $N_H\subset S$ where $H$ has the type $n=6$ or $4$
($\rk N_H=14$), then $\Aut(X,S)_0=H$.
Otherwise, if only $N_H\subset S$ where $H$ has the type $n=3$ or $2$
($\rk N_H=12$), then $\Aut(X,S)_0=H$.
Otherwise, if only $N_H\subset S$ where $H$ has the type $n=1$
($\rk N_H=8$), then $\Aut(X,S)_0=H\cong C_2$.
Otherwise, $\Aut(X,S)_0$ is trivial.

Let us assume that
a K3 surface $X$ has $18$ non-singular rational
curves which define the graph $9\aaa_2$. In \cite{Barth}, Barth
described the primitive sublattice $\overline{I}\subset S_X$
generated by classes
$$
\alpha_{11}, \alpha_{12},\,\dots ,\ \alpha_{19},\alpha_{29}
$$
of these curves. In particular, $\rk \overline{I}=18$
and  $\overline{I}^\ast/\overline{I}\cong (\bz/3\bz)^3$.

Let us consider marking $S\subset N_i$ of $X$ by $S$
which contains the primitive sublattice $\overline{I}\subset S$.
It was showed in \cite{Barth} that
$$
P(X)\cap S=\{\alpha_{11}, \alpha_{12},\,\dots ,\
\alpha_{19},\alpha_{29}\}\ .
$$
Thus, $9\aaa_2$ should be a subgraph of the Dynkin graph of
$N_i$, and the corresponding sublattice $\overline{I}\subset N_i$ is
a primitive sublattice. By classification of Niemeier
lattices, this is possible only for $N(12A_2)$. Thus, $X$
can be marked by the Niemeier lattice $N_{22}=N(12A_2)$ only.

\newpage

{\bf Case 23.} For the Niemeier lattice $N_{23}$, we have
$$
N=N_{23}=N(24A_1)=[24A_1,\ [1(00000101001100110101111)]]=
[24A_1,\
$$
$$
\epsilon_1+
\epsilon_7+\epsilon_9+\epsilon_{12}+\epsilon_{13}+\epsilon_{16}+\epsilon_{17}+
\epsilon_{19}+\epsilon_{21}+\epsilon_{22}+\epsilon_{23}+\epsilon_{24},\ $$
$$
\epsilon_1+
\epsilon_2+\epsilon_8+\epsilon_{10}+\epsilon_{13}+\epsilon_{14}+\epsilon_{17}+
\epsilon_{18}+\epsilon_{20}+\epsilon_{22}+\epsilon_{23}+\epsilon_{24},\
$$
$$
\epsilon_1+
\epsilon_2+\epsilon_3+\epsilon_9+\epsilon_{11}+\epsilon_{14}+\epsilon_{15}+
\epsilon_{18}+\epsilon_{19}+\epsilon_{21}+\epsilon_{23}+\epsilon_{24},\ $$
$$
\epsilon_1+
\epsilon_2+\epsilon_3+\epsilon_4+\epsilon_{10}+\epsilon_{12}+\epsilon_{15}+
\epsilon_{16}+\epsilon_{19}+\epsilon_{20}+\epsilon_{22}+\epsilon_{24},\ $$
$$
\epsilon_1+
\epsilon_2+\epsilon_3+\epsilon_4+\epsilon_5+\epsilon_{11}+\epsilon_{13}+
\epsilon_{16}+\epsilon_{17}+\epsilon_{20}+\epsilon_{21}+\epsilon_{23},\
$$
$$
\epsilon_1+
\epsilon_3+\epsilon_4+\epsilon_5+\epsilon_6+\epsilon_{12}+\epsilon_{14}+
\epsilon_{17}+\epsilon_{18}+\epsilon_{21}+\epsilon_{22}+\epsilon_{24},\
$$
$$
\epsilon_1+
\epsilon_2+\epsilon_4+\epsilon_5+\epsilon_6+\epsilon_7+\epsilon_{13}+
\epsilon_{15}+\epsilon_{18}+\epsilon_{19}+\epsilon_{22}+\epsilon_{23},\
$$
$$
\epsilon_1+
\epsilon_3+\epsilon_5+\epsilon_6+\epsilon_7+\epsilon_8+\epsilon_{14}+
\epsilon_{16}+\epsilon_{19}+\epsilon_{20}+\epsilon_{23}+\epsilon_{24},\
$$
$$
\epsilon_1+
\epsilon_2+\epsilon_4+\epsilon_6+\epsilon_7+\epsilon_8+\epsilon_9+
\epsilon_{15}+\epsilon_{17}+\epsilon_{20}+\epsilon_{21}+\epsilon_{24},\
$$
$$
\epsilon_1+
\epsilon_2+\epsilon_3+\epsilon_5+\epsilon_7+\epsilon_8+\epsilon_9+
\epsilon_{10}+\epsilon_{16}+\epsilon_{18}+\epsilon_{21}+\epsilon_{22},\
$$
$$
\epsilon_1+
\epsilon_3+\epsilon_4+\epsilon_6+\epsilon_8+\epsilon_9+\epsilon_{10}+
\epsilon_{11}+\epsilon_{17}+\epsilon_{19}+\epsilon_{22}+\epsilon_{23},\
$$
$$
\epsilon_1+
\epsilon_4+\epsilon_5+\epsilon_7+\epsilon_9+\epsilon_{10}+\epsilon_{11}+
\epsilon_{12}+\epsilon_{18}+\epsilon_{20}+\epsilon_{23}+\epsilon_{24},\
$$
$$
\epsilon_1+
\epsilon_2+\epsilon_5+\epsilon_6+\epsilon_8+\epsilon_{10}+\epsilon_{11}+
\epsilon_{12}+\epsilon_{13}+\epsilon_{19}+\epsilon_{21}+\epsilon_{24},\
$$
$$
\epsilon_1+
\epsilon_2+\epsilon_3+\epsilon_6+\epsilon_7+\epsilon_{9}+\epsilon_{11}+
\epsilon_{12}+\epsilon_{13}+\epsilon_{14}+\epsilon_{20}+\epsilon_{22},\
$$
$$
\epsilon_1+
\epsilon_3+\epsilon_4+\epsilon_7+\epsilon_8+\epsilon_{10}+\epsilon_{12}+
\epsilon_{13}+\epsilon_{14}+\epsilon_{15}+\epsilon_{21}+\epsilon_{23},\
$$
$$
\epsilon_1+
\epsilon_4+\epsilon_5+\epsilon_8+\epsilon_9+\epsilon_{11}+\epsilon_{13}+
\epsilon_{14}+\epsilon_{15}+\epsilon_{16}+\epsilon_{22}+\epsilon_{24},\
$$
$$
\epsilon_1+
\epsilon_2+\epsilon_5+\epsilon_6+\epsilon_9+\epsilon_{10}+\epsilon_{12}+
\epsilon_{14}+\epsilon_{15}+\epsilon_{16}+\epsilon_{17}+\epsilon_{23},\
$$
$$
\epsilon_1+
\epsilon_3+\epsilon_6+\epsilon_7+\epsilon_{10}+\epsilon_{11}+\epsilon_{13}+
\epsilon_{15}+\epsilon_{16}+\epsilon_{17}+\epsilon_{18}+\epsilon_{24},\
$$
$$
\epsilon_1+
\epsilon_2+\epsilon_4+\epsilon_7+\epsilon_8+\epsilon_{11}+\epsilon_{12}+
\epsilon_{14}+\epsilon_{16}+\epsilon_{17}+\epsilon_{18}+\epsilon_{19},\
$$
$$
\epsilon_1+
\epsilon_3+\epsilon_5+\epsilon_8+\epsilon_9+\epsilon_{12}+\epsilon_{13}+
\epsilon_{15}+\epsilon_{17}+\epsilon_{18}+\epsilon_{19}+\epsilon_{20},\
$$
$$
\epsilon_1+
\epsilon_4+\epsilon_6+\epsilon_9+\epsilon_{10}+\epsilon_{13}+\epsilon_{14}+
\epsilon_{16}+\epsilon_{18}+\epsilon_{19}+\epsilon_{20}+\epsilon_{21},\
$$
$$
\epsilon_1+
\epsilon_5+\epsilon_7+\epsilon_{10}+\epsilon_{11}+\epsilon_{14}+\epsilon_{15}+
\epsilon_{17}+\epsilon_{19}+\epsilon_{20}+\epsilon_{21}+\epsilon_{22},\
$$
$$
\epsilon_1+
\epsilon_6+\epsilon_8+\epsilon_{11}+\epsilon_{12}+\epsilon_{15}+\epsilon_{16}+
\epsilon_{18}+\epsilon_{20}+\epsilon_{21}+\epsilon_{22}+\epsilon_{23},\
$$
$$
\epsilon_1+
\epsilon_7+\epsilon_9+\epsilon_{12}+\epsilon_{13}+\epsilon_{16}+\epsilon_{17}+
\epsilon_{19}+\epsilon_{21}+\epsilon_{22}+\epsilon_{23}+\epsilon_{24}]\
$$
where we denote $\alpha_i=\alpha_{1i}$ and $\epsilon_i=\epsilon_{1i}$ for
$i=1,2,\dots,24$. The group $A(N_{23})$ is the Mathieu group $M_{24}$ generated by
$$
\varphi_1=(\alpha_1)(\alpha_2\alpha_3\dots\alpha_{23}\alpha_{24}),
$$
$$
\varphi_2=(\alpha_{4}\alpha_{18}\alpha_{11} \alpha_8 \alpha_{10})
(\alpha_{5}\alpha_{14}\alpha_{15} \alpha_{20} \alpha_{6})
(\alpha_{9}\alpha_{19}\alpha_{12} \alpha_{13} \alpha_{24})
(\alpha_{16}\alpha_{21}\alpha_{23} \alpha_{22} \alpha_{17}),
$$
$$
\varphi_3=
(\alpha_{1}\alpha_{2})(\alpha_{3}\alpha_{24})
(\alpha_{4}\alpha_{13})(\alpha_{5}\alpha_{17})
(\alpha_{6}\alpha_{19})(\alpha_{7}\alpha_{11})
(\alpha_{8}\alpha_{21})(\alpha_{9}\alpha_{15})
(\alpha_{10}\alpha_{22})(\alpha_{12}\alpha_{18})
$$
$$
(\alpha_{14}\alpha_{23})(\alpha_{16}\alpha_{20}). \hskip10cm
$$
(see \cite[Ch. 16]{CS}).

Using GAP Progam \cite{GAP}, we obtain the following classification.

\newpage

\centerline {\bf Classification of KahK3 conjugacy classes for $A(N_{23})$.}

\vskip1cm

{\bf n=81,} $H\cong M_{20}$ ($|H|=960$, $i=11357$):
$\rk N_H=19$, $(N_H)^\ast/N_H\cong \bz/40\bz\times
(\bz/2\bz)^2$ and
$\det(K((q_{N_H})_2))\equiv
\pm 2^5\cdot 5 \mod (\bz_2^\ast)^2$.
$$
H_{81,1}=[(\alpha_1\alpha_{19}\alpha_{6})(\alpha_{3}\alpha_{24}\alpha_{16})
(\alpha_{4}\alpha_{20}\alpha_{9})(\alpha_{7}\alpha_{10}\alpha_{8})
(\alpha_{11}\alpha_{23}\alpha_{13})(\alpha_{14}\alpha_{21}\alpha_{15}),\ \
$$
$$
(\alpha_{1}\alpha_{11}\alpha_{23}\alpha_{8})(\alpha_{4}\alpha_{6}\alpha_{13}\alpha_{20})
(\alpha_{5}\alpha_{16})(\alpha_{7}\alpha_{17}\alpha_{9}\alpha_{14})
(\alpha_{10}\alpha_{21}\alpha_{19}\alpha_{15})(\alpha_{12}\alpha_{24})]
$$
with orbits (here and in what follows we show orbits with more than one elements only)
$
\{\alpha_1,\alpha_{19}, \alpha_{11}, \alpha_{6},\alpha_{15}, \alpha_{23},
\alpha_{13}, \alpha_{14}, \alpha_{10},
\alpha_8, \alpha_{20}, \alpha_{21}, \alpha_7,\alpha_{9},\alpha_4,\alpha_{17}\},\
\{\alpha_{3},\alpha_{24},\alpha_{16},\alpha_{12},\alpha_{5}\}$;
$$
H_{81,2}=[
(\alpha_{1}\alpha_{16}\alpha_{9})
(\alpha_{4}\alpha_{10}\alpha_{12})
(\alpha_{5}\alpha_{13}\alpha_{19})
(\alpha_{6}\alpha_{20}\alpha_{17})
(\alpha_{7}\alpha_{23}\alpha_{24})
(\alpha_{8}\alpha_{21}\alpha_{11}),\
$$
$$
(\alpha_{1}\alpha_{11}\alpha_{23}\alpha_{8})
(\alpha_{4}\alpha_{6}\alpha_{13}\alpha_{20})
(\alpha_{5}\alpha_{16})
(\alpha_{7}\alpha_{17}\alpha_{9}\alpha_{14})
(\alpha_{10}\alpha_{21}\alpha_{19}\alpha_{15})
(\alpha_{12}\alpha_{24})]
$$
with orbits
$
\{\alpha_{1},\alpha_{16},\alpha_{11},\alpha_{9},\alpha_{5},\alpha_{8},\alpha_{23},
\alpha_{14},\alpha_{13},\alpha_{21},\alpha_{24},\alpha_{7},\alpha_{19},
\alpha_{20},\alpha_{12},\alpha_{17},\alpha_{15},\alpha_{4},\alpha_{6},\alpha_{10} \}$.

\medskip

{\bf n=80,} $H\cong F_{384}$ ($|H|=384$, $i=18135$):
$\rk N_H=19$, $(N_H)^\ast/N_H\cong (\bz/8\bz)^2\times
\bz/4\bz$ and
$\det(K((q_{N_H})_2))\equiv \pm 2^8 \mod (\bz_2^\ast)^2$.
$$
H_{80,1}=
[(\alpha_{1}\alpha_{8}\alpha_{15})
(\alpha_{4}\alpha_{20}\alpha_{11})
(\alpha_{6}\alpha_{7}\alpha_{21})
(\alpha_{9}\alpha_{13}\alpha_{23})
(\alpha_{14}\alpha_{17}\alpha_{19})
(\alpha_{18}\alpha_{24}\alpha_{22}),
$$
$$
(\alpha_{1}\alpha_{11}\alpha_{23}\alpha_{8})
(\alpha_{4}\alpha_{6}\alpha_{13}\alpha_{20})
(\alpha_{5}\alpha_{16})
(\alpha_{7}\alpha_{17}\alpha_{9}\alpha_{14})
(\alpha_{10}\alpha_{21}\alpha_{19}\alpha_{15})
(\alpha_{12}\alpha_{24})]
$$
with orbits
$
\{\alpha_{1},\alpha_{8},\alpha_{11},\alpha_{15},\alpha_{4},\alpha_{23},
\alpha_{10},\alpha_{20},\alpha_{6},\alpha_{9},\alpha_{21},
\alpha_{7},\alpha_{13},\alpha_{14},\alpha_{19},\alpha_{17}\},\{\alpha_{5},\alpha_{16}\},
\newline
\{\alpha_{12},\alpha_{24},\alpha_{22},\alpha_{18}\}$.

\medskip

{\bf n=79,} $H\cong {\frak A}_{6}$ ($|H|=360$, $i=118$):
$\rk N_H=19$ and $(N_H)^\ast/N_H\cong \bz/60\bz\times \bz/3\bz$.
$$
H_{79,1}=
[(\alpha_{2}\alpha_{22}\alpha_{15}\alpha_{8}\alpha_{23})
(\alpha_{3}\alpha_{18}\alpha_{12}\alpha_{7}\alpha_{4})
(\alpha_{5}\alpha_{21}\alpha_{19}\alpha_{20}\alpha_{11})
(\alpha_{9}\alpha_{10}\alpha_{13}\alpha_{14}\alpha_{24}),\
$$
$$
(\alpha_{2}\alpha_{17}\alpha_{15}\alpha_{8}\alpha_{23})
(\alpha_{3}\alpha_{14}\alpha_{10}\alpha_{13}\alpha_{9})
(\alpha_{4}\alpha_{24}\alpha_{18}\alpha_{7}\alpha_{12})
(\alpha_{5}\alpha_{11}\alpha_{20}\alpha_{16}\alpha_{21})]
$$
with orbits
$
\{\alpha_{2},\alpha_{22},\alpha_{17},\alpha_{15},\alpha_{8},\alpha_{23}\},
\{\alpha_{3},\alpha_{18},\alpha_{14},\alpha_{12},\alpha_{7},\alpha_{24},
\alpha_{10},\alpha_{4},\alpha_{9},\alpha_{13}\},
\{\alpha_{5},\alpha_{21},\alpha_{11},
\newline
\alpha_{19},\alpha_{20},\alpha_{16}\}$;
$$
H_{79,2}=
[(\alpha_{2}\alpha_{8}\alpha_{18}\alpha_{20}\alpha_{10})
(\alpha_{3}\alpha_{12}\alpha_{11}\alpha_{22}\alpha_{14})
(\alpha_{4}\alpha_{15}\alpha_{24}\alpha_{19}\alpha_{9})
(\alpha_{5}\alpha_{23}\alpha_{17}\alpha_{7}\alpha_{13}),\
$$
$$
(\alpha_{2}\alpha_{13})(\alpha_{3}\alpha_{22})
(\alpha_{4}\alpha_{9})(\alpha_{7}\alpha_{18})
(\alpha_{8}\alpha_{14})(\alpha_{11}\alpha_{20})
(\alpha_{12}\alpha_{23})(\alpha_{16}\alpha_{19})]
$$
with orbits
$
\{\alpha_{2},\alpha_{8},\alpha_{13},\alpha_{18},\alpha_{14},\alpha_{5},
\alpha_{20},\alpha_{7},\alpha_{3},\alpha_{23},\alpha_{10},\alpha_{11},
\alpha_{12},\alpha_{22},\alpha_{17}\},
\{\alpha_{4},\alpha_{15},\alpha_{9},
\newline
\alpha_{24},\alpha_{19},\alpha_{16}\}.
$

\medskip

{\bf n=78,} $H\cong {\frak A}_{4,4}$ ($|H|=288$, $i=1026$):
$\rk N_H=19$, $(N_H)^\ast/N_H\cong \bz/24\bz\times
\bz/6\bz\times \bz/2\bz$ and
$
\det(K((q_{N_{H}})_2)
\equiv \pm 2^5\cdot 3^2 \mod (\bz_2^\ast)^2.
$
$$
H_{78,1}=
[(\alpha_{1}\alpha_{19}\alpha_{8})
(\alpha_{3}\alpha_{22}\alpha_{15})
(\alpha_{5}\alpha_{23}\alpha_{20})
(\alpha_{6}\alpha_{21}\alpha_{9})
(\alpha_{7}\alpha_{11}\alpha_{16})
(\alpha_{12}\alpha_{18}\alpha_{14}),\ \
$$
$$
(\alpha_{1}\alpha_{6}\alpha_{23}\alpha_{8})
(\alpha_{4}\alpha_{11}\alpha_{16}\alpha_{7})
(\alpha_{5}\alpha_{21}\alpha_{14}\alpha_{12})
(\alpha_{9}\alpha_{19}\alpha_{20}\alpha_{18})
(\alpha_{10}\alpha_{13})(\alpha_{15}\alpha_{22}),\ \
$$
$$
(\alpha_{2}\alpha_{13})(\alpha_{3}\alpha_{22})
(\alpha_{4}\alpha_{9})(\alpha_{7}\alpha_{18})
(\alpha_{8}\alpha_{14})(\alpha_{11}\alpha_{20})
(\alpha_{12}\alpha_{23})(\alpha_{16}\alpha_{19})]
$$
with orbits
$
\{\alpha_{1},\alpha_{19},\alpha_{6},\alpha_{8},\alpha_{20},\alpha_{16},
\alpha_{21},\alpha_{23},\alpha_{14},\alpha_{5},\alpha_{18},\alpha_{11},
\alpha_{7},\alpha_{9},\alpha_{12},\alpha_{4}\},
\{\alpha_{2},\alpha_{13},\alpha_{10}\}, \newline
\{\alpha_{3},\alpha_{22},\alpha_{15}\}$;
$$
H_{78,2}=[
(\alpha_{4}\alpha_{11}\alpha_{12})
(\alpha_{5}\alpha_{18}\alpha_{7})
(\alpha_{6}\alpha_{23}\alpha_{8})
(\alpha_{9}\alpha_{14}\alpha_{16})
(\alpha_{10}\alpha_{24}\alpha_{13})
(\alpha_{19}\alpha_{20}\alpha_{21}),\
$$
$$
(\alpha_{1}\alpha_{18}\alpha_{7}\alpha_{5})
(\alpha_{4}\alpha_{20}\alpha_{21}\alpha_{8})
(\alpha_{6}\alpha_{16}\alpha_{9}\alpha_{12})
(\alpha_{10}\alpha_{24})
(\alpha_{11}\alpha_{14}\alpha_{23}\alpha_{19})
(\alpha_{17}\alpha_{22}),\
$$
$$
(\alpha_{2}\alpha_{13})(\alpha_{3}\alpha_{22})
(\alpha_{4}\alpha_{9})(\alpha_{7}\alpha_{18})
(\alpha_{8}\alpha_{14})(\alpha_{11}\alpha_{20})
(\alpha_{12}\alpha_{23})(\alpha_{16}\alpha_{19})]
$$
with orbits
$
\{\alpha_{1},\alpha_{18},\alpha_{7},\alpha_{5}\},
\{\alpha_{2},\alpha_{13},\alpha_{10},\alpha_{24}\},
\{\alpha_{3},\alpha_{22},\alpha_{17}\},
\{\alpha_{4},\alpha_{11},\alpha_{20},\alpha_{9},\alpha_{12},
\alpha_{14},\alpha_{21},\alpha_{6},
\newline
\alpha_{23},\alpha_{16},\alpha_{8},\alpha_{19}\}$.

\medskip

{\bf n=77,} $H\cong {T}_{192}$ ($|H|=192$, $i=1493$):
$\rk N_H=19$, $(N_H)^\ast/N_H\cong \bz/12\bz\times
(\bz/4\bz)^2$ and
$\det(K((q_{N_H})_2))\equiv
\pm 2^6\cdot 3 \mod (\bz_2^\ast)^2$.
$$
H_{77,1}=[
(\alpha_{1}\alpha_{15}\alpha_{8})
(\alpha_{3}\alpha_{24}\alpha_{17})
(\alpha_{4}\alpha_{20}\alpha_{16})
(\alpha_{5}\alpha_{13}\alpha_{6})
(\alpha_{10}\alpha_{21}\alpha_{23})
(\alpha_{12}\alpha_{14}\alpha_{22}),\
$$
$$
(\alpha_{1}\alpha_{11}\alpha_{23}\alpha_{8})
(\alpha_{4}\alpha_{6}\alpha_{13}\alpha_{20})
(\alpha_{5}\alpha_{16})
(\alpha_{7}\alpha_{17}\alpha_{9}\alpha_{14})
(\alpha_{10}\alpha_{21}\alpha_{19}\alpha_{15})
(\alpha_{12}\alpha_{24})]
$$
with orbits
$
\{\alpha_{1},\alpha_{15},\alpha_{11},\alpha_{8},\alpha_{10},
\alpha_{23},\alpha_{21},\alpha_{19}\},
\{\alpha_{3},\alpha_{24},\alpha_{17},\alpha_{12},\alpha_{9},
\alpha_{14},\alpha_{22},\alpha_{7}\},
\{\alpha_{4},\alpha_{20},\alpha_{6},
\newline
\alpha_{16},\alpha_{5},\alpha_{13}\}.
$

\medskip

{\bf n=76,} $H\cong {H}_{192}$ ($|H|=192$, $i=955$):
$\rk N_H=19$, $(N_H)^\ast/N_H\cong \bz/24\bz\times
(\bz/4\bz)^2$ and
$\det(K((q_{N_H})_2)\equiv \pm 2^7\cdot 3 \mod (\bz_2^\ast)^2$.
$$
H_{76,1}=[
(\alpha_{1}\alpha_{17}\alpha_{19}\alpha_{5})
(\alpha_{4}\alpha_{24}\alpha_{21}\alpha_{9})
(\alpha_{7}\alpha_{23}\alpha_{20}\alpha_{22})
(\alpha_{8}\alpha_{16}\alpha_{14}\alpha_{15})
(\alpha_{10}\alpha_{13})(\alpha_{12}\alpha_{18}),\
$$
$$
(\alpha_{2}\alpha_{13})(\alpha_{3}\alpha_{22})
(\alpha_{4}\alpha_{9})(\alpha_{7}\alpha_{18})
(\alpha_{8}\alpha_{14})(\alpha_{11}\alpha_{20})
(\alpha_{12}\alpha_{23})(\alpha_{16}\alpha_{19})]
$$
with orbits
$
\{\alpha_{1},\alpha_{17},\alpha_{19},\alpha_{5},\alpha_{16},
\alpha_{14},\alpha_{15},\alpha_{8}\},
\{\alpha_{2},\alpha_{13},\alpha_{10}\},
\{\alpha_{3},\alpha_{22},\alpha_{7},\alpha_{23},\alpha_{18},
\alpha_{20},\alpha_{12},
\newline
\alpha_{11}\},
\{\alpha_{4},\alpha_{24},\alpha_{9},\alpha_{21}\}$;
$$
H_{76,2}=[
(\alpha_{1}\alpha_{18}\alpha_{9}\alpha_{21})
(\alpha_{4}\alpha_{12}\alpha_{5}\alpha_{19})
(\alpha_{6}\alpha_{23}\alpha_{7}\alpha_{20})
(\alpha_{8}\alpha_{14}\alpha_{11}\alpha_{16})
(\alpha_{10}\alpha_{13})(\alpha_{17}\alpha_{24}),\
$$
$$
(\alpha_{2}\alpha_{13})(\alpha_{3}\alpha_{22})
(\alpha_{4}\alpha_{9})(\alpha_{7}\alpha_{18})
(\alpha_{8}\alpha_{14})(\alpha_{11}\alpha_{20})
(\alpha_{12}\alpha_{23})(\alpha_{16}\alpha_{19})]
$$
with orbits
$
\{\alpha_{1},\alpha_{18},\alpha_{9},\alpha_{7},\alpha_{21},
\alpha_{4},\alpha_{20},\alpha_{12},
\alpha_{6},\alpha_{11},\alpha_{5},\alpha_{23},\alpha_{16},
\alpha_{19},\alpha_{8},\alpha_{14}\},
\{\alpha_{2},\alpha_{13},\alpha_{10}\},
\newline
\{\alpha_{3},\alpha_{22}\},
\{\alpha_{17},\alpha_{24}\}$;
$$
H_{76,3}=[
(\alpha_{1}\alpha_{19}\alpha_{6}\alpha_{5})
(\alpha_{3}\alpha_{8}\alpha_{15}\alpha_{7})
(\alpha_{4}\alpha_{14})
(\alpha_{9}\alpha_{22}\alpha_{23}\alpha_{10})
(\alpha_{11}\alpha_{21}\alpha_{20}\alpha_{16})
(\alpha_{12}\alpha_{18}),
$$
$$
(\alpha_{2}\alpha_{13})(\alpha_{3}\alpha_{22})
(\alpha_{4}\alpha_{9})(\alpha_{7}\alpha_{18})
(\alpha_{8}\alpha_{14})(\alpha_{11}\alpha_{20})
(\alpha_{12}\alpha_{23})(\alpha_{16}\alpha_{19})]
$$
with orbits
$
\{\alpha_{1},\alpha_{19},\alpha_{6},\alpha_{16},\alpha_{5},
\alpha_{11},\alpha_{21},\alpha_{20}\},
\{\alpha_{2},\alpha_{13}\},
\{\alpha_{3},\alpha_{8},\alpha_{22},\alpha_{15},\alpha_{14},
\alpha_{23},\alpha_{7},\alpha_{4},\alpha_{10},
\newline
\alpha_{12},\alpha_{18},\alpha_{9}\}.
$

\medskip

{\bf n=75,} $H\cong 4^2{\frak A}_4$ ($|H|=192$, $i=1023$):
$\rk N_H=18$, $(N_H)^\ast/N_H\cong (\bz/8\bz)^2\times
(\bz/2\bz)^2$ and
$\det(K((q_{N_H})_2) \equiv \pm 2^8 \mod (\bz_2^\ast)^2$.
$$
H_{75,1}=[
(\alpha_{3}\alpha_{16})(\alpha_{4}\alpha_{5})
(\alpha_{6}\alpha_{21})(\alpha_{10}\alpha_{20})
(\alpha_{11}\alpha_{12})(\alpha_{13}\alpha_{17})
(\alpha_{14}\alpha_{22})(\alpha_{23}\alpha_{24}),\
$$
$$
(\alpha_{1}\alpha_{9}\alpha_{7})
(\alpha_{3}\alpha_{21}\alpha_{22})
(\alpha_{4}\alpha_{14}\alpha_{24})
(\alpha_{5}\alpha_{10}\alpha_{16})
(\alpha_{6}\alpha_{20}\alpha_{23})
(\alpha_{12}\alpha_{17}\alpha_{13}),\
$$
$$
(\alpha_{1}\alpha_{6}\alpha_{22})
(\alpha_{3}\alpha_{24}\alpha_{16})
(\alpha_{4}\alpha_{20}\alpha_{9})
(\alpha_{5}\alpha_{10}\alpha_{7})
(\alpha_{11}\alpha_{13}\alpha_{12})
(\alpha_{14}\alpha_{18}\alpha_{21})]
$$
with orbits
$
\{\alpha_{1},\alpha_{9},\alpha_{6},\alpha_{7},\alpha_{4},
\alpha_{21},\alpha_{20},
\alpha_{22},\alpha_{5},\alpha_{14},\alpha_{10},\alpha_{23},
\alpha_{3},\alpha_{24},\alpha_{18},\alpha_{16}\},
\{\alpha_{11},\alpha_{12},\alpha_{13},
\newline
\alpha_{17}\}$.

\medskip

{\bf n=74,} $H\cong L_2(7)$ ($|H|=168$, $i=42$):
$\rk N_H=19$ and $(N_H)^\ast/N_H \cong \bz/28\bz \times \bz/7\bz$.
$$
H_{74,1}=[
(\alpha_{6}\alpha_{8})(\alpha_{9}\alpha_{24})
(\alpha_{10}\alpha_{22})(\alpha_{11}\alpha_{14})
(\alpha_{12}\alpha_{13})(\alpha_{15}\alpha_{21})
(\alpha_{17}\alpha_{20})(\alpha_{18}\alpha_{23}),\
$$
$$
(\alpha_{1}\alpha_{6}\alpha_{22})
(\alpha_{3}\alpha_{24}\alpha_{16})
(\alpha_{4}\alpha_{20}\alpha_{9})
(\alpha_{5}\alpha_{10}\alpha_{7})
(\alpha_{11}\alpha_{13}\alpha_{12})
(\alpha_{14}\alpha_{18}\alpha_{21})]
$$
with orbits
$
\{\alpha_{1},\alpha_{6},\alpha_{8},\alpha_{22},
\alpha_{10},\alpha_{7},\alpha_{5}\},
\{\alpha_{3},\alpha_{24},\alpha_{9},\alpha_{16},
\alpha_{4},\alpha_{20},\alpha_{17}\},
\{\alpha_{11},\alpha_{14},\alpha_{13},
\newline
\alpha_{18},\alpha_{12},\alpha_{23},\alpha_{21},\alpha_{15}\};
$
$$
H_{74,2}=[
(\alpha_{5}\alpha_{11})(\alpha_{6}\alpha_{10})
(\alpha_{7}\alpha_{17})(\alpha_{12}\alpha_{19})
(\alpha_{13}\alpha_{24})(\alpha_{14}\alpha_{15})
(\alpha_{16}\alpha_{22})(\alpha_{18}\alpha_{20}),\
$$
$$
(\alpha_{1}\alpha_{6}\alpha_{22})
(\alpha_{3}\alpha_{24}\alpha_{16})
(\alpha_{4}\alpha_{20}\alpha_{9})
(\alpha_{5}\alpha_{10}\alpha_{7})
(\alpha_{11}\alpha_{13}\alpha_{12})
(\alpha_{14}\alpha_{18}\alpha_{21})]
$$
with orbits
$
\{\alpha_{1},\alpha_{6},\alpha_{10},\alpha_{22},\alpha_{7},\alpha_{16},\alpha_{17},
\alpha_{5},\alpha_{3},\alpha_{11},\alpha_{24},\alpha_{13},\alpha_{12},\alpha_{19}\},
\{\alpha_{4},\alpha_{20},\alpha_{18},\alpha_{9},
\newline
\alpha_{21},\alpha_{14},\alpha_{15}\}.
$

\medskip

{\bf n=73,} $H\cong 2^4D_{10}$ ($|H|=160$, $i=234$):
$$
H_{73,1}=[
(\alpha_{1}\alpha_{4}\alpha_{20}\alpha_{9}\alpha_{19})
(\alpha_{3}\alpha_{5}\alpha_{24}\alpha_{16}\alpha_{12})
(\alpha_{7}\alpha_{21}\alpha_{17}\alpha_{15}\alpha_{13})
(\alpha_{8}\alpha_{11}\alpha_{23}\alpha_{14}\alpha_{10}),
$$
$$
(\alpha_{4}\alpha_{13})(\alpha_{5}\alpha_{12})
(\alpha_{6}\alpha_{17})(\alpha_{7}\alpha_{9})
(\alpha_{8}\alpha_{15})(\alpha_{11}\alpha_{21})
(\alpha_{14}\alpha_{20})(\alpha_{16}\alpha_{24})]
$$
with $Clos(H_{73,1})=H_{81,1}$ above;
$$
H_{73,2}=[
(\alpha_{1}\alpha_{6}\alpha_{4}\alpha_{24}\alpha_{15})
(\alpha_{5}\alpha_{8}\alpha_{19}\alpha_{17}\alpha_{9})
(\alpha_{7}\alpha_{12}\alpha_{11}\alpha_{23}\alpha_{20})
(\alpha_{10}\alpha_{14}\alpha_{13}\alpha_{16}\alpha_{21}),
$$
$$
(\alpha_{1}\alpha_{9}\alpha_{23}\alpha_{13})
(\alpha_{4}\alpha_{10}\alpha_{7}\alpha_{19})
(\alpha_{5}\alpha_{15}\alpha_{24}\alpha_{8})
(\alpha_{6}\alpha_{14})
(\alpha_{11}\alpha_{12}\alpha_{21}\alpha_{16})
(\alpha_{17}\alpha_{20})]
$$
with $Clos(H_{73,2})=H_{81,2}$ above.

\medskip

{\bf n=72,} $H\cong {\frak A}_4^2$ ($|H|=144$, $i=184$):
$$
H_{72,1}=[
(\alpha_{1}\alpha_{4}\alpha_{23}\alpha_{7}\alpha_{21}\alpha_{11})
(\alpha_{2}\alpha_{10}\alpha_{13})
(\alpha_{3}\alpha_{15}\alpha_{22})
(\alpha_{5}\alpha_{9}\alpha_{14}\alpha_{18}\alpha_{6}\alpha_{19})
(\alpha_{8}\alpha_{20})(\alpha_{12}\alpha_{16}),
$$
$$
(\alpha_{1}\alpha_{4}\alpha_{5}\alpha_{21}\alpha_{7}\alpha_{6})
(\alpha_{2}\alpha_{10}\alpha_{13})
(\alpha_{3}\alpha_{22}\alpha_{15})
(\alpha_{8}\alpha_{23}\alpha_{16}\alpha_{14}\alpha_{12}\alpha_{11})
(\alpha_{9}\alpha_{18})(\alpha_{19}\alpha_{20})]
$$
with $Clos(H_{72,1})=H_{78,1}$ above;
$$
H_{72,2}=[
(\alpha_{2}\alpha_{10})
(\alpha_{3}\alpha_{17}\alpha_{22})
(\alpha_{4}\alpha_{8}\alpha_{11}\alpha_{21}\alpha_{12}\alpha_{14})
(\alpha_{5}\alpha_{18}\alpha_{7})
(\alpha_{6}\alpha_{20}\alpha_{19}\alpha_{9}\alpha_{16}\alpha_{23})
(\alpha_{13}\alpha_{24}),
$$
$$
(\alpha_{1}\alpha_{5})
(\alpha_{3}\alpha_{17}\alpha_{22})
(\alpha_{4}\alpha_{19}\alpha_{12}\alpha_{9}\alpha_{11}\alpha_{8})
(\alpha_{6}\alpha_{23}\alpha_{20}\alpha_{21}\alpha_{14}\alpha_{16})
(\alpha_{7}\alpha_{18})
(\alpha_{10}\alpha_{13}\alpha_{24})]
$$
with $Clos(H_{72,2})=H_{78,2}$ above.

\medskip

{\bf n=71,} $H\cong F_{128}$ ($|H|=128$, $i=931$):
$$
H_{71,1}=[
(\alpha_{6}\alpha_{15})(\alpha_{8}\alpha_{21})
(\alpha_{9}\alpha_{23})(\alpha_{10}\alpha_{13})
(\alpha_{11}\alpha_{17})(\alpha_{12}\alpha_{22})
(\alpha_{14}\alpha_{20})(\alpha_{18}\alpha_{24}),
$$
$$
(\alpha_{4}\alpha_{7})(\alpha_{5}\alpha_{16})
(\alpha_{6}\alpha_{14})(\alpha_{8}\alpha_{11})
(\alpha_{9}\alpha_{13})(\alpha_{12}\alpha_{24})
(\alpha_{15}\alpha_{21})(\alpha_{17}\alpha_{20}),
$$
$$
(\alpha_{1}\alpha_{6}\alpha_{13}\alpha_{14}\alpha_{19}\alpha_{17}\alpha_{9}\alpha_{20})
(\alpha_{4}\alpha_{11}\alpha_{23}\alpha_{21}\alpha_{7}\alpha_{15}\alpha_{10}\alpha_{8})
(\alpha_{5}\alpha_{16})(\alpha_{12}\alpha_{22}\alpha_{24}\alpha_{18})]
$$
with $Clos(H_{71,1})=H_{80,1}$ above.

\medskip

{\bf n=70,} $H\cong {\frak S}_5$ ($|H|=120$, $i=34$):
$\rk N_H=19$ and
$(N_H)^\ast/N_H\cong \bz/60\bz \times
\bz/5\bz$.
$$
H_{70,1}=[
(\alpha_{5}\alpha_{11})(\alpha_{7}\alpha_{12})
(\alpha_{8}\alpha_{13})(\alpha_{9}\alpha_{20})
(\alpha_{10}\alpha_{23})(\alpha_{14}\alpha_{21})
(\alpha_{16}\alpha_{17})(\alpha_{19}\alpha_{22}),
$$
$$
(\alpha_{2}\alpha_{17}\alpha_{15}\alpha_{8}\alpha_{23})
(\alpha_{3}\alpha_{14}\alpha_{10}\alpha_{13}\alpha_{9})
(\alpha_{4}\alpha_{24}\alpha_{18}\alpha_{7}\alpha_{12})
(\alpha_{5}\alpha_{11}\alpha_{20}\alpha_{16}\alpha_{21})]
$$
with orbits
$
\{\alpha_{2},\alpha_{17},\alpha_{16},\alpha_{15},\alpha_{21},
\alpha_{8},\alpha_{14},\alpha_{5},\alpha_{13},\alpha_{23},
\alpha_{10},\alpha_{11},\alpha_{9},\alpha_{20},\alpha_{3}\},
\{\alpha_{4},\alpha_{24},\alpha_{18},\alpha_{7},
\newline
\alpha_{12}\},
\{\alpha_{19},\alpha_{22}\}.
$
$$
H_{70,2}=[
(\alpha_{3}\alpha_{9})(\alpha_{4}\alpha_{8})
(\alpha_{6}\alpha_{19})(\alpha_{7}\alpha_{24})
(\alpha_{10}\alpha_{13})(\alpha_{11}\alpha_{20})
(\alpha_{12}\alpha_{17})(\alpha_{14}\alpha_{22}),
$$
$$
(\alpha_{2}\alpha_{17}\alpha_{15}\alpha_{8}\alpha_{23})
(\alpha_{3}\alpha_{14}\alpha_{10}\alpha_{13}\alpha_{9})
(\alpha_{4}\alpha_{24}\alpha_{18}\alpha_{7}\alpha_{12})
(\alpha_{5}\alpha_{11}\alpha_{20}\alpha_{16}\alpha_{21})]
$$
with orbits
$
\{\alpha_{2},\alpha_{17},\alpha_{12},\alpha_{15},\alpha_{4},
\alpha_{8},\alpha_{24},\alpha_{23},\alpha_{7},\alpha_{18}\},
\{\alpha_{3},\alpha_{9},\alpha_{14},\alpha_{22},\alpha_{10},\alpha_{13}\},
\{\alpha_{5},\alpha_{11},\alpha_{20},
\newline
\alpha_{16},\alpha_{21}\},\
\{\alpha_{6},\alpha_{19}\}.
$

\medskip

{\bf n=69,} $H\cong (Q_8 * Q_8)\rtimes C_3$ ($|H|=96$, $i=204$):
$$
H_{69,1}=[
(\alpha_{3}\alpha_{14}\alpha_{7})
(\alpha_{4}\alpha_{20}\alpha_{5})
(\alpha_{6}\alpha_{16}\alpha_{13})
(\alpha_{8}\alpha_{23}\alpha_{11})
(\alpha_{9}\alpha_{22}\alpha_{17})
(\alpha_{10}\alpha_{15}\alpha_{19}),
$$
$$
(\alpha_{1}\alpha_{11}\alpha_{15}\alpha_{21}\alpha_{19}\alpha_{23})
(\alpha_{3}\alpha_{12}\alpha_{14}\alpha_{22}\alpha_{24}\alpha_{17})
(\alpha_{4}\alpha_{20}\alpha_{5})
(\alpha_{6}\alpha_{16}\alpha_{13})
(\alpha_{7}\alpha_{9})(\alpha_{8}\alpha_{10}),
$$
$$
(\alpha_{1}\alpha_{8}\alpha_{11}\alpha_{21}\alpha_{10}\alpha_{19})
(\alpha_{3}\alpha_{7}\alpha_{12}\alpha_{22}\alpha_{9}\alpha_{24})
(\alpha_{4}\alpha_{6}\alpha_{5})(\alpha_{13}\alpha_{20}\alpha_{16})
(\alpha_{14}\alpha_{17})(\alpha_{15}\alpha_{23})]
$$
with $Clos(H_{69,1})=H_{77,1}$ above.

\medskip

{\bf n=68,} $H\cong 2^3D_{12}$, ($|H|=96$, $i=195$):
$$
H_{68,1}=[
(\alpha_{1}\alpha_{5}\alpha_{7}\alpha_{18})
(\alpha_{3}\alpha_{22})
(\alpha_{4}\alpha_{19}\alpha_{21}\alpha_{14})
(\alpha_{6}\alpha_{11}\alpha_{9}\alpha_{23})
(\alpha_{8}\alpha_{16}\alpha_{20}\alpha_{12})
(\alpha_{10}\alpha_{13}),\
$$
$$
(\alpha_{1}\alpha_{4}\alpha_{5}\alpha_{21}\alpha_{7}\alpha_{6})
(\alpha_{2}\alpha_{10}\alpha_{13})
(\alpha_{3}\alpha_{22}\alpha_{15})
(\alpha_{8}\alpha_{23}\alpha_{16}\alpha_{14}\alpha_{12}\alpha_{11})
(\alpha_{9}\alpha_{18})
(\alpha_{19}\alpha_{20})]
$$
with $Clos(H_{68,1})=H_{78,1}$ above;
$$
H_{68,2}=[
(\alpha_{2}\alpha_{10}\alpha_{24}\alpha_{13})
(\alpha_{4}\alpha_{21}\alpha_{6}\alpha_{9})
(\alpha_{5}\alpha_{7})
(\alpha_{8}\alpha_{19}\alpha_{20}\alpha_{11})
(\alpha_{12}\alpha_{14}\alpha_{16}\alpha_{23})
(\alpha_{17}\alpha_{22}),
$$
$$
(\alpha_{1}\alpha_{5})
(\alpha_{3}\alpha_{17}\alpha_{22})
(\alpha_{4}\alpha_{19}\alpha_{12}\alpha_{9}\alpha_{11}\alpha_{8})
(\alpha_{6}\alpha_{23}\alpha_{20}\alpha_{21}\alpha_{14}\alpha_{16})
(\alpha_{7}\alpha_{18})
(\alpha_{10}\alpha_{13}\alpha_{24})]
$$
with $Clos(H_{68,2})=H_{78,2}$ above.

\medskip

{\bf n=67,} $H\cong 4^2D_6$, ($|H|=96$, $i=64$):
$$
H_{67,1}=[
(\alpha_{4}\alpha_{7}\alpha_{19})
(\alpha_{6}\alpha_{8}\alpha_{9})
(\alpha_{10}\alpha_{11}\alpha_{21})
(\alpha_{12}\alpha_{18}\alpha_{22})
(\alpha_{13}\alpha_{15}\alpha_{20})
(\alpha_{14}\alpha_{23}\alpha_{17}),
$$
$$
(\alpha_{1}\alpha_{6}\alpha_{13}\alpha_{14}\alpha_{19}\alpha_{17}\alpha_{9}\alpha_{20})
(\alpha_{4}\alpha_{11}\alpha_{23}\alpha_{21}\alpha_{7}\alpha_{15}\alpha_{10}\alpha_{8})
(\alpha_{5}\alpha_{16})(\alpha_{12}\alpha_{22}\alpha_{24}\alpha_{18})]
$$
with $Clos(H_{67,1})=H_{80,1}$ above.

\medskip

{\bf n=66,} $H\cong 2^4 C_6$ ($|H|=96$, $i=70$):
$$
H_{66,1}=[
(\alpha_{1}\alpha_{5}\alpha_{15}\alpha_{8}\alpha_{16}\alpha_{14})
(\alpha_{2}\alpha_{13}\alpha_{10})
(\alpha_{3}\alpha_{7}\alpha_{11}\alpha_{23}\alpha_{12}\alpha_{22})
(\alpha_{4}\alpha_{24}\alpha_{21})
(\alpha_{17}\alpha_{19})(\alpha_{18}\alpha_{20}),
$$
$$
(\alpha_{2}\alpha_{10}\alpha_{13})(\alpha_{5}\alpha_{17}\alpha_{8})
(\alpha_{7}\alpha_{23}\alpha_{22})(\alpha_{9}\alpha_{24}\alpha_{21})
(\alpha_{11}\alpha_{18}\alpha_{12})(\alpha_{15}\alpha_{19}\alpha_{16})]
$$
with $Clos(H_{66,1})=H_{76,1}$ above;
$$
H_{66,2}=[
(\alpha_{1}\alpha_{4}\alpha_{5}\alpha_{16}\alpha_{18}\alpha_{11})
(\alpha_{2}\alpha_{10}\alpha_{13})(\alpha_{3}\alpha_{22})
(\alpha_{6}\alpha_{8}\alpha_{20}\alpha_{19}\alpha_{23}\alpha_{21})
(\alpha_{9}\alpha_{14}\alpha_{12})(\alpha_{17}\alpha_{24}),
$$
$$
(\alpha_{2}\alpha_{10}\alpha_{13})(\alpha_{3}\alpha_{22})
(\alpha_{4}\alpha_{9}\alpha_{20}\alpha_{8}\alpha_{14}\alpha_{11})
(\alpha_{5}\alpha_{23}\alpha_{7}\alpha_{21}\alpha_{18}\alpha_{12})
(\alpha_{6}\alpha_{19}\alpha_{16})(\alpha_{17}\alpha_{24})]
$$
with $Clos(H_{66,2})=H_{76,2}$ above;
$$
H_{66,3}=[
(\alpha_{3}\alpha_{8}\alpha_{18})(\alpha_{4}\alpha_{22}\alpha_{23})
(\alpha_{5}\alpha_{19}\alpha_{21})(\alpha_{6}\alpha_{20}\alpha_{11})
(\alpha_{7}\alpha_{12}\alpha_{15})(\alpha_{9}\alpha_{14}\alpha_{10}),
$$
$$
(\alpha_{1}\alpha_{5}\alpha_{20}\alpha_{21}\alpha_{6}\alpha_{16})
(\alpha_{2}\alpha_{13})
(\alpha_{3}\alpha_{4}\alpha_{7}\alpha_{22}\alpha_{18}\alpha_{9})
(\alpha_{8}\alpha_{15}\alpha_{12})(\alpha_{10}\alpha_{14}\alpha_{23})
(\alpha_{11}\alpha_{19})]
$$
with $Clos(H_{66,3})=H_{76,3}$ above.

\medskip

{\bf n=65,} $H\cong2^4D_6$ ($|H|=96$, $i=227$):
$\rk N_H=18$,
$(N_H)^\ast/N_H\cong \bz/24\bz \times
\bz/4\bz\times (\bz/2\bz)^2$ and
$\det(K((q_{N_H})_2)\equiv \pm 2^7\cdot 3 \mod (\bz_2^\ast)^2$.
$$
H_{65,1}=[
(\alpha_{1}\alpha_{6}\alpha_{21})(\alpha_{4}\alpha_{17}\alpha_{14})
(\alpha_{7}\alpha_{15}\alpha_{8})(\alpha_{9}\alpha_{10}\alpha_{23})
(\alpha_{11}\alpha_{20}\alpha_{19})(\alpha_{12}\alpha_{18}\alpha_{24}),
$$
$$
(\alpha_{1}\alpha_{11}\alpha_{23}\alpha_{8})
(\alpha_{4}\alpha_{6}\alpha_{13}\alpha_{20})
(\alpha_{5}\alpha_{16})
(\alpha_{7}\alpha_{17}\alpha_{9}\alpha_{14})
(\alpha_{10}\alpha_{21}\alpha_{19}\alpha_{15})
(\alpha_{12}\alpha_{24})]
$$
with orbits
$
\{\alpha_{1},\alpha_{19},\alpha_{6},\alpha_{7},\alpha_{4},
\alpha_{17},\alpha_{20},\alpha_{11},\alpha_{14},
\alpha_{21},\alpha_{15},\alpha_{23},
\alpha_{8},\alpha_{10},\alpha_{9},\alpha_{13}\},
\{\alpha_{5},\alpha_{16}\}, \newline
\{\alpha_{12},\alpha_{24},\alpha_{18}\}$;
$$
H_{65,2}=[
(\alpha_{1}\alpha_{10}\alpha_{19})(\alpha_{4}\alpha_{17}\alpha_{24})
(\alpha_{5}\alpha_{7}\alpha_{6})(\alpha_{9}\alpha_{20}\alpha_{16}),
(\alpha_{11}\alpha_{15}\alpha_{21})(\alpha_{12}\alpha_{13}\alpha_{14}),
$$
$$
(\alpha_{1}\alpha_{11}\alpha_{23}\alpha_{8})
(\alpha_{4}\alpha_{6}\alpha_{13}\alpha_{20})
(\alpha_{5}\alpha_{16})
(\alpha_{7}\alpha_{17}\alpha_{9}\alpha_{14})
(\alpha_{10}\alpha_{21}\alpha_{19}\alpha_{15})
(\alpha_{12}\alpha_{24})];
$$
with orbits
$
\{\alpha_{1},\alpha_{11},\alpha_{10},\alpha_{23},\alpha_{19},
\alpha_{15},\alpha_{8},\alpha_{21}\},
\{\alpha_{4},\alpha_{17},\alpha_{7},\alpha_{9},\alpha_{24},\alpha_{6},
\alpha_{20},\alpha_{13},\alpha_{5},\alpha_{16},
\newline
\alpha_{14},\alpha_{12}\}$.
$$
H_{65,3}=[
(\alpha_{1}\alpha_{7}\alpha_{17})(\alpha_{3}\alpha_{22}\alpha_{5})
(\alpha_{4}\alpha_{23}\alpha_{20})(\alpha_{6}\alpha_{14}\alpha_{8})
(\alpha_{9}\alpha_{13}\alpha_{11})(\alpha_{10}\alpha_{19}\alpha_{21}),
$$
$$
(\alpha_{1}\alpha_{11}\alpha_{23}\alpha_{8})
(\alpha_{4}\alpha_{6}\alpha_{13}\alpha_{20})
(\alpha_{5}\alpha_{16})
(\alpha_{7}\alpha_{17}\alpha_{9}\alpha_{14})
(\alpha_{10}\alpha_{21}\alpha_{19}\alpha_{15})
(\alpha_{12}\alpha_{24})];
$$
with orbits
$
\{\alpha_{1},\alpha_{7},\alpha_{23},\alpha_{8},\alpha_{11},\alpha_{17},\alpha_{20},\alpha_{9},
\alpha_{6},\alpha_{4},\alpha_{14},\alpha_{13}\},
\{\alpha_{3},\alpha_{22},\alpha_{16},\alpha_{5}\},
\{\alpha_{10},\alpha_{19},
\newline
\alpha_{15},\alpha_{21}\},
\{\alpha_{12},\alpha_{24}\}
$.
$$
H_{65,4}=[
(\alpha_{1}\alpha_{21}\alpha_{23})(\alpha_{3}\alpha_{22}\alpha_{12})
(\alpha_{4}\alpha_{20}\alpha_{13})(\alpha_{5}\alpha_{18}\alpha_{16})
(\alpha_{8}\alpha_{11}\alpha_{10})(\alpha_{9}\alpha_{14}\alpha_{17}),
$$
$$
(\alpha_{1}\alpha_{11}\alpha_{23}\alpha_{8})
(\alpha_{4}\alpha_{6}\alpha_{13}\alpha_{20})
(\alpha_{5}\alpha_{16})
(\alpha_{7}\alpha_{17}\alpha_{9}\alpha_{14})
(\alpha_{10}\alpha_{21}\alpha_{19}\alpha_{15})
(\alpha_{12}\alpha_{24})]
$$
with orbits
$
\{\alpha_{1},\alpha_{8},\alpha_{21},\alpha_{23},\alpha_{15},
\alpha_{11},\alpha_{10},\alpha_{19}\},
\{\alpha_{3},\alpha_{22},\alpha_{24},\alpha_{12}\},
\{\alpha_{4},\alpha_{13},\alpha_{20},\alpha_{6}\},
\{\alpha_{5},\alpha_{16},
\newline
\alpha_{18}\},\{\alpha_{7},\alpha_{17},\alpha_{14},\alpha_{9},\}
$.

\medskip

{\bf n=64,} $H\cong 2^4 C_5$ ($|H|=80$, $i=49$):
$$
H_{64,1}=[
(\alpha_{3}\alpha_{5}\alpha_{16}\alpha_{12}\alpha_{24})
(\alpha_{4}\alpha_{11}\alpha_{15}\alpha_{9}\alpha_{23})
(\alpha_{6}\alpha_{20}\alpha_{21}\alpha_{10}\alpha_{8})
(\alpha_{7}\alpha_{13}\alpha_{17}\alpha_{19}\alpha_{14}),
$$
$$
(\alpha_{1}\alpha_{13})(\alpha_{4}\alpha_{23})
(\alpha_{6}\alpha_{21})(\alpha_{7}\alpha_{10})
(\alpha_{8}\alpha_{17})(\alpha_{9}\alpha_{19})
(\alpha_{11}\alpha_{14})(\alpha_{15}\alpha_{20})]
$$
with $Clos(H_{64,1})=H_{81,1}$ above;
$$
H_{64,2}=[
(\alpha_{1}\alpha_{4}\alpha_{5}\alpha_{14}\alpha_{8})
(\alpha_{6}\alpha_{21}\alpha_{19}\alpha_{9}\alpha_{16})
(\alpha_{7}\alpha_{12}\alpha_{20}\alpha_{11}\alpha_{23})
(\alpha_{10}\alpha_{13}\alpha_{24}\alpha_{17}\alpha_{15}),
$$
$$
(\alpha_{4}\alpha_{13})(\alpha_{5}\alpha_{12})
(\alpha_{6}\alpha_{17})(\alpha_{7}\alpha_{9})
(\alpha_{8}\alpha_{15})(\alpha_{11}\alpha_{21})
(\alpha_{14}\alpha_{20})(\alpha_{16}\alpha_{24})]
$$
with $Clos(H_{64,2})=H_{81,2}$ above.

\medskip

{\bf n=63,} $H\cong M_9$ ($|H|=72$, $i=41$):
$\rk N_H=19$ and $(N_H)^\ast/N_H\cong \bz/18\bz \times
\bz/6\bz\times \bz/2\bz$.
$$
H_{63,1}=[
(\alpha_{1}\alpha_{9}\alpha_{17}\alpha_{14})
(\alpha_{4}\alpha_{19})
(\alpha_{5}\alpha_{15}\alpha_{6}\alpha_{21})
(\alpha_{7}\alpha_{11}\alpha_{23}\alpha_{22})
(\alpha_{10}\alpha_{13}\alpha_{24}\alpha_{16})
(\alpha_{12}\alpha_{20}),
$$
$$
(\alpha_{1}\alpha_{11}\alpha_{23}\alpha_{8})
(\alpha_{4}\alpha_{6}\alpha_{13}\alpha_{20})
(\alpha_{5}\alpha_{16})
(\alpha_{7}\alpha_{17}\alpha_{9}\alpha_{14})
(\alpha_{10}\alpha_{21}\alpha_{19}\alpha_{15})
(\alpha_{12}\alpha_{24})]
$$
with orbits
$
\{\alpha_{1},\alpha_{23},\alpha_{9},\alpha_{17},\alpha_{22},
\alpha_{7},\alpha_{11},\alpha_{14},\alpha_{8}\},
\{\alpha_{4},\alpha_{20},\alpha_{19},\alpha_{16},\alpha_{12},
\alpha_{6},\alpha_{24},\alpha_{10},\alpha_{13},
\newline
\alpha_{15},\alpha_{21},\alpha_{5}\}$.

\medskip

{\bf n=62,} $H\cong N_{72}$ ($|H|=72$, $i=40$):
$\rk N_H=19$,
$(N_H)^\ast/N_H\cong \bz/36\bz \times
(\bz/3\bz)^2$ and
$\det(K((q_{N_H})_3))\equiv - 2^2\cdot 3^4\mod (\bz_3^\ast)^2$.
$$
H_{62,1}=[
(\alpha_{1}\alpha_{4})(\alpha_{2}\alpha_{15})
(\alpha_{7}\alpha_{21})(\alpha_{8}\alpha_{20})
(\alpha_{10}\alpha_{22})(\alpha_{11}\alpha_{16})
(\alpha_{12}\alpha_{23})(\alpha_{14}\alpha_{19}),
$$
$$
(\alpha_{1}\alpha_{19})
(\alpha_{2}\alpha_{18}\alpha_{3}\alpha_{23}\alpha_{10}\alpha_{22})
(\alpha_{4}\alpha_{20}\alpha_{14})(\alpha_{5}\alpha_{12}\alpha_{15})
(\alpha_{6}\alpha_{16}\alpha_{9}\alpha_{24}\alpha_{17}\alpha_{11})
(\alpha_{7}\alpha_{21})]
$$
with orbits
$
\{\alpha_{1},\alpha_{4},\alpha_{19},\alpha_{8},\alpha_{20},\alpha_{14}\},
\{\alpha_{2},\alpha_{15},\alpha_{10},\alpha_{18},
\alpha_{12},\alpha_{22},\alpha_{23},\alpha_{5},\alpha_{3}\},
\{\alpha_{6},\alpha_{11},\alpha_{9},\alpha_{16},
\newline
\alpha_{17},\alpha_{24}\},
\{\alpha_{7},\alpha_{21}\}.
$

\medskip

{\bf n=61,} $H\cong {\frak A}_{4,3}$ ($|H|=72$, $i=43$):
$\rk N_H=18$ and
$(N_H)^\ast/N_H\cong (\bz/12\bz)^2 \times \bz/3\bz$.
$$
H_{61,1}=[
(\alpha_{1}\alpha_{10}\alpha_{23})(\alpha_{2}\alpha_{4}\alpha_{11})
(\alpha_{3}\alpha_{21}\alpha_{20})(\alpha_{8}\alpha_{9}\alpha_{17})
(\alpha_{13}\alpha_{22}\alpha_{14})(\alpha_{16}\alpha_{19}\alpha_{24}),
$$
$$
(\alpha_{2}\alpha_{14})(\alpha_{3}\alpha_{17})(\alpha_{4}\alpha_{8})(\alpha_{7}\alpha_{15})
(\alpha_{9}\alpha_{13})(\alpha_{10}\alpha_{12})(\alpha_{11}\alpha_{21})(\alpha_{16}\alpha_{19}),
$$
$$
(\alpha_{2}\alpha_{13})(\alpha_{3}\alpha_{22})(\alpha_{4}\alpha_{9})(\alpha_{7}\alpha_{18})
(\alpha_{8}\alpha_{14})(\alpha_{11}\alpha_{20})(\alpha_{12}\alpha_{23})(\alpha_{16}\alpha_{19})]
$$
with orbits
$
\{\alpha_{1},\alpha_{10},\alpha_{23},\alpha_{12}\},
\{\alpha_{2},\alpha_{4},\alpha_{21},\alpha_{13},\alpha_{11}\alpha_{8},
\alpha_{20},\alpha_{17},\alpha_{22}\alpha_{14},\alpha_{9},\alpha_{3}\},
\{\alpha_{7},\alpha_{18},\alpha_{15}\}, \newline
\{\alpha_{16},\alpha_{19},\alpha_{24}\}.
$

\medskip

{\bf n=60,} $H\cong \Gamma_{26}a_2$ ($|H|=64$, $i=136$):
$$
H_{60,1}=[
(\alpha_{4}\alpha_{7})(\alpha_{5}\alpha_{16})
(\alpha_{6}\alpha_{8})(\alpha_{10}\alpha_{23})
(\alpha_{11}\alpha_{14})(\alpha_{15}\alpha_{20})
(\alpha_{17}\alpha_{21})(\alpha_{18}\alpha_{22}),
$$
$$
(\alpha_{1}\alpha_{6}\alpha_{23}\alpha_{8}\alpha_{19}\alpha_{17}\alpha_{10}\alpha_{21})
(\alpha_{4}\alpha_{11}\alpha_{13}\alpha_{20}\alpha_{7}\alpha_{15}\alpha_{9}\alpha_{14})
(\alpha_{5}\alpha_{16})
(\alpha_{12}\alpha_{18}\alpha_{24}\alpha_{22}),
$$
$$
(\alpha_{1}\alpha_{6}\alpha_{13}\alpha_{14}\alpha_{19}\alpha_{17}\alpha_{9}\alpha_{20})
(\alpha_{4}\alpha_{11}\alpha_{23}\alpha_{21}\alpha_{7}\alpha_{15}\alpha_{10}\alpha_{8})
(\alpha_{5}\alpha_{16})
(\alpha_{12}\alpha_{22}\alpha_{24}\alpha_{18})]
$$
with $Clos(H_{60,1})=H_{80,1}$ above.

\medskip

{\bf n=59,} $H\cong \Gamma_{23}a_2$ ($|H|=64$, $i=35$):
$$
H_{59,1}=
(\alpha_{4}\alpha_{19})(\alpha_{5}\alpha_{16})
(\alpha_{6}\alpha_{13}\alpha_{17}\alpha_{23})
(\alpha_{8}\alpha_{14}\alpha_{20}\alpha_{21})
(\alpha_{9}\alpha_{11}\alpha_{10}\alpha_{15})
(\alpha_{12}\alpha_{22}\alpha_{18}\alpha_{24}),
$$
$$
(\alpha_{1}\alpha_{6}\alpha_{7}\alpha_{11})
(\alpha_{4}\alpha_{15}\alpha_{19}\alpha_{17})
(\alpha_{8}\alpha_{10}\alpha_{20}\alpha_{13})
(\alpha_{9}\alpha_{21}\alpha_{23}\alpha_{14})
(\alpha_{12}\alpha_{24})(\alpha_{18}\alpha_{22})]
$$
has $Clos(H_{59,1})=H_{80,1}$ above.

\medskip

{\bf n=58,} $H\cong \Gamma_{22}a_1$ ($|H|=64$, $i=32$):
$$
H_{58,1,1}=[
(\alpha_{1}\alpha_{6}\alpha_{7}\alpha_{11})
(\alpha_{4}\alpha_{15}\alpha_{19}\alpha_{17})
(\alpha_{8}\alpha_{10}\alpha_{20}\alpha_{13})
(\alpha_{9}\alpha_{21}\alpha_{23}\alpha_{14})
(\alpha_{12}\alpha_{24})(\alpha_{18}\alpha_{22}),
$$
$$
(\alpha_{1}\alpha_{6}\alpha_{13}\alpha_{14}\alpha_{19}\alpha_{17}\alpha_{9}\alpha_{20})
(\alpha_{4}\alpha_{11}\alpha_{23}\alpha_{21}\alpha_{7}\alpha_{15}\alpha_{10}\alpha_{8})
(\alpha_{5}\alpha_{16})
(\alpha_{12}\alpha_{22}\alpha_{24}\alpha_{18})]
$$
has $Clos(H_{58,1,1})=H_{80,1}$ above;
$$
H_{58,1,2}=[
(\alpha_{1}\alpha_{9}\alpha_{7}\alpha_{23})
(\alpha_{4}\alpha_{10}\alpha_{19}\alpha_{13})
(\alpha_{6}\alpha_{14}\alpha_{11}\alpha_{21})
(\alpha_{8}\alpha_{17}\alpha_{20}\alpha_{15})
(\alpha_{12}\alpha_{22})(\alpha_{18}\alpha_{24}),
$$
$$
(\alpha_{1}\alpha_{6}\alpha_{13}\alpha_{14}\alpha_{19}\alpha_{17}\alpha_{9}\alpha_{20})
(\alpha_{4}\alpha_{11}\alpha_{23}\alpha_{21}\alpha_{7}\alpha_{15}\alpha_{10}\alpha_{8})
(\alpha_{5}\alpha_{16})
(\alpha_{12}\alpha_{22}\alpha_{24}\alpha_{18})]
$$
has $Clos(H_{58,1,2})=H_{80,1}$ above.

\medskip

{\bf n=57,} $H\cong \Gamma_{13}a_1$ ($|H|=64$, $i=242$):
$$
H_{57,1}=[
(\alpha_{1}\alpha_{3}\alpha_{7}\alpha_{23})
(\alpha_{4}\alpha_{20}\alpha_{6}\alpha_{10})
(\alpha_{5}\alpha_{14}\alpha_{21}\alpha_{22})
(\alpha_{9}\alpha_{24}\alpha_{18}\alpha_{16})
(\alpha_{11}\alpha_{13})(\alpha_{12}\alpha_{17}),
$$
$$
(\alpha_{1}\alpha_{10}\alpha_{7}\alpha_{20})
(\alpha_{3}\alpha_{6}\alpha_{23}\alpha_{4})
(\alpha_{5}\alpha_{24}\alpha_{21}\alpha_{16})
(\alpha_{9}\alpha_{22}\alpha_{18}\alpha_{14})
(\alpha_{11}\alpha_{12})(\alpha_{13}\alpha_{17}),
$$
$$
(\alpha_{3}\alpha_{24})(\alpha_{4}\alpha_{6})
(\alpha_{5}\alpha_{21})(\alpha_{10}\alpha_{14})
(\alpha_{11}\alpha_{17})(\alpha_{12}\alpha_{13})
(\alpha_{16}\alpha_{23})(\alpha_{20}\alpha_{22}),
$$
$$
(\alpha_{3}\alpha_{16})(\alpha_{4}\alpha_{5})
(\alpha_{6}\alpha_{21})(\alpha_{10}\alpha_{20})
(\alpha_{11}\alpha_{12})(\alpha_{13}\alpha_{17})
(\alpha_{14}\alpha_{22})(\alpha_{23}\alpha_{24})]
$$
with $Clos(H_{57,1})=H_{75,1}$ above.

\medskip

{\bf n=56,} $H\cong \Gamma_{25}a_1$ ($|H|=64$, $i=138$):
$\rk N_H=18$, $(N_H)^\ast/N_H\cong
\bz/8\bz\times (\bz/4\bz)^3$ and
$\det(K((q_{N_H})_2))\equiv \pm 2^9\mod (\bz_2^\ast)^9$.
$$
H_{56,1}=[
(\alpha_{5}\alpha_{23})(\alpha_{6}\alpha_{10})(\alpha_{7}\alpha_{9})
(\alpha_{8}\alpha_{16})(\alpha_{13}\alpha_{18})(\alpha_{14}\alpha_{24})
(\alpha_{15}\alpha_{20})(\alpha_{19}\alpha_{21}),
$$
$$
(\alpha_{1}\alpha_{14})(\alpha_{2}\alpha_{23})(\alpha_{3}\alpha_{5})
(\alpha_{6}\alpha_{18})(\alpha_{7},\alpha_{8})(\alpha_{9}\alpha_{13})
(\alpha_{10}\alpha_{16})(\alpha_{17}\alpha_{24}),
$$
$$
(\alpha_{1}\alpha_{14})(\alpha_{2}\alpha_{24})(\alpha_{3}\alpha_{5})
(\alpha_{6}\alpha_{10})(\alpha_{12}\alpha_{22})(\alpha_{16}\alpha_{18})
(\alpha_{17}\alpha_{23})(\alpha_{19}\alpha_{20})]
$$
with orbits
$
 \{\alpha_{1}, \alpha_{14}, \alpha_{24}, \alpha_{17}, \alpha_{2}, 
\alpha_{23}, \alpha_{5}, \alpha_{3} \}, \{ \alpha_{6}, \alpha_{10}, \alpha_{18},
\alpha_{16}, \alpha_{13}, \alpha_{8}, \alpha_{9}, \alpha_{7} \},
\{ \alpha_{12}, \alpha_{22} \},
\newline
\{ \alpha_{15}, \alpha_{20}, \alpha_{19}, \alpha_{21}\}$;

$$
H_{56,2}=[
(\alpha_{2}\alpha_{3})(\alpha_{5}\alpha_{6})(\alpha_{7}\alpha_{18})
(\alpha_{8}\alpha_{23})(\alpha_{10}\alpha_{20})(\alpha_{11}\alpha_{17})
(\alpha_{15}\alpha_{16})(\alpha_{19}\alpha_{24}),
$$
$$
(\alpha_{1}\alpha_{14})(\alpha_{2}\alpha_{23})(\alpha_{3}\alpha_{5})
(\alpha_{6}\alpha_{18})(\alpha_{7}\alpha_{8})(\alpha_{9}\alpha_{13})
(\alpha_{10}\alpha_{16})(\alpha_{17}\alpha_{24}),
$$
$$
(\alpha_{1}\alpha_{14})(\alpha_{2}\alpha_{24})(\alpha_{3}\alpha_{5})
(\alpha_{6}\alpha_{10})(\alpha_{12}\alpha_{22})(\alpha_{16}\alpha_{18})
(\alpha_{17}\alpha_{23})(\alpha_{19}\alpha_{20})]
$$
with orbits
$
 \{\alpha_{1}, \alpha_{14}\}, \{\alpha_{2}, \alpha_{3}, \alpha_{23}, 
\alpha_{24}, \alpha_{5},\alpha_{8}, \alpha_{17}, \alpha_{19}, \alpha_{6},
\alpha_{7}, \alpha_{11}, \alpha_{20}, \alpha_{18}, 
\alpha_{10},\alpha_{16},\alpha_{15}\},
\newline
\{\alpha_{9},\alpha_{13}\},
\{\alpha_{12}, \alpha_{22} \}$.

\medskip

{\bf n=55,} $H\cong {\frak A}_5$ ($|H|=60$, $i=5$):
$\rk N_H=18$ and $(N_H)^\ast/N_H\cong \bz/30\bz\times \bz/10\bz$.
$$
H_{55,1}=[
(\alpha_{2}\alpha_{13}\alpha_{18}\alpha_{7}\alpha_{5})
(\alpha_{3}\alpha_{22}\alpha_{12}\alpha_{11}\alpha_{19})
(\alpha_{4}\alpha_{6}\alpha_{15}\alpha_{9}\alpha_{14})
(\alpha_{16}\alpha_{21}\alpha_{20}\alpha_{17}\alpha_{23}),
$$
$$
(\alpha_{2}\alpha_{13})(\alpha_{3}\alpha_{22})
(\alpha_{4}\alpha_{9})(\alpha_{7}\alpha_{18})
(\alpha_{8}\alpha_{14})(\alpha_{11}\alpha_{20})
(\alpha_{12}\alpha_{23})(\alpha_{16}\alpha_{19})]
$$
with orbits
$
\{\alpha_{2},\alpha_{13},\alpha_{18},\alpha_{7},\alpha_{5}\},
\{\alpha_{3},\alpha_{22},\alpha_{12},\alpha_{11},\alpha_{23},
\alpha_{19},\alpha_{20},\alpha_{16},\alpha_{17},\alpha_{21}\},
\{\alpha_{4},\alpha_{6},\alpha_{9},
\newline
\alpha_{15},\alpha_{14},\alpha_{8}\}$;
$$
H_{55,2}=[
(\alpha_{2}\alpha_{8}\alpha_{7}\alpha_{17}\alpha_{11})
(\alpha_{3}\alpha_{14}\alpha_{10}\alpha_{13}\alpha_{12})
(\alpha_{4}\alpha_{19}\alpha_{16}\alpha_{15}\alpha_{9})
(\alpha_{6}\alpha_{23}\alpha_{20}\alpha_{18}\alpha_{22}),
$$
$$
(\alpha_{2}\alpha_{13})(\alpha_{3}\alpha_{22})
(\alpha_{4}\alpha_{9})(\alpha_{7}\alpha_{18})
(\alpha_{8}\alpha_{14})(\alpha_{11}\alpha_{20})
(\alpha_{12}\alpha_{23})(\alpha_{16}\alpha_{19})]
$$
with orbits
$
\{\alpha_{2},\alpha_{8},\alpha_{13},\alpha_{7},\alpha_{14},
\alpha_{12},\alpha_{17},\alpha_{18},\alpha_{10},\alpha_{3},
\alpha_{23},\alpha_{11},\alpha_{22},\alpha_{20},\alpha_{6}\},
\{\alpha_{4},\alpha_{19},\alpha_{9},
\newline
\alpha_{16},\alpha_{15}\}$.

\medskip

{\bf n=54,} $H\cong T_{48}$ ($|H|=48$, $i=29$):
$\rk N_H=19$ and $(N_H)^\ast/N_H\cong
\bz/24\bz\times \bz/8\bz\times \bz/2\bz$.
$$
H_{54,1}=[
(\alpha_{1}\alpha_{13}\alpha_{22})
(\alpha_{3}\alpha_{10}\alpha_{7})
(\alpha_{4}\alpha_{23}\alpha_{24})
(\alpha_{5}\alpha_{9}\alpha_{19})
(\alpha_{6}\alpha_{8}\alpha_{14})
(\alpha_{11}\alpha_{17}\alpha_{21}),
$$
$$
(\alpha_{2}\alpha_{13})(\alpha_{3}\alpha_{22})
(\alpha_{4}\alpha_{9})(\alpha_{7}\alpha_{18})
(\alpha_{8}\alpha_{14})(\alpha_{11}\alpha_{20})
(\alpha_{12}\alpha_{23})(\alpha_{16}\alpha_{19})]
$$
with orbits
$
\{\alpha_{1},\alpha_{2},\alpha_{7},\alpha_{22},
\alpha_{10},\alpha_{3},\alpha_{13},\alpha_{18}\},
\{\alpha_{4},\alpha_{12},\alpha_{16},\alpha_{19},\alpha_{24},
\alpha_{9},\alpha_{5},\alpha_{23}\},\{\alpha_{6},\alpha_{8},\alpha_{14}\},
\newline
\{\alpha_{11},\alpha_{21},\alpha_{20},\alpha_{17}\}$.

\medskip

{\bf n=53,} $H\cong 2^2Q_{12}$ ($|H|=48$, $i=30$):
$$
H_{53,1}=[
(\alpha_{2}\alpha_{10}\alpha_{13})
(\alpha_{3}\alpha_{15}\alpha_{22})
(\alpha_{4}\alpha_{16}\alpha_{11})
(\alpha_{6}\alpha_{8}\alpha_{14})
(\alpha_{9}\alpha_{20}\alpha_{19})
(\alpha_{12}\alpha_{23}\alpha_{21}),
$$
$$
(\alpha_{1}\alpha_{4}\alpha_{14}\alpha_{20})
(\alpha_{5}\alpha_{9}\alpha_{23}\alpha_{16})
(\alpha_{6}\alpha_{19}\alpha_{12}\alpha_{7})
(\alpha_{8}\alpha_{18}\alpha_{21}\alpha_{11})
(\alpha_{10}\alpha_{13})(\alpha_{15}\alpha_{22})]
$$
has $Clos(H_{53,1})=H_{78,1}$ above;
$$
H_{53,2}=[
(\alpha_{1}\alpha_{5}\alpha_{7})
(\alpha_{2}\alpha_{13})
(\alpha_{3}\alpha_{17}\alpha_{22})
(\alpha_{4}\alpha_{16}\alpha_{23}\alpha_{9}\alpha_{8}\alpha_{11})
(\alpha_{6}\alpha_{12}\alpha_{14}\alpha_{21}\alpha_{20}\alpha_{19})
(\alpha_{10}\alpha_{24}),
$$
$$
(\alpha_{2}\alpha_{10}\alpha_{13}\alpha_{24})
(\alpha_{4}\alpha_{8}\alpha_{9}\alpha_{16})
(\alpha_{5}\alpha_{18})
(\alpha_{6}\alpha_{12}\alpha_{21}\alpha_{20})
(\alpha_{11}\alpha_{14}\alpha_{23}\alpha_{19})
(\alpha_{17}\alpha_{22})]
$$
has $Clos(H_{53,2})=H_{78,2}$ above.

\medskip

{\bf n=52,} $H\cong 2^2(C_2\times C_6)$ ($|H|=48$, $i=49$):
$$
H_{52,1}=[
(\alpha_{1}\alpha_{8}\alpha_{21}\alpha_{5}\alpha_{12}\alpha_{6})
(\alpha_{2}\alpha_{10}\alpha_{13})(\alpha_{3}\alpha_{15}\alpha_{22})
(\alpha_{4}\alpha_{18}\alpha_{16}\alpha_{9}\alpha_{7}\alpha_{20})
(\alpha_{11}\alpha_{19})(\alpha_{14}\alpha_{23}),
$$
$$
(\alpha_{1}\alpha_{4})(\alpha_{5}\alpha_{9})
(\alpha_{6}\alpha_{18})(\alpha_{7}\alpha_{21})
(\alpha_{8}\alpha_{19})(\alpha_{11}\alpha_{12})
(\alpha_{14}\alpha_{20})(\alpha_{16}\alpha_{23})]
$$
has $Clos(H_{52,1})=H_{78,1}$ above;
$$
H_{52,2}=[
(\alpha_{2}\alpha_{13})
(\alpha_{3}\alpha_{22}\alpha_{17})
(\alpha_{4}\alpha_{23}\alpha_{12}\alpha_{9}\alpha_{11}\alpha_{20})
(\alpha_{5}\alpha_{7}\alpha_{18})
(\alpha_{6}\alpha_{14}\alpha_{16}\alpha_{21}\alpha_{19}\alpha_{8})
(\alpha_{10}\alpha_{24}),
$$
$$
(\alpha_{1}\alpha_{5})(\alpha_{2}\alpha_{10})
(\alpha_{4}\alpha_{6})(\alpha_{7}\alpha_{18})
(\alpha_{9}\alpha_{21})(\alpha_{11}\alpha_{23})
(\alpha_{13}\alpha_{24})(\alpha_{14}\alpha_{19})]
$$
has $Clos(H_{52,2})=H_{78,2}$ above.

\medskip

{\bf n=51,} $H\cong C_2\times {\frak S}_4$ ($|H|=48$, $i=48$):
$\rk N_H=18$, $(N_H)^\ast/N_H\cong
(\bz/12\bz)^2\times (\bz/2\bz)^2$ and
$\det(K((q_{N_H})_2))\equiv \pm 2^6\mod (\bz_2^\ast)^2$.
$$
H_{51,1}=[
(\alpha_{1}\alpha_{6}\alpha_{23}\alpha_{8})
(\alpha_{4}\alpha_{11}\alpha_{16}\alpha_{7})
(\alpha_{5}\alpha_{21}\alpha_{14}\alpha_{12})
(\alpha_{9}\alpha_{19}\alpha_{20}\alpha_{18})
(\alpha_{10}\alpha_{13})(\alpha_{15}\alpha_{22}),
$$
$$
(\alpha_{2}\alpha_{13})(\alpha_{3}\alpha_{22})
(\alpha_{4}\alpha_{9})(\alpha_{7}\alpha_{18})
(\alpha_{8}\alpha_{14})(\alpha_{11}\alpha_{20})
(\alpha_{12}\alpha_{23})(\alpha_{16}\alpha_{19})]
$$
with orbits
$
\{\alpha_{1},\alpha_{6},\alpha_{5},\alpha_{21},
\alpha_{23},\alpha_{12},\alpha_{8},\alpha_{14}\},
\{\alpha_{2},\alpha_{13},\alpha_{10}\},
\{\alpha_{3},\alpha_{22},\alpha_{15}\},
\{\alpha_{4},\alpha_{7},\alpha_{9},\alpha_{16},
\newline
\alpha_{11},\alpha_{18},\alpha_{20},\alpha_{19}\}$;
$$
H_{51,2}=[
(\alpha_{1}\alpha_{18}\alpha_{7}\alpha_{5})
(\alpha_{4}\alpha_{20}\alpha_{21}\alpha_{8})
(\alpha_{6}\alpha_{16}\alpha_{9}\alpha_{12})
(\alpha_{10}\alpha_{24})
(\alpha_{11}\alpha_{14}\alpha_{23}\alpha_{19})
(\alpha_{17}\alpha_{22}),
$$
$$
(\alpha_{2}\alpha_{13})(\alpha_{3}\alpha_{22})
(\alpha_{4}\alpha_{9})(\alpha_{7}\alpha_{18})
(\alpha_{8}\alpha_{14})(\alpha_{11}\alpha_{20})
(\alpha_{12}\alpha_{23})(\alpha_{16}\alpha_{19})]
$$
with orbits
$
\{\alpha_{1},\alpha_{5},\alpha_{7},\alpha_{18}\},
\{\alpha_{2},\alpha_{13}\},
\{\alpha_{3},\alpha_{22},\alpha_{17}\},
\{\alpha_{4},\alpha_{9},\alpha_{23},\alpha_{21},\alpha_{6},
\alpha_{11},\alpha_{16},\alpha_{8},\alpha_{19},
\newline
\alpha_{14},\alpha_{12},\alpha_{20}\},
\{\alpha_{10},\alpha_{24}\}$;
$$
H_{51,3}=[
(\alpha_{1}\alpha_{15}\alpha_{18}\alpha_{20})
(\alpha_{3}\alpha_{6}\alpha_{23}\alpha_{17})
(\alpha_{4}\alpha_{19})
(\alpha_{7}\alpha_{21}\alpha_{24}\alpha_{11})
(\alpha_{8}\alpha_{9}\alpha_{14}\alpha_{16})
(\alpha_{12}\alpha_{22}),
$$
$$
(\alpha_{2}\alpha_{13})(\alpha_{3}\alpha_{22})
(\alpha_{4}\alpha_{9})(\alpha_{7}\alpha_{18})
(\alpha_{8}\alpha_{14})(\alpha_{11}\alpha_{20})
(\alpha_{12}\alpha_{23})(\alpha_{16}\alpha_{19})]
$$
with orbits
$
\{\alpha_{1},\alpha_{15},\alpha_{21},\alpha_{24},
\alpha_{18},\alpha_{7},\alpha_{20},\alpha_{11}\},
\{\alpha_{2},\alpha_{13}\},
\{\alpha_{3},\alpha_{12},\alpha_{6},\alpha_{23},\alpha_{22},\alpha_{17}\},
\{\alpha_{4},\alpha_{9},
\newline
\alpha_{19},\alpha_{16},\alpha_{14},\alpha_{8}\}$;
$$
H_{51,4}=[
(\alpha_{1}\alpha_{17}\alpha_{16}\alpha_{19})
(\alpha_{4}\alpha_{21}\alpha_{13}\alpha_{10})
(\alpha_{5}\alpha_{15}\alpha_{8}\alpha_{14})
(\alpha_{6}\alpha_{24})
(\alpha_{7}\alpha_{20}\alpha_{18}\alpha_{12})
(\alpha_{11}\alpha_{23}),
$$
$$
(\alpha_{2}\alpha_{13})(\alpha_{3}\alpha_{22})
(\alpha_{4}\alpha_{9})(\alpha_{7}\alpha_{18})
(\alpha_{8}\alpha_{14})(\alpha_{11}\alpha_{20})
(\alpha_{12}\alpha_{23})(\alpha_{16}\alpha_{19})]
$$
with orbits
$
\{\alpha_{1},\alpha_{17},\alpha_{16},\alpha_{19}\},
\{\alpha_{2},\alpha_{13},\alpha_{9},\alpha_{4},\alpha_{21},\alpha_{10}\},
\{\alpha_{3},\alpha_{22}\},
\{\alpha_{5},\alpha_{15},\alpha_{8},\alpha_{14}\},
\{\alpha_{6},\alpha_{24}\},
\newline
\{\alpha_{7},\alpha_{18},\alpha_{23},\alpha_{11},\alpha_{12},\alpha_{20}\}$.

\medskip

{\bf n=50,} $H\cong 4^2C_3$ ($|H|=48$, $i=3$):
$$
H_{50,1}=[
(\alpha_{3}\alpha_{4}\alpha_{22})(\alpha_{5}\alpha_{20}\alpha_{23})
(\alpha_{6}\alpha_{10}\alpha_{16})(\alpha_{7}\alpha_{9}\alpha_{18})
(\alpha_{11}\alpha_{17}\alpha_{12})(\alpha_{14}\alpha_{24}\alpha_{21}),
$$
$$
(\alpha_{1}\alpha_{3}\alpha_{7}\alpha_{23})
(\alpha_{4}\alpha_{20}\alpha_{6}\alpha_{10})
(\alpha_{5}\alpha_{14}\alpha_{21}\alpha_{22})
(\alpha_{9}\alpha_{24}\alpha_{18}\alpha_{16})
(\alpha_{11}\alpha_{13})(\alpha_{12}\alpha_{17})]
$$
has $Clos(H_{50,1})=H_{75,1}$ above.

\medskip

{\bf n=49,} $H\cong 2^4 C_3$ ($|H|=48$, $i=50$):
$\rk N_H=17$, $(N_H)^\ast/N_H \cong
\bz/24\bz\times (\bz/2\bz)^4$ and
$\det(K((q_{N_H})_2))\equiv \pm 2^7\cdot 3\mod (\bz_2^\ast)^2$.
$$
H_{49,1}=[
(\alpha_{1}\alpha_{22}\alpha_{19})
(\alpha_{3}\alpha_{16}\alpha_{17})
(\alpha_{4}\alpha_{20}\alpha_{9})
(\alpha_{7}\alpha_{10}\alpha_{8})
(\alpha_{12}\alpha_{13}\alpha_{23})
(\alpha_{14}\alpha_{18}\alpha_{21}),
$$
$$
(\alpha_{2}\alpha_{12})(\alpha_{3}\alpha_{8})
(\alpha_{4}\alpha_{20})(\alpha_{7}\alpha_{16})
(\alpha_{9}\alpha_{11})(\alpha_{13}\alpha_{23})
(\alpha_{14}\alpha_{22})(\alpha_{18}\alpha_{19}),
$$
$$
(\alpha_{2}\alpha_{13})(\alpha_{3}\alpha_{22})
(\alpha_{4}\alpha_{9})(\alpha_{7}\alpha_{18})
(\alpha_{8}\alpha_{14})(\alpha_{11}\alpha_{20})
(\alpha_{2}\alpha_{23})(\alpha_{16}\alpha_{19})]
$$
with orbits
$
\{\alpha_{1},\alpha_{22},\alpha_{21},\alpha_{10},\alpha_{19},\alpha_{14},
\alpha_{8},\alpha_{17},\alpha_{18},\alpha_{7},\alpha_{3},\alpha_{16}\},
\{\alpha_{2},\alpha_{12},\alpha_{13},\alpha_{23}\}, \newline
\{\alpha_{4},\alpha_{20},\alpha_{11},\alpha_{9}\}$;
$$
H_{49,2}=[
(\alpha_{1}\alpha_{21}\alpha_{6})
(\alpha_{2}\alpha_{14}\alpha_{4})
(\alpha_{3}\alpha_{8}\alpha_{7})
(\alpha_{9}\alpha_{22}\alpha_{23})
(\alpha_{11}\alpha_{19}\alpha_{20})
(\alpha_{12}\alpha_{18}\alpha_{13}),
$$
$$
(\alpha_{2}\alpha_{12})(\alpha_{3}\alpha_{8})
(\alpha_{4}\alpha_{20})(\alpha_{7}\alpha_{16})
(\alpha_{9}\alpha_{11})(\alpha_{13}\alpha_{23})
(\alpha_{14}\alpha_{22})(\alpha_{18}\alpha_{19}),
$$
$$
(\alpha_{2}\alpha_{13})(\alpha_{3}\alpha_{22})
(\alpha_{4}\alpha_{9})(\alpha_{7}\alpha_{18})
(\alpha_{8}\alpha_{14})(\alpha_{11}\alpha_{20})
(\alpha_{2}\alpha_{23})(\alpha_{16}\alpha_{19})]
$$
with orbits
$
\{\alpha_{1},\alpha_{21},\alpha_{6}\},
\{\alpha_{2},\alpha_{14},\alpha_{12},\alpha_{11},\alpha_{23},
\alpha_{19},\alpha_{4},\alpha_{22},\alpha_{18}, \alpha_{3},
\alpha_{9},\alpha_{13},\alpha_{7},
\newline
\alpha_{20},\alpha_{8},\alpha_{16}\}$;
$$
H_{49,3}=[
(\alpha_{1}\alpha_{6}\alpha_{21})
(\alpha_{3}\alpha_{13}\alpha_{22})
(\alpha_{4}\alpha_{18}\alpha_{7})
(\alpha_{8}\alpha_{23}\alpha_{14})
(\alpha_{10}\alpha_{24}\alpha_{17})
(\alpha_{16}\alpha_{20}\alpha_{19}),
$$
$$
(\alpha_{2}\alpha_{22})(\alpha_{3}\alpha_{13})
(\alpha_{4}\alpha_{18})(\alpha_{7}\alpha_{9})
(\alpha_{10}\alpha_{24})(\alpha_{11}\alpha_{20})
(\alpha_{15}\alpha_{17})(\alpha_{16}\alpha_{19}),
$$
$$
(\alpha_{2}\alpha_{13})(\alpha_{3}\alpha_{22})
(\alpha_{4}\alpha_{9})(\alpha_{7}\alpha_{18})
(\alpha_{8}\alpha_{14})(\alpha_{11}\alpha_{20})
(\alpha_{2}\alpha_{23})(\alpha_{16}\alpha_{19})]
$$
with orbits
$
\{\alpha_{1},\alpha_{6},\alpha_{21}\},
\{\alpha_{2},\alpha_{22},\alpha_{3},\alpha_{13}\},
\{\alpha_{4},\alpha_{18},\alpha_{9},\alpha_{7}\},
\{\alpha_{8},\alpha_{23},\alpha_{14},\alpha_{12}\},
\{\alpha_{10},\alpha_{24},
\newline
\alpha_{15},\alpha_{17}\},
\{\alpha_{11},\alpha_{20},\alpha_{19},\alpha_{16}\}$.

\medskip

{\bf n=48,} $H\cong {\frak S}_{3,3}$ ($|H|=36$, $i=10$):
$\rk N_H=18$, $(N_H)^\ast/N_H \cong
\bz/18\bz\times
\bz/6\bz\times (\bz/3\bz)^2$ and
$\det(K((q_{N_H})_3))\equiv -2^2\cdot 3^5\mod (\bz_3^\ast)^2$.
$$
H_{48,1}=[
(\alpha_{1}\alpha_{19})
(\alpha_{2}\alpha_{22}\alpha_{10}\alpha_{18}\alpha_{3}\alpha_{23})
(\alpha_{4}\alpha_{11}\alpha_{9})
(\alpha_{5}\alpha_{12})
(\alpha_{6}\alpha_{14}\alpha_{24}\alpha_{17}\alpha_{20}\alpha_{16})
(\alpha_{7}\alpha_{13}\alpha_{21}),
$$
$$
(\alpha_{1}\alpha_{19})
(\alpha_{2}\alpha_{18}\alpha_{3}\alpha_{23}\alpha_{10}\alpha_{22})
(\alpha_{4}\alpha_{20}\alpha_{14})
(\alpha_{5}\alpha_{12}\alpha_{15})
(\alpha_{6}\alpha_{16}\alpha_{9}\alpha_{24}\alpha_{17}\alpha_{11})
(\alpha_{7}\alpha_{21})]
$$
with orbits
$
\{\alpha_{1},\alpha_{19}\},
\{\alpha_{2},\alpha_{18},\alpha_{23},\alpha_{3},\alpha_{10},\alpha_{22}\},
\{\alpha_{4},\alpha_{11},\alpha_{20},\alpha_{9},\alpha_{24},\alpha_{14},\alpha_{6},
\alpha_{16},\alpha_{17}\},
\newline
\{\alpha_{5},\alpha_{12},\alpha_{15}\},
\{\alpha_{7},\alpha_{21},\alpha_{13}\}$.

\medskip

{\bf n=47,} $H\cong C_3\times {\frak A}_4$ ($|H|=36$, $i=11$):
$$
H_{47,1}=[
(\alpha_{2}\alpha_{8}\alpha_{9})(\alpha_{3}\alpha_{17}\alpha_{22})
(\alpha_{4}\alpha_{14}\alpha_{13})(\alpha_{7}\alpha_{15}\alpha_{18})
(\alpha_{10}\alpha_{23}\alpha_{12})(\alpha_{11}\alpha_{21}\alpha_{20}),
$$
$$
(\alpha_{1}\alpha_{10})
(\alpha_{2}\alpha_{9}\alpha_{4}\alpha_{13}\alpha_{11}\alpha_{20})
(\alpha_{3}\alpha_{8}\alpha_{17}\alpha_{14}\alpha_{22}\alpha_{21})
(\alpha_{7}\alpha_{15}\alpha_{18})
(\alpha_{12}\alpha_{23})
(\alpha_{16}\alpha_{19}\alpha_{24})]
$$
has $Clos(H_{47,1})=H_{61,1}$ above.

\medskip

{\bf n=46,} $H\cong 3^2 C_4$ ($|H|=36$, $i=9$):
$\rk N_H=18$ and $(N_H)^\ast/N_H\cong
\bz/18\bz\times \bz/6\bz\times \bz/3\bz$.
$$
H_{46,1}=[
(\alpha_{1}\alpha_{8}\alpha_{9}\alpha_{6})
(\alpha_{3}\alpha_{21}\alpha_{22}\alpha_{19})
(\alpha_{4}\alpha_{14})(\alpha_{5}\alpha_{16})
(\alpha_{7}\alpha_{20}\alpha_{11}\alpha_{23})
(\alpha_{12}\alpha_{17}\alpha_{18}\alpha_{13})
$$
$$
(\alpha_{2}\alpha_{13})(\alpha_{3}\alpha_{22})
(\alpha_{4}\alpha_{9})(\alpha_{7}\alpha_{18})
(\alpha_{8}\alpha_{14})(\alpha_{11}\alpha_{20})
(\alpha_{12}\alpha_{23})(\alpha_{16}\alpha_{19})]
$$
with orbits
$
\{\alpha_{1},\alpha_{8},\alpha_{9},\alpha_{6},\alpha_{4},\alpha_{14}\},
\{\alpha_{2},\alpha_{20},\alpha_{18},\alpha_{11},\alpha_{23},\alpha_{17},
\alpha_{13},\alpha_{12},\alpha_{7}\},
\{\alpha_{3},\alpha_{21},\alpha_{22},
\newline
\alpha_{5},\alpha_{19},\alpha_{16}\}.
$

\medskip

{\bf n=45,} $H\cong \Gamma_{6}a_2$ ($|H|=32$, $i=44$):
$$
H_{45,1}=[
(\alpha_{4}\alpha_{7})(\alpha_{5}\alpha_{16})
(\alpha_{6}\alpha_{14})(\alpha_{8}\alpha_{11})
(\alpha_{9}\alpha_{13})(\alpha_{12}\alpha_{24})
(\alpha_{15}\alpha_{21})(\alpha_{17}\alpha_{20}),
$$
$$
(\alpha_{1}\alpha_{8}\alpha_{13}\alpha_{11}\alpha_{19}\alpha_{21}\alpha_{9}\alpha_{15})
(\alpha_{4}\alpha_{20}\alpha_{23}\alpha_{6}\alpha_{7}\alpha_{14}\alpha_{10}\alpha_{17})
(\alpha_{5}\alpha_{16})
(\alpha_{12}\alpha_{18}\alpha_{24}\alpha_{22}),
$$
$$
(\alpha_{1}\alpha_{6}\alpha_{13}\alpha_{14}\alpha_{19}\alpha_{17}\alpha_{9}\alpha_{20})
(\alpha_{4}\alpha_{11}\alpha_{23}\alpha_{21}\alpha_{7}\alpha_{15}\alpha_{10}\alpha_{8})
(\alpha_{5}\alpha_{16})
(\alpha_{12}\alpha_{22}\alpha_{24}\alpha_{18})]
$$
has $Clos(H_{45,1})=H_{80,1}$ above.

\medskip

{\bf n=44,} $H\cong \Gamma_{3}e$  ($|H|=32$, $i=11$):
$$
H_{44,1}=[
(\alpha_{1}\alpha_{8}\alpha_{15}\alpha_{23}\alpha_{4}\alpha_{14}\alpha_{6}\alpha_{13})
(\alpha_{5}\alpha_{16})
(\alpha_{7}\alpha_{21}\alpha_{17}\alpha_{10}\alpha_{19}\alpha_{20}\alpha_{11}\alpha_{9})
(\alpha_{12}\alpha_{24}\alpha_{22}\alpha_{18}),
$$
$$
(\alpha_{5}\alpha_{16})(\alpha_{7}\alpha_{19})
(\alpha_{8}\alpha_{9})(\alpha_{10}\alpha_{14})
(\alpha_{11}\alpha_{17})(\alpha_{12}\alpha_{22})
(\alpha_{13}\alpha_{20})(\alpha_{21}\alpha_{23})]
$$
has $Clos(H_{44,1})=H_{80,1}$ above.

\medskip

{\bf n=43,} $H\cong \Gamma_{7}a_2$ ($|H|=32$, $i=7$):
$$
H_{43,1}=[
(\alpha_{1}\alpha_{6}\alpha_{23}\alpha_{8}\alpha_{19}\alpha_{17}\alpha_{10}\alpha_{21})
(\alpha_{4}\alpha_{11}\alpha_{13}\alpha_{20}\alpha_{7}\alpha_{15}\alpha_{9}\alpha_{14})
(\alpha_{5}\alpha_{16})
(\alpha_{12}\alpha_{18}\alpha_{24}\alpha_{22}),
$$
$$
(\alpha_{1}\alpha_{6}\alpha_{13}\alpha_{14}\alpha_{19}\alpha_{17}\alpha_{9}\alpha_{20})
(\alpha_{4}\alpha_{11}\alpha_{23}\alpha_{21}\alpha_{7}\alpha_{15}\alpha_{10}\alpha_{8})
(\alpha_{5}\alpha_{16})
(\alpha_{12}\alpha_{22}\alpha_{24}\alpha_{18})]
$$
has $Clos(H_{43,1})=H_{80,1}$ above.

\medskip

{\bf n=42,} $H\cong \Gamma_4c_2$ {$|H|=32$, $i=31$):
$$
H_{42,1}=[
(\alpha_{3}\alpha_{24})(\alpha_{4}\alpha_{6})(\alpha_{5}\alpha_{21})
(\alpha_{10}\alpha_{14})(\alpha_{11}\alpha_{17})(\alpha_{12}\alpha_{13})
(\alpha_{16}\alpha_{23})(\alpha_{20}\alpha_{22}),
$$
$$
(\alpha_{1}\alpha_{4}\alpha_{9}\alpha_{5})
(\alpha_{3}\alpha_{20}\alpha_{24}\alpha_{14})
(\alpha_{6}\alpha_{18}\alpha_{21}\alpha_{7})
(\alpha_{10}\alpha_{16}\alpha_{22}\alpha_{23})
(\alpha_{11}\alpha_{12})(\alpha_{13}\alpha_{17}),
$$
$$
(\alpha_{1}\alpha_{3}\alpha_{7}\alpha_{23})
(\alpha_{4}\alpha_{20}\alpha_{6}\alpha_{10})
(\alpha_{5}\alpha_{14}\alpha_{21}\alpha_{22})
(\alpha_{9}\alpha_{24}\alpha_{18}\alpha_{16})
(\alpha_{11}\alpha_{13})(\alpha_{12}\alpha_{17})]
$$
has $Clos(H_{42,1})=H_{75,1}$ above.

\medskip

{\bf n=41,} $H\cong \Gamma_7 a_1$, ($|H|=32$, $i=6$):
$$
H_{41,1,1}=
[(\alpha_{2}\alpha_{17})(\alpha_{5}\alpha_{23}\alpha_{14}\alpha_{24})
(\alpha_{6}\alpha_{8}\alpha_{9}\alpha_{18})(\alpha_{7}\alpha_{16}\alpha_{10}\alpha_{13})
(\alpha_{12}\alpha_{22})(\alpha_{15}\alpha_{20}\alpha_{21}\alpha_{19}),
$$
$$
(\alpha_{1}\alpha_{5})(\alpha_{2}\alpha_{23})(\alpha_{3}\alpha_{14})
(\alpha_{7}\alpha_{9})(\alpha_{8}\alpha_{13})(\alpha_{12}\alpha_{22})
(\alpha_{17}\alpha_{24})(\alpha_{19}\alpha_{20}) ]
$$
with $Clos(H_{41,1,1})=H_{56,1}$ above;
$$
H_{41,1,2}=
[ (\alpha_{2}\alpha_{17})(\alpha_{5}\alpha_{23}\alpha_{14}\alpha_{24})
(\alpha_{6}\alpha_{8}\alpha_{9}\alpha_{18})(\alpha_{7}\alpha_{16}\alpha_{10}\alpha_{13})
(\alpha_{12}\alpha_{22})(\alpha_{15}\alpha_{20}\alpha_{21}\alpha_{19}),
$$
$$
(\alpha_{1}\alpha_{5})(\alpha_{2}\alpha_{24})(\alpha_{3}\alpha_{14})
(\alpha_{6}\alpha_{16})(\alpha_{7}\alpha_{13})(\alpha_{8}\alpha_{9})
(\alpha_{10}\alpha_{18})(\alpha_{17}\alpha_{23}) ]
$$
with $Clos(H_{41,1,2})=H_{56,1}$ above;
$$
H_{41,2}=
[ (\alpha_{2}\alpha_{3}\alpha_{17}\alpha_{11})
(\alpha_{5}\alpha_{16}\alpha_{15}\alpha_{6})
(\alpha_{7}\alpha_{10}\alpha_{19}\alpha_{23})
(\alpha_{8}\alpha_{24}\alpha_{20}\alpha_{18})
(\alpha_{9}\alpha_{13})(\alpha_{12}\alpha_{22}),
$$
$$
(\alpha_{1}\alpha_{14})(\alpha_{3}\alpha_{7})(\alpha_{5}\alpha_{8})
(\alpha_{9}\alpha_{13})(\alpha_{10}\alpha_{24})
(\alpha_{11}\alpha_{20})(\alpha_{15}\alpha_{19})(\alpha_{16}\alpha_{17}) ]
$$
with $Clos(H_{41,2})=H_{56,2}$ above.

\medskip

{\bf n=40,} $H\cong Q_8*Q_8$ ($|H|=32$, $i=49$):
$\rk N_H=17$, $(N_H)^\ast/N_H\cong
(\bz/4\bz)^5$ and
$\det(K((q_{N_H})_2))\equiv \pm 2^{10}\mod (\bz_2^\ast)^2$.
$$
H_{40,1}=
[(\alpha_{1}\alpha_{16})(\alpha_{2}\alpha_{18})
(\alpha_{4}\alpha_{9})(\alpha_{5}\alpha_{23})
(\alpha_{10}\alpha_{12})(\alpha_{14}\alpha_{20})
(\alpha_{15}\alpha_{21})(\alpha_{19}\alpha_{24}),
$$
$$
(\alpha_{3}\alpha_{5})(\alpha_{6}\alpha_{12})
(\alpha_{8}\alpha_{13})(\alpha_{10}\alpha_{22})
(\alpha_{15}\alpha_{21})(\alpha_{16}\alpha_{19})
(\alpha_{17}\alpha_{23})(\alpha_{18}\alpha_{20}),
$$
$$
(\alpha_{1}\alpha_{16})(\alpha_{2}\alpha_{20})
(\alpha_{3}\alpha_{6})(\alpha_{5}\alpha_{10})
(\alpha_{12}\alpha_{23})(\alpha_{14}\alpha_{18})
(\alpha_{17}\alpha_{22})(\alpha_{19}\alpha_{24}),
$$
$$
(\alpha_{1}\alpha_{14})(\alpha_{2}\alpha_{24})
(\alpha_{3}\alpha_{5})(\alpha_{6}\alpha_{10})
(\alpha_{12}\alpha_{22})(\alpha_{16}\alpha_{18})
(\alpha_{17}\alpha_{23})(\alpha_{19}\alpha_{20})]
$$
with orbits
$
\{\alpha_{1},\alpha_{16},\alpha_{14},\alpha_{19},\alpha_{18},\alpha_{20},
\alpha_{24},\alpha_{2}\},
\{\alpha_{3},\alpha_{5},\alpha_{6},\alpha_{23},\alpha_{10},\alpha_{12},
\alpha_{17},\alpha_{22}\},
\{\alpha_{4},\alpha_{9}\},
\newline
\{\alpha_{8},\alpha_{13}\},\{\alpha_{15},\alpha_{21}\}$.

\medskip

{\bf n=39,} $H\cong 2^4C_2$ ($|H|=32$, $i=27$):
$\rk N_H=17$, $(N_H)^\ast/N_H\cong
\bz/8\bz\times (\bz/4\bz)^2\times (\bz/2\bz)^2$ and
$\det(K((q_{N_H})_2))\equiv \pm 2^9\mod (\bz_2^\ast)^2$.
$$
H_{39,1}=
(\alpha_{5}\alpha_{10})(\alpha_{7}\alpha_{13})
(\alpha_{8}\alpha_{12})(\alpha_{9}\alpha_{22})
(\alpha_{11}\alpha_{23})(\alpha_{14}\alpha_{16})
(\alpha_{17}\alpha_{21})(\alpha_{19}\alpha_{20}),
$$
$$
(\alpha_{2}\alpha_{12})(\alpha_{3}\alpha_{8})
(\alpha_{4}\alpha_{20})(\alpha_{7}\alpha_{16})
(\alpha_{9}\alpha_{11})(\alpha_{13}\alpha_{23})
(\alpha_{14}\alpha_{22})(\alpha_{18}\alpha_{19}),
$$
$$
(\alpha_{2}\alpha_{13})(\alpha_{3}\alpha_{22})
(\alpha_{4}\alpha_{9})(\alpha_{7}\alpha_{18})
(\alpha_{8}\alpha_{14})(\alpha_{11}\alpha_{20})
(\alpha_{12}\alpha_{23})(\alpha_{16}\alpha_{19})]
$$
with orbits
$
\{\alpha_{2},\alpha_{12},\alpha_{13},\alpha_{3},\alpha_{18},\alpha_{8},
\alpha_{23},\alpha_{19},\alpha_{7},\alpha_{22},\alpha_{4},
\alpha_{14},\alpha_{20},\alpha_{11},\alpha_{16},\alpha_{9}\},
\{\alpha_{5},\alpha_{10}\}, \newline
\{\alpha_{17},\alpha_{21}\}$;
$$
H_{39,2}=
(\alpha_{5}\alpha_{18})(\alpha_{6}\alpha_{7})
(\alpha_{10}\alpha_{17})(\alpha_{11}\alpha_{20})
(\alpha_{12}\alpha_{13})(\alpha_{14}\alpha_{22})
(\alpha_{15}\alpha_{16})(\alpha_{19}\alpha_{24}),
$$
$$
(\alpha_{2}\alpha_{12})(\alpha_{3}\alpha_{8})
(\alpha_{4}\alpha_{20})(\alpha_{7}\alpha_{16})
(\alpha_{9}\alpha_{11})(\alpha_{13}\alpha_{23})
(\alpha_{14}\alpha_{22})(\alpha_{18}\alpha_{19}),
$$
$$
(\alpha_{2}\alpha_{13})(\alpha_{3}\alpha_{22})
(\alpha_{4}\alpha_{9})(\alpha_{7}\alpha_{18})
(\alpha_{8}\alpha_{14})(\alpha_{11}\alpha_{20})
(\alpha_{12}\alpha_{23})(\alpha_{16}\alpha_{19})]
$$
with orbits
$
\{\alpha_{2},\alpha_{12},\alpha_{13},\alpha_{23}\},
\{\alpha_{3},\alpha_{8},\alpha_{22},\alpha_{14}\},
\{\alpha_{4},\alpha_{20},\alpha_{9},\alpha_{11}\},
\{\alpha_{5},\alpha_{18},\alpha_{24},\alpha_{6},\alpha_{19},
\newline
\alpha_{7},\alpha_{15},\alpha_{16}\},\{\alpha_{10},\alpha_{17}\}$;
$$
H_{39,3}=
(\alpha_{3}\alpha_{4})(\alpha_{5}\alpha_{18})
(\alpha_{6}\alpha_{19})(\alpha_{7}\alpha_{24})
(\alpha_{8}\alpha_{9})(\alpha_{11}\alpha_{22})
(\alpha_{14}\alpha_{20})(\alpha_{15}\alpha_{16}),
$$
$$
(\alpha_{2}\alpha_{12})(\alpha_{3}\alpha_{8})
(\alpha_{4}\alpha_{20})(\alpha_{7}\alpha_{16})
(\alpha_{9}\alpha_{11})(\alpha_{13}\alpha_{23})
(\alpha_{14}\alpha_{22})(\alpha_{18}\alpha_{19}),
$$
$$
(\alpha_{2}\alpha_{13})(\alpha_{3}\alpha_{22})
(\alpha_{4}\alpha_{9})(\alpha_{7}\alpha_{18})
(\alpha_{8}\alpha_{14})(\alpha_{11}\alpha_{20})
(\alpha_{12}\alpha_{23})(\alpha_{16}\alpha_{19})]
$$
with orbits
$
\{\alpha_{2},\alpha_{12},\alpha_{13},\alpha_{23}\},\{\alpha_{3},\alpha_{4},
\alpha_{8},\alpha_{22},\alpha_{14},\alpha_{20},\alpha_{9},\alpha_{11}\},
\{\alpha_{5},\alpha_{18},\alpha_{6},\alpha_{24},
\alpha_{19},\alpha_{7}, \newline
\alpha_{15},\alpha_{16}\}$.

\medskip

{\bf n=38,} $H\cong T_{24}$ ($|H|=24$, $i=3$):
$$
H_{38,1}=[
(\alpha_{2}\alpha_{7}\alpha_{22})
(\alpha_{3}\alpha_{18}\alpha_{13})
(\alpha_{5}\alpha_{12}\alpha_{23})
(\alpha_{6}\alpha_{14}\alpha_{8})
(\alpha_{11}\alpha_{17}\alpha_{20})
(\alpha_{16}\alpha_{19}\alpha_{24}),
$$
$$
(\alpha_{1}\alpha_{2}\alpha_{10}\alpha_{18})
(\alpha_{3}\alpha_{13}\alpha_{22}\alpha_{7})
(\alpha_{4}\alpha_{16}\alpha_{9}\alpha_{12})
(\alpha_{5}\alpha_{23}\alpha_{24}\alpha_{19})
(\alpha_{11}\alpha_{17})(\alpha_{20}\alpha_{21})]
$$
has $Clos(H_{38,1})=H_{54,1}$ above.

\medskip

{\bf n=37,} $H\cong T_{24}$, ($|H|=24$, $i=3$):
$$
H_{37,1}=[
(\alpha_{4}\alpha_{16}\alpha_{20})
(\alpha_{5}\alpha_{6}\alpha_{13})
(\alpha_{7}\alpha_{12}\alpha_{14})
(\alpha_{8}\alpha_{11}\alpha_{15})
(\alpha_{9}\alpha_{24}\alpha_{17})
(\alpha_{10}\alpha_{19}\alpha_{23}),
$$
$$
(\alpha_{1}\alpha_{8}\alpha_{21}\alpha_{10})
(\alpha_{3}\alpha_{7}\alpha_{22}\alpha_{9})
(\alpha_{4}\alpha_{13})(\alpha_{6}\alpha_{20})
(\alpha_{11}\alpha_{23}\alpha_{19}\alpha_{15})
(\alpha_{12}\alpha_{17}\alpha_{24}\alpha_{14})]
$$
has $Clos(H_{37,1})=H_{77,1}$ above.

\medskip

{\bf n=36,} $H\cong C_3\rtimes D_8$ ($|H|=24$, $i=8$):
$$
H_{36,1}=[
(\alpha_{3}\alpha_{22})(\alpha_{4}\alpha_{11})
(\alpha_{7}\alpha_{15})(\alpha_{8}\alpha_{9})
(\alpha_{10}\alpha_{23})(\alpha_{13}\alpha_{21})
(\alpha_{14}\alpha_{20})(\alpha_{16}\alpha_{24}),
$$
$$
(\alpha_{1}\alpha_{10})
(\alpha_{2}\alpha_{9}\alpha_{4}\alpha_{13}\alpha_{11}\alpha_{20})
(\alpha_{3}\alpha_{8}\alpha_{17}\alpha_{14}\alpha_{22}\alpha_{21})
(\alpha_{7}\alpha_{15}\alpha_{18})
(\alpha_{12}\alpha_{23})
(\alpha_{16}\alpha_{19}\alpha_{24})]
$$
has $Clos(H_{36,1})=H_{61,1}$ above.

\medskip

{\bf n=35,} $H\cong C_2\times {\frak A}_4$, ($|H|=24$, $i=13$):
$$
H_{35,1}=[
(\alpha_{1}\alpha_{12})(\alpha_{4}\alpha_{11})
(\alpha_{5}\alpha_{8})(\alpha_{6}\alpha_{14})
(\alpha_{7}\alpha_{16})(\alpha_{9}\alpha_{19})
(\alpha_{18}\alpha_{20})(\alpha_{21}\alpha_{23}),
$$
$$
(\alpha_{1}\alpha_{5})
(\alpha_{2}\alpha_{10}\alpha_{13})
(\alpha_{3}\alpha_{15}\alpha_{22})
(\alpha_{4}\alpha_{20}\alpha_{11}\alpha_{9}\alpha_{16}\alpha_{19})
(\alpha_{6}\alpha_{12}\alpha_{14}\alpha_{21}\alpha_{8}\alpha_{23})
(\alpha_{7}\alpha_{18})]
$$
has $Clos(H_{35,1})=H_{51,1}$ above;
$$
H_{35,2}=[
(\alpha_{1}\alpha_{5})(\alpha_{4}\alpha_{9})
(\alpha_{6}\alpha_{21})(\alpha_{7}\alpha_{18})
(\alpha_{8}\alpha_{12})(\alpha_{11}\alpha_{19})
(\alpha_{14}\alpha_{23})(\alpha_{16}\alpha_{20}),
$$
$$
(\alpha_{2}\alpha_{13})
(\alpha_{3}\alpha_{17}\alpha_{22})
(\alpha_{4}\alpha_{20}\alpha_{11}\alpha_{9}\alpha_{12}\alpha_{23})
(\alpha_{5}\alpha_{18}\alpha_{7})
(\alpha_{6}\alpha_{8}\alpha_{19}\alpha_{21}\alpha_{16}\alpha_{14})
(\alpha_{10}\alpha_{24})]
$$
has $Clos(H_{35,2})=H_{51,2}$ above;
$$
H_{35,3}=[
(\alpha_{1}\alpha_{7})(\alpha_{4}\alpha_{19})
(\alpha_{6}\alpha_{17})(\alpha_{8}\alpha_{14})
(\alpha_{11}\alpha_{15})(\alpha_{12}\alpha_{22})
(\alpha_{18}\alpha_{24})(\alpha_{20}\alpha_{21}),
$$
$$
(\alpha_{1}\alpha_{11}\alpha_{24}\alpha_{21}\alpha_{18}\alpha_{15})
(\alpha_{2}\alpha_{13})
(\alpha_{3}\alpha_{17}\alpha_{12})
(\alpha_{4}\alpha_{16}\alpha_{8},\alpha_{19}\alpha_{9}\alpha_{14})
(\alpha_{6}\alpha_{22}\alpha_{23})(\alpha_{7}\alpha_{20})]
$$
has $Clos(H_{35,3})=H_{51,3}$ above;
$$
H_{35,4}=[
(\alpha_{1}\alpha_{16}\alpha_{17})
(\alpha_{2}\alpha_{4}\alpha_{21})
(\alpha_{5}\alpha_{8}\alpha_{15})
(\alpha_{7}\alpha_{11}\alpha_{20})
(\alpha_{9}\alpha_{13}\alpha_{10})
(\alpha_{12}\alpha_{18}\alpha_{23}),
$$
$$
(\alpha_{2}\alpha_{4}\alpha_{21}\alpha_{9}\alpha_{13}\alpha_{10})
(\alpha_{3}\alpha_{22})(\alpha_{6}\alpha_{24})
(\alpha_{7}\alpha_{23}\alpha_{12}\alpha_{18}\alpha_{11}\alpha_{20})
(\alpha_{8}\alpha_{14}\alpha_{15})(\alpha_{16}\alpha_{19}\alpha_{17})]
$$
has $Clos(H_{35,4})=H_{51,4}$ above.

\medskip

{\bf n=34}, $H\cong {\frak S}_4$ ($|H|=24$, $i=12$):
$\rk N_H=17$ and $(N_H)^\ast/N_H\cong
(\bz/12\bz)^2\times \bz/4\bz$.
$$
H_{34,1}=[
(\alpha_{1}\alpha_{18}\alpha_{15}\alpha_{22})
(\alpha_{2}\alpha_{4}\alpha_{20}\alpha_{11})
(\alpha_{3}\alpha_{7})
(\alpha_{8}\alpha_{10}\alpha_{16}\alpha_{24})
(\alpha_{9}\alpha_{12}\alpha_{23}\alpha_{13})
(\alpha_{14}\alpha_{19}),
$$
$$
(\alpha_{2}\alpha_{13})(\alpha_{3}\alpha_{22})
(\alpha_{4}\alpha_{9})(\alpha_{7}\alpha_{18})
(\alpha_{8}\alpha_{14})(\alpha_{11}\alpha_{20})
(\alpha_{12}\alpha_{23})(\alpha_{16}\alpha_{19})]
$$
with orbits
$
\{\alpha_{1},\alpha_{3},\alpha_{22},\alpha_{15},\alpha_{7},\alpha_{18}\},
\{\alpha_{2},\alpha_{13},\alpha_{20},\alpha_{9},
\alpha_{12},\alpha_{4},\alpha_{23},\alpha_{11}\},
\{\alpha_{8},\alpha_{14},\alpha_{16},\alpha_{10},
\newline
\alpha_{19},\alpha_{24}\}$;
$$
H_{34,2}=[
(\alpha_{1}\alpha_{16}\alpha_{19}\alpha_{15})
(\alpha_{3}\alpha_{5}\alpha_{12}\alpha_{14})
(\alpha_{4}\alpha_{9})
(\alpha_{7}\alpha_{20}\alpha_{23}\alpha_{22})
(\alpha_{8}\alpha_{11}\alpha_{17}\alpha_{18})
(\alpha_{10}\alpha_{13}),
$$
$$
(\alpha_{2}\alpha_{13})(\alpha_{3}\alpha_{22})
(\alpha_{4}\alpha_{9})(\alpha_{7}\alpha_{18})
(\alpha_{8}\alpha_{14})(\alpha_{11}\alpha_{20})
(\alpha_{12}\alpha_{23})(\alpha_{16}\alpha_{19})]
$$
with orbits
$
\{\alpha_{1},\alpha_{16},\alpha_{19},\alpha_{15}\},
\{\alpha_{2},\alpha_{10},\alpha_{13}\},
\{\alpha_{3},\alpha_{22},\alpha_{12},\alpha_{18},\alpha_{17},
\alpha_{20},\alpha_{7},\alpha_{11},\alpha_{23},\alpha_{8},
\newline
\alpha_{5},\alpha_{14}\},\{\alpha_{4},\alpha_{9}\}$;
$$
H_{34,3}=[
(\alpha_{1}\alpha_{6}\alpha_{18}\alpha_{7})
(\alpha_{3}\alpha_{24}\alpha_{20}\alpha_{17})
(\alpha_{4}\alpha_{9})
(\alpha_{8}\alpha_{13}\alpha_{14}\alpha_{19})
(\alpha_{10}\alpha_{15}\alpha_{12}\alpha_{23})
(\alpha_{11}\alpha_{22}),
$$
$$
(\alpha_{2}\alpha_{13})(\alpha_{3}\alpha_{22})
(\alpha_{4}\alpha_{9})(\alpha_{7}\alpha_{18})
(\alpha_{8}\alpha_{14})(\alpha_{11}\alpha_{20})
(\alpha_{12}\alpha_{23})(\alpha_{16}\alpha_{19})]
$$
with orbits
$
\{\alpha_{1},\alpha_{18},\alpha_{6},\alpha_{7}\},
\{\alpha_{2},\alpha_{8},\alpha_{14},\alpha_{16},\alpha_{13},\alpha_{19}\},
\{\alpha_{3},\alpha_{22},\alpha_{20},\alpha_{24},\alpha_{11},\alpha_{17}\},
\newline
\{\alpha_{4},\alpha_{9}\},
\{\alpha_{10},\alpha_{12},\alpha_{15},\alpha_{23}\}$;
$$
H_{34,4}=[
(\alpha_{1}\alpha_{6}\alpha_{16}\alpha_{19})
(\alpha_{4}\alpha_{14}\alpha_{23}\alpha_{9})
(\alpha_{5}\alpha_{11}\alpha_{20}\alpha_{21})
(\alpha_{7}\alpha_{8}\alpha_{12}\alpha_{18})
(\alpha_{13}\alpha_{15})(\alpha_{17}\alpha_{22}),
$$
$$
(\alpha_{2}\alpha_{13})(\alpha_{3}\alpha_{22})
(\alpha_{4}\alpha_{9})(\alpha_{7}\alpha_{18})
(\alpha_{8}\alpha_{14})(\alpha_{11}\alpha_{20})
(\alpha_{12}\alpha_{23})(\alpha_{16}\alpha_{19})]
$$
with orbits
$
\{\alpha_{1},\alpha_{16},\alpha_{6},\alpha_{19}\},
\{\alpha_{2},\alpha_{15},\alpha_{13}\},
\{\alpha_{3},\alpha_{17},\alpha_{22}\},
\{\alpha_{4},\alpha_{7},\alpha_{18},\alpha_{23},
\alpha_{9},\alpha_{12},\alpha_{8},\alpha_{14}\},
\newline
\{\alpha_{5},\alpha_{11},\alpha_{20},\alpha_{21}\}$.

\medskip

{\bf n=33,} $H\cong C_7\rtimes C_3$ ($|H|=21$, $i=1$):
$\rk N_H=18$ and $(N_H)^\ast/N_H\cong (\bz/7\bz)^3$.
$$
H_{33,1}=[
(\alpha_{1}\alpha_{6}\alpha_{5})(\alpha_{3}\alpha_{20}\alpha_{17})
(\alpha_{7}\alpha_{19}\alpha_{22})(\alpha_{9}\alpha_{24}\alpha_{16})
(\alpha_{12}\alpha_{21}\alpha_{14})(\alpha_{13}\alpha_{15}\alpha_{18}),
$$
$$
(\alpha_{1}\alpha_{7}\alpha_{6}\alpha_{22}\alpha_{19}\alpha_{10}\alpha_{5})
(\alpha_{3}\alpha_{17}\alpha_{4}\alpha_{9}\alpha_{24}\alpha_{20}\alpha_{16})
(\alpha_{11}\alpha_{12}\alpha_{21}\alpha_{13}\alpha_{14}\alpha_{18}\alpha_{15})]
$$
with orbits
$
\{\alpha_{1},\alpha_{6},\alpha_{7},\alpha_{5},\alpha_{22},\alpha_{19},\alpha_{10}\},
\{\alpha_{3},\alpha_{20},\alpha_{17},\alpha_{16},\alpha_{4},\alpha_{9},\alpha_{24}\},
\{\alpha_{11},\alpha_{12},\alpha_{21},\alpha_{14},\alpha_{13},
\newline
\alpha_{18},\alpha_{15}\}$.

\medskip

{\bf n=32,} $H\cong Hol(C_5)$ ($|H|=20$, $i=3$):
$\rk N_H=18$ and
$(N_H)^\ast/N_H \cong (\bz/10\bz)^2 \times \bz/5\bz$.
$$
H_{32,1}=[
(\alpha_{2}\alpha_{21}\alpha_{17}\alpha_{11})
(\alpha_{3}\alpha_{13}\alpha_{9}\alpha_{14})
(\alpha_{4}\alpha_{18}\alpha_{24}\alpha_{12})
(\alpha_{5}\alpha_{8})
(\alpha_{15}\alpha_{16}\alpha_{23}\alpha_{20})
(\alpha_{19}\alpha_{22}),
$$
$$
(\alpha_{2}\alpha_{17}\alpha_{15}\alpha_{8}\alpha_{23})
(\alpha_{3}\alpha_{14}\alpha_{10}\alpha_{13}\alpha_{9})
(\alpha_{4}\alpha_{24}\alpha_{18}\alpha_{7}\alpha_{12})
(\alpha_{5}\alpha_{11}\alpha_{20}\alpha_{16}\alpha_{21})]
$$
with orbits
$
\{\alpha_{2},\alpha_{21},\alpha_{17},\alpha_{11},\alpha_{5},
\alpha_{15},\alpha_{20},\alpha_{8},\alpha_{16},\alpha_{23}\},
\{\alpha_{3},\alpha_{13},\alpha_{9},\alpha_{14},\alpha_{10}\},
\{\alpha_{4},\alpha_{18},\alpha_{24},
\newline
\alpha_{12},\alpha_{7}\},
\{\alpha_{19},\alpha_{22}\}$.

\medskip

{\bf n=31,} $H\cong C_3\times D_6$ ($|H|=18$, $i=3$):
$$
H_{31,1}=[
(\alpha_{2}\alpha_{3}\alpha_{10})
(\alpha_{4}\alpha_{11}\alpha_{9})
(\alpha_{6}\alpha_{20}\alpha_{24})
(\alpha_{7}\alpha_{13}\alpha_{21})
(\alpha_{14}\alpha_{16}\alpha_{17})
(\alpha_{18}\alpha_{22}\alpha_{23}),
$$
$$
(\alpha_{1}\alpha_{19})
(\alpha_{2}\alpha_{18}\alpha_{10}\alpha_{22}\alpha_{3}\alpha_{23})
(\alpha_{4}\alpha_{17}\alpha_{20}\alpha_{9}\alpha_{14}\alpha_{6})
(\alpha_{5}\alpha_{15}\alpha_{12})
(\alpha_{7}\alpha_{13})
(\alpha_{11}\alpha_{16}\alpha_{24})]
$$
has $Clos(H_{31,1})=H_{48,1}$.

\medskip

{\bf n=30,} $H\cong {\frak A}_{3,3}$ ($|H|=18$, $i=4$):
$\rk N_H=16$ and
$(N_H)^\ast/N_H\cong \bz/9\bz\times (\bz/3\bz)^4$.
$$
H_{30,1}=[
(\alpha_{3}\alpha_{21})(\alpha_{4}\alpha_{9})
(\alpha_{6}\alpha_{22})(\alpha_{7}\alpha_{10})
(\alpha_{11}\alpha_{12})(\alpha_{14}\alpha_{16})
(\alpha_{15}\alpha_{17})(\alpha_{18}\alpha_{24}),
$$
$$
(\alpha_{1}\alpha_{24}\alpha_{18})(\alpha_{3}\alpha_{14}\alpha_{22})
(\alpha_{4}\alpha_{20}\alpha_{9})(\alpha_{5}\alpha_{7}\alpha_{10})
(\alpha_{6}\alpha_{16}\alpha_{21})(\alpha_{15}\alpha_{19}\alpha_{17}),
$$
$$
(\alpha_{1}\alpha_{6}\alpha_{22})(\alpha_{3}\alpha_{24}\alpha_{16})
(\alpha_{4}\alpha_{20}\alpha_{9})(\alpha_{5}\alpha_{10}\alpha_{7})
(\alpha_{11}\alpha_{13}\alpha_{12})(\alpha_{14}\alpha_{18}\alpha_{21})]
$$
with orbits
$
\{\alpha_{1},\alpha_{24},\alpha_{21},\alpha_{18},
\alpha_{6},\alpha_{3},\alpha_{16},\alpha_{22},\alpha_{14}\},
\{\alpha_{4},\alpha_{9},\alpha_{20}\},
\{\alpha_{5},\alpha_{7},\alpha_{10},\}
\newline
\{\alpha_{11},\alpha_{12},\alpha_{13}\},
\{\alpha_{15},\alpha_{17},\alpha_{19}\}$.

\medskip

{\bf n=29,} $H\cong Q_{16}$ ($|H|=16$, $i=9$):
$$
H_{29,1}=[
(\alpha_{1}\alpha_{4}\alpha_{19}\alpha_{7})
(\alpha_{5}\alpha_{16})
(\alpha_{6}\alpha_{8}\alpha_{17}\alpha_{21})
(\alpha_{9}\alpha_{23}\alpha_{13}\alpha_{10})
(\alpha_{11}\alpha_{14}\alpha_{15}\alpha_{20})
(\alpha_{12}\alpha_{24}),
$$
$$
(\alpha_{1}\alpha_{6}\alpha_{13}\alpha_{14}\alpha_{19}
\alpha_{17}\alpha_{9}\alpha_{20})
(\alpha_{4}\alpha_{11}\alpha_{23}\alpha_{21}\alpha_{7}
\alpha_{15}\alpha_{10}\alpha_{8})
(\alpha_{5}\alpha_{16})
(\alpha_{12}\alpha_{22}\alpha_{24}\alpha_{18})]
$$
has $Clos(H_{29,1})=H_{80,1}$ above.

\medskip

{\bf n=28,} $H\cong \Gamma_2 d$ ($|H|=16$, $i=6$):
$$
H_{28,1}=[
(\alpha_{1}\alpha_{6}\alpha_{13}\alpha_{14}
\alpha_{19}\alpha_{17}\alpha_{9}\alpha_{20})
(\alpha_{4}\alpha_{11}\alpha_{23}\alpha_{21}\alpha_{7}
\alpha_{15}\alpha_{10}\alpha_{8})
(\alpha_{5}\alpha_{16})
(\alpha_{12}\alpha_{22}\alpha_{24}\alpha_{18}),
$$
$$
(\alpha_{1}\alpha_{4})(\alpha_{6}\alpha_{15})
(\alpha_{7}\alpha_{19})(\alpha_{8}\alpha_{14})
(\alpha_{9}\alpha_{10})(\alpha_{11}\alpha_{17})
(\alpha_{13}\alpha_{23})(\alpha_{20}\alpha_{21})]
$$
has $Clos(H_{28,1})=H_{80,1}$ above.

\medskip

{\bf n=27,} $H\cong C_2\times Q_8$ ($|H|=16$, $i=12$):
$$
H_{27,1}=
[(\alpha_{1}\alpha_{4}\alpha_{9}\alpha_{5})
(\alpha_{3}\alpha_{20}\alpha_{24}\alpha_{14})
(\alpha_{6}\alpha_{18}\alpha_{21}\alpha_{7})
(\alpha_{10}\alpha_{16}\alpha_{22}\alpha_{23})
(\alpha_{11}\alpha_{12})(\alpha_{13}\alpha_{17}),
$$
$$
(\alpha_{1}\alpha_{3}\alpha_{9}\alpha_{24})
(\alpha_{4}\alpha_{14}\alpha_{5}\alpha_{20})
(\alpha_{6}\alpha_{22}\alpha_{21}\alpha_{10})
(\alpha_{7}\alpha_{23}\alpha_{18}\alpha_{16})
(\alpha_{11}\alpha_{17})(\alpha_{12}\alpha_{13}),
$$
$$
(\alpha_{1}\alpha_{7})(\alpha_{3}\alpha_{23})
(\alpha_{4}\alpha_{6})(\alpha_{5}\alpha_{21})
(\alpha_{9}\alpha_{18})(\alpha_{10}\alpha_{20})
(\alpha_{14}\alpha_{22})(\alpha_{16}\alpha_{24})]
$$
with $Clos(H_{27,1})=H_{75,1}$ above.

\medskip

{\bf n=26,} $H\cong SD_{16}$ ($|H|=16$, $i=8$):
$\rk N_H=18$ and
$(N_H)^\ast/N_H\cong (\bz/8\bz)^2\times
\bz/4\bz\times \bz/2\bz$.
$$
H_{26,1}=[
(\alpha_{1}\alpha_{12}\alpha_{15}\alpha_{11})
(\alpha_{3}\alpha_{5}\alpha_{18}\alpha_{17})
(\alpha_{4}\alpha_{10})(\alpha_{6}\alpha_{9})
(\alpha_{7}\alpha_{14}\alpha_{22}\alpha_{8})
(\alpha_{16}\alpha_{20}\alpha_{19}\alpha_{23}),
$$
$$
(\alpha_{2}\alpha_{13})(\alpha_{3}\alpha_{22})
(\alpha_{4}\alpha_{9})(\alpha_{7}\alpha_{18})
(\alpha_{8}\alpha_{14})(\alpha_{11}\alpha_{20})
(\alpha_{12}\alpha_{23})(\alpha_{16}\alpha_{19})]
$$
with orbits
$
\{\alpha_{1},\alpha_{12},\alpha_{19},\alpha_{15},
\alpha_{23},\alpha_{20},\alpha_{11},\alpha_{16}\},
\{\alpha_{2},\alpha_{13}\},
\{\alpha_{3},\alpha_{5},\alpha_{22},\alpha_{7},
\alpha_{18},\alpha_{8},\alpha_{17},\alpha_{14}\},
\newline
\{\alpha_{4},\alpha_{10},\alpha_{9},\alpha_{6}\}$.

\medskip

{\bf n=25,} $H\cong C_4^2$ ($|H|=16$, $i=2$):
$$
H_{25,1}=[
(\alpha_{1}\alpha_{4}\alpha_{9}\alpha_{5})
(\alpha_{3}\alpha_{20}\alpha_{24}\alpha_{14})
(\alpha_{6}\alpha_{18}\alpha_{21}\alpha_{7})
(\alpha_{10}\alpha_{16}\alpha_{22}\alpha_{23})
(\alpha_{11}\alpha_{12})(\alpha_{13}\alpha_{17}),
$$
$$
(\alpha_{1}\alpha_{3}\alpha_{7}\alpha_{23})
(\alpha_{4}\alpha_{20}\alpha_{6}\alpha_{10})
(\alpha_{5}\alpha_{14}\alpha_{21}\alpha_{22})
(\alpha_{9}\alpha_{24}\alpha_{18}\alpha_{16})
(\alpha_{11}\alpha_{13})(\alpha_{12}\alpha_{17})]
$$
with $Clos(H_{25,1})=H_{75,1}$ above.
\footnote{These calculations show that for a symplectic group
$G=(C_4)^2$ on a
K\"ahlerian K3 surface, the group
$S_{(G)}^\ast/S_{(G)}=S_{(4,4)}^\ast /S_{(4,4)}\cong
(\bz/8\bz)^2\times (\bz/2\bz)^2$.
We must correct our calculation of this group in
\cite[Prop. 10.1]{Nik0}.}

\medskip

{\bf n=24,} $H\cong Q_8 * C_4$ ($|H|=16$, $i=13$):
$$
H_{24,1}=[
(\alpha_{3}\alpha_{5})(\alpha_{6}\alpha_{12})
(\alpha_{8}\alpha_{13})(\alpha_{10}\alpha_{22})
(\alpha_{15}\alpha_{21})(\alpha_{16}\alpha_{19})
(\alpha_{17}\alpha_{23})(\alpha_{18}\alpha_{20}),
$$
$$
(\alpha_{1}\alpha_{16})(\alpha_{2}\alpha_{18})
(\alpha_{4}\alpha_{9})(\alpha_{5}\alpha_{23})
(\alpha_{10}\alpha_{12})(\alpha_{14}\alpha_{20})
(\alpha_{15}\alpha_{21})(\alpha_{19}\alpha_{24}),
$$
$$
(\alpha_{1}\alpha_{2}\alpha_{24}\alpha_{14})
(\alpha_{3}\alpha_{6}\alpha_{17}\alpha_{22})
(\alpha_{4}\alpha_{9})
(\alpha_{5}\alpha_{12}\alpha_{23}\alpha_{10})
(\alpha_{8}\alpha_{13})
(\alpha_{16}\alpha_{18}\alpha_{19}\alpha_{20})
$$
with $Clos(H_{24,1})=H_{40,1}$ above.

\medskip

{\bf n=23,} $H\cong \Gamma_2 c_1$ ($|H|=16$, $i=3$):
$$
H_{23,1}=[
(\alpha_{2}\alpha_{7}\alpha_{18}\alpha_{13})
(\alpha_{3}\alpha_{9}\alpha_{4}\alpha_{22})
(\alpha_{5}\alpha_{10})
(\alpha_{8}\alpha_{16}\alpha_{20}\alpha_{23})
(\alpha_{11}\alpha_{19}\alpha_{14}\alpha_{12})
(\alpha_{17}\alpha_{21}),
$$
$$
(\alpha_{2}\alpha_{8}\alpha_{3}\alpha_{12})
(\alpha_{4}\alpha_{19}\alpha_{18}\alpha_{20})
(\alpha_{5}\alpha_{10})
(\alpha_{7}\alpha_{14}\alpha_{9}\alpha_{23})
(\alpha_{11}\alpha_{22}\alpha_{16}\alpha_{13})
(\alpha_{17}\alpha_{21})]
$$
with $Clos(H_{23,1})=H_{39,1}$ above;
$$
H_{23,2}=[
(\alpha_{3}\alpha_{14}\alpha_{8}\alpha_{22})
(\alpha_{4}\alpha_{20}\alpha_{9}\alpha_{11})
(\alpha_{5}\alpha_{7}\alpha_{15}\alpha_{19})
(\alpha_{6}\alpha_{18}\alpha_{24}\alpha_{16})
(\alpha_{10}\alpha_{17})
(\alpha_{12}\alpha_{13}),
$$
$$
(\alpha_{2}\alpha_{12})(\alpha_{4}\alpha_{11})
(\alpha_{5}\alpha_{15})(\alpha_{6}\alpha_{24})
(\alpha_{7}\alpha_{18})(\alpha_{9}\alpha_{20})
(\alpha_{13}\alpha_{23})(\alpha_{16}\alpha_{19})]
$$
with $Clos(H_{23,2})=H_{39,2}$ above;
$$
H_{23,3}=[
(\alpha_{2}\alpha_{12})
(\alpha_{3}\alpha_{4}\alpha_{22}\alpha_{11})
(\alpha_{5}\alpha_{16}\alpha_{6}\alpha_{7})
(\alpha_{8}\alpha_{9}\alpha_{14}\alpha_{20})
(\alpha_{13}\alpha_{23})
(\alpha_{15}\alpha_{18}\alpha_{24}\alpha_{19}),
$$
$$
(\alpha_{2}\alpha_{13})(\alpha_{4}\alpha_{20})
(\alpha_{5}\alpha_{6})(\alpha_{7}\alpha_{19})
(\alpha_{9}\alpha_{11})(\alpha_{12}\alpha_{23})
(\alpha_{15}\alpha_{24})(\alpha_{16}\alpha_{18})]
$$
with $Clos(H_{23,3})=H_{39,3}$ above.

\medskip

{\bf n=22,} $H\cong C_2\times D_8$ ($|H|=16$, $i=11$):
$\rk N_H=16$, $(N_H)^\ast/N_H\cong
(\bz/4\bz)^4\times (\bz/2\bz)^2$, and
$\det(K(q_{N_H})_2)\equiv \pm 2^{10}\mod (\bz_2^\ast)^2$.
$$
H_{22,1}=[
(\alpha_{4}\alpha_{7})(\alpha_{8}\alpha_{12})
(\alpha_{9}\alpha_{18})(\alpha_{10}\alpha_{15})
(\alpha_{11}\alpha_{20})(\alpha_{14}\alpha_{23})
(\alpha_{16}\alpha_{19})(\alpha_{17}\alpha_{24}),
$$
$$
(\alpha_{2}\alpha_{12})(\alpha_{3}\alpha_{8})
(\alpha_{4}\alpha_{20})(\alpha_{7}\alpha_{16})
(\alpha_{9}\alpha_{11})(\alpha_{13}\alpha_{23})
(\alpha_{14}\alpha_{22})(\alpha_{18}\alpha_{19}),
$$
$$
(\alpha_{2}\alpha_{13})(\alpha_{3}\alpha_{22})
(\alpha_{4}\alpha_{9})(\alpha_{7}\alpha_{18})
(\alpha_{8}\alpha_{14})(\alpha_{11}\alpha_{20})
(\alpha_{12}\alpha_{23})(\alpha_{16}\alpha_{19})]
$$
with orbits
$
\{\alpha_{2},\alpha_{12},\alpha_{13},\alpha_{3},\alpha_{8},
\alpha_{23},\alpha_{22},\alpha_{14}\},
\{\alpha_{4},\alpha_{7},\alpha_{20},\alpha_{9},\alpha_{18},
\alpha_{16},\alpha_{11},\alpha_{19}\},
\newline
\{\alpha_{10},\alpha_{15}\},\{\alpha_{17},\alpha_{24}\}$;
$$
H_{22,2}=[
(\alpha_{3}\alpha_{4})(\alpha_{6}\alpha_{15})
(\alpha_{8}\alpha_{11})(\alpha_{9}\alpha_{22})
(\alpha_{12}\alpha_{23})(\alpha_{14}\alpha_{20})
(\alpha_{16}\alpha_{19})(\alpha_{17}\alpha_{21}),
$$
$$
(\alpha_{2}\alpha_{22})(\alpha_{3}\alpha_{13})
(\alpha_{4}\alpha_{18})(\alpha_{7}\alpha_{9})
(\alpha_{10}\alpha_{24})(\alpha_{11}\alpha_{20})
(\alpha_{15}\alpha_{17})(\alpha_{16}\alpha_{19}),
$$
$$
(\alpha_{2}\alpha_{13})(\alpha_{3}\alpha_{22})
(\alpha_{4}\alpha_{9})(\alpha_{7}\alpha_{18})
(\alpha_{8}\alpha_{14})(\alpha_{11}\alpha_{20})
(\alpha_{12}\alpha_{23})(\alpha_{16}\alpha_{19})]
$$
with orbits
$
\{\alpha_{2},\alpha_{22},\alpha_{13},\alpha_{7},\alpha_{9},
\alpha_{3},\alpha_{18},\alpha_{4}\},
\{\alpha_{6},\alpha_{15},\alpha_{21},\alpha_{17}\}, 
\{\alpha_{8},\alpha_{11},\alpha_{14},\alpha_{20}\},
\{\alpha_{10},\alpha_{24}\}, \newline
\{\alpha_{12},\alpha_{23}\}, \{\alpha_{16},\alpha_{19}\}$;
$$
H_{22,3}=[
(\alpha_{1}\alpha_{10})(\alpha_{3}\alpha_{22})
(\alpha_{4}\alpha_{9})(\alpha_{5}\alpha_{24})
(\alpha_{6}\alpha_{15})(\alpha_{8}\alpha_{14})
(\alpha_{11}\alpha_{20})(\alpha_{17}\alpha_{21}),
$$
$$
(\alpha_{2}\alpha_{22})(\alpha_{3}\alpha_{13})
(\alpha_{4}\alpha_{18})(\alpha_{7}\alpha_{9})
(\alpha_{10}\alpha_{24})(\alpha_{11}\alpha_{20})
(\alpha_{15}\alpha_{17})(\alpha_{16}\alpha_{19}),
$$
$$
(\alpha_{2}\alpha_{13})(\alpha_{3}\alpha_{22})
(\alpha_{4}\alpha_{9})(\alpha_{7}\alpha_{18})
(\alpha_{8}\alpha_{14})(\alpha_{11}\alpha_{20})
(\alpha_{12}\alpha_{23})(\alpha_{16}\alpha_{19})]
$$
with orbits
$
\{\alpha_{1},\alpha_{10},\alpha_{5},\alpha_{24}\},
\{\alpha_{2},\alpha_{22},\alpha_{13},\alpha_{3}\},
\{\alpha_{4},\alpha_{9},\alpha_{18},\alpha_{7}\},
\{\alpha_{6},\alpha_{15},\alpha_{21},\alpha_{17}\},
\{\alpha_{8},\alpha_{14}\}, \newline
\{\alpha_{11},\alpha_{20}\},
\{\alpha_{12},\alpha_{23}\},\{\alpha_{16},\alpha_{19}\}$.

\medskip

{\bf n=21,} $H\cong C_2^4$ ($|H|=16$, $i=14$):
$\rk N_H=15$, $(N_H)^\ast/N_H \cong
\bz/8\bz\times (\bz/2\bz)^6$, and
$\det(K(q_{N_H})_2)\equiv \pm 2^{9}\mod (\bz_2^\ast)^2$.
\footnote{These calculations show that for a symplectic group
$G=(C_2)^4$ on a
K\"ahlerian K3 surface, the group
$S_{(G)}^\ast/S_{(G)}=S_{(2,2,2,2)}^\ast /S_{(2,2,2,2)}\cong
\bz/8\bz\times (\bz/2\bz)^6$.
We must correct our calculation of this group in
\cite[Prop. 10.1]{Nik0}.}
$$
H_{21,1}=
[(\alpha_{2}\alpha_{20})(\alpha_{3}\alpha_{10})
(\alpha_{5}\alpha_{6})(\alpha_{8}\alpha_{11})
(\alpha_{9}\alpha_{21})(\alpha_{12}\alpha_{22})
(\alpha_{17}\alpha_{23})(\alpha_{19}\alpha_{24}),
$$
$$
(\alpha_{2}\alpha_{19})(\alpha_{3}\alpha_{5})
(\alpha_{6}\alpha_{10})(\alpha_{8}\alpha_{9})
(\alpha_{11}\alpha_{21})(\alpha_{12}\alpha_{23})
(\alpha_{17}\alpha_{22})(\alpha_{20}\alpha_{24}),
$$
$$
(\alpha_{1}\alpha_{16})(\alpha_{2}\alpha_{20})
(\alpha_{3}\alpha_{6})(\alpha_{5}\alpha_{10})
(\alpha_{12}\alpha_{23})(\alpha_{14}\alpha_{18})
(\alpha_{17}\alpha_{22})(\alpha_{19}\alpha_{24}),
$$
$$
(\alpha_{1}\alpha_{14})(\alpha_{2}\alpha_{24})
(\alpha_{3}\alpha_{5})(\alpha_{6}\alpha_{10})
(\alpha_{12}\alpha_{22})(\alpha_{16}\alpha_{18})
(\alpha_{17}\alpha_{23})(\alpha_{19}\alpha_{20})]
$$
with orbits {
$
\{\alpha_{1},\alpha_{16},\alpha_{14},\alpha_{18}\},
\{\alpha_{2},\alpha_{20},\alpha_{19},\alpha_{24}\},
\{\alpha_{3},\alpha_{10},\alpha_{5},\alpha_{6}\}, 
\{\alpha_{8},\alpha_{11},\alpha_{9},\alpha_{21}\},
\newline
\{\alpha_{12},\alpha_{22},\alpha_{23},\alpha_{17}\}
$;

$$
H_{21,2}=
[(\alpha_{1}\alpha_{3})(\alpha_{2}\alpha_{23})
(\alpha_{5}\alpha_{14})(\alpha_{6}\alpha_{16})
(\alpha_{10}\alpha_{18})(\alpha_{12}\alpha_{20})
(\alpha_{17}\alpha_{24})(\alpha_{19}\alpha_{22}),
$$
$$
(\alpha_{1}\alpha_{2})(\alpha_{3}\alpha_{23})
(\alpha_{5}\alpha_{17})(\alpha_{6}\alpha_{12})
(\alpha_{10}\alpha_{22})(\alpha_{14}\alpha_{24})
(\alpha_{16}\alpha_{20})(\alpha_{18}\alpha_{19}),
$$
$$
(\alpha_{1}\alpha_{16})(\alpha_{2}\alpha_{20})
(\alpha_{3}\alpha_{6})(\alpha_{5}\alpha_{10})
(\alpha_{12}\alpha_{23})(\alpha_{14}\alpha_{18})
(\alpha_{17}\alpha_{22})(\alpha_{19}\alpha_{24}),
$$
$$
(\alpha_{1}\alpha_{14})(\alpha_{2}\alpha_{24})
(\alpha_{3}\alpha_{5})(\alpha_{6}\alpha_{10})
(\alpha_{12}\alpha_{22})(\alpha_{16}\alpha_{18})
(\alpha_{17}\alpha_{23})(\alpha_{19}\alpha_{20})]
$$
with orbits
$
\{\alpha_{1},\alpha_{3},\alpha_{2},\alpha_{16},\alpha_{14},
\alpha_{23},\alpha_{6},\alpha_{5},
\alpha_{20},\alpha_{24},\alpha_{18},\alpha_{12},\alpha_{17},
\alpha_{10},\alpha_{19},\alpha_{22}\}
$.

\medskip

{\bf n=20,} $H\cong Q_{12}$ ($|H|=12$, $i=1$):
$$
H_{20,1}=[
(\alpha_{1}\alpha_{10}\alpha_{23}\alpha_{12})
(\alpha_{2}\alpha_{3}\alpha_{21}\alpha_{9})
(\alpha_{4}\alpha_{22}\alpha_{8}\alpha_{20})
(\alpha_{7}\alpha_{15})
(\alpha_{11}\alpha_{17}\alpha_{14}\alpha_{13})
(\alpha_{16}\alpha_{19}),
$$
$$
(\alpha_{2}\alpha_{4}\alpha_{11})(\alpha_{3}\alpha_{17}\alpha_{22})
(\alpha_{7}\alpha_{18}\alpha_{15})(\alpha_{8}\alpha_{14}\alpha_{21})
(\alpha_{9}\alpha_{13}\alpha_{20})(\alpha_{16}\alpha_{24}\alpha_{19})]
$$
has $Clos(H_{20,1})=H_{61,1}$ above.

\medskip

{\bf n=19,} $H\cong C_2\times C_6$ ($|H|=12$, $i=5$):
$$
H_{19,1}=[
(\alpha_{1}\alpha_{23})(\alpha_{2}\alpha_{21})
(\alpha_{3}\alpha_{9})(\alpha_{4}\alpha_{8})
(\alpha_{10}\alpha_{12})(\alpha_{11}\alpha_{14})
(\alpha_{13}\alpha_{17})(\alpha_{20}\alpha_{22}),
$$
$$
(\alpha_{1}\alpha_{10})
(\alpha_{2}\alpha_{9}\alpha_{4}\alpha_{13}\alpha_{11}\alpha_{20})
(\alpha_{3}\alpha_{8}\alpha_{17}\alpha_{14}\alpha_{22}\alpha_{21})
(\alpha_{7}\alpha_{15}\alpha_{18})
(\alpha_{12}\alpha_{23})(\alpha_{16}\alpha_{19}\alpha_{24})]
$$
has $Clos(H_{19,1})=H_{61,1}$ above.

\medskip

{\bf n=18,} $H\cong D_{12}$ ($|H|=12$, $i=4$):
$\rk N_H=16$ and $(N_H)^\ast/N_H\cong (\bz/6\bz)^4$.
$$
H_{18,1}=[
(\alpha_{2}\alpha_{17})(\alpha_{3}\alpha_{10})
(\alpha_{4}\alpha_{21})(\alpha_{5}\alpha_{18})
(\alpha_{8}\alpha_{20})(\alpha_{11}\alpha_{23})
(\alpha_{13}\alpha_{22})(\alpha_{15}\alpha_{24}),
$$
$$
(\alpha_{2}\alpha_{13})(\alpha_{3}\alpha_{22})
(\alpha_{4}\alpha_{9})(\alpha_{7}\alpha_{18})
(\alpha_{8}\alpha_{14})(\alpha_{11}\alpha_{20})
(\alpha_{12}\alpha_{23})(\alpha_{16}\alpha_{19})]
$$
with orbits
$
\{\alpha_{2},\alpha_{13},\alpha_{3},\alpha_{22},\alpha_{10},\alpha_{17}\},
\{\alpha_{4},\alpha_{9},\alpha_{21}\},
\{\alpha_{5},\alpha_{18},\alpha_{7}\},
\{\alpha_{8},\alpha_{14},\alpha_{23},\alpha_{12},
\alpha_{11},\alpha_{20}\},
\newline
\{\alpha_{15},\alpha_{24}\},\{\alpha_{16},\alpha_{19},\}$.

\medskip

{\bf n=17,} $H\cong {\frak A}_4$ ($|H|=12$, $i=3$):
$\rk N_H=16$ and $(N_H)^\ast/N_H\cong (\bz/12\bz)^2\times (\bz/2\bz)^2$.
$$
H_{17,1}=[
(\alpha_{1}\alpha_{6}\alpha_{10})(\alpha_{2}\alpha_{13}\alpha_{19})
(\alpha_{3}\alpha_{18}\alpha_{14})(\alpha_{4}\alpha_{20}\alpha_{7})
(\alpha_{8}\alpha_{11}\alpha_{12})(\alpha_{9}\alpha_{22}\alpha_{23}),
$$
$$
(\alpha_{2}\alpha_{13})(\alpha_{3}\alpha_{22})(\alpha_{4}\alpha_{9})
(\alpha_{7}\alpha_{18})(\alpha_{8}\alpha_{14})(\alpha_{11}\alpha_{20})
(\alpha_{12}\alpha_{23})(\alpha_{16}\alpha_{19})]
$$
with orbits
$
\{\alpha_{1},\alpha_{6},\alpha_{10},\},
\{\alpha_{2},\alpha_{13},\alpha_{16},\alpha_{19},\},
\{\alpha_{3},\alpha_{18},\alpha_{22},\alpha_{11},\alpha_{14},\alpha_{7},
\alpha_{23},\alpha_{20},\alpha_{12},\alpha_{8},
\newline
\alpha_{4},\alpha_{9}\}$;
$$
H_{17,2}=[
(\alpha_{1}\alpha_{22}\alpha_{20})(\alpha_{3}\alpha_{11}\alpha_{15})
(\alpha_{4}\alpha_{9}\alpha_{19})(\alpha_{7}\alpha_{21}\alpha_{12})
(\alpha_{8}\alpha_{14}\alpha_{13})(\alpha_{18}\alpha_{24}\alpha_{23}),
$$
$$
(\alpha_{2}\alpha_{13})(\alpha_{3}\alpha_{22})
(\alpha_{4}\alpha_{9})(\alpha_{7}\alpha_{18})
(\alpha_{8}\alpha_{14})(\alpha_{11}\alpha_{20})
(\alpha_{12}\alpha_{23})(\alpha_{16}\alpha_{19})]
$$
with orbits
$
\{\alpha_{1},\alpha_{22},\alpha_{15},\alpha_{20},\alpha_{3},
\alpha_{11}\},\{\alpha_{2},\alpha_{13},\alpha_{8},\alpha_{14}\},
\{\alpha_{4},\alpha_{9},\alpha_{16},\alpha_{19}\},
\{\alpha_{7},\alpha_{21},\alpha_{18},\alpha_{12},
\newline
\alpha_{24},\alpha_{23}\}$;
$$
H_{17,3}=[
(\alpha_{1}\alpha_{5}\alpha_{17})(\alpha_{3}\alpha_{7}\alpha_{18})
(\alpha_{4}\alpha_{9}\alpha_{13})
(\alpha_{6}\alpha_{24}\alpha_{21})(\alpha_{8}\alpha_{19}\alpha_{14})
(\alpha_{11}\alpha_{12}\alpha_{23}),
$$
$$
(\alpha_{2}\alpha_{13})(\alpha_{3}\alpha_{22})
(\alpha_{4}\alpha_{9})(\alpha_{7}\alpha_{18})
(\alpha_{8}\alpha_{14})(\alpha_{11}\alpha_{20})
(\alpha_{12}\alpha_{23})(\alpha_{16}\alpha_{19})]
$$
with orbits
$
\{\alpha_{1},\alpha_{5},\alpha_{17}\},
\{\alpha_{2},\alpha_{13},\alpha_{4},\alpha_{9}\},
\{\alpha_{3},\alpha_{7},\alpha_{22},\alpha_{18}\},
\{\alpha_{6},\alpha_{24},\alpha_{21}\},
\{\alpha_{8},\alpha_{19},\alpha_{14},\alpha_{16}\},\newline
\{\alpha_{11},\alpha_{12},\alpha_{20},\alpha_{23}\}$.

\medskip

{\bf n=16,} $H\cong D_{10}$ ($|H|=10$, $i=1$):
$\rk N_H=16$ and $(N_H)^\ast/N_H\cong (\bz/5\bz)^4$.
$$
H_{16,1}=[
(\alpha_{3}\alpha_{13})(\alpha_{5}\alpha_{21})
(\alpha_{7}\alpha_{18})(\alpha_{8}\alpha_{15})
(\alpha_{10}\alpha_{14})(\alpha_{11}\alpha_{16})
(\alpha_{12}\alpha_{24})(\alpha_{17}\alpha_{23}),
$$
$$
(\alpha_{2}\alpha_{17}\alpha_{15}\alpha_{8}\alpha_{23})
(\alpha_{3}\alpha_{14}\alpha_{10}\alpha_{13}\alpha_{9})
(\alpha_{4}\alpha_{24}\alpha_{18}\alpha_{7}\alpha_{12})
(\alpha_{5}\alpha_{11}\alpha_{20}\alpha_{16}\alpha_{21})]
$$
with orbits
$
\{\alpha_{2},\alpha_{17},\alpha_{23},\alpha_{15},\alpha_{8}\},
\{\alpha_{3},\alpha_{13},\alpha_{14},\alpha_{9},\alpha_{10}\},
\{\alpha_{4},\alpha_{24},\alpha_{12},\alpha_{18},\alpha_{7}\},
\{\alpha_{5},\alpha_{21},\alpha_{11},
\newline
\alpha_{16},\alpha_{20}\}$.

\medskip

{\bf n=15,} $H\cong C_3^2$ ($|H|=9$, $i=2$):
$$
H_{15,1}=[
(\alpha_{1}\alpha_{3}\alpha_{21})(\alpha_{5}\alpha_{10}\alpha_{7})
(\alpha_{6}\alpha_{24}\alpha_{14})
(\alpha_{11}\alpha_{12}\alpha_{13})(\alpha_{15}\alpha_{19}\alpha_{17})
(\alpha_{16}\alpha_{18}\alpha_{22}),
$$
$$
(\alpha_{1}\alpha_{6}\alpha_{22})(\alpha_{3}\alpha_{24}\alpha_{16})
(\alpha_{4}\alpha_{20}\alpha_{9})(\alpha_{5}\alpha_{10}\alpha_{7})
(\alpha_{11}\alpha_{13}\alpha_{12})(\alpha_{14}\alpha_{18}\alpha_{21})]
$$
has $Clos(H_{15,1})=H_{30,1}$.

\medskip

{\bf n=14,} $H\cong C_8$ ($|H|=8$, $i=1$):
$$
H_{14,1}=[
(\alpha_{1}\alpha_{12}\alpha_{16}\alpha_{23}
\alpha_{15}\alpha_{11}\alpha_{19}\alpha_{20})
(\alpha_{2}\alpha_{13})
(\alpha_{3}\alpha_{8}\alpha_{22}\alpha_{5}
\alpha_{18}\alpha_{14}\alpha_{7}\alpha_{17})
(\alpha_{4}\alpha_{6}\alpha_{9}\alpha_{10})]
$$
has $Clos(H_{14,1})=H_{26,1}$ above.

\medskip

{\bf n=13,} $H\cong Q_8$ ($|H|=8$, $i=4$):
$$
H_{13,1}=[
(\alpha_{1}\alpha_{2}\alpha_{24}\alpha_{14})
(\alpha_{3}\alpha_{6}\alpha_{17}\alpha_{22})
(\alpha_{4}\alpha_{9})
(\alpha_{5}\alpha_{12}\alpha_{23}\alpha_{10})
(\alpha_{8}\alpha_{13})
(\alpha_{16}\alpha_{18}\alpha_{19}\alpha_{20}),
$$
$$
(\alpha_{1}\alpha_{16}\alpha_{24}\alpha_{19})
(\alpha_{2}\alpha_{20}\alpha_{14}\alpha_{18})
(\alpha_{3}\alpha_{10}\alpha_{17}\alpha_{12})
(\alpha_{5}\alpha_{6}\alpha_{23}\alpha_{22})
(\alpha_{8}\alpha_{13})
(\alpha_{15}\alpha_{21})]
$$
with $Clos(H_{13,1})=H_{40,1}$ above.

\medskip

{\bf n=12,} $H\cong Q_8$ ($|H|=8$, $i=4$):
$\rk N_H=17$, $(N_H)^\ast/N_H \cong
(\bz/8\bz)^2\times (\bz/2\bz)^3$, and
$K((q_{N_H})_2)\cong q_\theta^{(2)}(2)\oplus q^\prime$.
$$
H_{12,1}=[
(\alpha_{1}\alpha_{9}\alpha_{23}\alpha_{7})
(\alpha_{4}\alpha_{19}\alpha_{13}\alpha_{10})
(\alpha_{5}\alpha_{12})
(\alpha_{6}\alpha_{21}\alpha_{20}\alpha_{15})
(\alpha_{8}\alpha_{14}\alpha_{11}\alpha_{17})
(\alpha_{6}\alpha_{24}),
$$
$$
(\alpha_{1}\alpha_{11}\alpha_{23}\alpha_{8})
(\alpha_{4}\alpha_{6}\alpha_{13}\alpha_{20})
(\alpha_{5}\alpha_{16})
(\alpha_{7}\alpha_{17}\alpha_{9}\alpha_{14})
(\alpha_{10}\alpha_{21}\alpha_{19}\alpha_{15})
(\alpha_{12}\alpha_{24})]
$$
with orbits
$
\{\alpha_{1},\alpha_{9},\alpha_{11},\alpha_{23},
\alpha_{14},\alpha_{7},\alpha_{17},\alpha_{8}\},
\{\alpha_{4},\alpha_{19},\alpha_{6},\alpha_{13},
\alpha_{15},\alpha_{10},\alpha_{21},\alpha_{20}\},
\{\alpha_{5},\alpha_{12},\alpha_{16},
\newline
\alpha_{24}\}$.

\medskip

{\bf n=11,} $H\cong C_2\times C_4$ ($|H|=8$, $i=2$):
$$
H_{11,1}=[
(\alpha_{2}\alpha_{8}\alpha_{3}\alpha_{12})
(\alpha_{4}\alpha_{11}\alpha_{18}\alpha_{16})
(\alpha_{7}\alpha_{19}\alpha_{9}\alpha_{20})
(\alpha_{10}\alpha_{15})
(\alpha_{13}\alpha_{14}\alpha_{22}\alpha_{23})
(\alpha_{17}\alpha_{24}),
$$
$$
(\alpha_{2}\alpha_{13})(\alpha_{3}\alpha_{22})
(\alpha_{4}\alpha_{9})(\alpha_{7}\alpha_{18})
(\alpha_{8}\alpha_{14})(\alpha_{11}\alpha_{20})
(\alpha_{12}\alpha_{23})(\alpha_{16}\alpha_{19})]
$$
has $Clos(H_{11,1})=H_{22,1}$ above;
$$
H_{11,2}=[
(\alpha_{2}\alpha_{3}\alpha_{7}\alpha_{4})
(\alpha_{6}\alpha_{17}\alpha_{21}\alpha_{15})
(\alpha_{8}\alpha_{11}\alpha_{14}\alpha_{20})
(\alpha_{9}\alpha_{13}\alpha_{22}\alpha_{18})
(\alpha_{10}\alpha_{24})(\alpha_{16}\alpha_{19}),
$$
$$
(\alpha_{2}\alpha_{13})(\alpha_{3}\alpha_{22})
(\alpha_{4}\alpha_{9})(\alpha_{7}\alpha_{18})
(\alpha_{8}\alpha_{14})(\alpha_{11}\alpha_{20})
(\alpha_{12}\alpha_{23})(\alpha_{16}\alpha_{19})]
$$
has $Clos(H_{11,2})=H_{22,2}$ above;
$$
H_{11,3}=[
(\alpha_{1}\alpha_{10}\alpha_{5}\alpha_{24})
(\alpha_{2}\alpha_{3}\alpha_{13}\alpha_{22})
(\alpha_{4}\alpha_{18}\alpha_{9}\alpha_{7})
(\alpha_{6}\alpha_{15}\alpha_{21}\alpha_{17})
(\alpha_{8}\alpha_{14})(\alpha_{16}\alpha_{19}),
$$
$$
(\alpha_{2}\alpha_{13})(\alpha_{3}\alpha_{22})
(\alpha_{4}\alpha_{9})(\alpha_{7}\alpha_{18})
(\alpha_{8}\alpha_{14})(\alpha_{11}\alpha_{20})
(\alpha_{12}\alpha_{23})(\alpha_{16}\alpha_{19})]
$$
has $Clos(H_{11,3})=H_{22,3}$ above.
\footnote{These calculations show that for a symplectic group
$G=C_2\times C_4$ on a
K\"ahlerian K3 surface, the group
$S_{(G)}^\ast/S_{(G)}=S_{(2,4)}^\ast /S_{(2,4)}\cong
(\bz/4\bz)^4\times (\bz/2\bz)^2$.
We must correct our calculation of this group in
\cite[Prop. 10.1]{Nik0}.}

\medskip

{\bf n=10,} $H\cong D_8$ ($|H|=8$, $i=3$):
$\rk N_H=15$ and $(N_H)^\ast/N_H\cong
(\bz/4\bz)^5$.
$$
H_{10,1}\cong [
(\alpha_{2}\alpha_{16})(\alpha_{3}\alpha_{12})
(\alpha_{4}\alpha_{13})(\alpha_{7}\alpha_{11})
(\alpha_{8}\alpha_{9})(\alpha_{10}\alpha_{17})
(\alpha_{14}\alpha_{19})(\alpha_{15}\alpha_{21}),
$$
$$
(\alpha_{2}\alpha_{13})(\alpha_{3}\alpha_{22})
(\alpha_{4}\alpha_{9})(\alpha_{7}\alpha_{18})
(\alpha_{8}\alpha_{14})(\alpha_{11}\alpha_{20})
(\alpha_{12}\alpha_{23})(\alpha_{16}\alpha_{19})]
$$
with orbits
$
\{\alpha_{2},\alpha_{16},\alpha_{13},\alpha_{8},
\alpha_{19},\alpha_{9},\alpha_{4},\alpha_{14}\},
\{\alpha_{3},\alpha_{12},\alpha_{22},\alpha_{23}\},
\{\alpha_{7},\alpha_{11},\alpha_{18},\alpha_{20}\}, \newline
\{\alpha_{10},\alpha_{17}\},
\{\alpha_{15},\alpha_{21}\}$;
$$
H_{10,2}=[
(\alpha_{2}\alpha_{13})(\alpha_{3}\alpha_{24})
(\alpha_{4}\alpha_{8})(\alpha_{5}\alpha_{11})
(\alpha_{10}\alpha_{15})(\alpha_{17}\alpha_{22})
(\alpha_{18}\alpha_{23})(\alpha_{20}\alpha_{21}),
$$
$$
(\alpha_{2}\alpha_{13})(\alpha_{3}\alpha_{22})
(\alpha_{4}\alpha_{9})(\alpha_{7}\alpha_{18})
(\alpha_{8}\alpha_{14})(\alpha_{11}\alpha_{20})
(\alpha_{12}\alpha_{23})(\alpha_{16}\alpha_{19})]
$$
with orbits
$
\{\alpha_{2},\alpha_{13}\},
\{\alpha_{3},\alpha_{24},\alpha_{22},\alpha_{17}\},
\{\alpha_{4},\alpha_{8},\alpha_{9},\alpha_{14}\},
\{\alpha_{5},\alpha_{11},\alpha_{21},\alpha_{20}\},
\{\alpha_{7},\alpha_{18}, \newline \alpha_{12},\alpha_{23}\},
\{\alpha_{10},\alpha_{15}\},\{\alpha_{16},\alpha_{19}\}$.

\medskip

{\bf n=9,} $H\cong C_2^3$ ($|H|=8$, $i=5$):
$\rk N_H=14$, $(N_H)^\ast/N_H \cong
(\bz/4\bz)^2\times (\bz/2\bz)^6$ and
$\det(K(q_{N_H})_2)\equiv \pm 2^{10}\mod (\bz_2^\ast)^2$.
$$
H_{9,1}=[
(\alpha_{3}\alpha_{8})(\alpha_{4}\alpha_{9})
(\alpha_{5}\alpha_{15})(\alpha_{6}\alpha_{24})
(\alpha_{7}\alpha_{19})(\alpha_{11}\alpha_{20})
(\alpha_{14}\alpha_{22})(\alpha_{16}\alpha_{18}),\
$$
$$
(\alpha_{2}\alpha_{12})(\alpha_{3}\alpha_{8})
(\alpha_{4}\alpha_{20})(\alpha_{7}\alpha_{16})
(\alpha_{9}\alpha_{11})(\alpha_{13}\alpha_{23})
(\alpha_{14}\alpha_{22})(\alpha_{18}\alpha_{19}),\
$$
$$
(\alpha_{2}\alpha_{13})(\alpha_{3}\alpha_{22})
(\alpha_{4}\alpha_{9})(\alpha_{7}\alpha_{18})
(\alpha_{8}\alpha_{14})(\alpha_{11}\alpha_{20})
(\alpha_{12}\alpha_{23})(\alpha_{16}\alpha_{19})]
$$
with orbits
$
\{\alpha_{2},\alpha_{12},\alpha_{13},\alpha_{23}\},
\{\alpha_{3},\alpha_{8},\alpha_{22},\alpha_{14}\},
\{\alpha_{4},\alpha_{9},\alpha_{20},\alpha_{11}\},
\{\alpha_{5},\alpha_{15}\},
\{\alpha_{6},\alpha_{24}\}, \newline
\{\alpha_{7},\alpha_{19},\alpha_{16},\alpha_{18}\}$;
$$
H_{9,2}=[
(\alpha_{2}\alpha_{3})(\alpha_{4}\alpha_{18})(\alpha_{7}\alpha_{9})
(\alpha_{8}\alpha_{12})(\alpha_{11}\alpha_{16})(\alpha_{13}\alpha_{22})
(\alpha_{14}\alpha_{23})(\alpha_{19}\alpha_{20}),
$$
$$
(\alpha_{2}\alpha_{12})(\alpha_{3}\alpha_{8})(\alpha_{4}\alpha_{20})
(\alpha_{7}\alpha_{16})(\alpha_{9}\alpha_{11})(\alpha_{13}\alpha_{23})
(\alpha_{14}\alpha_{22})(\alpha_{18}\alpha_{19}),
$$
$$
(\alpha_{2}\alpha_{13})(\alpha_{3}\alpha_{22})(\alpha_{4}\alpha_{9})
(\alpha_{7}\alpha_{18})(\alpha_{8}\alpha_{14})(\alpha_{11}\alpha_{20})
(\alpha_{12}\alpha_{23})(\alpha_{16}\alpha_{19})]
$$
with orbits
$
\{\alpha_{2},\alpha_{3},\alpha_{12},\alpha_{13},
\alpha_{8},\alpha_{22},\alpha_{23},\alpha_{14}\},
\{\alpha_{4},\alpha_{18},\alpha_{20},\alpha_{9},
\alpha_{19},\alpha_{7},\alpha_{11}\alpha_{16}\}$;
$$
H_{9,3}=[
(\alpha_{2}\alpha_{4})(\alpha_{3}\alpha_{7})(\alpha_{6}\alpha_{21})
(\alpha_{9}\alpha_{13})(\alpha_{10}\alpha_{24})(\alpha_{11}\alpha_{20})
(\alpha_{12}\alpha_{23})(\alpha_{18}\alpha_{22}),
$$
$$
(\alpha_{2}\alpha_{22})(\alpha_{3}\alpha_{13})(\alpha_{4}\alpha_{18})
(\alpha_{7}\alpha_{9})(\alpha_{10}\alpha_{24})(\alpha_{11}\alpha_{20})
(\alpha_{15}\alpha_{17})(\alpha_{16}\alpha_{19}),
$$
$$
(\alpha_{2}\alpha_{13})(\alpha_{3}\alpha_{22})(\alpha_{4}\alpha_{9})
(\alpha_{7}\alpha_{18})(\alpha_{8}\alpha_{14})(\alpha_{11}\alpha_{20})
(\alpha_{12}\alpha_{23})(\alpha_{16}\alpha_{19})]
$$
with orbits
$
\{\alpha_{2},\alpha_{4},\alpha_{22},\alpha_{13},
\alpha_{18},\alpha_{9},\alpha_{3},\alpha_{7}\},
\{\alpha_{6},\alpha_{21}\},\{\alpha_{8},\alpha_{14}\},
\{\alpha_{10},\alpha_{24}\},\{\alpha_{11},\alpha_{20}\},
\{\alpha_{12},
\newline
\alpha_{23}\},
\{\alpha_{15},\alpha_{17}\}, \{\alpha_{16},\alpha_{19}\}$;
$$
H_{9,4}=[
(\alpha_{1}\alpha_{5})(\alpha_{6}\alpha_{21})(\alpha_{8}\alpha_{14})
(\alpha_{10}\alpha_{24})(\alpha_{11}\alpha_{20})(\alpha_{12}\alpha_{23})
(\alpha_{15}\alpha_{17})(\alpha_{16}\alpha_{19}),
$$
$$
(\alpha_{2}\alpha_{22})(\alpha_{3}\alpha_{13})(\alpha_{4}\alpha_{18})
(\alpha_{7}\alpha_{9})(\alpha_{10}\alpha_{24})(\alpha_{11}\alpha_{20})
(\alpha_{15}\alpha_{17})(\alpha_{16}\alpha_{19}),
$$
$$
(\alpha_{2}\alpha_{13})(\alpha_{3}\alpha_{22})(\alpha_{4}\alpha_{9})
(\alpha_{7}\alpha_{18})(\alpha_{8}\alpha_{14})(\alpha_{11}\alpha_{20})
(\alpha_{12}\alpha_{23})(\alpha_{16}\alpha_{19})]
$$
with orbits
$
\{\alpha_{1},\alpha_{5}\},
\{\alpha_{2},\alpha_{22},\alpha_{13},\alpha_{3}\},
\{\alpha_{4},\alpha_{18},\alpha_{9},\alpha_{7}\},
\{\alpha_{6},\alpha_{21}\},\{\alpha_{8},\alpha_{14}\},
\{\alpha_{10},\alpha_{24}\},
\{\alpha_{11},
\newline
\alpha_{20}\},\{\alpha_{12},\alpha_{23}\},
\{\alpha_{15},\alpha_{17}\},\{\alpha_{16},\alpha_{19}\}$.

\medskip

{\bf n=8,} $H\cong C_7$ ($|H|=7$, $i=1$):
$$
H_{8,1}=[
(\alpha_{1}\alpha_{5}\alpha_{10}\alpha_{19}\alpha_{22}\alpha_{6}\alpha_{7})
(\alpha_{3}\alpha_{16}\alpha_{20}\alpha_{24}\alpha_{9}\alpha_{4}\alpha_{17})
(\alpha_{11}\alpha_{15}\alpha_{18}\alpha_{14}\alpha_{13}\alpha_{21}\alpha_{12})]
$$
has $Clos(H_{8,1})=H_{33,1}$ above.

\medskip

{\bf n=7,} $H\cong C_6$ ($|H|=6$, $i=2$):
$$
H_{7,1}=[
(\alpha_{2}\alpha_{17}\alpha_{13}\alpha_{3}\alpha_{10}\alpha_{22})
(\alpha_{4}\alpha_{21}\alpha_{9})
(\alpha_{5}\alpha_{7}\alpha_{18})
(\alpha_{8}\alpha_{11}\alpha_{12}\alpha_{23}\alpha_{20}\alpha_{14})
(\alpha_{15}\alpha_{24})(\alpha_{16}\alpha_{19})]
$$
has $Clos(H_{7,1})=H_{18,1}$ above.

\medskip

{\bf n=6,} $H\cong D_6$ ($|H|=6$, $i=1$):
$\rk N_H=14$ and $(N_H)^\ast/N_H\cong (\bz/6\bz)^2\times (\bz/3\bz)^5$.
$$
H_{6,1}=[
(\alpha_{2}\alpha_{14})(\alpha_{3}\alpha_{17})(\alpha_{4}\alpha_{8})
(\alpha_{7}\alpha_{15})(\alpha_{9}\alpha_{13})(\alpha_{10}\alpha_{12})
(\alpha_{11}\alpha_{21})(\alpha_{16}\alpha_{19}),
$$
$$
(\alpha_{2}\alpha_{13})(\alpha_{3}\alpha_{22})(\alpha_{4}\alpha_{9})
(\alpha_{7}\alpha_{18})(\alpha_{8}\alpha_{14})(\alpha_{11}\alpha_{20})
(\alpha_{12}\alpha_{23})(\alpha_{16}\alpha_{19})]
$$
with orbits
$
\{\alpha_{2},\alpha_{14},\alpha_{9},\alpha_{4},\alpha_{13},\alpha_{8}\},
\{\alpha_{3},\alpha_{17},\alpha_{22}\},\{\alpha_{7},\alpha_{15},\alpha_{18}\},
\{\alpha_{10},\alpha_{12},\alpha_{23}\},
\{\alpha_{11},\alpha_{21},
\newline
\alpha_{20}\},\{\alpha_{16},\alpha_{19}\}$.

\medskip

{\bf n=5,} $H\cong C_5$ ($|H|=5$, $i=1$):
$$
H_{5,1}=[
(\alpha_{2}\alpha_{8}\alpha_{17}\alpha_{23}\alpha_{15})
(\alpha_{3}\alpha_{13}\alpha_{14}\alpha_{9}\alpha_{10})
(\alpha_{4}\alpha_{7}\alpha_{24}\alpha_{12}\alpha_{18})
(\alpha_{5}\alpha_{16}\alpha_{11}\alpha_{21}\alpha_{20})]
$$
has $Clos(H_{5,1})=H_{16,1}$ above.

\medskip

{\bf n=4,} $H\cong C_4$ ($|H|=4$, $i=1$):
$\rk N_H=14$ and $(N_H)^\ast/N_H \cong
(\bz/4\bz)^4\times (\bz/2\bz)^2$.
$$
H_{4,1}=[
(\alpha_{2}\alpha_{22}\alpha_{23}\alpha_{6})
(\alpha_{3}\alpha_{5}\alpha_{9}\alpha_{16})
(\alpha_{7}\alpha_{10}\alpha_{20}\alpha_{13})
(\alpha_{8}\alpha_{15})
(\alpha_{11}\alpha_{19}\alpha_{24}\alpha_{14})
(\alpha_{12}\alpha_{21})]
$$
with orbits
$
\{\alpha_{2},\alpha_{22},\alpha_{23},\alpha_{6}\},
\{\alpha_{3},\alpha_{5},\alpha_{9},\alpha_{16}\},
\{\alpha_{7},\alpha_{10},\alpha_{20},\alpha_{13}\},
\{\alpha_{8},\alpha_{15}\},
\{\alpha_{11},\alpha_{19},\alpha_{24},
\newline
\alpha_{14}\},
\{\alpha_{12},\alpha_{21}\}$.

\medskip

{\bf n=3,} $H\cong C_2^2$ ($|H|=4$, $i=2$):
$\rk N_H=12$ and $(N_H)^\ast/N_H\cong (\bz/4\bz)^2\times (\bz/2\bz)^6$.
$$
H_{3,1}=[
(\alpha_{2}\alpha_{12})(\alpha_{3}\alpha_{8})(\alpha_{4}\alpha_{20})
(\alpha_{7}\alpha_{16})(\alpha_{9}\alpha_{11})(\alpha_{13}\alpha_{23})
(\alpha_{14}\alpha_{22})(\alpha_{18}\alpha_{19}),
$$
$$
(\alpha_{2}\alpha_{13})(\alpha_{3}\alpha_{22})(\alpha_{4}\alpha_{9})
(\alpha_{7}\alpha_{18})(\alpha_{8}\alpha_{14})(\alpha_{11}\alpha_{20})
(\alpha_{12}\alpha_{23})(\alpha_{16}\alpha_{19})]
$$
with orbits
$
\{\alpha_{2},\alpha_{12},\alpha_{13},\alpha_{23}\},
\{\alpha_{3},\alpha_{8},\alpha_{22},\alpha_{14}\},
\{\alpha_{4},\alpha_{20},\alpha_{9},\alpha_{11}\},
\{\alpha_{7},\alpha_{16},\alpha_{18},\alpha_{19}\}$;
$$
H_{3,2}=[
(\alpha_{2}\alpha_{22})(\alpha_{3}\alpha_{13})(\alpha_{4}\alpha_{18})
(\alpha_{7}\alpha_{9})(\alpha_{10}\alpha_{24})(\alpha_{11}\alpha_{20})
(\alpha_{15}\alpha_{17})(\alpha_{16}\alpha_{19}),
$$
$$
(\alpha_{2}\alpha_{13})(\alpha_{3}\alpha_{22})(\alpha_{4}\alpha_{9})
(\alpha_{7}\alpha_{18})(\alpha_{8}\alpha_{14})(\alpha_{11}\alpha_{20})
(\alpha_{12}\alpha_{23})(\alpha_{16}\alpha_{19})]
$$
with orbits
$
\{\alpha_{2},\alpha_{22},\alpha_{13},\alpha_{3}\},
\{\alpha_{4},\alpha_{18},\alpha_{9},\alpha_{7}\},
\{\alpha_{8},\alpha_{14}\},\{\alpha_{10},\alpha_{24}\},
\{\alpha_{11},\alpha_{20}\},\{\alpha_{12},\alpha_{23}\},
\newline
\{\alpha_{15},\alpha_{17}\},\{\alpha_{16},\alpha_{19}\}$;
$$
H_{3,3}=[
(\alpha_{1}\alpha_{24})(\alpha_{3}\alpha_{22})(\alpha_{4}\alpha_{9})
(\alpha_{5}\alpha_{10})(\alpha_{6}\alpha_{17})(\alpha_{12}\alpha_{23})
(\alpha_{15}\alpha_{21})(\alpha_{16}\alpha_{19}),
$$
$$
(\alpha_{2}\alpha_{13})(\alpha_{3}\alpha_{22})
(\alpha_{4}\alpha_{9})(\alpha_{7}\alpha_{18})
(\alpha_{8}\alpha_{14})(\alpha_{11}\alpha_{20})
(\alpha_{12}\alpha_{23})(\alpha_{16}\alpha_{19})]
$$
with orbits
$
\{\alpha_{1},\alpha_{24}\},\{\alpha_{2},\alpha_{13}\},
\{\alpha_{3},\alpha_{22}\},\{\alpha_{4},\alpha_{9}\},
\{\alpha_{5},\alpha_{10}\},\{\alpha_{6},\alpha_{17}\},
\{\alpha_{7},\alpha_{18}\},\{\alpha_{8},\alpha_{14}\},
\newline
\{\alpha_{11},\alpha_{20}\},
\{\alpha_{12},\alpha_{23}\},
\{\alpha_{15},\alpha_{21}\},\{\alpha_{16},\alpha_{19}\}$.

\medskip

{\bf n=2,} $H\cong C_3$ ($|H|=3$, $i=1$):
$\rk N_H=12$ and $(N_H)^\ast/N_H\cong (\bz/3\bz)^6$.
$$
H_{2,1}=[
(\alpha_{1}\alpha_{13}\alpha_{7})(\alpha_{2}\alpha_{19}\alpha_{22})
(\alpha_{5}\alpha_{8}\alpha_{9})(\alpha_{6}\alpha_{10}\alpha_{11})
(\alpha_{12}\alpha_{23}\alpha_{20})(\alpha_{15}\alpha_{18}\alpha_{16})]
$$
with orbits
$
\{\alpha_{1},\alpha_{13},\alpha_{7}\},\{\alpha_{2},\alpha_{19},\alpha_{22}\},
\{\alpha_{5},\alpha_{8},\alpha_{9}\},\{(\alpha_{6},\alpha_{10},\alpha_{11}\},
\{\alpha_{12},\alpha_{23},\alpha_{20}\},\{\alpha_{15},\alpha_{18},
\newline
\alpha_{16}\}$.

\medskip

{\bf n=1,} $H\cong C_2$ ($|H|=2$, $i=1$):
$\rk N_H=8$ and $(N_H)^\ast/N_H\cong (\bz/2\bz)^8$.
$$
H_{1,1}=[
(\alpha_{2}\alpha_{23})(\alpha_{3}\alpha_{9})(\alpha_{5}\alpha_{16})
(\alpha_{6}\alpha_{22})(\alpha_{7}\alpha_{20})(\alpha_{10}\alpha_{13})
(\alpha_{11}\alpha_{24})(\alpha_{14}\alpha_{19})]
$$
with orbits
$\{\alpha_{2},\alpha_{23}\},\{\alpha_{3},\alpha_{9}\},\{\alpha_{5},\alpha_{16}\},
\{\alpha_{6},\alpha_{22}\},\{\alpha_{7},\alpha_{20}\},\{\alpha_{10},\alpha_{13}\},
\{\alpha_{11},\alpha_{24}\},\{\alpha_{14},\alpha_{19}\}$.


\medskip

Let $X$ be marked by a primitive sublattice
$S\subset N=N_{23}=N(24A_1)$. Then $S$
must satisfy Theorem \ref{th:primembb3} and
$\Gamma(P(S))\subset \Gamma(P(N_{23}))=24\aaa_1$.
Any such $S$ gives marking of some $X$ and $P(X)\cap S=P(S)$.

If $N_H\subset S$ where $H$ has the type
$n=81$, $80$, $79$, $78$, $77$, $76$, $74$, $70$, $63$, $62$ or $54$
($\rk N_H=19$), then $\Aut(X,S)_0=H$.
Otherwise, if only $N_H\subset S$ where $H$ has the type
$n=75$, $65$, $61$, $56$, $55$, $51$, $48$, $46$, $33$, $32$ or $26$
($\rk N_H=18$), then$\Aut(X,S)_0=H$.
Otherwise, if only $N_H\subset S$ where $H$ has the type
$n=49$, $40$, $39$,  $34$ or $12$
($\rk N_H=17$), then $\Aut(X,S)_0=H$.
Otherwise, if only $N_H\subset S$ where $H$ has the type
$n=30$, $22$, $18$, $17$ or $16$
($\rk N_H=16$), then $\Aut(X,S)_0=H$.
Otherwise, if only $N_H\subset S$ where $H$ has the type
$n=21$ or $10$
($\rk N_H=15$), then $\Aut(X,S)_0=H$.
Otherwise, if only $N_H\subset S$ where $H$ has the type
$n=9$, $6$ or $4$
($\rk N_H=14$), then $\Aut(X,S)_0=H$.
Otherwise, if only $N_H\subset S$ where $H$ has the type
$n=3$ or $2$
($\rk N_H=12$), then $\Aut(X,S)_0=H$.
Otherwise, if only $N_H\subset S$ where $H$ has the type $n=1$
($\rk N_H=8$), then $\Aut(X,S)_0=H\cong C_2$.
Otherwise, $\Aut(X,S)_0$ is trivial.


\medskip

Let us assume that
a K3 surface $X$ has $16$ non-intersecting
non-singular rational
curves. It is known \cite{Nik-1} that classes
$$
\alpha_1,\ \alpha_2,\ \dots, \alpha_{16}
$$
of these curves generate a primitive
sublattice $\Pi\subset S_X$
which was described in \cite{PS} and \cite{Nik-1}, and $X$ is
a Kummer surface. In particular, $\rk \Pi=16$ and
$\Pi^\ast/\Pi\cong (\bz/2\bz)^6$.

Let us consider marking $S\subset N_i$ of $X$ by $S$
which contains the primitive sublattice $\Pi\subset S$.
It is known \cite{Nik-1} that
$$
P(X)\cap S=\{\alpha_1,\dots,\alpha_{16}\}.
$$
Thus, $16\aaa_1$ should be a subgraph of the Dynkin graph of
$N_i$, and the corresponding sublattice $\Pi\subset N_i$ is
a primitive sublattice. By classification of Niemeier
lattices, this is possible only for $N(24A_1)$. Thus, $X$
can be marked by the Niemeier lattice $N_{23}=N(24A_1)$ only.
For this case, $\Aut(X,S)_0$ is a subgroup of
$$
Kum =\{\phi \in A(N(24A_1))\ |\ \phi(\Pi)=\Pi\}.
$$
In \cite{TW}, in particular, some of these subgroups
were considered.

\medskip

Finally, our considerations give the following result.

\begin{theorem} For each of Niemeier lattices $N_i$,
$i=1,\,2,\,3$, $5$---$9$, $11$---$23$,
there exists a K\"ahlerian K3 surface $X$ such that $X$ can be marked
by the Niemeier lattice $N_i$ only.
\label{onlymarkings}
\end{theorem}

We believe that the same result is valid for the remaining
Niemeier lattices
$N_4$, $N_{10}$. By Kond\=o's
trick from \cite{Kon} which we mentioned in Sec.
\ref{sec:KalerNiemmark},
any K\"ahlerian K3 surface which is marked by the
Leech lattice $N_{24}$ can be also marked by one of Niemeier lattices $N_i$,
$i=1$ --- $23$.
On the other hand, it is natural to mark K3 surfaces with empty set
$P(X)\cap S$ by the Leech lattice $N_{24}$.
{\it All Niemeier lattices are important for marking
of K\"ahlerian K3 surfaces.}

\section{Appendix: Programs}
\label{sec:programs}
Here we give programs using GP/PARI Calculator, Version 2.2.13.

\medskip

Program 0: niemeier$\backslash$general1.txt

\medskip

\noindent
$\backslash\backslash$for Niemeier lattice Niem given by \hfill

\noindent
$\backslash\backslash$the 24$\backslash$times 24 symmetric matrix r
in root basis \hfill

\noindent
$\backslash\backslash$represented by basic of the size 24 vectors
(0,0,...,0,1,0,...,0)$\tilde{ }$\hfill

\noindent
$\backslash\backslash$and its rational cording matrix \hfill

\noindent
$\backslash\backslash$cord of size (\ .\ $\backslash$times\ 24)\hfill

\noindent
$\backslash\backslash$and its sublattice SUBL given by\hfill

\noindent
$\backslash\backslash$ (24$\backslash$times\ .\ ) rational matrix
SUBL \hfill

\noindent
$\backslash\backslash$it calculates basis of its primitive
sublattice\hfill

\noindent
$\backslash\backslash$SUBLpr (SUBL$\backslash$otimes Q)$\backslash$cap
Niem as matrix SUBLpr\hfill

\noindent
$\backslash\backslash$and calculates invariants DSUBLpr of\hfill

\noindent
$\backslash\backslash$SUBLpr$\backslash$subset SUBLpr* \hfill

\noindent
$\backslash\backslash$and calculates the matrix rSUBLpr of SUBLpr in
this\hfill

\noindent
$\backslash\backslash$elementary divisors (Smith) basis SUBLpr\hfill

\noindent
a=matrix(24,24+matsize(cord)[1]);\hfill

\noindent
for(i=1,24,a[i,i]=1);for(i=1,matsize(cord)[1],a[,24+i]=cord[i,]~$\widetilde{ }$
\ );\hfill

\noindent
L=a;N=SUBL;\hfill

\noindent
$\backslash$r niemeier$\backslash$latt4.txt;\hfill

\noindent
SUBLpr1=Npr;R=r;B=SUBLpr1;\hfill

\noindent
$\backslash$r niemeier$\backslash$latt2.txt;\hfill

\noindent
SUBLpr=BB;DSUBLpr=D;rSUBLpr=G;

\vskip0.5cm

Program 1: niemeier$\backslash$latt1.txt

\medskip

\noindent
$\backslash\backslash$for a non-degenerate lattice\hfill

\noindent
$\backslash\backslash$L given by a symmetric integer matrix l\hfill

\noindent
$\backslash\backslash$in some generators\hfill

\noindent
$\backslash\backslash$calculates the elementary divisors
(Smith) basis of L\hfill

\noindent
$\backslash\backslash$as a matrix b and

\noindent
$\backslash\backslash$calculates the matrix
ll=b$\widetilde{\ }$*l*b\hfill

\noindent
$\backslash\backslash$of L in the bases b\hfill

\noindent
$\backslash\backslash$calculates invariants d
of L$\backslash$subset L*\hfill

\noindent
ww=matsnf(l,1);uu=ww[1];vv=ww[2];dd=ww[3];\hfill

\noindent
nn=matsize(l)[1];nnn=nn;for(i=1,nn,if(dd[i,i]==0,nnn=nnn-1));\hfill

\noindent
b=matrix(nn,nnn,X,Y,vv[X,Y+nn-nnn]);\hfill

\noindent
ll=b$\widetilde{\ }$*l*b;\hfill

\noindent
d=vector(nnn,X,dd[X+nn-nnn,X+nn-nnn]);\hfill

\noindent
kill(ww);kill(uu);kill(vv);kill(dd);kill(nn);kill(nnn);\hfill

\vskip0.5cm

Program 2: niemeier$\backslash$latt2.txt

\medskip

\noindent
$\backslash\backslash$for a non-degenerate lattice L\hfill

\noindent
$\backslash\backslash$given by an integer quadratic n x n matrix R\hfill

\noindent
$\backslash\backslash$and generators B (n x m) matrix with
rational coefficients\hfill

\noindent
$\backslash\backslash$calculates invariants D of
L$\backslash$subset L* of this lattice\hfill

\noindent
$\backslash\backslash$the elementary divisors (Smith) basis BB of this
lattice, (n x mm) matrix,\hfill

\noindent
$\backslash\backslash$and matrix G=BB$\widetilde{\ }$*R*BB of L
in this basis\hfill

\noindent
l=B$\widetilde{\ }$*R*B;\hfill

\noindent
$\backslash$r niemeier$\backslash$latt1.txt;\hfill

\noindent
BB=B*b;G=BB$\widetilde{\ }$*R*BB;D=d;\hfill

\vskip0.5cm

Program 3: niemeier$\backslash$latt3.txt

\medskip

\noindent
$\backslash\backslash$for a module M\hfill

\noindent
$\backslash\backslash$given by rational columns\hfill

\noindent
$\backslash\backslash$of matrix M, it finds its basis\hfill

\noindent
$\backslash\backslash$as a matrix MM\hfill

\noindent
$\backslash\backslash$and finds matrix VV such that MM=M*VV\hfill

\noindent
gg=gcd(M);M1=M/gg;

\noindent
ww=matsnf(M1,1);uu=ww[1];vv=ww[2];dd=ww[3];\hfill

\noindent
mm=matsize(dd)[1];nn=matsize(dd)[2];\hfill

\noindent
nnn=nn;for(i=1,nn,if(dd[,i]==0,nnn=nnn-1));\hfill

\noindent
VV=matrix(nn,nnn);\hfill

\noindent
nnnn=0;for(i=1,nn,if(dd[,i]==0,,nnnn=nnnn+1;VV[,nnnn]=vv[,i]));\hfill

\noindent
M2=M1*VV;MM=M2*gg;\hfill

\noindent
kill(gg);kill(M1);kill(ww);kill(uu);kill(vv);kill(dd);kill(mm);\hfill

\noindent
kill(nn);kill(nnn);kill(nnnn);kill(M2);\hfill

\vskip0.5cm

Program 4: niemeier$\backslash$latt4.txt

\medskip

\noindent
$\backslash\backslash$For a module L given by\hfill

\noindent
$\backslash\backslash$a rational m x n matrix L\hfill

\noindent
$\backslash\backslash$such that L contains all basic columns\hfill

\noindent
$\backslash\backslash$(0,..,0,1,0,..0)$\widetilde{\ }$\hfill

\noindent
$\backslash\backslash$and its submodule N given by rational\hfill

\noindent
$\backslash\backslash$matrix N\hfill

\noindent
$\backslash\backslash$it finds basis of the primitive\hfill

\noindent
$\backslash\backslash$submodule
$N_{pr}$=(N$\backslash$otimes Q)$\backslash$cap L\hfill

\noindent
$\backslash\backslash$as the matrix Npr\hfill

\noindent
ggg=gcd(N);N1=N/ggg;\hfill

\noindent
M=L;\hfill

\noindent
$\backslash$r niemeier$\backslash$latt3.txt;\hfill

\noindent
L1=MM;kill(VV);\hfill

\noindent
N2=L1\^{ }-1*N1;\hfill

\noindent
ww=matsnf(N2,1);uu=ww[1];vv=ww[2];dd=ww[3];\hfill

\noindent
N3=N2*vv;mm=matsize(dd)[1];nn=matsize(dd)[2];\hfill

\noindent
nnn=nn;for(i=1,nn,if(dd[,i]==0,nnn=nnn-1));\hfill

\noindent
N4=matrix(mm,nnn);\hfill

\noindent
nnnn=0;\hfill

\noindent
for(i=1,nn,if(dd[,i]==0,,nnnn=nnnn+1;ddd=gcd(dd[,i]);$\backslash$

\noindent
N4[,nnnn]=N3[,i]/ddd));\hfill

\noindent
Npr=L1*N4;\hfill

\noindent
kill(ggg);kill(N1);kill(M);kill(L1);kill(MM);\hfill

\noindent
kill(N2);kill(ww);kill(uu);kill(vv);kill(dd);\hfill

\noindent
kill(N3);kill(mm);kill(nn);kill(nnn);kill(nnnn);\hfill

\noindent
kill(ddd);kill(N4);\hfill


V.V. Nikulin \par Deptm. of Pure Mathem. The University of
Liverpool, Liverpool\par L69 3BX, UK; \vskip1pt Steklov
Mathematical Institute,\par ul. Gubkina 8, Moscow 117966, GSP-1,
Russia

vnikulin@liv.ac.uk \, \ \ \  vvnikulin@list.ru

Personal page: http://vnikulin.com

\end{document}